\documentclass[11pt]{article}
\usepackage{etex}
\usepackage{pgf,tikz}
\usetikzlibrary{arrows}
\usepackage[margin=1in]{geometry}
\usepackage{stackengine}

\usepackage[page]{appendix}

\usepackage{color}
\usepackage[T1]{fontenc}
\usepackage[normalem]{ulem}
\usepackage[english]{babel}
\usepackage{verbatim}
\usepackage{graphicx}
\usepackage{enumerate}
\usepackage{amsmath,amssymb,amsfonts,amsthm,amscd,mathrsfs}
\usepackage{array}
\usepackage{amsmath,amssymb,graphicx,stmaryrd,enumerate,bbm,alltt}
\usepackage{dsfont}
\usepackage{comment}

\usepackage[T1]{fontenc}
\usepackage{babel}
\usepackage[bookmarksopen, bookmarksnumbered]{hyperref}

\usepackage{url}
\usepackage{pgfplots}

\usepackage{caption}

\newtheorem{theorem}{Theorem}[section]

\newtheorem{lemma}[theorem]{Lemma}
\newtheorem{cla}[theorem]{Claim}
\newtheorem{prop}[theorem]{Proposition}
\newtheorem{cor}[theorem]{Corollary}

\theoremstyle{definition}
\newtheorem{defn}[theorem]{Definition}
\newtheorem{rem}[theorem]{Remark}

\input xy
\xyoption{all}

\DeclareMathOperator{\spc}{sp}
\DeclareMathOperator{\argmax}{argmax}
\DeclareMathOperator{\dist}{dist}
\DeclareMathOperator{\diam}{diam}
\DeclareMathOperator{\radius}{radi}
\DeclareMathOperator{\Area}{Area}

\newcommand{\Z}{\mathbb Z}
\newcommand{\R}{\mathbb R}

\newcommand{\boo}{\mathbf 0}

\newcommand{\bt}{\mathbf t}

\newcommand{\be}{\boldsymbol{\mathbf{e}}}
\newcommand{\blambda}{\boldsymbol{\lambda}}
\newcommand{\olambda}{\overline{\blambda}}

\newcommand{\cC}{\mathcal C}
\newcommand{\cL}{\mathcal L}
\newcommand{\cP}{\mathcal P}
\newcommand{\cR}{\mathcal R}
\newcommand{\cS}{\mathcal S}

\newcommand{\haS}{\hat{\mathcal{S}}}
\newcommand{\cW}{\mathcal W}
\newcommand{\cV}{\mathbf v}
\newcommand{\caT}{\mathcal T}

\newcommand{\oT}{\mathring {\mathcal T}}
\newcommand{\fH}{\mathfrak H}
\newcommand{\fP}{\mathfrak P}
\newcommand{\fT}{\mathfrak T}

\newcommand{\prob}{\mathbb{P}}
\newcommand{\event}{\mathcal{E}}
\newcommand{\froz}{\mathcal{O}}
\newcommand{\tL}{\tilde L}

\makeatletter
\renewcommand\tableofcontents{%
  \null\hfill\textbf{\Large\contentsname}\hfill\null\par
  \@mkboth{\MakeUppercase\contentsname}{\MakeUppercase\contentsname}%
  \@starttoc{toc}%
}

\g@addto@macro\normalsize{%
  \setlength\abovedisplayskip{5pt}
  \setlength\belowdisplayskip{5pt}
  \setlength\abovedisplayshortskip{3pt}
  \setlength\belowdisplayshortskip{3pt}
}

\makeatother

\numberwithin{equation}{section}

\begin{document}
\title{Anderson-Bernoulli localization on the 3D lattice and discrete unique continuation principle}

\author{Linjun Li
\thanks{Department of Mathematics, University of Pennsylvania, Philadelphia, PA, USA. e-mail: linjun@sas.upenn.edu} \and Lingfu Zhang
\thanks{Department of Mathematics, Princeton University, Princeton, NJ, USA. e-mail: lingfuz@math.princeton.edu}
}
\date{}

\maketitle
 
\begin{abstract}
We consider the Anderson model with Bernoulli potential on the 3D lattice $\Z^{3}$, and prove localization of eigenfunctions corresponding to eigenvalues near zero, the lower boundary of the spectrum.
We follow the framework by \cite{bourgain2005localization} and \cite{ding2020localization}, and our main contribution is a 3D discrete unique continuation, which says that any eigenfunction of the harmonic operator with bounded potential cannot be too small on a significant fractional portion of all the points.
Its proof relies on geometric arguments about the 3D lattice.
\end{abstract}

\tableofcontents

\section{Introduction}   \label{sec:intro}

\subsection{Main result and background}   \label{sec:mainres}
In the 3D Anderson-Bernoulli model on the
lattice, we consider the random Schr\"{o}dinger operator
$H:=-\Delta+\delta V$, acting on the space $\ell^2(\Z^3)$.
Here $\delta > 0$ is the \emph{disorder strength}, $\Delta$ is the discrete Laplacian:
\begin{equation}
\Delta u(a) =- 6 u(a)+ \sum_{b \in \Z^3, |a-b|=1}u(b)  ,\; \forall u \in \ell^2(\Z^3), a \in \Z^3,
\end{equation}
and $V:\Z^3 \rightarrow \{0,1\}$ is the Bernoulli random potential; i.e. for each $a \in \Z^3$, $V(a)=1$ with probability $\frac{1}{2}$ independently.
Here and throughout this paper, $|\cdot|$ denotes the Euclidean norm.

Our main result is as follows.
\begin{theorem}  \label{thm:main}
There exists $\lambda_{*}>0$, depending on $\delta$, such that almost surely the following holds.
For any function $u : \Z^3 \rightarrow \R$ and $\lambda \in [0,\lambda_{*}]$, if $H u=\lambda u$ and $\inf_{n \geq 0} \sup_{a \in \Z^3} (1+|a|)^{-n}|u(a)|< \infty $, we have $\inf_{t>0} \sup_{a \in \Z^3} \exp(t|a|) |u(a)|<\infty$. 
\end{theorem}
In literature, this phenomenon is sometimes called ``Anderson localization'' (near the edge of the spectrum).
It also implies that $H$ has pure point spectrum in $[0, \lambda_{*}]$ (see e.g. \cite[Section 7]{kirsch2007invitation}).
Note that this is related to but different from ``dynamical localization'' (see e.g. discussions in \cite[Section 7.1]{aizenman2015random}).

The Anderson models are widely used to describe spectral and transport properties of disordered media, such as moving quantum mechanical particles, or electrons in a metal with impurities.
The mathematical study of their localization phenomena can be traced back to the 1980s (see e.g. \cite{kunz1980spectre}), and since then there have been many results in models on both discrete and continuous spaces.
In most early works, some regularity conditions on the distribution of the random potential are needed.
In \cite{frohlich1983absence}, Fr{\"o}hlich and Spencer used a multi-scale analysis argument to show that if $\{V(a):a\in \Z^d\}$ are i.i.d. bounded random variables with bounded probability density, then the resolvent decays exponentially when $\delta$ is large enough or energy is sufficiently small. Then in \cite{frohlich1985constructive}, together with Martinelli and Scoppola, they proved Anderson localization under the same condition.
This result was strengthened later by \cite{carmona1987anderson}, where the same results were proved under the condition that the distribution of $\{V(a):a\in \Z^d\}$ are i.i.d., bounded, and H{\"o}lder continuous.

It remains an interesting problem to remove these regularity conditions.
As described at the beginning of \cite{damanik2002localization}, when using the Anderson models to study alloy type materials, it is natural to expect the random potential to take only finitely many values.
A particular case is where the random potential are i.i.d. Bernoulli variables.

For the particular case of $d=1$, 
in the above mentioned paper \cite{carmona1987anderson} the authors proved that for the discrete model on $\Z$, 
Anderson localization holds for the full spectrum when the i.i.d. random potential is non-degenerate and has some finite moment. This includes the Bernoulli case.
In \cite{bucaj2019localization} a new proof is given for the case where the random potential has bounded support. 
In \cite{damanik2002localization}, the continuous model on $\R$ was studied, and Anderson localization was proved for the full spectrum when the i.i.d. random potential is non-degenerate and has bounded support.

For higher dimensions, a breakthrough was then made by Bourgain and Kenig.
In \cite{bourgain2005localization}, they studied the continuous model $\R^d$, for $d \geq 2$, and proved Anderson-Bernoulli localization near the bottom of the spectrum. 
An important ingredient is the \emph{unique continuation principle} in $\R^{d}$, i.e. \cite[Lemma 3.10]{bourgain2005localization}. It roughly says that, if $u:\R^d \rightarrow \R$ satisfies $\Delta u= V u$ for some bounded $V$ on $\R^{d}$, and $u$ is also bounded, then $u$ can not be too small on any ball with positive radius. 
Using this unique continuation principle together with the Sperner lemma, they proved a Wegner estimate, which was used to prove the exponential decay of the resolvent.
In doing this, many aspects of the usual multi-scale analysis framework were adapted; and in particular, they introduced the idea of ``free sites''.
See \cite{bourgain2005anderson} for some more discussions.
Later, Germinet and Klein \cite{germinet2012comprehensive} incorporated the new ideas of \cite{bourgain2005localization} and proved localization (in a strong form) near the bottom of the spectrum in the continuous model, for any non-degenerate potential with bounded support.

The Anderson-Bernoulli localization on lattices in higher dimensions remained open. There were efforts toward this goal by relaxing the condition that $V$ only takes two values (see \cite{imbrie2017localization}). Recently, the work of Ding and Smart \cite{ding2020localization} proved Anderson-Bernoulli localization near the edge of the spectrum on the 2D lattice.
As discussed in \cite[Section 1]{bourgain2005localization}, the approach there cannot be directly applied to the lattice model, due to the lack of a discrete version of the unique continuation principle.
A crucial difference between the lattice $\Z^d$ and $\R^d$ is that one could construct a function $u:\Z^d \rightarrow \R$, such that $\Delta u= V u$ holds for some bounded $V$, but $u$ is supported on a lower dimensional set (see Remark \ref{rem:sharp} below for an example on 3D lattice).
Hence, a suitable ``discrete unique continuation principle'' in $\Z^d$ would state that, if a function $u$ satisfies $-\Delta u + V u=0$ in a finite (hyper)cube, then $u$ can not be too small (compared to its value at the origin) on a substantial portion of the (hyper)cube.
In \cite{ding2020localization}, a randomized version of the discrete unique continuation principle on $\Z^2$ was proved.
The proof was inspired by \cite{buhovsky2017discrete}, where unique continuation principle was proved for harmonic functions (i.e. $V=0$) on $\Z^2$.
An important observation exploited in \cite{buhovsky2017discrete} is that the harmonic function has a polynomial structure. More recently, following this line, \cite{li2020anderson} studied the 2D lattice model with $\frac{1}{2}$-Bernoulli potential and large disorder, and localization was proved outside finitely many small intervals. 

Our Theorem \ref{thm:main} in this paper settles the Anderson-Bernoulli localization near the edge of the spectrum on the 3D lattice.
Our proof follows the framework of \cite{bourgain2005localization} and \cite{ding2020localization}.
Our main contribution is the proof of a 3D discrete unique continuation principle.
Unlike the 2D case, where some randomness is required, in 3D our discrete unique continuation principle is deterministic, and allows the potential $V$ to be an arbitrary bounded function.
It is also robust, in the sense that certain ``sparse set'' can be removed and the result still holds; and this makes it stand for the multi-scale analysis framework (see Theorem \ref{thm:qucF} below). The most innovative part of our proof is to explore the geometry of the 3D lattice.

Let us also mention that Anderson localization is not expected through the whole spectrum in $\Z^3$, when the potential is small.
There might be a localization-delocalization transition.
To be more precise, it is conjectured that there exists $\delta_0>0$ such that, for any $\delta<\delta_0$, $-\Delta+\delta V$ has purely absolutely continuous spectrum in some spectrum range (see e.g. \cite{simon2000schrodinger}). Localization and delocalization phenomenons are also studied for other models, see e.g. \cite[Chapter 16]{aizenman2015random} and \cite{anantharaman2019quantum} for regular tree graphs and expander graphs, and see \cite{bourgade2020random, bourgade2018random2, yang2020random} and \cite{shcherbina2017characteristic, shcherbina2019universality} for random band matrices.

\subsection{An outline of the proof of the 3D discrete unique continuation principle}

In this subsection we explain the most important ideas in the proof of the 3D discrete unique continuation principle.

The formal statement of the 3D discrete unique continuation principle is Theorem \ref{thm:qucF} below.
It is stated to fit the framework of \cite{bourgain2005localization} and \cite{ding2020localization}.
To make a clear outline, we state a simplified version here.
\begin{defn}\label{def:cube}
For any $a \in \Z^3$, and $r \in \R_+$, the set 
$a + \left([-r, r] \cap \Z\right)^3$ is called a \emph{cube}, or \emph{$2r$-cube}, and we denote it by $Q_r(a)$.
Particularly, we also denote $Q_r:= Q_r(\mathbf{0})$. 
\end{defn}
\begin{theorem}   \label{thm:quc}
There exists constant $p>\frac{3}{2}$ such that the following holds.
For each $K > 0$, there is $C_{1}>0$, such that for any large enough $n\in \Z_+$, and functions $u, V: \Z^3 \rightarrow \R$ with
\begin{equation}
    \Delta u=Vu
\end{equation}
in $Q_n$ and $\| V \|_{\infty} \leq K $,
we have that
\begin{equation}   \label{eq:quc}
    \left| \left\{ a \in Q_{n} : |u(a)| \geq \exp(-C_{1} n )  |u(\mathbf{0})| \right\} \right| \geq n^{p}.
\end{equation}
\end{theorem}

\begin{rem}
The power of $\frac{3}{2}$ should not be optimal.
We state it this way because it is precisely what we need (in the proof of Lemma \ref{lem:WegnerF} below).
\end{rem}
To prove Theorem \ref{thm:quc}, we first prove a version with a loose control on the magnitude of the function but with a two-dimensional support.
It is a simplified version of Theorem \ref{thm:wqucF} below.
\begin{theorem}   \label{thm:wquc}
For each $K > 0$,
there is $C_{2}$ depending only on $K$, such that for any $n \in \Z_+$ and functions $u, V: \Z^3 \rightarrow \R$ with
\begin{equation}\label{eq:harmo}
    \Delta u=Vu
\end{equation}
in $Q_n$ and $\| V \|_{\infty} \leq K $,
we have that
\begin{equation}   \label{eq:wquc}
    \left| \left\{ a \in Q_{n} : |u(a)| \geq \exp(-C_{2} n^{3})  |u(\mathbf{0})| \right\} \right| \geq C_3 n^2(\log_2 n)^{-1}.
\end{equation}
Here $C_3$ is a universal constant.
\end{theorem}

\begin{rem}\label{rem:sharp}
The power of $n^2$ can not be improved.
Consider the case where $V\equiv 0$, and $u:(x,y,z)\mapsto(-1)^{x}\exp(s z) \mathds{1}_{x=y}$, where $s \in \R_+$ is the constant satisfying $\exp(s)+\exp(-s)=6$.
One can check that $\Delta u_{0}\equiv 0$, while $|\{a \in Q_n:u_{0}(a) \neq 0\}|=|\{(x,y,z) \in Q_{n}:x=y\}|=(2n+1)^2$. 
\end{rem}

To prove Theorem \ref{thm:quc}, we find many disjoint translations of $Q_{n^{1/3}}$ inside $Q_n$, and use Theorem \ref{thm:wquc} on each of these translations.
This is made precise by Theorem \ref{thm:low} in Section \ref{sec:quc}. 
The foundation of the arguments there is the ``cone property'', given in Section \ref{sec:cone}, which says that from any point in $\Z^3$, we can find a chain of points, where $|u|$ decays at most exponentially.
Such property is also used in other parts of the paper.

The proof of Theorem \ref{thm:wquc} is based on geometric arguments on $\Z^3$.
We consider four collections of planes in $\R^3$.
\begin{defn}  \label{defn:pv}
Let $\be_1:=(1, 0, 0)$, $\be_2 := (0,1,0)$, and $\be_3 := (0,0,1)$ to be the standard basis of $\R^3$, and denote $\blambda_1 := \be_1+\be_2+\be_3$, $\blambda_2 := -\be_1+\be_2+\be_3$, $\blambda_3 := \be_1-\be_2+\be_3$, $\blambda_4:= -\be_1-\be_2+\be_3$.
For any $k \in \Z$, and $\tau \in \{1, 2, 3, 4\}$, 
denote $\cP_{\tau, k} := \left\{ a \in \R^3: a\cdot \blambda_{\tau} = k \right\}$.
\end{defn}

We note that the intersection of $\Z^3$ with each of these planes is a \emph{2D triangular lattice}.
Besides, there is a family of regular tetrahedrons in $\R^3$, whose four faces are orthogonal to $\blambda_1, \blambda_2, \blambda_3, \blambda_4$, respectively.
Using these tetrahedrons, we construct some polyhedrons $\mathfrak P \subset \R^3$, called \emph{pyramid}.
For each of these pyramid $\mathfrak P$, the boundary $\partial \mathfrak P$ consists of subsets of some of the planes $\cP_{\tau, k}$ (where $\tau \in \{1, 2, 3, 4\}$ and $k \in \Z$). See Figure \ref{fig:py} for an illustration.
Using these observations, we lower bound $\left|\left\{ a \in Q_{n} : |u(a)| \geq \exp(-C_{2} n^{3})  |u(\mathbf{0})| \right\} \cap \partial \mathfrak P\right|$.

To be more precise, we define such 2D triangular lattice as follows.
\begin{defn}\label{def:basic2D}
In $\R^2$, denote $\xi :=(-1, 0)$ and $\eta:=\left(\frac{1}{2}, \frac{\sqrt{3}}{2}\right)$.
Define the \emph{triangular lattice} as $\Lambda:=\{s\xi + t\eta: s,t \in \Z \}$.
For $a \in \Lambda$ and $n\in \Z_{\geq 0}$, denote 
\begin{equation}
T_{a;n}:=\left\{a+s\xi+t\eta: t,s\in\Z,-n\leq t \leq 2n, t-n\leq s \leq n \right\}.
\end{equation}
Then $T_{a;n}$ is an equilateral triangle of lattice points with center $a$, such that on each side there are $3n+1$ lattice points.
\end{defn}
Now we state the bound we need.
\begin{theorem}\label{thm:bou}
There exist constants $C_{4} >5$ and $\epsilon_{1} >0$ such that the following is true. For any $n\in\Z_+$ and any function $u: T_{\boo;n} \rightarrow \R$, if $|u ( a ) +u ( a - \xi ) +u ( a + \eta )| <C_{4}^{-n} |u (
  \mathbf{0} ) |$ for any $a \in T_{\boo;\left \lfloor\frac{n}{2}\right \rfloor}$,
  then
\begin{equation}  \label{eq:bou}
  \left| \{a \in T_{\boo;n} : | u ( a) | >C_{4}^{-n} |u ( \mathbf{0} ) | \}  \right| >
     \epsilon_{1} n^{2}.
\end{equation}
\end{theorem}
This theorem can be seen as a triangular version of \cite[Theorem(A)]{buhovsky2017discrete}.
Our proof is also similar to the arguments there, using the fact that the function $u$ has an approximate polynomial structure.

\subsection*{Organization of remaining text}
In Section \ref{sec:cone}, we state and prove the ``cone properties''.
In Section \ref{sec:fra}, we introduce our discrete unique continuation (Theorem \ref{thm:qucF}), and explain how to prove the resolvent estimate (Theorem \ref{thm:resonentexpo}) from it, by adapting the framework from \cite{bourgain2005localization} and \cite{ding2020localization}.
The next three sections are devoted to the proof of our discrete unique continuation (Theorem \ref{thm:qucF}):
in Section \ref{sec:pol} we prove the estimates on triangular lattice, i.e. Theorem \ref{thm:bou} and its corollaries, using arguments similar to those in \cite[Section 3]{buhovsky2017discrete};
in Section \ref{sec:pyr}, we state and prove Theorem \ref{thm:wqucF} (a stronger version of Theorem \ref{thm:wquc}) by constructing pyramids and using Theorem \ref{thm:bou};
finally, in Section \ref{sec:quc} we do induction on scales, and deduce Theorem \ref{thm:qucF} from Theorem \ref{thm:wqucF}.

We have three appendices. 
In Appendix \ref{sec:aux} we state some auxiliary results from \cite{ding2020localization} that are used in the general framework.
Appendix \ref{sec:app} is devoted to the base case of the multi-scale analysis in the general framework. 
In Appendix \ref{app:proof-of-main} we give some details on deducing Anderson localization (Theorem \ref{thm:main}) from decay of the resolvent (Theorem \ref{thm:resonentexpo}), following existing arguments (from \cite{bourgain2005localization, bourgain2005anderson, germinet2012comprehensive}).

\subsection*{Acknowledgement}
The authors thank Professor Jian Ding, the advisor of Linjun Li, for introducing this problem to them, explaining the idea of ``free sites'' from \cite{bourgain2005localization}, reading early versions of this paper, and providing very helpful suggestions on formulating the text. 
The authors thank Professor Charles Smart for explaining the ideas in the proof of \cite[Lemma 5.1]{ding2020localization} (i.e. Lemma \ref{lem:vareigen}).
The authors also thank anonymous referees for reading this paper carefully, and for their valuable feedbacks which led to many improvements in the text.

\section{Cone properties} \label{sec:cone}
In this section we state and prove the ``cone properties'', which are widely used throughout the rest of this paper.
\begin{defn}   \label{defn:cone}
For each $a \in \Z^3$, and $\tau \in \left\{1,2,3\right\}$, denote the cone
\begin{equation}
\cC_a^{\tau} := \left\{ b \in \Z^3: |(b-a) \cdot \be_{\tau}| \geq \sum_{\tau' \in \left\{1,2,3\right\}\setminus \left\{\tau\right\} }|(b-a) \cdot \be_{\tau'}|  \right\}.
\end{equation}
For each $k \in \Z$, let $\cC_a^{\tau}(k) := \cC_a^{\tau} \cap \left\{b \in \Z^3: (b-a) \cdot \be_{\tau} = k \right\}$ be a section of the cone.
We also denote $\cC:=\cC_{\mathbf{0}}^{3}$, for simplicity of notations.
\end{defn}
First, we have the ``local cone property''.
\begin{lemma}   \label{lem:walk}
For any $u:\Z^3 \rightarrow \R$, $a \in \Z^3$, and $v\in \left\{\pm \be_1, \pm \be_2, \pm \be_3  \right\}$, if $|\Delta u(a+v)| \leq K|u(a+v)|$, we have
\begin{equation}  \label{eq:walk}
\max_{b \in a + v + \left\{\mathbf{0}, \pm \be_1, \pm \be_2, \pm \be_3  \right\} \setminus \left\{a\right\}}
|u(b)| \geq (K+11)^{-1}|u(a)|.
\end{equation}
\end{lemma}
\begin{proof}
Without loss of generality we assume that $v = \be_1$.
We have
\begin{multline}
|u(a)|
\leq (6+K)|u(a+\be_1)| + |u(a+2\be_1)| + |u(a+\be_1-\be_2)| + |u(a+\be_1+\be_2)| \\ + |u(a+\be_1+\be_3)| + |u(a+\be_1-\be_3)|
\leq
(K+11) \max_{b \in a + \be_1 + \left\{\mathbf{0}, \pm \be_1, \pm \be_2, \pm \be_3  \right\} \setminus \left\{a\right\} } |u(b)|,
\end{multline}
and our conclusion follows.
\end{proof}

With Lemma \ref{lem:walk}, we can inductively construct an oriented ``chain'' from $\boo$ to the boundary of a cube, and inside a cone.

\begin{lemma}   \label{lem:chain}
Let $K \in \R_+$, and $u, V: \Z^3 \rightarrow \R$, such that $\|V\|_{\infty} \leq K$, and $\Delta u = Vu$ in $Q_n$ for some $n \in \Z_+$.
For any $a \in Q_{n-2}$, $\tau \in \left\{1,2,3\right\}$, $\iota \in \left\{1, -1\right\}$, and $k \in \Z_{\geq 0}$, if $\cC_a^{\tau}(\iota k) \subset Q_n$,
then there exists $w \in \Z_{\geq 0}$, and a sequence of points $a = a_0, a_1, \cdots, a_w \in \cC_a^{\tau}\cap Q_{n}$, such that for any $1 \leq i \leq w$, we have $a_i - a_{i-1} \in ( \iota\be_{\tau} + \left\{\mathbf{0},\pm \be_1, \pm \be_2, \pm \be_3\right\} ) \setminus \left\{\mathbf{0}\right\}$, $|u(a_i)| \geq (K+11)^{-1}|u(a_{i-1})|$; and $(a_w-a) \cdot (\iota \be_{\tau}) \in \left\{ k-1, k  \right\}$.
\end{lemma}
\begin{proof}
We prove the case where $\iota = 1$, and the other case follows the same arguments.

We define the sequence inductively.
Let $a_0 := a$.
Suppose we have $a_i \in \cC_a^{\tau}$, with $0 \leq (a_i-a) \cdot \be_{\tau} < k-1$.
Then $a_i + \be_{\tau} + \left\{\mathbf{0}, \pm \be_1, \pm \be_2, \pm \be_3\right\} \subset Q_n$.
Let
\begin{equation}
a_{i+1}:= \argmax_{b \in a_i + \be_{\tau} + \left\{\mathbf{0}, \pm \be_1, \pm \be_2, \pm \be_3\right\}\setminus \left\{a_i\right\}} |u(b)|.
\end{equation}
Then we have that $a_{i+1} - a_{i} \in \be_{\tau} + \left\{\mathbf{0}, \pm \be_1, \pm \be_2, \pm \be_3\right\} \setminus \left\{\mathbf{0}\right\}$, $0 \leq (a_{i+1}-a) \cdot \be_{\tau} \leq k$, and $a_{i+1} \in \cC_a^{\tau}$.
By Lemma \ref{lem:walk}, we also have that $|u(a_{i+1})| \geq (K+11)^{-1}|u(a_i)|$.
This process will terminate when $(a_i-a) \cdot \be_{\tau} \geq k-1$ for some $i\in \Z_{\geq 0}$.
Then we let $w=i$; and from the construction we know that $(a_i-a) \cdot \be_{\tau} \in \left\{k-1, k\right\}$.
Thus we get the desired sequence of lattice points.
\end{proof}

We also have a Dirichlet boundary version, whose proof is similar.
\begin{lemma}   \label{lem:eichain}
Take any $n \in \Z_+$, $K\in\R_+$, and $u, V:Q_n \rightarrow \R$, such that $\|V\|_{\infty} \leq K$ and $\Delta u = Vu$ with Dirichlet boundary condition.
For any $a \in Q_{n}$, $\tau \in \left\{1,2,3\right\}$, $\iota \in \left\{1, -1\right\}$, and $k \in \Z_{\geq 0}$, if $\cC_a^{\tau}(\iota k) \cap Q_n \neq \emptyset$, then the result of Lemma \ref{lem:chain} still holds.
In particular, we have $a_w \in (\cC_a^{\tau}(\iota (k-1))\cup \cC_{a}^{\tau}(\iota k))\cap Q_{n}$ and $|u(a_{w})|\geq (K+11)^{-k}|u(a)|$.
\end{lemma}

\begin{proof}
Again we only prove the case where $\iota = 1$, and define the sequence inductively.
The only difference is that, given some $a_i \in \cC_a^{\tau}$, if $0\leq (a_i-a) \cdot \be_{\tau} < k-1$, now we let
\begin{equation}
a_{i+1}:= \argmax_{b \in (a_i + \be_{\tau} + \left\{\mathbf{0}, \pm \be_1, \pm \be_2, \pm \be_3\right\}\setminus \left\{a_i\right\}) \cap Q_{n}} |u(b)|.
\end{equation}
By the Dirichlet boundary condition, 
we still have that $a_{i+1} - a_{i} \in \be_{\tau} + \left\{\mathbf{0}, \pm \be_1, \pm \be_2, \pm \be_3\right\} \setminus \left\{\mathbf{0}\right\}$, $0\leq(a_{i+1}-a) \cdot \be_{\tau} \leq k$, $a_{i+1} \in \cC_a^{\tau} \cap Q_{n}$, and $|u(a_{i+1})| \geq (K+11)^{-1}|u(a_i)|$.
\end{proof}

\section{General framework}  \label{sec:fra}
This section is about the framework, based on the arguments in \cite{ding2020localization}.
We formally state the discrete unique continuation principle (Theorem \ref{thm:qucF}), and explain how to deduce Theorem \ref{thm:main} from it.
For some results from \cite{ding2020localization} that are used in this section, we record them in Appendix \ref{sec:aux} for easy reference purpose.

As in \cite{ding2020localization}, these arguments essentially work for any i.i.d. potential $V$ that is bounded and nontrivial.
For simplicity we only study the $\frac{1}{2}$-Bernoulli case with disorder strength $\delta = 1$.
Borrowing the formalism from \cite{bourgain2005localization} and \cite{ding2020localization}, we allow $V$ to take values in the interval $[0, 1]$, for the purpose of controlling the number of eigenvalues in proving the Wegner estimate (in the proof of Claim \ref{cla:conts} below).
In other words, we study the operator $H=-\Delta+V$, where $V$ takes value in the space $[0, 1]^{\Z^3}$, equipped with the usual Borel sigma-algebra, and the distribution is given by the product of the $\frac{1}{2}$-Bernoulli measure (which is supported on $\{0, 1\}^{\Z^3}$).

We let $\spc(H)$ be the spectrum of $H$, then it is well known that, almost surely  $\spc(H)=[0,13]$ (see, e.g. \cite[Corollary 3.13]{aizenman2015random}).
For any cube $Q \subset \Z^3$, let $P_{Q}:\ell^2(\Z^3) \rightarrow \ell^2(Q)$ be the projection operator onto cube $Q$, i.e. $P_{Q}u=u | _{Q}$.
Define $H_{Q}:=P_{Q} H P_{Q}^{\dag}$, where $P_{Q}^{\dag}$ is the adjoint of $P_Q$. Then $H_{Q}:\ell^{2}(Q)\rightarrow \ell^{2}(Q)$ is the restriction of $H$ on $Q$ with Dirichlet boundary condition. 

Throughout this section, by ``dyadic'', we mean a number being an integer power of $2$.

The following result on decay of the resolvent is a 3D version of Theorem \cite[Theorem 1.4]{ding2020localization}, and it directly implies Theorem \ref{thm:main}.
\begin{theorem}\label{thm:resonentexpo}
There exist $\kappa_{0}>0$, $0<\lambda_{*}<1$ and $L_{*}>1$ such that
\begin{equation}
    \mathbb{P}\left[ \left|(H_{Q_L}-\lambda)^{-1}(a,b)\right| \leq \exp\left(L^{1-\lambda_{*}}-\lambda_{*} |a-b|\right), \;\forall a,b \in Q_{L}\right] \geq 1-L^{-\kappa_{0}}
\end{equation}
for any $\lambda \in [0,\lambda_{*}]$ and dyadic scale $L \geq L_{*}$.
\end{theorem}
From Theorem \ref{thm:resonentexpo}, the arguments in \cite[Section 7]{bourgain2005localization} prove Anderson localization in $[0,\lambda_{*}]$ (Theorem \ref{thm:main}).
See Appendix \ref{app:proof-of-main} for the details.

To prove Theorem \ref{thm:resonentexpo}, we will prove a 3D analog of \cite[Theorem 8.3]{ding2020localization}, i.e. Theorem \ref{thm:multiscale} below. Except for replacing all 2D objects by 3D objects, the essential differences are:
\begin{enumerate}
    \item We need more information on the \emph{the frozen sites} defined in \cite{ding2020localization}, rather than only knowing they're ``$\eta_k$-regular'' (see  \cite[Definition 3.4]{ding2020localization}). 
    \item We need a 3D Wegner estimate, an analog of \cite[Lemma 5.6]{ding2020localization}.
\end{enumerate}

We now set up some geometric notations. 
\begin{defn}
For any sets $A,B \subset \R^3$, let
\begin{equation}
        \dist(A,B):=\inf_{a \in A,b \in B}|a-b|,
\end{equation}
and
\begin{equation}
    \diam(A):=\sup_{a,b \in A}|a-b|.
\end{equation}
If $A=\{b \in \R^3:|a-b| < r\}$, for some $r>0$ and $a \in \R^3$, we call $A$ a \emph{(open) ball} and denote its \emph{radius} as $\radius(A):=r$.
\end{defn}

The following definitions are used to describe the frozen sites, and are stronger than being ``$\eta_k$-regular'' in \cite{ding2020localization}.
\begin{defn}\label{def:F}
Let $d \in \Z_{\geq 0}$, $N \in \Z_+$, and $C,\varepsilon>0$, $l \geq 1$.
A set $Z \subset \R^3$ is called \emph{$(N,l,\varepsilon)$-scattered} if $Z=\bigcup_{j \in \Z_+,1 \leq t \leq N}Z^{(j,t)}$ is a union of open balls such that,
\begin{enumerate}
    \item for each $j \in \Z_+$ and $t\in \{1,\cdots,N\}$, $\radius(Z^{(j,t)})=l$;
    \item for any $j \neq j' \in \Z_+$ and $t\in \{1,2,\cdots,N\}$, $\dist(Z^{(j,t)},Z^{(j',t)}) \geq l^{1+\varepsilon}$.
\end{enumerate}

A set $Z \subset \R^3$ is called \emph{$C$-unitscattered},
if we can write $Z=\bigcup_{j\in \Z_+} Z^{(j)}$, where each $Z^{(j)} \subset \R^3$ is an open unit ball with center in $\Z^{3}$ 
and
\begin{equation}
    \forall j\neq j' \in \Z_+, \dist(Z^{(j)},Z^{(j')})\geq C.
\end{equation}

Let $l_{1}, \cdots l_d > 1$, we say that the vector $\Vec{l}=(l_1,l_2,\cdots,l_d)$ is \emph{$\varepsilon$-geometric} if for each $2 \leq i \leq d$, we have $l_{i-1}^{1+2\varepsilon} \leq l_i$.
Given a vector of positive reals $\Vec{l}=(l_1,l_2,\cdots,l_d)$, a set $E\subset \R^3$ is called an \emph{$(N,\Vec{l},C,\varepsilon)$-graded set} if
there exist sets $E_{0}, \cdots, E_d \subset \R^3$, such that $E=\bigcup_{i=0}^{d}E_i$ and the following holds:
\begin{enumerate}
    \item $\Vec{l}$ is $\varepsilon$-geometric,
    \item $E_0$ is a $C$-unitscattered set,
    \item for any $1 \leq i\leq d$, $E_i$ is an $(N,l_i,\varepsilon)$-scattered set.
\end{enumerate}
For each $1 \leq i \leq d$, we say that $l_i$ is the \emph{$i$-th scale length} of $E$.
In particular, $l_1$ is called the \emph{first scale length}.
We also denote $l_0 := 1$.

Let $A \subset \R^3$, and $E$ be an $(N,\vec{l},C,\varepsilon)$-graded set and $\overline{C},\overline\varepsilon>0$.
Then $E$ is said to be \emph{$(\overline{C},\overline\varepsilon)$-normal in} $A$,
if $E_0\cap A \neq \emptyset$ implies $\overline{C}\leq \diam(A)$, and $E_i\cap A \neq \emptyset$ implies $l_i \leq \diam(A)^{1-\frac{\overline\varepsilon}{2}}$ for any $i\in \{1,\cdots,d\}$.
\end{defn}

In \cite{ding2020localization}, a 2D Wegner estimate \cite[Lemma 5.6]{ding2020localization} is proved and used in the multi-scale analysis.
We will prove the 3D Wegner estimate based on our 3D discrete unique continuation, and we need to accommodate the frozen sites which emerge from the multi-scale analysis.
For this we refine Theorem \ref{thm:quc} as follows.
\begin{theorem}   \label{thm:qucF}
There exists a constant $p>\frac{3}{2}$, such that for any $N \in \Z_+$, $K \in \R_+$, and small enough $\varepsilon \in \R_+$, there exist $C_{\varepsilon,K},C_{\varepsilon,N}>0$ to make the following statement hold.

Take $n \in \Z_+$ with $n>C_{\varepsilon,N}^{4}$ and functions $u, V: \Z^3 \rightarrow \R$ satisfying
\begin{equation}\label{eq:harmoF}
    \Delta u=Vu,
\end{equation}
and $\| V \|_{\infty} \leq K $ in $Q_n$.
Let $\Vec{l}$ be a vector of positive reals, and $E\subset \Z^3$ be an $(N,\Vec{l},\varepsilon^{-1},\varepsilon)$-graded set with the first scale length $l_1>C_{\varepsilon,N}$ and be $(1,\varepsilon)$-normal in $Q_n$.
Then we have that
\begin{equation}   \label{eq:qucF}
    \left| \left\{ a \in Q_{n}\setminus E : |u(a)| \geq \exp(-C_{\varepsilon,K} n )  |u(\mathbf{0})| \right\} \right| \geq n^{p}.
\end{equation}
\end{theorem}
Assuming Theorem \ref{thm:qucF}, we can prove the 3D Wegner estimate. For simplicity of notations, for any $A\subset \Z^{3}$, we denote $V_{A}:=V|_{A}$, the restriction of the potential function $V$ on $A$. 
\begin{lemma}[3D Wegner estimate]\label{lem:WegnerF}
  There exists $\varepsilon_0>0$ such that, if
  \begin{enumerate}
      \item $\varepsilon > \delta>0$, $\varepsilon$ is small enough, and $\overline{\lambda} \in \spc(H)=[0,13]$,
      \item $N_1 \geq 1$ is an integer and $\Vec{l}$ is a vector of positive reals,
      \item $L_0>\cdots>L_5\geq C_{\varepsilon,\delta,N_{1}}$ with $L_{j}^{1-2\delta}\geq L_{j+1}\geq L_{j}^{1-\frac{1}{2}\varepsilon}$ for $j=0,1,2,3,4$, where $C_{\varepsilon,\delta,N_{1}}$ is a (large enough) constant, and $L_0$, $L_3$ are dyadic, 
      \item $Q \subset \Z^3$ and $Q$ is an $L_0$-cube,
      \item $Q'_1,Q'_2,\cdots,Q'_{N_1} \subset Q$, and $Q_k'$ is an $L_3$-cube for each $k=1,2,\cdots,N_{1}$ (we call them ``defects''),
      \item $G \subset \bigcup_{k=1}^{N_{1}}Q'_{k}$ with $0<|G|<L_{0}^{\delta}$,
      \item $E$ is a $(1000 N_{1},\Vec{l},\varepsilon^{-1},\varepsilon)$-graded set with the first scale length $l_1 \geq C_{\varepsilon,\delta,N_{1}}$ and $\mathscr{V}:E\cap Q \rightarrow \{0,1\}$,
      \item for any $L_3$-cube $Q'\subset Q\setminus \bigcup_{k=1}^{N_1}Q'_{k}$, $E$ is $(1,\varepsilon)$-normal in $Q'$,
      \item 
      for any $V:\Z^3\to [0,1]$ with
      $V_{E\cap Q}=\mathscr{V}$, $|\lambda-\overline{\lambda}|\leq \exp(-L_5)$ and $H_{Q} u=\lambda u$, we have
      \begin{equation}
          \exp(L_4)\|u\|_{\ell^2(Q\setminus \bigcup_{k}Q'_{k})} \leq \|u\|_{\ell^{2}(Q)}\leq (1+L_{0}^{-\delta})\|u\|_{\ell^2(G)}.
      \end{equation}
  \end{enumerate}
  Then
  \begin{equation}
      \mathbb{P}\left[\|(H_{Q}-\overline{\lambda})^{-1}\|\leq \exp(L_1)\big| \; V_{E\cap Q}=\mathscr{V}\right]\geq 1-L_{0}^{C\varepsilon-\varepsilon_0},
  \end{equation}
  where $C$ is a universal constant, and $\|\cdot\|$ denotes the operator norm.
\end{lemma}
The proof is similar to that of \cite[Lemma 5.6]{ding2020localization}, after changing 2D notations to corresponding 3D notations.
The major difference is in Claim \ref{cla:loca} and \ref{cla:move-out-eigen} (corresponding to \cite[Claim 5.9 5.10]{ding2020localization}), where Theorem \ref{thm:qucF} is used.
This is also the reason why we need the constant $p>\frac{3}{2}$ in Theorem \ref{thm:qucF}.
\begin{proof}[Proof of Lemma \ref{lem:WegnerF}]
Let $\varepsilon_{0}<p-\frac{3}{2}$ where $p > \frac{3}{2}$ is the constant in Theorem \ref{thm:qucF}. In this proof, we will use $c, C$ to denote small and large universal constants. 
    
    We let $\lambda_1 \geq \lambda_2 \geq \cdots \geq \lambda_{(L_{0}+1)^{3}}$ be the eigenvalues of $H_{Q}$. 
    For each $1 \leq k \leq (L_0+1)^3$,
    choose eigenfunctions $u_k$ such that $\|u_k\|_{\ell^{2}(Q)}=1$ and $H_{Q} u_{k}= \lambda_{k} u_{k}$. We may think of $\lambda_{k}$ and $u_{k}$ as deterministic functions of the potential $V_{Q}\in [0,1]^{Q}$.
    
         Let $E'=\left(\bigcup_{k=1}^{N_1}Q'_{k}\right) \cup (E\cap Q)$, then for any event $\event$,
    \begin{equation}
        \prob\left[\event\big|\; V_{E\cap Q}=\mathscr{V}\right]=2^{-|E'\setminus E|}\sum_{\mathscr{V}':E'\rightarrow\{0,1\}, \mathscr{V}'|_{E\cap Q}=\mathscr{V}}\prob\left[\event\big|\;V_{E'}=\mathscr{V}'\right].
    \end{equation}
   By the simple fact that the average is bounded from above by the maximum, we only need to prove
    \begin{equation}
      \mathbb{P}\left[\|(H_{Q}-\overline{\lambda})^{-1}\|> \exp(L_1)\big| \; V_{E'}=\mathscr{V}'\right]\leq L_{0}^{C\varepsilon-\varepsilon_0},
  \end{equation}
    for any $\mathscr{V}':E'\rightarrow\{0,1\}$ with $\mathscr{V}'|_{E\cap Q}=\mathscr{V}$.
    \begin{cla}\label{cla:prelocal}
        There is a constant $C_{N_1}$ such that the following is true. 
        Suppose $u$ satisfies $H_{Q} u= \lambda u$ for some $\lambda \in [0,13]$.
        Then there is $a'\in \Z^3$, such that $Q_{\frac{L_{3}}{2}}(a') \subset Q \setminus \bigcup_{k}Q'_{k}$, and
        \begin{equation}
            |u(a')| \geq \exp(-C_{N_1} L_{3}) \|u\|_{\ell^{\infty}(Q)}.
        \end{equation}
    \end{cla}
    \begin{proof}
         Without loss of generality, we assume $Q=Q_{\frac{L_0}{2}}(\mathbf{0})$.
         Take $a_0 \in Q$ such that $|u(a_{0})|=\|u\|_{\ell^{\infty}(Q)}$.
         We assume without loss of generality that $a_0 \cdot \be_{\tau}\leq 0$, for each $\tau \in \{1, 2, 3\}$.
         Since each $Q'_{k}$ is an $L_3$-cube, by the Pigeonhole principle, there is $x'_{0} \in [a_0 \cdot \be_1+100 N_{1} L_3, a_0 \cdot \be_1+200 N_{1} L_3]$, such that 
         \begin{equation}
             \{b \in Q: b \cdot \be_1 \in [x'_{0}-16L_{3},x'_{0}+16L_{3}]\}\cap \bigcup_{k=1}^{N_1}Q'_{k} =\emptyset.
         \end{equation}
         Now we iteratively apply the cone property Lemma \ref{lem:eichain} with $K=13$.
         Recall the notations of cones from Definition \ref{defn:cone}, and note that $(K+11)<\exp(5)$.
We find 
\begin{equation}\label{eq:a_1-position}
    a_{1}\in (\cC_{a_0}^{1}(x'_{0}-a_0 \cdot \be_1)\cup \cC_{a_0}^{1}(x'_{0}-a_0 \cdot \be_1+1)) \cap Q
\end{equation}
with
         \begin{equation}
               |u(a_{1})| \geq \exp(-1000 N_{1} L_{3}) |u(a_{0})|,
         \end{equation}
and $a_{2} \in (\cC_{a_1}^{2}(4L_{3})\cup \cC_{a_1}^{2}(4L_{3}+1)) \cap Q$ with 
         \begin{equation}
             |u(a_{2})| \geq \exp(-(1000 N_{1}+20) L_{3}) |u(a_{0})|,
         \end{equation}
and $a_{3} \in (\cC_{a_2}^{3}(2L_{3})\cup \cC_{a_2}^{3}(2L_{3}+1)) \cap Q$ with 
         \begin{equation}
             |u(a_{3})| \geq \exp(-(1000 N_{1}+30) L_{3}) |u(a_{0})|.
         \end{equation}
         By \eqref{eq:a_1-position},
         we have $|a_{1} \cdot \be_1-x'_{0}|\leq 1$ and $-\frac{L_0}{2}\leq a_{1} \cdot \be_{\tau}\leq 200 N_{1} L_{3}+1$ for $\tau=2,3$. Then $|a_{2} \cdot \be_1-x'_{0}|\leq 4L_{3}+2$, and 
         $-\frac{L_0}{2}+4L_{3} \leq  a_2 \cdot \be_2 \leq (200N_{1}+4)L_{3}+2$, and 
         $-\frac{L_0}{2} \leq  a_2 \cdot \be_3 \leq (200N_{1}+4)L_{3}+2$.
         Finally, we have $|a_3 \cdot \be_1-x'_{0}|\leq 6L_{3}+3$, and $-\frac{L_0}{2}+2L_{3}-1 \leq  a_3 \cdot \be_2 \leq (200N_{1}+6)L_{3}+3$, and 
         $-\frac{L_0}{2}+2L_{3} \leq  a_3 \cdot \be_3 \leq (200N_{1}+6)L_{3}+3$. This implies $Q_{\frac{L_{3}}{2}}(a_{3}) \subset Q\setminus \bigcup_{k=1}^{N_1}Q'_{k}$ and the claim follows by letting $a'=a_{3}$ and $C_{N_1}=1000N_{1}+30$.
    \end{proof}
    
    \begin{cla}\label{cla:loca}
        For any $\lambda \in [0,13]$,  $H_{Q} u= \lambda u$ implies
        \begin{equation}
            \left|\left\{a\in Q:|u(a)| \geq \exp\left(-\frac{L_2}{4} \right) \|u\|_{\ell^{2}(Q)} \right\} \setminus E' \right| \geq \left(\frac{L_{3}}{2}\right)^{p} .
        \end{equation}
    \end{cla}

    \begin{proof}
        By applying Claim \ref{cla:prelocal} to $u$, we can find a cube $Q_{\frac{L_{3}}{2}}(a') \subset Q \setminus \bigcup_{k}Q'_{k}$ for some $a' \in \Z^3$, such that $|u(a')| \geq \exp(-C_{N_1} L_{3}) \|u\|_{\ell^{\infty}(Q)}\geq \exp(-C_{N_1} L_{3})(L_{0}+1)^{-\frac{3}{2}} \|u\|_{\ell^{2}(Q)}$.
        By Condition $8$, $E$ is $(1,\varepsilon)$-normal in $Q_{\frac{L_{3}}{2}}(a')$. Applying Theorem \ref{thm:qucF} to cube $Q_{\frac{L_{3}}{2}}(a')$ with graded set $E$, function $u$, and $K=13$,
        and letting $\frac{1}{4}C_{\varepsilon,\delta,N_1}^{2\delta}>C_{N_1}+C_{\varepsilon,K}$ where $C_{\varepsilon,K}$ is the constant in Theorem \ref{thm:qucF},
        the claim follows. 
    \end{proof}
    
    \begin{cla}\label{cla:move-out-eigen}
    Let $s_{i}=\exp(-L_{1}+(L_2-L_4+C)i)$ for each $i \in \Z$.
    For $1 \leq k_1 \leq k_2 \leq (L_{0}+1)^{3}$ and $0 \leq \ell \leq C L_{0}^{\delta}$, we have
    \begin{equation}
        \mathbb{P}\left[\mathcal{E}_{k_1,k_2,\ell} \big|\; V_{E'} = \mathscr{V}' \right] \leq CL_{0}^{\frac{3}{2}}L_{3}^{-p} 
    \end{equation}
    where $\mathcal{E}_{k_1,k_2,\ell}$ denotes the event 
    \begin{equation}
        |\lambda_{k_1}-\overline{\lambda}|,|\lambda_{k_2}-\overline{\lambda}|< s_{\ell},\;  |\lambda_{k_1 -1}-\overline{\lambda}|,|\lambda_{k_2+1}-\overline{\lambda}|\geq s_{\ell+1}.
    \end{equation}
    \end{cla}

    \begin{proof}
        For $i=0,1$, we let $\mathcal{E}_{k_1,k_2,\ell,i}$ denote the event 
        \begin{equation}
            \mathcal{E}_{k_1,k_2,\ell}\ \cap\ \left\{\left|\left\{a\in Q:|u_{k_1}(a)| \geq \exp\left(-\frac{L_2}{4} \right), V(a)=i \right\}\setminus E' \right| \geq \frac{L_{3}^{p}}{8} \right\} \cap \{V_{E'}=\mathscr{V}'\}.
        \end{equation}
        Since we are under the event $V_{E'}=\mathscr{V}'$, we can view $\mathcal{E}_{k_1,k_2,\ell,0}$ and $\mathcal{E}_{k_1,k_2,\ell,1}$ as subsets of $\{0,1\}^{Q\setminus E'}$.
        Observe that $\mathcal{E}_{k_1,k_2,\ell}\cap \{V_{E'}=\mathscr{V}'\} \subset \mathcal{E}_{k_1,k_2,\ell,0} \cup \mathcal{E}_{k_1,k_2,\ell,1}$ by Claim \ref{cla:loca}. 
        
Fix $i \in \{0, 1\}$. For each $\omega\in \mathcal{E}_{k_1,k_2,\ell,i}$, we denote
\begin{equation}
S_1(\omega):=\{a \in Q \setminus E':\omega(a)=1-i\},
\end{equation}
and
\begin{equation}
S_2(\omega):=\left\{a \in Q\setminus E': \omega(a)=i,  |u_{k_1}(a)|\geq \exp\left(- \frac{L_2}{4} \right)\right\}.
\end{equation}
By definition of $\mathcal{E}_{k_1,k_2,\ell,i}$, we have $|S_2(\omega)|\geq \frac{L_{3}^{p}}{8}$. 
For each $\omega \in \event_{k_1,k_2,\ell,i}$, $a\in S_{2}(\omega)$, we define $\omega^{a}$ as
        \begin{equation}
            \omega^{a}(a):=1-\omega(a),\; \omega^{a}(a'):=\omega(a'),\;  \forall a' \in Q\setminus E', a'\neq a.
        \end{equation}
        We claim that $\omega^{a} \not\in \event_{k_1,k_2,\ell,i}$.
        In the case where $i=0$, because of Condition $9$ and $a \not\in \bigcup_{k}Q'_{k}$, we have $\sum_{|\lambda_{k}-\overline{\lambda}| < \exp(-L_{5})}u_{k}(a)^{2}<\exp(-c L_{4})$. Now we apply Lemma \ref{lem:vareigen} to $H_{Q}-\overline{\lambda}+s_{\ell}$
        with $r_1=2s_{\ell}$, $r_2=s_{\ell+1}$, $r_3=\exp(-\frac{1}{2} L_2)$, $r_4=\exp(-c L_4)$ and $r_5=\exp(-L_5)$. 
        Then $\lambda_{k_1}$ moves out of interval $(\overline{\lambda}-s_{\ell},\overline{\lambda}+s_{\ell})$ when $\omega(a)$ is changed from $0$ to $1$. 
        Thus we have $\omega^{a} \not\in \event_{k_1,k_2,\ell,0}$. The case where $i=1$ is similar.
        
        From this, we know that
        for any two $\omega,\omega' \in \event_{k_1,k_2,\ell,i}$, $S_1(\omega) \subset S_1(\omega')$ implies $S_1(\omega')\cap S_{2}(\omega)=\emptyset$.
        Since $|Q\setminus E'| \leq (L_0+1)^{3}-(L_{3}+1)^{3} \leq L_{0}^3$, 
        we can apply Theorem \ref{thm:sperner} with set
        $\{S_{1}(\omega):\omega \in \event_{k_1,k_2,\ell,i}\}$ and $\rho=\frac{1}{8} L_{0}^{-3}L_{3}^{p}$, and we conclude that $\prob[\event_{k_1,k_2,\ell,i}| \; V_{E'} = \mathscr{V}'] \leq C L_{0}^{\frac{3}{2}} L_{3}^{-p}$.
    \end{proof}
    \begin{cla}   \label{cla:conts}
            There is a set $K \subset \{1,2,\cdots,(L_{0}+1)^{3}\}$ depending only on $E'$ and $\mathscr{V}'$, such that $|K|\leq CL_{0}^{\delta}$ and 
            \begin{equation}
             \{\|(H_{Q}-\overline{\lambda})^{-1}\|>\exp(L_{1})\}\cap \{V_{E'} =\mathscr{V}'\} \subset \bigcup_{\substack{k_1,k_2 \in K\\ k_{1}\leq k_{2}}}\bigcup_{0\leq \ell \leq CL_{0}^{\delta}} \event_{k_1,k_2,\ell}.
            \end{equation}
        \end{cla}
        \begin{proof}
             Conditioning on $V_{E'}=\mathscr{V}'$, we view $\lambda_k$ and $u_k$ as functions on $[0,1]^{Q\setminus E'}$.
             Let $1\leq k_{1}<\cdots<k_{m}\leq (L_{0}+1)^{3}$ be all indices $k_{i}$ such that there is at least one $\omega \in [0,1]^{Q\setminus E'}$ with $|\lambda_{k_{i}}(w)-\overline{\lambda}|\leq \exp(-L_{2})$. To prove the claim, it suffices to prove that $m\leq C L_{0}^{\delta}$. Indeed, then we can always find an $0\leq \ell \leq m$ such that the annulus $\left[\overline{\lambda}-s_{\ell +1},\overline{\lambda}+s_{\ell +1}\right]\setminus \left[\overline{\lambda}-s_{\ell},\overline{\lambda}+s_{\ell}\right]$ contains no eigenvalue of $H_{Q}$. 
             
             Since $\bigcup_{k} Q'_{k}\subset E'$, Condition $9$ implies that for any $\omega\in [0,1]^{Q\setminus E'}$ with $|\lambda_{k_{i}}(\omega)-\overline{\lambda}|\leq \exp(-L_{5})$, we have $\|u_{k_{i}}(\omega)\|_{\ell^{\infty}(Q\setminus E')}\leq \exp(-L_{4})$. In particular, if there is $\omega_{0}\in \left[0,1\right]^{Q\setminus E'}$ such that $|\lambda_{k_{i}}(\omega_{0})-\overline{\lambda}|\leq \exp(-L_{2})$, then by eigenvalue variation,   
             \begin{equation}\label{eq:eigen-variation-bdd}
                 |\lambda_{k_{i}}(\omega) -\overline{\lambda}|\leq \exp(-L_{4})
             \end{equation}
             holds for all $\omega \in \{0,1\}^{Q\setminus E'}$. Indeed, let $\omega_{t}=(1-t)\omega_{0}+t\omega$ for $t\in [0,1]$. We compute
             \begin{align}
                 \begin{split}
                     |\lambda_{k_{i}}(\omega_{t})-\overline{\lambda}|\leq& |\lambda_{k_{i}}(\omega_{0})-\overline{\lambda}| +
                     \int_{0}^{t} \|u_{k_{i}}(\omega_{s})\|^{2}_{\ell^{2}(Q\setminus E')} ds\\
                     \leq& \exp(-L_{2})+
                     \int_{0}^{t} |Q| \exp(-2L_{4})+\mathds{1}_{|\lambda_{k_{i}}(\omega_{s})-\overline{\lambda}|\geq \exp(-L_{5})} ds\\
                     \leq& \exp(-L_{4})+ 
                     \mathds{1}_{\max_{0\leq s\leq t}|\lambda_{k_{i}}(\omega_{s})-\overline{\lambda}|\geq \exp(-L_{5})}
                 \end{split}
             \end{align}
             and conclude by continuity. By \eqref{eq:eigen-variation-bdd} and Condition $9$, for all $\omega\in \{0,1\}^{Q\setminus E'}$, we have $1=\|u_{k_{i}}(\omega)\|_{\ell^{2}(Q)}\geq \|u_{k_{i}}(\omega)\|_{\ell^{2}(G)}\geq 1-C L_{0}^{-\delta}$. In particular, we have 
             $|\langle u_{k_{i}}(\omega), u_{k_{j}}(\omega) \rangle_{\ell^{2}(G)}-\mathds{1}_{i=j}| \leq CL_{0}^{-\delta}\leq (5|G|)^{-\frac{1}{2}}$. By Lemma \ref{lem:app-almost-orth} we have that $m\leq C|G|\leq CL_{0}^{\delta}$. 
        \end{proof}
        Finally,
        \begin{equation}
            \mathbb{P}[\|(H_{Q}-\overline{\lambda})^{-1}\| > \exp(L_1)| \; V_{E'} = \mathscr{V}'] \leq \sum_{k_1,k_2 \in K}\sum_{1 \leq \ell \leq CL_{0}^{\delta}} \prob[\event_{k_1,k_2,\ell}|\; V_{E'} = \mathscr{V}'] 
        \end{equation}
        and thus
        \begin{equation}
        \mathbb{P}[\|(H_{Q}-\overline{\lambda})^{-1}\| > \exp(L_1)| \; V_{E'} = \mathscr{V}']
            \leq CL_{0}^{\frac{3}{2}+3\delta}L_{3}^{-p}
            \leq L_{0}^{C\varepsilon-\varepsilon_{0}},
        \end{equation}
so our conclusion follows.
\end{proof}

We now prove Theorem \ref{thm:resonentexpo} by a multi-scale analysis argument.

In the remaining part of this section, by ``dyadic cube'', we mean a cube $Q_{2^{n}}(a)$ for some $a \in 2^{n-1}\Z^3$ and $n \in \Z_{+}$. For each $k,m\in \Z_{+}$ and each $2k$-cube $Q$, we denote by $m Q$ the $2 m k$-cube with the same center as $Q$.

\begin{theorem}[Multi-scale Analysis]\label{thm:multiscale}
There exists $\kappa > 0$, such that for any $\varepsilon_{*}>0$, there are
  \begin{enumerate}
      \item $\varepsilon_{*}>\varepsilon>\nu >\delta>0$,
      \item $M,N\in \Z_+$,
      \item dyadic scales $L_k$, for $k\in \Z_{\geq 0}$, with $\left\lfloor \log_{2}L_{k+1}^{1-6 \varepsilon} \right\rfloor=\log_{2} L_{k}$,
      \item decay rates $1 \geq m_{k} \geq L_{k}^{-\delta}$ for $k \in \Z_{\geq 0}$,
    \end{enumerate}
    such that for any $0 \leq \overline{\lambda} \leq \exp(-L_{M}^{\delta})$, we have random sets $\mathcal{O}_{k} \subset \R^3$ for $k\in \Z_{\geq 0}$ with $\mathcal{O}_{k} \subset \mathcal{O}_{k+1}$ (depending on the Bernoulli potential $V$),
    and the following six statements hold for any $k \in \Z_{\geq 0}$:
  \begin{enumerate}
      \item When $k \leq M$, $\mathcal{O}_{k} \cap \Z^{3}= \left\lceil \varepsilon^{-1} \right\rceil \Z^{3}$.
      \item When $k\geq M+1$, $\mathcal{O}_{k}$ is an $(N,\Vec{l},(2\varepsilon)^{-1},2\varepsilon)$-graded random set with $\vec{l}=(L_{M+1}^{1-2\varepsilon},L_{M+2}^{1-2\varepsilon},\cdots,L_{k}^{1-2\varepsilon})$.
      \item For any $L_k$-cube $Q$, the set $\mathcal{O}_k$ is $(1,2\varepsilon)$-normal in $Q$.
      \item For any $i\in\Z_{\geq 0}$ and any dyadic $2^iL_k$-cube $Q$, the set $\mathcal{O}_{k}\cap Q$ is $V_{\mathcal{O}_{k-1}\cap 3Q}$-measurable.
      \item For any dyadic $L_k$-cube $Q$,
      it is called \emph{good} (otherwise \emph{bad}), if for any potential $V': \Z^{3}\rightarrow [0,1]$ with  $V'_{\mathcal{O}_k\cap Q}=V_{\mathcal{O}_k\cap Q}$, we have
      \begin{equation}
          |(H'_{Q}-\overline{\lambda})^{-1}(x,y)|\leq \exp(L_{k}^{1-\varepsilon}-m_{k}|x-y|),\; \forall x,y \in Q.
      \end{equation}
      Here $H'_Q$ is the restriction of $-\Delta+V'$ on $Q$ with Dirichlet boundary condition.
      Then $Q$ is good with probability at least $1-L_{k}^{-\kappa}$.
      \item $m_{k} = m_{k-1}-L_{k-1}^{-\nu}$ when $k \geq M+1$. 
  \end{enumerate}
\end{theorem}

\begin{proof}
Throughout the proof, we use $c,C$ to denote small and large universal constants.

Let $\kappa$ be any number with $0<\kappa<\varepsilon_{0}$, where $\varepsilon_{0}$ is from Lemma \ref{lem:WegnerF}.
    Let small reals $\varepsilon,\delta,\nu$ satisfy Condition $1$ and to be determined.
    Let $M \in \Z_{+}$ satisfy $\frac{3}{5}\delta<(1-6\varepsilon)^{M}<\frac{4}{5}\delta$; such $M$ must exist as long as $\varepsilon<\frac{1}{24}$.
    Leave $N$ to be determined, and let $L_0$ be large enough with $L_{0} \geq \max\left\{C_{\delta,\varepsilon}, C_{\varepsilon,\delta,N}\right\}$, where $C_{\delta,\varepsilon}$ is the constant in Proposition \ref{prop:basecase} and $C_{\varepsilon,\delta,N}$ is the constant in Lemma \ref{lem:WegnerF} (with $N_1=N$).
     For $k>0$, let $L_{k}$ be dyadic numbers satisfying Condition $3$. Fix $\overline{\lambda}\in \left[0, \exp(-L_{M}^{\delta})\right]$.
     
    When $k=0,1,\cdots,M$, let $\mathcal{O}_{k}=\bigcup_{a \in \left\lceil \varepsilon^{-1} \right\rceil \Z^3} o_{a}$, where $o_{a}$ is the open unit ball centered at $a$. Then Statement 1, 3, 4 hold. Let $m_{k}:=L_{k}^{-\delta}$. 
    Proposition \ref{prop:basecase} implies Statement $5$ for $k=1,2,\cdots,M$.
    
    We now prove by induction for $k>M$.
    Assume that Statement $1$ to $6$ hold for all $k'<k$.

    For any $0< k' < k$, by Lemma \ref{lem:propdr}, any bad dyadic $L_{k'}$-cube $Q$ must contain a bad $L_{k'-1}$-cube.
For any $0 < i \leq k$, and a bad $L_{k-i}$-cube $Q'\subset Q$, we call $Q'$ a \emph{hereditary bad $L_{k-i}$-subcube of $Q$}, if there exists a sequence $Q'=\overline Q_i \subset \overline Q_{i-1} \subset \cdots \subset \overline Q_1 \subset Q$, where for each $j=1, \cdots, i$, $\overline Q_j$ is a bad $L_{k-j}$-cube. We also call such sequence $\{\overline Q_{j}\}_{1\leq j \leq i}$ a hereditary bad chain of length $i$.
Note that the set of hereditary bad chains of $Q$ is $V_{\froz_{k-1}\cap Q}$-measurable. 
    \begin{cla}  \label{cla:readyprob}
    When $\varepsilon$ is small enough, there exists $N'$ depending on $M,\kappa,\delta,\varepsilon$, such that,
    for any dyadic $L_k$-cube $Q$,
    \begin{equation}
        \prob\left[\text{$Q$ has no more than $N'$ hereditary bad chain of length $M$}\right] \geq 1-L_{k}^{-10}.
    \end{equation}
    \end{cla}
    \begin{proof}
         Writing $N'=(N'')^{M}$, we have
         \begin{align}
             \begin{split}
             & \prob[\text{$Q$ has more than $N'$ hereditary bad chain of length $M$}]\\
             \leq & \sum_{\substack{Q'\subset Q\\\text{$Q'$ is a dyadic $L_{j}$-cube}\\ k-M<j\leq k}} \prob[\text{$Q'$ contains more than $N''$ bad $L_{j-1}$-subcubes}].
             \end{split}
        \end{align}
        We can use inductive hypothesis to bound this by
        \begin{align}
            \begin{split}
              & \sum_{\substack{Q'\subset Q\\\text{$Q'$ is a dyadic $L_{j}$-cube}\\ k-M<j\leq k}} \left(\frac{L_{j}}{L_{j-1}}\right)^{CN''} (L_{j-1}^{-\kappa})^{cN''}
             \leq  \sum_{k-M<j\leq k} \left(\frac{L_{k}}{L_{j}}\right)^{C} \left(\frac{L_{j}}{L_{j-1}}\right)^{CN''} (L_{j-1}^{-\kappa})^{cN''}\\
             \leq & C M L_{k}^{C} \max_{k-M<j\leq k}L_{j-1}^{(C\varepsilon-c\kappa)N''}
             \leq  C M L_{k}^{C} (L_{k}^{(C\varepsilon-c\kappa)N''}+L_{k}^{(C\varepsilon-c\kappa)\delta N''}).
             \end{split}
         \end{align}
         Here we used that $L_{k-M} > L_{k}^{\frac{\delta}{2}}$ in the last inequality.
         The claim follows by taking $\varepsilon$ sufficiently small (depending on $\kappa$) and $N''$ large enough (depending on $M,\kappa,\delta,\varepsilon$).
    \end{proof}
    Now we let $N:=1000N'$.
    We call a dyadic $L_{k}$-cube $Q$ \emph{ready} if $Q$ has no more than $N'$ hereditary bad chain of length $M$. The event that $Q$ is ready is $V_{\froz_{k-1} \cap Q}$-measurable.
    
    Suppose $Q$ is an $L_k$-cube and is ready.
    Let $Q'''_{1},\cdots,Q'''_{N'}\subset Q$ be a complete list of all hereditary bad $L_{k-M}$-subcubes of $Q$.
    Let $Q''_{1},\cdots,Q''_{N'} \subset Q$ be
    the corresponding bad $L_{k-1}$-cubes, such that $Q'''_{i} \subset Q''_{i}$ for each $i=1,2,\cdots,N'$.
    These cubes are chosen in a way such that $\{Q''_{1},\cdots,Q''_{N'}\}$ contains all the bad $L_{k-1}$-cubes in $Q$.
    
    By Lemma \ref{lem:scov}, we can choose a dyadic scale $L'$ satisfying
    \begin{equation}\label{eq:def-L'}
        L_{k}^{1-3\varepsilon} \leq L' \leq L_{k}^{1-2\varepsilon}
    \end{equation}
    and disjoint $L'$-cubes $Q'_{1},\cdots,Q'_{N'}\subset Q$ such that, for every $Q''_{i}$, there is a $Q'_{j}$ such that $Q''_{i}\subset Q'_{j}$ and $\dist(Q''_{i},Q \setminus Q'_{j})\geq \frac{L'}{8}$. For each $i=1,2,\cdots,N'$, we let $O_{Q,i}$ be the ball in $\R^3$, with the same center as $Q'_{i}$ and with radius $L_{k}^{1-2\varepsilon}$. 
    We can choose $O_{Q,i},Q''_{i},Q'''_{i}$ in a $V_{\froz_{k-1}\cap Q}$-measurable way. 
    
    Now we let $\froz_{k}$ be the union of $\froz_{k-1}$ and balls $O_{Q,1},\cdots,O_{Q,N'}$, for each ready $L_{k}$-cube $Q$; i.e. 
    \begin{equation}\label{eq:defF}
        \froz_{k}:=\froz_{k-1} \cup \left(\bigcup_{\text{$Q$ is an $L_k$-cube and is ready}}\left( \bigcup_{i=1}^{N'} O_{Q,i} \right)\right),
    \end{equation}
and let $m_{k}=m_{k-1}-L_{k-1}^{-\nu}$.
From induction hypothesis we have $m_k\geq L_{k-1}^{-\delta}-L_{k-1}^{-\nu}\geq L_k^{-\delta}$.

    We now verify Statement $2$ to $6$.
    First note that Statement $4$ and $6$ hold for $k$ by the above construction. 
    \begin{cla}
    Statement $2$ and $3$ hold for $k$.
    \end{cla}
    \begin{proof}
         From \eqref{eq:defF}, we let $\tilde{\froz}_{k'}:=\bigcup_{\text{$Q$ is an $L_{k'}$-cube and is ready}} \bigcup_{i=1}^{N'} O_{Q,i}$ for $k'>M$.
         Then we have that $\froz_{k}=\froz_{M} \cup \left(\bigcup_{k'=M+1}^{k}\tilde{\froz}_{k'}\right)$, and we claim that
\begin{enumerate}
    \item $\froz_{M}$ is $(2\varepsilon)^{-1}$-unitscattered,
    \item $\tilde{\mathcal{O}}_{k'}$ is an $(N,L_{k'}^{1-2\varepsilon},2\varepsilon)$-scattered set for each $k'>M$.
\end{enumerate}
By these two claims, Statement $2$ holds by Condition $3$.
  
    Now we check these two claims.
    For the first one, just note that $\mathcal{O}_{M}=\bigcup_{a \in \left\lceil \varepsilon^{-1} \right\rceil \Z^3} o_{a}$, then use Definition \ref{def:F}.
    For the second one, when $k'>M$ the set $\tilde{\froz}_{k'}$ is the union of $N'$ balls $O_{Q,1},O_{Q,2},\cdots,O_{Q,N'}$ for each ready $L_{k'}$-cube $Q$, and each ball $O_{Q,i}$ has radius $L_{k'}^{1-2\varepsilon}$. Denote the collection of dyadic $L_{k'}$-cubes by $\mathcal{Q}_{k'}:=\left\{Q_{\frac{L_{k'}}{2}}(a):a \in \frac{L_{k'}}{4} \Z^3\right\}$. We can divide $\mathcal{Q}_{k'}$ into at most $1000$ subsets $\mathcal{Q}_{k'}=\bigcup_{t=1}^{1000}\mathcal{Q}^{(t)}_{k'}$, such that any two $L_{k'}$-cubes in the same subset have distance larger than $L_{k'}$, i.e.
\begin{equation}\label{eq:distance-differ}
    \text{$\dist(Q, Q') \geq L_{k'}$ for any $t \in \left\{1,2,\cdots,1000\right\}$ and any $Q\neq Q' \in \mathcal{Q}^{(t)}_{k'}$. }
\end{equation}
For each $1 \leq t \leq 1000$ and $1 \leq j \leq N'$, let $\mathfrak{O}^{(t,j)}_{k'}=\left\{O_{Q,j}:\text{$Q$ is ready and $Q \in \mathcal{Q}^{(t)}_{k'}$}\right\} $. Then for any two $O \neq O' \in \mathfrak{O}^{(t,j)}_{k'}$, by \eqref{eq:distance-differ}, we have
\begin{equation}
    \dist(O, O') \geq 
    L_{k'}-2L_{k'}^{1-2\varepsilon}
    \geq
    L_{k'}^{1-4\varepsilon^2}=(\radius(O))^{1+2\varepsilon}= (\radius(O'))^{1+2\varepsilon}. 
\end{equation}
From Definition \ref{def:F}, we have that $\tilde{\mathcal{O}}_{k'}=\bigcup_{1 \leq t \leq 1000,1 \leq j \leq N'} \left(\bigcup \mathfrak{O}^{(t,j)}_{k'}\right)$ is an $(N,L_{k'}^{1-2\varepsilon},2\varepsilon)$-scattered set since $N=1000 N'$. Thus the second claim holds.

Finally, since $\radius (O_{Q,i})=L_{k'}^{1-2\varepsilon}< \diam(Q)^{1-\varepsilon}$ for any ready $L_{k'}$-cube $Q$ and $1\leq i \leq N'$, we have that $\mathcal{O}_{k}$ is $(1,2\varepsilon)$-normal in any $L_{k}$-cube. Hence Statement $3$ holds.
\end{proof}
Now it remains to check Statement $5$ for $k$.
\begin{cla}  \label{cla:k51}
    If $Q$ is an $L_{k}$-cube and $Q$ is ready, then for any $1 \leq i \leq N'$, we have
    \begin{equation}
        \exp(c L_{k-1}^{1-\delta}) \|u\|_{\ell^{\infty}\left(Q'_{i}\setminus \bigcup_{j=1}^{N'} Q''_{j}\right)} \leq \|u\|_{\ell^{2}(Q'_{i})} \leq (1+\exp(-c L_{k-M}^{1-\delta}))\|u\|_{\ell^{2}\left(Q'_{i}\cap \bigcup_{j=1}^{N'}Q'''_{j}\right)},
    \end{equation}
    for any
    $\lambda \in \R$ with $|\lambda-\overline{\lambda}|\leq \exp(-2 L_{k-1}^{1-\varepsilon})$, and any $u:Q'_{i}\rightarrow \R$ with $H_{Q'_{i}}u=\lambda u$.
\end{cla}
\begin{proof}
     If $a\in Q'_{i}\setminus \bigcup_{j=1}^{N'}Q'''_{j}$, then there is a $j'=1,\cdots,M$ and a good $L_{k-j'}$-cube $Q''\subset Q'_{i}$ with $a\in Q''$ and $\dist(a,Q_{i}'\setminus Q'')\geq \frac{1}{8}L_{k-j'}$. Moreover, if $a\in Q'_{i}\setminus \bigcup_{j=1}^{N'} Q''_{j}$, then we can take $j'=1$. By the definition of good and Lemma \ref{lem:app-pertu}, 
     \begin{equation}
         |u(a)|\leq 2\exp\left(L_{k-j'}^{1-\varepsilon} -\frac{1}{8}m_{k-j'}L_{k-j'}\right) \|u\|_{\ell^{1}(Q'_{i})}\leq \exp(-c L_{k-j'}^{1-\delta})\|u\|_{\ell^{2}(Q'_{i})}.
     \end{equation}
     In particular, we see that 
     \begin{equation}
         \|u\|_{\ell^{\infty}\left(Q'_{i}\setminus \bigcup_{j=1}^{N'} Q''_{j}\right)}\leq \exp(-c L_{k-1}^{1-\delta}) \|u\|_{\ell^{2}(Q'_{i})}
     \end{equation}
     and
     \begin{equation}
         \|u\|_{\ell^{\infty}(Q'_{i}\setminus \bigcup_{j=1}^{N'}Q'''_{j})}\leq \exp(-c L_{k-M}^{1-\delta})\|u\|_{\ell^{2}(Q'_{i})}.
     \end{equation}
     These together imply the claim.
\end{proof}
\begin{cla}  \label{cla:k52}
    If $Q$ is an $L_{k}$-cube, and for any $1 \leq i \leq N'$, $\event_{i}(Q)$ denotes the event that 
    \begin{equation}
        \text{Q is ready and $\|(H_{Q'_{i}}-\overline{\lambda})^{-1}\| \leq \exp(L_{k}^{1-4\varepsilon})$},
    \end{equation}
    then $\prob[\event_{i}(Q)]\geq 1-L_{k}^{C\varepsilon-\varepsilon_{0}}$.
\end{cla}
    \begin{proof}
         Recall that the event where $Q$ is ready is $V_{\froz_{k-1} \cap Q}$-measurable, and subcubes $Q'_{i}$'s are also $V_{\froz_{k-1} \cap Q}$-measurable. 
         Assuming $\varepsilon>5\delta$, we apply Lemma \ref{lem:WegnerF} with $2\varepsilon>\delta>0$, $N_1=N'$, and to the cube $Q'_{i}$ with scales $L' \geq L_{k}^{1-4\varepsilon} \geq L_{k}^{1-5\varepsilon} \geq L_{k-1} \geq L_{k-1}^{1-2\delta} \geq 2 L_{k-1}^{1-\varepsilon}$ (recall that $L'$ is the scale chosen above satisfying \eqref{eq:def-L'}), defects $\left\{Q''_{j}:Q''_{j} \subset Q'_{i}\right\}$ , $G=\bigcup_{1\leq j \leq N':Q'''_{j} \subset Q'_{i}}Q'''_{j}$, and $E=\froz_{k-1}$. 
         Note that $L_{k}^{\frac{\delta}{2}} < L_{k-M} < L_{k}^{\frac{9\delta}{10}}$.
        Condition 9 of Lemma \ref{lem:WegnerF} is given by Claim \ref{cla:k51}. By Claim \ref{cla:readyprob} this claim follows.
    \end{proof}
\begin{cla} \label{cla:k53}
    If $Q$ is an $L_{k}$-cube and $\event_{1}(Q),\cdots,\event_{N'}(Q)$ hold, then $Q$ is good.
\end{cla}
    \begin{proof}
         We apply Lemma \ref{lem:propdr} to the cube $Q$ with small parameters $\varepsilon>\nu >0$, scales $L_{k}\geq L_{k}^{1-\varepsilon}\geq L' \geq L_{k}^{1-3\varepsilon}\geq L_{k}^{1-4\varepsilon}\geq L_{k-1}\geq L_{k-1}^{1-\varepsilon}$, and defects $Q'_{1},\cdots,Q'_{N'}$. We conclude that 
         \begin{equation}
             |(H_{Q}-\overline{\lambda})^{-1}(a,b)|\leq \exp(L_{k}^{1-\varepsilon}-m_{k}|a-b|).
         \end{equation}
         Since $Q'_{i} \subset \froz_{k}$ when $Q$ is ready,
         the events $\mathcal{E}_{i}(Q)$ are $V_{\mathcal{O}_{k}\cap Q}$-measurable, thus $Q$ is good. 
    \end{proof}

By combining Claim \ref{cla:k52}, Claim \ref{cla:k53}, and letting $C\varepsilon<\varepsilon_{0}-\kappa$, we have that Statement $5$ holds for $k$. Thus the induction principle proves the theorem.
\end{proof}
\begin{proof}[Proof of Theorem \ref{thm:resonentexpo}]
Apply Theorem \ref{thm:multiscale} with any $\varepsilon_{*}<\frac{\kappa}{100}$, then there are $\left\{L_{k}\right\}_{k \in \Z_{\geq 0}}$, $\left\{m_{k}\right\}_{k \in \Z_{\geq 0}}$, $\varepsilon$, $\delta$, $\nu$, $N$ and $M$ such that the statements of Theorem \ref{thm:multiscale} hold. Let $k_*\in\Z_+$ be large enough with $k_*\geq M+2$ and let $L_{*}= L_{k_*}$. Fix dyadic scale $L\geq L_{*}$, and let $k$ be the maximal integer such that $L \geq L_{k+1}$. Then $L_{k}^{1+6\varepsilon} \leq L_{k+1} \leq L < L_{k+2} \leq L_{k}^{1+15\varepsilon}$. Denote
     \begin{equation}
         \mathcal{Q}:=\left\{Q:\text{$Q$ is a dyadic $L_{k}$-cube and $Q \cap Q_{L}\neq \emptyset$}\right\}. 
     \end{equation}
 Then $Q_{L} \subset \bigcup_{Q\in \mathcal{Q}}Q$ and $|\mathcal{Q}|\leq 1000\left(\frac{L}{L_{k}}\right)^{3} \leq L_{k}^{100\varepsilon}\leq L_{k}^{100 \varepsilon_{*}}$. By elementary observations, for any $a \in Q_{L}$, there is a $Q\in \mathcal{Q}$ such that $a \in Q$ and $\dist(a,Q_{L}\setminus Q) \geq \frac{1}{8}L_{k}$.
 Fix $\lambda \in [0,\exp(-L_{M}^{\delta})]$. For each $Q\in \mathcal{Q}$, define $A_{Q}$ to be the following event:
     \begin{equation}
         \text{$|(H_{Q}-\lambda)^{-1}(a,b)|\leq \exp(L_{k}^{1-\varepsilon}-m_{k}|a-b|)$ for each $a,b \in Q$}.
     \end{equation}
     By Lemma \ref{lem:propdr}, $\bigcap_{Q \in \mathcal{Q}} A_{Q}$ implies 
     \begin{equation}
         |(H_{Q_L}-\lambda)^{-1}(a,b)|\leq \exp(L^{1-\varepsilon}-m|a-b|), \forall a,b \in Q_L,
     \end{equation}
     where $m=m_{k}-L_{k}^{-\delta}$. 
     Note that for $k\geq k_*-1\geq M+1$ we have
    \begin{equation}
        m=m_{k}-L_{k}^{-\delta} \geq L_{k_*-2}^{-\delta}-L_{k_*-2}^{-\nu}-\cdots-L_{k-1}^{-\nu}-L_{k}^{-\delta}>\delta_{0}
    \end{equation}     
     for some $\delta_{0}>0$ independent of $k$.
     Here the inequalities are by Condition $4$ and Statement $6$ in Theorem \ref{thm:multiscale}, and the fact that $L_{k}$ increases super-exponentially and $k_*$ is large enough.
     
     By Theorem $\ref{thm:multiscale}$, for each $Q \in \mathcal{Q}$ we have
     \begin{equation}
         \prob[A_{Q}]\geq 1-L_{k}^{-\kappa}.
     \end{equation}
     Thus
     \begin{equation}
         \prob\left[\bigcap_{Q \in \mathcal{Q}} A_{Q}\right] \geq 1-|\mathcal{Q}|L_{k}^{-\kappa} \geq 1-L_{k}^{-\kappa+ 100\varepsilon_{*}}.
     \end{equation}
     Hence our theorem follows by letting $\kappa_{0}=\frac{\kappa - 100\varepsilon_{*}}{1+15\varepsilon}$ and  $\lambda_{*}=\min\left\{\delta_{0},\exp(-L_{M}^{\delta}),\varepsilon\right\}$.
\end{proof}

\section{Polynomial arguments on triangular lattice}   \label{sec:pol}
The goal of this section is to prove Theorem \ref{thm:bou}, which is a triangular lattice version of \cite[Theorem (A)]{buhovsky2017discrete}.
Our proof closely follows that in \cite{buhovsky2017discrete}, which employs the polynomial structure of $u$ and the Remez inequality, and a Vitalli covering argument.

\subsection{Notations and basic bounds}
Before starting the proof, recall Definition \ref{def:basic2D} for some basic geometric objects.
Here we need more notations for geometric patterns in $\Lambda$.
\begin{defn}  \label{defn:pt}
We denote $\gamma:=\xi + \eta=\left(-\frac{1}{2}, \frac{\sqrt{3}}{2}\right)$. For each $b=s\xi+t\eta \in \Lambda$, we denote $\xi(b):=s$ and $\eta(b):=t$.
For $a \in \Lambda$ and $m\in \Z_{\geq 0}$, we denote the \emph{$\xi$-edge}, \emph{$\eta$-edge}, and \emph{$\gamma$-edge} of $T_{a;m}$ to be the sets 
\begin{equation}
    \begin{split}
        &\left\{a-m\eta+s\xi:-2m\leq s \leq m\right\}\cap \Lambda,\\ 
        &\left\{a+m\xi+t\eta:-m \leq t \leq 2m\right\} \cap \Lambda,\\ 
        &\left\{a-m\xi +s \xi+ s\eta:-m\leq s \leq 2m\right\}\cap \Lambda
    \end{split}
\end{equation}
respectively, each containing $3m+1$ points.
In this section, an \emph{edge} of $T_{a;m}$ means one of its $\xi$-edge, $\eta$-edge and $\gamma$-edge.

For $a \in \Lambda$ and $m,\ell \in \Z_{\geq 0}$, denote $P_{a;m,\ell}:=\left\{a+s\xi+t\eta:-\ell \leq t \leq 0, -m+t \leq s \leq 0 \right\} \cap \Lambda$, a trapezoid of lattice points. Especially, when $\ell=0$, $P_{a;m,\ell}=\left\{a+s\xi:0 \leq s \leq m\right\}$ is a segment parallel to $\xi$.
The $ \emph{lower edge}$ of $P_{a;m,\ell}$ is defined to be the set $P_{a-\ell \eta;m+\ell,0}$, and the $ \emph{upper edge}$ of $P_{a;m,\ell}$ is defined to be the set $P_{a;m,0}$. The $\emph{left leg}$ of $P_{a;m,\ell}$ is the set $\left\{a+t \eta:-\ell \leq t \leq 0\right\}\cap \Lambda$, and the $\emph{right leg}$ of $P_{a;m,\ell}$ is the set $\left\{a-m\xi-t\gamma: 0 \leq t \leq \ell \right\}\cap \Lambda$.

See Figure \ref{illu} for an illustration of $T_{a;m}$ and $P_{a;m,\ell}$.
\end{defn}
\begin{figure}
    \centering
    \includegraphics[width=\textwidth]{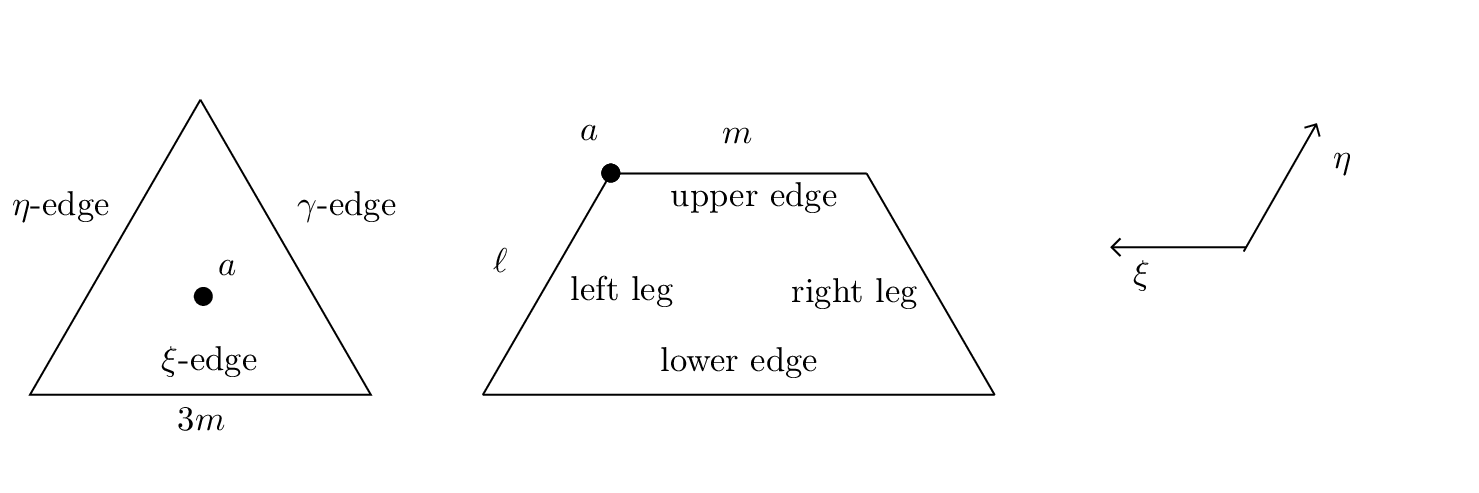}
    \caption{$T_{a;m}$ is the set of lattice points in the triangle region; $P_{a;m,\ell}$ is the set of lattice points in the trapezoid region.}
    \label{illu}
\end{figure}

The following lemma can be proved using a straight forward induction.
\begin{lemma}   \label{lem:tra}
Let $R, S \in \R_+$, $a \in \Lambda$, and $m \in \Z_+$.
Suppose $u:\Lambda\rightarrow \R$ satisfies
\begin{equation}\label{eq:har}
     |u ( b ) +u ( b-\xi ) +u ( b+\eta )| \leq R
\end{equation}
for any $b\in T_{a;m}$ with $\eta(b)-\xi(b)<m$, and $|u| \leq S$ on one of three edges of $T_{a;m}$. Then $|u(b)| \leq 2^{3m}S+(2^{3m}-1)R$ for each $b \in T_{a;m}$.
\end{lemma}

\begin{proof}
By symmetry, we only need to prove the result when $|u| \leq S$ on the $\xi$-edge of $T_{a;m}$.
Without loss of generality we also assume that $a=\boo$.

We claim that for each $k=0, 1, \cdots, 3m$, $|u(b)| \leq 2^{k}S+(2^{k}-1)R$ for any $b \in T_{\boo;m}$ with $\eta(b)=k-m$.
We prove this claim by induction on $k$. The base case of $k=0$ holds by the assumptions.
We suppose that the statement is true for $0, 1, \cdots, k$.
For any $b \in T_{\boo;m}$ with $\eta(b)=k-m$ and $\xi(b)>k-2m$, we have
$b,b-\xi \in T_{\boo;m}$ and $\eta ( b ) = \eta  (b-\xi) = k-m$. By \eqref{eq:har} and the induction hypothesis, 
\begin{equation}
    |u(b+\eta)|\leq |u(b)|+|u(b-\xi)| +R \leq 2(2^{k}S+(2^k-1)R)+R=2^{k+1}S+(2^{k+1}-1)R.
\end{equation}
Then our claim holds by induction, and the lemma follows from our claim.
\end{proof}

\subsection{Key lemmas via polynomial arguments}
In this subsection we prove two key results, Lemma \ref{pol} and \ref{lem:pra} below, which are analogous to \cite[Lemma 3.4]{buhovsky2017discrete} and \cite[Lemma 3.6]{buhovsky2017discrete}, respectively.
We will use the Remez inequality \cite{remez1936propriete}.
More precisely, we will use the following discrete version as stated and proved in \cite{buhovsky2017discrete}.
\begin{lemma} [\protect{\cite[Corollary 3.2]{buhovsky2017discrete}}] \label{Remez}
  Let $d, \ell \in \Z_{+}$, and $p$ be a polynomial with degree no more than $d$.
  For $M\in \R_+$, suppose that $|p| \leq M$ on at least $d+\ell$ integer points on a closed interval $I$, then on $I$ we have
  \begin{equation}
      |p| \leq \left(\frac{4|I|}{\ell}\right)^d M.
  \end{equation}
\end{lemma}
Now we prove the following bound of $|u|$ in a trapezoid, given that $|u|$ is small on the upper edge and on a substantial fraction of the lower edge of the trapezoid.
\begin{lemma}\label{pol}
Let $R, K \in \R_+$, $\ell,m \in \Z_+$ with $\ell \leq \frac{m}{10}$, and $a\in \Lambda$.
There is a universal constant $C_5 > 1$ (independent of $a,m,\ell,K,R$), such that the following is true.
Suppose $u:P_{a;m,\ell} \rightarrow \R$ is a function satisfying that:
\begin{enumerate}
    \item \eqref{eq:har} holds for any $b\in P_{a-\eta;m,\ell-1}$,
    \item $|u| \leq K$ on the upper edge of $P_{a;m,\ell}$,
    \item $|u| \leq K$ for at least half of the points in the lower edge of $P_{a;m,\ell}$. 
\end{enumerate}
Then $|u| \leq C_5^{\ell+m}(K+R)$ in $P_{a;m,\ell}$.
\end{lemma}

\begin{proof}
We assume without loss of generality that $a=\mathbf{0}$.
We first claim that there is a function $v:P_{\mathbf{0};m,\ell}\rightarrow \R$ satisfying the following four conditions:
\begin{enumerate}
    \item $v=0$ on $\left\{-t\eta : 1 \leq t \leq \ell\right\}$.
    \item $v=u$ on $P_{\mathbf{0};m,0}$.
    \item For each point $b \in P_{-\eta;m,\ell-1}$,
    \begin{equation}\label{eq:okm}
    v(b)+v(b-\xi)+v(b+\eta)=u(b)+u(b-\xi)+u(b+\eta).
\end{equation}
    \item $\|v\|_{\infty} \leq 4^{\ell+m}(K+R)$.
\end{enumerate}

We construct the function $v$ by first defining it on $\left\{-t\eta: 0 \leq t \leq \ell \right\}$ and $P_{\mathbf{0};m,0}$, then iterating \eqref{eq:okm} line by line.
More precisely, for $-m \leq s\leq 0$, we let $v(s\xi)=u(s\xi)$.
For each $t =  -1, -2 \cdots, -\ell$, we first define $v(t \eta)=0$, then define 
\begin{equation}
    v((s-1)\xi+t\eta):=-v(s\xi+t\eta)-v(s\xi+(t+1)\eta)+u (s\xi+t\eta) +u ((s-1)\xi+t\eta) +u (s\xi+(t+1)\eta)
\end{equation} 
for all $-m+t+1 \leq s\leq 0$.
Then we have defined $v(s\xi+t\eta)$ for $-\ell \leq t \leq 0$ and $-m+t \leq s \leq 0 $.
By our construction, $v$ satisfies Condition 1 to 3.

Now we prove $v$ satisfies Condition 4.
First, \eqref{eq:okm} implies that $|v(b)+v(b-\xi)+v(b+\eta)| \leq R$ for any $b\in P_{-\eta;m,\ell-1}$.
Using this and $|v|\leq K$ on $P_{\mathbf{0};m,0}$, by an induction similar to that in the construction of $v$, we can prove that
\begin{equation}
    |v(-s\xi-t\eta)| \leq 2^{s+t}K+(2^{s+t}-1)R
\end{equation}
for each $0 \leq t \leq \ell$ and $0 \leq s
\leq m+t $. In particular, $|v| \leq (K+R)4^{\ell+m}$ on any point in trapezoid $P_{\mathbf{0};m,\ell}$, and $v$ satisfies Condition 4.

Let $w:=u-v$, then $w=0$ on $P_{\mathbf{0};m,0}$ and $w(b)+w(b-\eta)+w(b-\gamma)=0$ for each $b \in P_{\mathbf{0};m,\ell-1} $. Also, $|w| \leq (K+R)4^{\ell+m}+K \leq (K+R)5^{\ell+m}$ on at least half of points in the lower edge of $P_{\mathbf{0};m,\ell}$. Since $\ell \leq \frac{m}{10}$, we have
\begin{equation}\label{poiu}
    \left| \left\{0 \leq s \leq m+\ell:|w(-s\xi-\ell \eta)| \leq (K+R)5^{\ell+m}\right\}\right|\geq \frac{m+\ell}{2} \geq 5\ell.
\end{equation}

We claim that for each $0 \leq t \leq \ell$, 
if we denote
\begin{equation}
    g_{t}(s)=(-1)^{s}w(-s\xi-t\eta), \; \forall 0 \leq s \leq m+t,s\in \Z,
\end{equation}
then $g_t$ is a polynomial of degree at most $t$.
We prove the claim by induction on $t$.
For $t=0$, this is true since $w=0$ on the upper edge of $P_{\mathbf{0};m, \ell}$.
Suppose the statement is true for $t$, then since
\begin{multline}
    g_{t+1}(s)-g_{t+1}(s+1) = 
    (-1)^{s}w(-s\xi-(t+1)\eta)-(-1)^{s-1}w((-s-1)\xi-(t+1)\eta)
    \\
    =-(-1)^{s}w(-s\xi-t\eta)=-g_{t}(s),
\end{multline}
for all $0 \leq s \leq m+t, s\in \Z$, we have that $g_{t+1}$ is a polynomial of degree at most $t+1$.
Hence our claim holds.

In particular, $g_{\ell}(s) = (-1)^s w(-s\xi-\ell \eta)$ is a polynomial of degree at most $\ell$.
Hence by \eqref{poiu} and Lemma \ref{Remez}, there exists a constant $C>0$ such that
\begin{equation}
    |w(-s\xi-\ell \eta)| \leq 5^{\ell+m}C^{\ell}(K+R)
\end{equation}
for each $0 \leq s \leq m+\ell$.
Thus on the lower edge of $P_{\mathbf{0};m,\ell}$,
\begin{equation}
    |u| \leq |w|+|v| \leq 5^{\ell+m}C^{\ell}(K+R)+4^{\ell+m}(K+R) \leq (5 C +4)^{\ell+m}(K+R),
\end{equation}
Finally, by an inductive argument similar to the proof of Lemma \ref{lem:tra}, and letting $C_5=10 C+8$, we get
\begin{equation}
    |u| \leq 2^{\ell}(5C+4)^{\ell+m}(K+R)+(2^{\ell}-1)R\leq C_5^{\ell+m}(K+R) 
\end{equation}
in $P_{\mathbf{0};m,\ell}$.
\end{proof}

Our next lemma is obtained by applying Lemma \ref{pol} repeatedly.

\begin{lemma}\label{lem:pra}
Let $m,\ell \in \Z_+$ with $\ell\leq m \leq 2\ell$, $K, R \in \R_+$, and $a \in \Lambda$.
Let $u:P_{a;m,\ell}\rightarrow \R$ be a function satisfying \eqref{eq:har} for each $b\in P_{a-\eta;m,\ell-1}$.
If $|u| \leq K$ on $P_{a;m,0}$ and $  \left|\left\{b\in P_{a;m,\ell}:|u(b)|>K\right\}\right| \leq \frac{1}{10^5}m\ell$,
then $|u| \leq (K+R)C_{6}^{\ell}$ in $P_{a;m, \left\lfloor \frac{\ell}{2} \right\rfloor}$, where $C_6>1$ is a universal constant.
\end{lemma}

\begin{proof}
If $\ell \leq 120$, then the result holds trivially since $\frac{1}{10^5}m\ell \leq \frac{2}{10^5}\ell^2 <1$.
From now on we assume that $\ell \geq 120$, and let $C_6 = C_5^{1000}$ where $C_5$ is the constant in Lemma \ref{pol}.

For each $k = 0,1,\cdots,29$, we choose an $l_{k} \in \left\{\left\lfloor \frac{2k}{60}\ell \right\rfloor, \left\lfloor \frac{2k}{60}\ell \right\rfloor+1, \cdots,\left\lfloor \frac{2k+1}{60} \ell \right\rfloor-1\right\}$ such that
\begin{equation}
    \left|\left\{b:|u(b)| \leq K\right\} \cap P_{a-l_{k}\eta;m+l_{k},0}\right| \geq \frac{1}{2}(m+l_k).
\end{equation}
Such $l_{k}$ must exist, since otherwise, 
\begin{equation}
    \left| \left\{b\in P_{a;m,\ell}:|u(b)|> K\right\}\right| > \frac{1}{2}\cdot\frac{1}{60}m\ell >  \frac{1}{10^5}m\ell,
\end{equation}
which contradicts with an assumption in the statement of this lemma.
In particular, we can take $l_{0}=0$.

From the definition, we have $l_{k+1}-l_{k} \leq \frac{1}{20}\ell \leq \frac{1}{20}m$ and $l_{k+1}-l_{k} \geq \frac{1}{60}\ell \geq \frac{1}{120}m$.
For each $k=0,1,\cdots,28$, 
let $P_k=P_{a-l_{k}\eta;m+l_{k},l_{k+1}-l_{k}}$,
then we claim that $|u| \leq C_{6}^{l_{k+1}}(K+R)$ on $P_k$.

We prove this claim by induction on $k$.
For $k=0$, we use Lemma \ref{pol} for $P_{a;m,l_{1}}$ to get
\begin{equation}
    |u| \leq (K+R)C_5^{l_{1}+m} \leq (K+R)C_5^{121l_{1}} \leq (K+R)C_{6}^{l_{1}}
\end{equation} in $P_0=P_{a;m,l_{1}}$.
Suppose the statement holds for $k$, then $|u| \leq (K+R)C_{6}^{l_{k+1}}$ in $P_{a-l_{k+1}\eta;m+l_{k+1},0}$ which is the upper edge of $P_{k+1}$.
We use Lemma \ref{pol} again for $P_{k+1}$, and get $|u| \leq (K+R)C_{6}^{l_{k+2}}$ in $P_{k+1}$.
Thus the claim follows.

Since $l_{29} \geq \frac{29}{30}\ell-1 \geq \left\lfloor \frac{1}{2}\ell \right\rfloor+1$ when $\ell \geq 120$, we have $P_{a;m,\left\lfloor \frac{\ell}{2} \right\rfloor}\subset\bigcup_{k=0}^{28}P_k$. Then the lemma is implied by this claim.
\end{proof}

\subsection{Proof of Theorem \ref{thm:bou}}
In this subsection we finish the proof of Theorem \ref{thm:bou}.
The key step is a triangular analogue of \cite[Corollary 3.7]{buhovsky2017discrete} (Lemma \ref{lem:gro} below); then we finish using a Vitalli covering argument.

\begin{proof}[Proof of Theorem \ref{thm:bou}]
Let $\epsilon_1 = \frac{1}{10^{18}}$, and $C_4=6C_6>6$ where $C_6$ is the constant in Lemma \ref{lem:pra}.
We note that now Theorem \ref{thm:bou} holds trivially when $n<10^9$, so below we assume that $n\geq 10^9$.

We argue by contradiction, i.e. we assume that
\begin{equation}  \label{eq:boup1}
\left|\left\{b\in T_{\boo;n} : | u (b) | >K \right\} \right| \leq \epsilon_{1} n^2,
\end{equation}
where we take $K=C_{4}^{-n}|u(\mathbf{0})|$.

We first define a notion of triangles on which $|u|$ is ``suitably bounded''. 
For this, we let $R=C_{4}^{-n}|u(\mathbf{0})|$ as well, and we define a triangle $T_{a;m}\subset T_{\boo;\left\lfloor\frac{n}{2}\right\rfloor}$ as being \emph{good} if $m$ is even and $|u| \leq (K+R) \left(\frac{C_{4}}{3}\right)^{3m}$ on any point in $T_{a;m}$.

We choose points $a_i \in T_{\boo;\left\lfloor \frac{n}{20} \right\rfloor}$ for $1 \leq i \leq \left\lfloor \frac{n^2}{10^6} \right\rfloor$, such that each $T_{a_i,2} \subset T_{\boo;\left\lfloor \frac{n}{20} \right\rfloor}$, and $T_{a_i,2}\cap T_{a_j,2}=\emptyset$ for any $i \neq j$.
Denote $S:=\left\{T_{a_i,2}:1 \leq i \leq \left\lfloor \frac{n^2}{10^6} \right\rfloor\right\}$.
By \eqref{eq:boup1}, for at least half of the triangles in $S$, $|u| \leq K$ on each of them.
Hence, there are at least $\frac{n^2}{10^{7}}$ good triangles in $S$.
Denote
\begin{equation}
    Q=\left\{a_{i}:1 \leq i \leq \left\lfloor \frac{n^2}{10^6} \right\rfloor,\text{ $T_{a_{i},2}$ is good}\right\}.
\end{equation}
For any $a \in Q$, let $l_{a}=\max\left\{l\in \Z_{+}:\text{$T_{a,l}$ is good and $T_{a,l}\subset T_{\boo;\left\lfloor\frac{n}{2}\right\rfloor}$}\right\}$.
Denote $X_a=T_{a;l_{a}}$ for each $a\in Q$.

If there exists $a \in Q$ with $l_a \geq \frac{n}{30}$, then this maximal triangle contains $\mathbf{0}$, and $|u(\mathbf{0})| \leq \left(\frac{C_4}{3}\right)^{3l_{a}}(K+R) \leq \left(\frac{C_4}{3}\right)^{n}(K+R) < |u(\mathbf{0})|$, which is impossible.
Hence $l_a \leq \frac{n}{30}$ for any $a\in Q$.
For any $a \in Q$, denote $Y_a := T_{a;4l_a}$.
Then $Y_{a} \subset T_{\boo;\left\lfloor\frac{n}{2}\right\rfloor}$.

We need the following result on good triangles.
\begin{lemma}  \label{lem:gro}
For any $m\in \Z_+$ and $a\in \Lambda$ the following is true.
Let $T_1=T_{a;2m}$, $T_2=T_{a;5m}$ and $T_3=T_{a;8m}$ (see Figure \ref{fig:duc_1} for an illustration).
If $T_3\subset T_{\boo;\left\lfloor\frac{n}{2}\right\rfloor}$, and $\left|\left\{b \in T_3:|u(b)|>K\right\} \right|\leq \frac{m^2}{10^6}$, and $T_{1}$ is good, then $T_2$ is also good.
\end{lemma}
We assume this result for now and continue our proof of Theorem \ref{thm:bou}.
We have that
\begin{equation}
    \left|\left\{b \in Y_{a}:|u(b)|>K\right\}\right| \geq \frac{l_{a}^2}{10^7},\;\;\forall a \in Q,
\end{equation}
since otherwise, by Lemma \ref{lem:gro} with $T_1=X_a$ and $T_3=Y_a$, there is a good triangle strictly containing $X_a$ and this contradicts with the maximal property of $X_a$. 

Finally we apply Vitalli's covering theorem to the collection of triangles $\left\{Y_{a}:a \in Q\right\}$.
We can find a subset $\tilde{Q}\subset Q$ such that $\left|\bigcup_{a \in \tilde{Q}}Y_{a}\right| \geq \frac{1}{16}|\bigcup_{a \in {Q}}Y_{a}|$, and $Y_a \cap Y_{a'}=\emptyset$ for any $a\neq a' \in \tilde{Q}$.
Hence 
\begin{equation}
    \left| \left\{a \in T_{\boo;\left\lfloor\frac{n}{2}\right\rfloor}: |u(a)| >K \right\} \right| \geq \frac{1}{10^7}\left|\bigcup_{a \in \tilde{Q}}Y_{a}\right| > \frac{1}{10^9} \left|\bigcup_{a \in {Q}}Y_{a}\right|.
\end{equation}
Since $Q \subset \bigcup_{a \in {Q}}Y_{a}$, we have $\left|\bigcup_{a \in {Q}}Y_{a}\right| \geq |Q| > \frac{n^2}{10^{7}}$, so $\left| \left\{a \in T_{\boo;\left\lfloor\frac{n}{2}\right\rfloor}: |u(a)| >K \right\} \right| > \frac{1}{10^{9}}\cdot \frac{n^2}{10^{7}} = \frac{n^2}{10^{16}}$.
This contradicts with our assumption \eqref{eq:boup1} since $\epsilon_{1}=\frac{1}{10^{18}}$.
\end{proof}

\begin{figure}
    \centering
    \includegraphics[width=\textwidth]{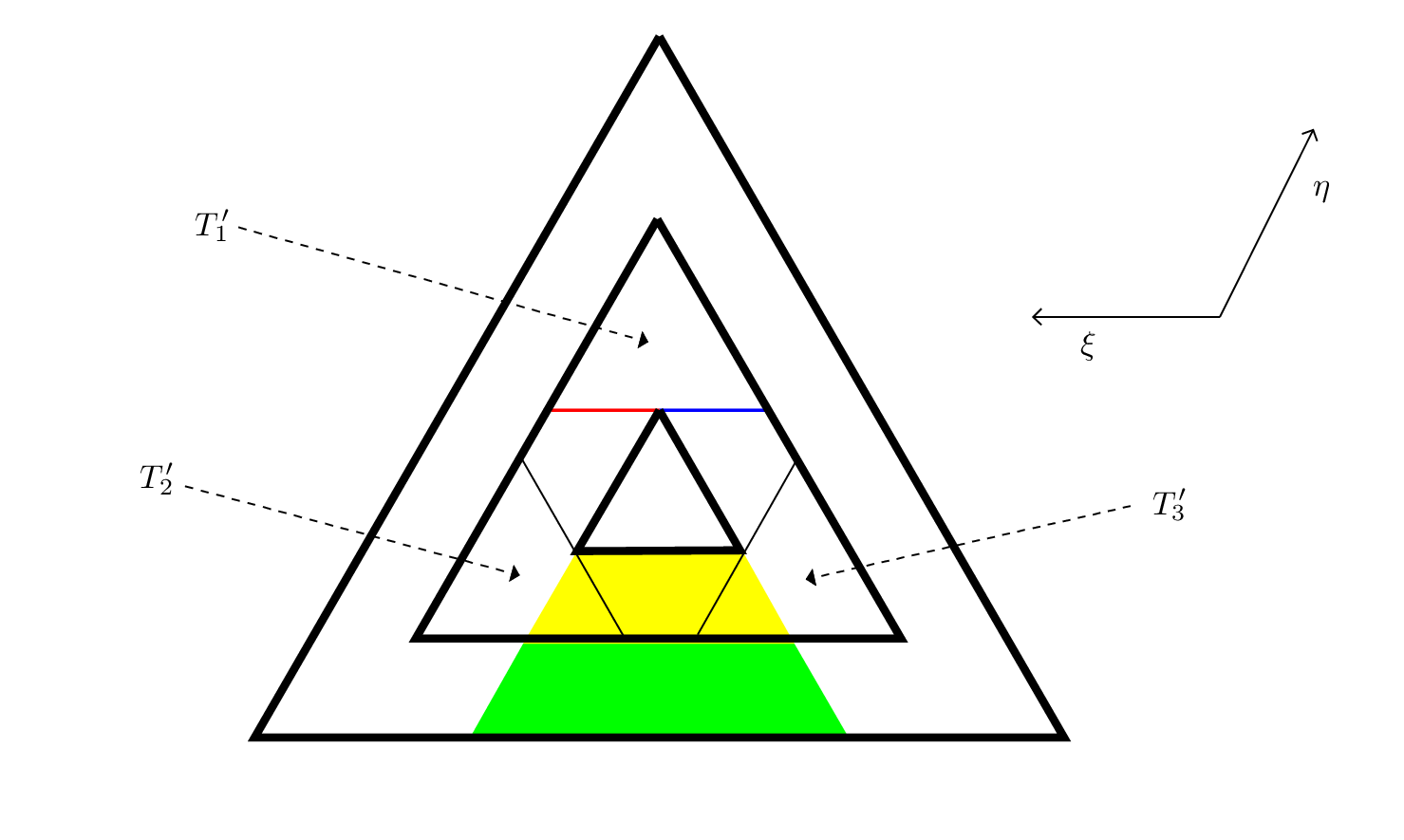}
    \caption{The thick lines indicate edges of $T_1$, $T_2$, and $T_3$. The blue segment indicates $L_1$ and the red segment indicates $L_{2}$. The yellow region indicates $P'_{1}$ and the union of yellow region and green region indicates $P_{1}$.}
    \label{fig:duc_1}
\end{figure}

It remains to prove Lemma \ref{lem:gro}.
\begin{proof}[Proof of Lemma \ref{lem:gro}]
We first note that $u$ satisfies \eqref{eq:har} for any $b\in T_{\boo;\left\lfloor\frac{n}{2}\right\rfloor}$.
Without loss of generality, we assume $a=\mathbf{0}$.

Define $F:\Lambda \rightarrow \Lambda$ to be the counterclockwise rotation around $\mathbf{0}$ by $\frac{2\pi}{3}$, i.e.
\begin{equation}
    F(s_1\xi+t_1\eta)=(t_1-s_1)\xi-s_1\eta
\end{equation} for any $s_1,t_1 \in \Z$.

We first consider the trapezoid $P_1:=P_{2m\xi-2m\eta;6m,6m}$.
The upper edge of $P_1$ is exactly the $\xi$-edge of $T_1$ and the lower edge of $P_1$ is contained in the $\xi$-edge of $T_3$. 
Denote $P'_1:=P_{2m\xi-2m\eta;6m,3m}$, $K_1:=(K+R)(2C_6)^{6m}$ and $K_2:=(K_1+R)C_{6}^{6m}$. Then $|u| \leq K_1$ in $T_1$ since $T_1$ is good.
In particular, $|u| \leq K_1$ on the upper edge of $P_1$.
We also have $ \left| \left\{b \in P_1:|u(b)|>K\right\}\right| \leq \frac{36}{10^5}m^2$, by $P_1\subset T_3$ and the assumption of this lemma.
Thus by Lemma \ref{lem:pra}, we deduce that $|u| \leq K_2$ in $P'_1$.

Let $P_2:=F(P_1)$ and $P_3:=F^{-1}(P_1)$.
A symmetric argument for $P_2$ and $P_3$ implies that $|u| \leq K_2$ also holds in $P'_2:=F(P'_1)$ and $P'_3:=F^{-1}(P'_1)$.

Consider the triangles $T'_1:=T_{3m\xi +6m\eta;2m}$, $T'_2:=T_{3m\xi-3m\eta; 2m}$ and $T'_3:=T_{-6m\xi-3m\eta;2m}$ (see Figure \ref{fig:duc_1}).
We have $T'_2=F(T'_1)$ and $T'_3=F^{-1}(T'_1)$.
We claim that $|u| \leq (K_2+R)2^{6m}$ in $\bigcup_{i=1,2,3}T'_i$.
By symmetry, we only need to prove the claim in $T'_1$. 
Denote $L_1 := \left\{s\xi+4m\eta: -m \leq s \leq 2m \right\}$ and $L_2 := \left\{s\xi+4m\eta: 2m \leq s \leq 5m \right\}$.
Note that the $\xi$-edge of triangle $T'_1$ is the set of points
\begin{equation}
    \left\{s\xi+4m\eta: -m \leq s \leq 5m \right\}= L_1 \cup L_2.
\end{equation}
Since
\begin{equation}
    F^{-1}(L_1)=\left\{-4m\xi+(s-4m)\eta:-m \leq s \leq 2m \right\} \subset P'_1,
\end{equation}
and
\begin{equation}
     F(L_2)=\left\{(4m+t)\xi+t\eta: -5m \leq t \leq -2m \right\} \subset P'_1,
\end{equation}
we have $L_1 \subset F(P'_1)=P'_2$ and $L_2 \subset F^{-1}(P'_1)=P'_3$.
Hence $|u| \leq K_2$ on $L_1 \cup L_2$, i.e. the $\xi$-edge of $T'_1$.
By Lemma \ref{lem:tra}, $|u| \leq (K_2+R)2^{6m}$ in $T'_1$, and our claim holds.

Since $\left(\bigcup_{i=1,2,3} T'_i\right) \cup \left(\bigcup_{i=1,2,3} P'_i\right) \cup T_1 = T_2$, we have $|u| \leq (K_2+R)2^{6m}$ in $T_2$.
We also have that
\begin{equation}
    2^{6m}(K_2+R) = 2^{12m}C_{6}^{12m}K+ (2^{12m}C_{6}^{12m}+2^{6m}C_{6}^{6m}+2^{6m})R \leq \left(\frac{C_4}{3}\right)^{15m}(K+R),
\end{equation}
so $T_2$ is good.
\end{proof}

To apply Theorem \ref{thm:bou} to prove Theorem \ref{thm:wqucF} in the next section, we actually need the following two corollaries.
\begin{cor}   \label{cor:poly}
Let $a \in \Lambda$, and $m, \ell\in\Z_{\geq 0}$ with $m \geq 2\ell$.
Take any nonempty
\begin{equation}
L\subset \left\{a-t\xi :t \in \Z, \ell \leq t \leq m-\ell \right\},
\end{equation}
and function $u: P_{a;m,\ell} \rightarrow \R$ such that
\begin{equation}\label{eq:cor1}
         |u ( b ) +u ( b-\xi ) +u ( b+\eta )| \leq C_4^{-2\ell}\min_{c\in L} |u(c)|,
\end{equation}
for any $b$ with $\left\{b,b-\xi, b+\eta\right\}  \subset P_{a;m,\ell}$. Then
\begin{equation}\label{eq:trapee0}
    \left| \left\{b \in P_{a;m,\ell} : |u(b)| \geq C_{4} ^{-2\ell}\min_{c\in L} |u(c)| \right\} \right| \geq \epsilon_2 (\ell+1)^2
\end{equation}
whenever $L$ contains at least one element; and
\begin{equation}\label{eq:trapee}
    \left| \left\{b \in P_{a;m,\ell} : |u(b)| \geq C_{4} ^{-2\ell}\min_{c\in L} |u(c)| \right\} \right| \geq \epsilon_2 (m+2) (\ell+1)
\end{equation}
if $m \geq 2\ell+2$ and $L= \left\{a-t\xi :t \in \Z, \ell+1 \leq t \leq m-\ell-1 \right\}$.
Here $\epsilon_2$ is a universal constant.
\end{cor}

\begin{proof}
If $\ell \leq 10^9$ then the conclusion holds trivially by taking $\epsilon_2$ small enough. From now on we assume $\ell>10^9$. 
We denote $P:=P_{a;m,\ell}$, for simplicity of notations. Without loss of generality, we assume that $\min_{c\in L} |u(c)| = 1$. 

First we prove
\begin{equation}\label{eq:trapep1}
\left| \left\{b \in P : |u(b)| \geq C_{4} ^{-2\ell} \right\} \right| \geq
\frac{\epsilon_1 (\ell+1)^2}{100},
\end{equation}
which implies \eqref{eq:trapee0}.
We take $a'\in L$.
By \eqref{eq:cor1}, for any $b \in P_{a-\xi;m-2,\ell-2}$ and $0<k_1<\ell$, if $|u(b)| \geq C_4 ^{-k_1}$, then $|u ( b-\eta)| \geq C_4 ^{-k_1-1}$ or $|u ( b-\gamma)| \geq C_4 ^{-k_1-1} $.
Thus we can inductively pick $a_1=a',a_2, \cdots ,a_{\left\lfloor\frac{\ell}{3}\right\rfloor} \in P$, such that for each $i=1,2, \cdots, \left\lfloor\frac{\ell}{3}\right\rfloor$, $|u(a_i)| \geq C_4 ^{-i+1}$, and $a_i=a'-s_i\xi-i\eta$ with $s_i - s_{i-1} \in \left\{0, 1\right\}$ for each $2 \leq i \leq \left\lfloor\frac{\ell}{3}\right\rfloor$. 
In particular, we have $\left|u\left(a_{\left\lfloor\frac{\ell}{3}\right\rfloor}\right)\right| \geq C_4 ^{-\ell}$.

Denote $T':=T_{a_{\left\lfloor\frac{\ell}{3}\right\rfloor};2\left\lfloor\frac{\ell}{18}\right\rfloor}$.
Then $T' \subset P$, and 
we can apply Theorem \ref{thm:bou} in $T'$ with $n=2\left\lfloor\frac{\ell}{18}\right\rfloor$, thus $\eqref{eq:trapep1}$ follows.

For the case where $L= \left\{a-t\xi :t \in \Z, \ell+1 \leq t \leq m-\ell-1 \right\}$, we prove
\begin{equation}\label{eq:trapep2}
    \left| \left\{b \in P : |u(b)| \geq C_{4} ^{-2\ell} \right\} \right| \geq \epsilon_1 \left(\frac{(m+2)(\ell+1)}{800}-\frac{(\ell+1)^2}{100}\right).
\end{equation}

When $m \leq 8\ell$, \eqref{eq:trapep2} is trivial.
From now on we assume that $m > 8\ell$.
Denote $l := \left\lceil \frac{m-2\ell-1}{4\ell} \right\rceil-1$.
We take $b_1:=a-(\ell+1)\xi$. Let $b_i:=b_1-4\ell(i-1)\xi$ where $i=2,\cdots,l$.
For each $1 \leq i \leq l$,
consider the trapezoid $P_i:=P_{b_i;2\ell,\ell}$.
We note that these trapezoids are disjoint, and $P_i \subset P$ for each $1 \leq i \leq l$ (see Figure \ref{fig:lots} for an illustration).
We apply the same arguments in the proof of \eqref{eq:trapep1}, with $P$ substituted by each $P_i$, and we get
\begin{equation}\label{eq:trapep3}
    \left| \left\{b \in P_i : |u(b)| \geq C_{4} ^{-2\ell} \right\} \right| \geq \frac{\epsilon_1 (\ell+1)^2}{100},
\end{equation}
for each $1 \leq i \leq l$.
By summing over all $i$ we get \eqref{eq:trapep2}.

Finally, we can deduce \eqref{eq:trapee} from $\eqref{eq:trapep1}$ and $\eqref{eq:trapep2}$.
\begin{figure}
    \centering
    \includegraphics[width=\textwidth]{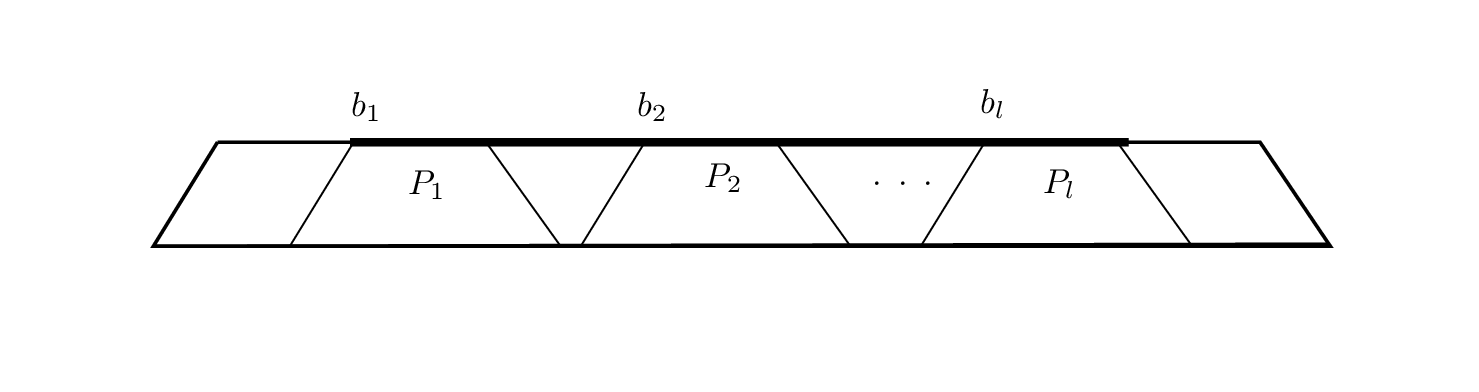}
    \caption{An illustration of $P_i$'s. The thick line indicates $L$.}
    \label{fig:lots}
\end{figure}
\end{proof}

For the next corollary, we set up notations for reversed trapezoids.
\begin{defn}  \label{defn:3rt}
For any $a\in \Lambda$, $m,\ell \in \Z_{\geq 0}$ with $\ell \leq m$,
we denote
\begin{equation}
P^{r}_{a;m,\ell}:=\left\{a-t\xi-s\eta:s \leq t \leq m, 0 \leq s \leq \ell\right\}\cap \Lambda,
\end{equation}
which is also a trapezoid, but its orientation is different from that of $P_{a;m,\ell}$ (see Figure \ref{revtra} for an illustration). We also denote $\left\{a-t\xi:0 \leq t \leq m\right\}\cap \Lambda$ to be the \emph{upper edge} of $P^{r}_{a;m,\ell}$.
\end{defn}
\begin{cor}\label{cor:poly2}
Let $a \in \Lambda$, and $m, \ell \in \Z_{\geq 0}$ with $m \geq \ell$.
Let $L$ be a nonempty subset of the upper edge of $P^{r}_{a;m,\ell}$.
\begin{figure}
    \centering
    \includegraphics{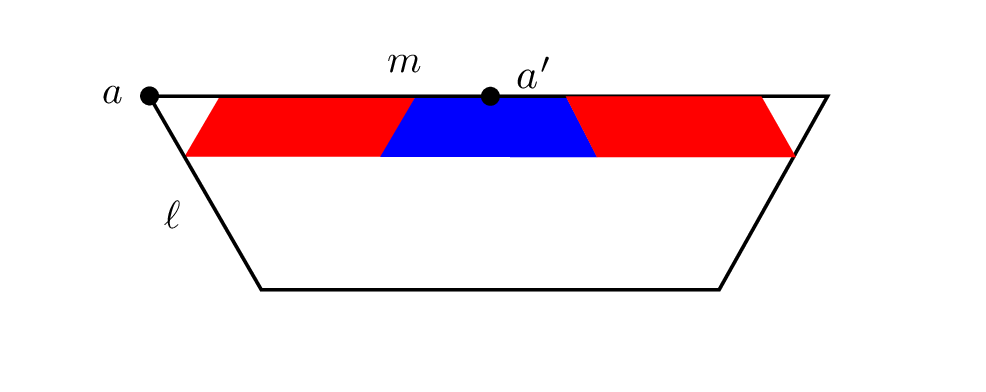}
    \caption{An illustration of Corollary \ref{cor:poly2}: $P^{r}_{a;m,\ell}$ is the set of lattice points in the region surrounded by black lines.
$P_{a'+\left(\left\lfloor\frac{\ell}{5}\right\rfloor+1\right)\xi;2\left\lfloor\frac{\ell}{5}\right\rfloor+2,\left\lfloor\frac{\ell}{5}\right\rfloor}$ is the blue region, and 
$P_{a-\left(\left\lfloor\frac{\ell}{5}\right\rfloor+2\right)\xi;m-2\left\lfloor\frac{\ell}{5}\right\rfloor-4,\left\lfloor\frac{\ell}{5}\right\rfloor}$
is the union of the blue and red regions.}
    \label{revtra}
\end{figure}
Take a function $u: P^{r}_{a;m,\ell} \rightarrow \R$ such that
\begin{equation}
         |u ( b ) +u ( b-\xi ) +u ( b+\eta )| \leq C_4^{-2\ell}\min_{c\in L} |u(c)|,
\end{equation}
for any $b$ with $\left\{b,b-\xi,b+\eta\right\} \subset P^{r}_{a;m,\ell}$.
Then
\begin{equation}\label{eq:trape}
    \left| \left\{b \in P^{r}_{a;m,\ell} : |u(b)| \geq C_{4} ^{-2\ell}\min_{c\in L} |u(c)| \right\} \right| \geq \epsilon_3 (\ell+1)^2,
\end{equation}
if $L=\left\{a-\left\lfloor\frac{m}{2}\right\rfloor\xi\right\}$ or $L=\left\{a-\left\lceil\frac{m}{2}\right\rceil \xi  \right\}$. And
\begin{equation}\label{eq:trape2}
    \left| \left\{b \in P^{r}_{a;m,\ell} : |u(b)| \geq C_{4} ^{-2\ell}\min_{c\in L} |u(c)| \right\} \right| \geq \epsilon_3 (m+2)(\ell+1),
\end{equation}
if $L=\left\{a-t\xi:t\in \Z, 1 \leq t \leq m-1\right\}$.
Here $\epsilon_3$ is a universal constant.
\end{cor}

\begin{proof}
If $m \leq 10^9$, then the conclusion holds trivially by taking $\epsilon_3$ small enough. From now on we assume that $m>10^9$.
If $L=\left\{a-\left\lfloor\frac{m}{2}\right\rfloor\xi\right\}$ or $L=\left\{a-\left\lceil\frac{m}{2}\right\rceil \xi  \right\}$, let $a'=a-\left\lfloor\frac{m}{2}\right\rfloor\xi$ or $a'=a-\left\lceil\frac{m}{2}\right\rceil\xi$ respectively. Consider $P_{a'+\left(\left\lfloor\frac{\ell}{5}\right\rfloor+1\right)\xi;2\left\lfloor\frac{\ell}{5}\right\rfloor+2,\left\lfloor\frac{\ell}{5}\right\rfloor}\subset P^{r}_{a;m,\ell}$ (blue region in Figure \ref{revtra}). Using Corollary \ref{cor:poly} for this trapezoid, we get \eqref{eq:trape}.

If $L=\left\{a-t\xi:t\in \Z, 1 \leq t \leq m-1\right\}$, consider $P_{a-\left(\left\lfloor\frac{\ell}{5}\right\rfloor+2\right)\xi;m-2\left\lfloor\frac{\ell}{5}\right\rfloor-4,\left\lfloor\frac{\ell}{5}\right\rfloor} \subset P^{r}_{a;m,\ell}$ (union of blue and red regions in Figure \ref{revtra}). Using Corollary \ref{cor:poly} for this trapezoid, we get \eqref{eq:trape2}.
\end{proof}

\section{Geometric substructure on 3D lattice}  \label{sec:pyr}
In this section we state and prove the following stronger version of Theorem \ref{thm:wquc} which incorporates a graded set (which is defined in Definition \ref{def:F}).
\begin{theorem}   \label{thm:wqucF}
For any $K\in\R_+$, $N\in\Z_+$, and small enough $\varepsilon\in\R_+$, we can find large $C_{2} \in \R_+$ depending only on $K$ and $C_{\varepsilon,N} \in \R_+$ depending only on $\varepsilon, N$, such that the following statement is true.

Take integer $n>C_{\varepsilon,N}$ and functions $u, V: \Z^3 \rightarrow \R$, satisfying 
\begin{equation}\label{eq:harmonica}
    \Delta u = V u
\end{equation}
in $Q_{n}$ and $\|V\|_{\infty} \leq K$.
Let $\Vec{l}$ be a vector of positive reals, and $E \subset \Z^3$ be any $(N, \Vec{l},\varepsilon^{-1}, \varepsilon)$-graded set, with  the first scale length $l_1 > C_{\varepsilon, N}$. If $E$ is $(1,2\varepsilon)$-normal in $Q_{n}$, then we have that
\begin{equation}   \label{eq:wqucF}
    \left| \left\{ a \in Q_{n} : |u(a)| \geq \exp(-C_{2} n^{3})  |u(\mathbf{0})| \right\}\setminus E \right| \geq C_3 n^2(\log_2 n)^{-1}.
\end{equation}
Here $C_3$ is a universal constant.
\end{theorem}
The first result we need is based on the ``cone property'' of the function $u$, as discussed in Section \ref{sec:cone}.
We remind the reader of the notations $\be_\tau$, for $\tau=1,2,3$; and $\blambda_\tau$, $\cP_{\tau,k}$, for $\tau \in \left\{1,2,3,4\right\}$ and $k\in \Z$, from Definition \ref{defn:pv}; and the cones from Definition \ref{defn:cone}.

\begin{prop}   \label{prop:estp}
Let $K \in \R_+$, $n \in \Z_+$, and $u, V$ satisfy \eqref{eq:harmonica} in $Q_n$, with $\|V\|_{\infty}\leq K$.
Then there exists $\tau \in \left\{1, 2, 3, 4\right\}$, such that for any $0 \leq i \leq \frac{n}{10}$ there is 
\begin{equation}
    a_i \in \left(\cP_{\tau, i}\cup \cP_{\tau, i+1}\right) \cap \cC \cap Q_{\frac{n}{10} + 1}
\end{equation}
with $|u(a_i)| \geq (K+11)^{-n}|u(\mathbf{0})|$.
\end{prop}

\begin{proof}
We can assume that $n\geq 10$ since otherwise this proposition holds obviously.
We argue by contradiction.
Denote $\Upsilon:=\left\{b \in Q_{n}:|u(b)|\geq (K+11)^{-n}|u(\mathbf{0})|\right\}$.
If the statement is not true, then for each $\tau \in \left\{1, 2, 3, 4\right\}$, there is $i_{\tau} \in \left[0, \frac{n}{10}\right]$, such that
\begin{equation}   \label{eq:pfestp1}
    \left(\cP_{\tau, i_\tau}\cup \cP_{\tau, i_\tau+1}\right) \cap \cC \cap \Upsilon \cap Q_{\frac{n}{10}+1}= \emptyset .
\end{equation}

Define
$B_{in}:=\bigcap_{\tau=1}^4\left\{ a \in \cC: a\cdot \blambda_{\tau} < i_{\tau} \right\}$,
$B_{bd}:=\bigcap_{\tau=1}^4\left\{ a \in \cC: a\cdot \blambda_{\tau} \leq i_{\tau} + 1 \right\} \setminus B_{in}$,
$B_{out}:= \cC \setminus \left( B_{in}\cup B_{bd} \right)$.
Then for any $a \in B_{in}$ and $b \in B_{out}$, we have $\|a - b \|_{1} \geq 3$.
Since $i_1, i_2, i_3, i_4 \leq \frac{n}{10}$, we have that
\begin{equation}
B_{bd} \subset \cC \cap \left\{ a \in \Z^3: |a \cdot \be_1| + |a\cdot \be_2| + a \cdot \be_3 \leq \frac{n}{10} + 1 \right\} \subset Q_{\frac{n}{10} + 1}.
\end{equation}
Then the condition \eqref{eq:pfestp1} implies that $\Upsilon \cap B_{bd} = \emptyset$.

We now apply Lemma \ref{lem:chain} to starting point $a_0 = \mathbf{0}$, in the $\be_3$ direction, and $k = n$.
Let $\mathbf{0} = a_0, a_1, \cdots, a_w \in \cC \cap \Z^3$ be the chain.
Then $a_0 \in B_{in}$, and $a_w \cdot \be_3 \geq n-1$, which implies that $a_w \in B_{out}$ (since otherwise, $a_w\cdot \be_3 = \frac{1}{4}\sum_{\tau=1}^4 a_w\cdot \blambda_\tau \leq \frac{1}{4}\sum_{\tau=1}^4 i_\tau+1\leq \frac{n}{10}+1$).
Thus $a_w\neq a_0$ and $w\geq 1$.
Since $|u(a_i)|\geq (K+11)^{-1} |u(a_{i-1})|$ for each $i=1, \cdots, w$, we also have that each $a_i \in \Upsilon$.
As $\Upsilon \cap B_{bd} = \emptyset$, we can find $1 \leq i \leq w$, such that $a_{i-1} \in B_{in}$ and $a_i \in B_{out}$.
This implies that $\|a_{i-1} - a_i \|_1 \geq 3$, which contradicts with the construction of the chain from Lemma \ref{lem:chain}.
\end{proof}

\begin{prop}  \label{prop:lvh}
For any $K\in\R_+$, $N\in\Z_+$, and small enough $\varepsilon>0$, we can find $C_7, C_{\varepsilon,N}\in\R_+$, where $C_7$ depends only on $K$ and $C_{\varepsilon,N}$ depends only on $\varepsilon, N$, such that following statement is true.

Take integer $n>C_{\varepsilon,N}$, and let functions $u, V$ satisfy \eqref{eq:harmonica} in $Q_n$, and $\|V\|_{\infty}\leq K$.
Let $\Vec{l}$ be a vector of positive reals,
and $E$ be an $(N,\Vec{l},\varepsilon^{-1},\varepsilon)$-graded set with the first scale length $l_1>C_{\varepsilon,N}$, and be $(1,2\varepsilon)$-normal in $Q_n$.
For any $\tau \in \left\{1,2,3,4\right\}$, $k\in \Z$, $0 \leq k \leq \frac{n}{10}$, and $a_0 \in \cP_{\tau, k} \cap Q_{\frac{n}{4}}$, there exists $h \in \Z_+$, such that
\begin{equation}  \label{eq:lvh}
    \left| 
    \left\{a \in Q_n \cap \bigcup_{i=0}^h \cP_{\tau, k + i}: |u(a)| \geq \exp(- C_{7} n^{3})|u(a_0)|\right\}\setminus E \right| > C_{8} hn(\log_2(n))^{-1} .
\end{equation}
Here $C_8$ is a universal constant.
\end{prop}

In Section \ref{ssec:deta},
Theorem \ref{thm:wqucF} is proved by applying Proposition \ref{prop:lvh} to each point $a_i$ obtained from Proposition \ref{prop:estp}.

The next two subsections are devoted to the proof of Proposition \ref{prop:lvh}.
We will work with $\tau=1$ only, and the cases where $\tau=2,3,4$ follow the same arguments.
Assuming the result does not hold, we can find many ``gaps'', i.e. intervals that do not intersect the set $\{|u(a)|:a \in Q_n \setminus E, a \cdot \blambda_{1} \geq k  \}$.
These gaps will allow us to construct geometric objects on $\Z^3$.
We first find many ``pyramids'' in $\left\{a \in Q_n: a \cdot \blambda_{1} \geq k \right\}$  (see Lemma \ref{lem:bas}), then we prove Proposition \ref{prop:lvh} assuming a lower bound on the number of desired points in each ``pyramid'' (Proposition \ref{prop:tetraF}).
In Section \ref{ssec:multila} we prove Proposition \ref{prop:tetraF}, by studying ``faces'' of each ``pyramid'', and using corollaries of Theorem \ref{thm:bou}.

\subsection{Decomposition into  pyramids}\label{ssec:decompo}
In this subsection we define pyramids in $Q_n$, and in the next subsection we study the structure of each of these pyramids.

We need some further geometric objects in $\R^3$.
\begin{defn}\label{def:tetra}
For simplicity of notations we denote $\olambda_2=\blambda_2=-\be_1+\be_2+\be_3$, $\olambda_3=\blambda_3=\be_1-\be_2+\be_3$, and $\olambda_4=-\blambda_4=\be_1+\be_2-\be_3$.
Then $\blambda_1\cdot\olambda_2=\blambda_1\cdot\olambda_3=\blambda_1\cdot\olambda_4=1$,
and $\olambda_2\cdot\olambda_3=\olambda_2\cdot\olambda_4=\olambda_3\cdot\olambda_4=-1$.

For any $a \in \R^3$, $r \in \Z_+$,
denote $\bt_r(a)=a+r\be_1+r\be_2+2r\be_3$.
Then $\bt_r(a)\cdot \olambda_2 = a\cdot \olambda_2+2r$, $\bt_r(a)\cdot \olambda_3 = a\cdot \olambda_3+2r$, and $\bt_r(a)\cdot \olambda_4 = a\cdot \olambda_4$.
Denote
\begin{equation}
\caT_{a, r} := \left\{b \in \cP_{1, a \cdot \blambda_1}:
b \cdot \olambda_\tau \leq \bt_r(a) \cdot \olambda_\tau, \forall \tau \in \{2,3,4\}
\right\},
\end{equation}
and let $\oT_{a, r}$ be the interior of $\caT_{a, r}$ in $\cP_{1, a \cdot \blambda_1}$.
Respectively, $\oT_{a, r}$ and $\caT_{a, r}$ are the open and closed equilateral triangles with side length $2\sqrt{2}r$ in the plane $\cP_{1, a \cdot \blambda_1}$,
and $a$ is the midpoint of one side.
When $a \in \Z^3$, there are $2r+1$ lattice points on each side of $\caT_{a, r}$.

We also take
\begin{equation}
\fT_{a,r}:= \left\{b \in \R^3: b\cdot \blambda_1 \geq a \cdot \blambda_1, \; b \cdot \olambda_\tau \leq \bt_r(a) \cdot \olambda_\tau, \forall \tau \in \{2,3,4\}
\right\},
\end{equation}
which is a (closed) regular tetrahedron, with four faces orthogonal to $\blambda_1, \blambda_2, \blambda_3, \blambda_4$ respectively.
The point $\bt_r(a)$ is a vertex of $\fT_{a,r}$, and $\caT_{a,r}$ is the face orthogonal to $\blambda_1$. (See Figure \ref{fig:td} for an illustration)

For any $k \in \Z$, denote $\pi_k(a)$ to be the orthogonal projection of $a$ onto $\cP_{1, k}$.
\end{defn}
The purpose of the following lemma is to find some triangles ($\caT_{a_i,r_i}$ for $a_i$, $r_i$ in Lemma \ref{lem:bas}) in $\cP_{1,k}\cup\cP_{1,k+1}$, and these triangles will be basements of pyramids to be constructed in the proof of Proposition \ref{prop:lvh}. 
\begin{lemma}   \label{lem:bas}
Let $N \in \Z_{+}$, and $\varepsilon>0$ and be small enough, then there exists $C_{\varepsilon,N}>0$ such that the following statement is true.
Suppose we have
\begin{enumerate}
    \item a function $u:\Z^3\to \R$,
    \item $n, k \in \Z$,
    $n>C_{\varepsilon,N}$, $k \in \Z\cap \left[0, \frac{n}{10}\right)$, $a_{0} \in \cP_{1, k} \cap Q_{\frac{n}{4}}$, 
    \item a vector of positive reals $\Vec{l}$, and an $(N,\Vec{l},\varepsilon^{-1},\varepsilon)$-graded set $E$ with the first scale length $l_1>C_{\varepsilon,N}$, and $E$ being $(1,2 \varepsilon)$-normal in $Q_{n}$,
    \item $D \in \R_+$, and $0 < g_1, \cdots, g_{100n} < |u(a_0)|$, such that $g_i \leq g_{i+1}\exp(-Dn)$ for each $1 \leq i \leq 100n - 1$.
\end{enumerate}
Then we can find $m \in \Z_+$, $r_1, r_2 \cdots, r_m \in \Z\cap \left[0, \frac{n}{32}\right)$, $a_1, a_2, \cdots, a_m \in \left(\cP_{1, k}\cup\cP_{1, k+1} \right) \cap Q_{\frac{n}{2}}$ and $s_{1},s_{2},\cdots,s_{m} \in \left\{1,2,\cdots, 100n\right\}$, satisfying the following conditions:
\begin{enumerate}
    \item $\sum_{i=1}^m (r_i+1) \geq \frac{n}{100}$.
    \item for each $1 \leq i \leq m$, we have $|u(a_{i})| \geq \exp(Dn)g_{s_i}$, and $|u(b)|<g_{s_i}$ for any $b \in (\oT_{\pi_{k}(a_{i}),r_{i}}\cup \oT_{\pi_{k+1}(a_{i}),r_{i}}) \cap \Z^3$.
    \item for any point $a \in \cP_{1, k}$, we have $a \in \caT_{\pi_k(a_i), r_i}$ for at most two $1 \leq i \leq m$.
    \item $E$ is $(\varepsilon^{-\frac{1}{2}},\varepsilon)$-normal in $\fT_{a_i, r_i}$ for each $1 \leq i \leq m$.
\end{enumerate}
\end{lemma}
\begin{proof}
Denote $R:= \left\{a \in (\cP_{1,k} \cup \cP_{1,k+1}) \cap Q_{\frac{n}{2}} : |u(a)| \geq \exp(Dn)g_1 \right\}$.
For each $a \in R$, denote
\begin{equation}
    I(a):=\max\left\{i \in \left\{1, \cdots, 100n\right\} : |u(a)|\geq \exp(Dn) g_i\right\},
\end{equation}
and we let $r(a)$ be the largest integer, such that
$0 \leq r(a) < \frac{n}{32}$, and
\begin{equation}\label{eq:maximaltran}
   |u(b)| \leq g_{I(a)},\; \forall b\in \left(\oT_{\pi_k(a), r(a)} \cup \oT_{\pi_{k+1}(a), r(a)}\right)\cap \Z^3.
\end{equation}

Suppose $\vec{l}=(l_1,l_2,\cdots,l_d)$. We write $E=\bigcup_{i=0}^{d}E_{i}$ where $E_i$ is a $(N,l_i,\varepsilon)$-scattered set for $0<i\leq d$, and $E_0$ is a $\varepsilon^{-1}$-unitscattered set. 
We write $E_{i}=\bigcup_{t=1}^N\bigcup_{j\in\Z_+}E_{i}^{(j,t)}$, where each $E_{i}^{(j,t)}$ is an open ball, and $\dist(E_i^{(j,t)}, E_i^{(j',t)})\geq l_i^{1+\varepsilon}$, $\forall j\neq j' \in \Z_+$.
We also write $E_{0}=\bigcup_{j \in \Z_{+}} o_{j}$ where each $o_{j}$ is an open unit ball, such that
$\forall j\neq j' \in \Z_+$ we have $\dist(o_{j},o_{j'})\geq \varepsilon^{-1}$.

If $r(a) \geq \frac{n}{100}$ for any $a \in R$, then Condition $1$ to $3$ hold by letting $m=1$, $a_1=a$, $r_1=r(a)$ and $s_1=I(a)$. 
Now we show that Condition $4$ also holds (when $C_{\varepsilon,N}$ is large enough).
Since $E$ is $(1,2\varepsilon)$-normal in $Q_{n}$,
\begin{equation}
    l_{i}<4n^{1-\varepsilon},
\end{equation}
whenever $E_{i} \cap Q_{n}\neq \emptyset$. 
Then since $n>C_{\varepsilon,N}$, by taking $C_{\varepsilon,N}$ large enough we have $n>300\varepsilon^{-\frac{1}{2}}$, and
\begin{equation}
    l_{i}<4n^{1-\varepsilon}< r(a)^{1-\frac{\varepsilon}{2}}.
\end{equation}
Thus $E$ is $(\varepsilon^{-\frac{1}{2}},\varepsilon)$-normal in $\fT_{a_1, r_1}$. From now on, we assume $r(a) < \frac{n}{100}$ for each $a \in R$.
We also assume that $n > 100$ by letting $C_{\varepsilon, N} > 100$.

For each $0<i\leq d$, $1 \leq t \leq N$, and $j\in\Z_+$, denote $B_{i}^{(j,t)}$ to be the open ball with radius $l_{i}^{1+\frac{2}{3}\varepsilon}$ and the same center as $E_{i}^{(j,t)}$.
Let $\tilde{B}_{i}^{(j,t)}:=B_{i}^{(j,t)} \cap \cP_{1, k}$, which is either a 2D open ball on the plane $\cP_{1, k}$, or $\emptyset$. 
For each $j \in \Z_{+}$, let $B_{j}$ be the open ball with radius $\varepsilon^{-\frac{2}{3}}$ and has the same center as $o_{j}$.
Denote $\tilde{B}_{j}:=B_{j} \cap \cP_{1,k}$.

We define a graph $G$ as follows.
The set of vertices of $G$ is 
\begin{multline}
V(G):=\left\{\caT_{\pi_{k}(a),r(a)+1}:a \in R\right\}
\\
\cup \left\{\tilde{B}_{i}^{(j,t)}: 1 \leq i\leq d, 1 \leq t\leq N, j\in\Z_+, \tilde{B}_{i}^{(j,t)} \neq \emptyset\right\} \cup \left\{\tilde{B}_{j}:j \in \Z_{+}, \tilde{B}_{j} \neq \emptyset\right\}.
\end{multline}
For any $v_1, v_2 \in V(G)$,
there is an edge connecting $v_1,v_2$ if and only if $v_1 \cap v_2\neq \emptyset$.
\begin{cla}
There is $a_{\infty} \in R$, such that $\caT_{\pi_{k}(a_0),r(a_0)+1}$ and $\caT_{\pi_{k}(a_{\infty}),r(a_{\infty})+1}$ are in the same connected component in $G$, 
and $\left(\caT_{\pi_{k}(a_{\infty}),r(a_{\infty})+1} \cup \caT_{\pi_{k+1}(a_{\infty}),r(a_{\infty})+1}\right) \cap \Z^3 \not\subset Q_{\frac{n}{2}}$.
\end{cla}
\begin{proof}
We let $b_0 : = a_0$.
For any $i \in \Z_{\geq 0}$, if $b_i \in R$, we choose
\begin{equation}
    b_{i+1} \in \Z^3 \cap \left(\oT_{\pi_k(b_i), r(b_i)+1} \cup \oT_{\pi_{k+1}(b_i), r(b_i)+1}\right)\setminus \left(\oT_{\pi_k(b_i), r(b_i)} \cup \oT_{\pi_{k+1}(b_i), r(b_i)}\right),
\end{equation}
with the largest $|u(b_{i+1})|$ (choose any one if not unique).
As $b_{i+1} \in \Z^3 \cap \left(\oT_{\pi_k(b_i), r(b_i)+1} \cup \oT_{\pi_{k+1}(b_i), r(b_i)+1}\right)$, we have that
\begin{equation}  \label{eq:lvh:ca1}
b_{i+1} \cdot (-\be_1-\be_2+2\be_3) \geq b_{i} \cdot (-\be_1-\be_2+2\be_3) + 1.
\end{equation}
By the definition of $r(b_i)$, we have that
$|u(b_{i+1})| \geq g_{I(b_i)} \geq \exp(Dn) g_{I(b_i)-1}$, thus $I(b_{i+1}) \geq I(b_i)-1$.

The construction terminates when we get some $q \in \Z_+$ such that $b_q \not\in R$.
We let $a_{\infty} := b_{q-1}$, and we show that it satisfies all the conditions.

From the construction, for each $i=0,\cdots, q-1$ we have that $\pi_k(b_{i+1}) \in \oT_{\pi_k(b_i), r(b_i)+1}$, so there is an edge in $G$ connecting $\caT_{\pi_{k}(b_i),r(b_i)+1}$ and $\caT_{\pi_{k}(b_{i+1}),r(b_{i+1})+1}$.
This implies that $\caT_{\pi_{k}(b_0),r(b_0)+1}$ and $\caT_{\pi_{k}(b_{q-1}),r(b_{q-1})+1}$ are in the same connected component in $G$.

If
$\left(\caT_{\pi_{k}(b_{q-1}),r(b_{q-1})+1} \cup \caT_{\pi_{k+1}(b_{q-1}),r(b_{q-1})+1}\right) \cap \Z^3 \subset Q_{\frac{n}{2}}$, we have $b_q \in Q_{\frac{n}{2}}$.
By \eqref{eq:lvh:ca1} we have that $b_{q} \cdot (-\be_1-\be_2+2\be_3) \geq b_{0} \cdot (-\be_1-\be_2+2\be_3) + q$.
Since $b_0, b_q \in Q_{\frac{n}{2}}$, we have $q \leq 4n$.
This means that $I(b_q) \geq I(b_0) - q \geq 100n-4n > 1$.
Then we have that $b_q \in R$, which contradicts with its construction.
This means that $a_{\infty}=b_{q-1}$ satisfies all the conditions stated in the claim.
\end{proof}
We define a weight on the graph $G$, by letting each vertex in $\left\{\caT_{\pi_{k}(a),r(a)+1}:a \in R\right\}$ (which are triangles) have weight $2$, and each other vertex (which are balls) have weight $1$.
The weights are defined this way for the purpose of proving Condition 4.
We then take a path $\gamma_{path} =\left\{v_1, v_2, \cdots, v_{p}\right\}$ such that $\pi_{k}(a_{0})\in v_1$ and $\pi_{k}(a_{\infty}) \in v_{p}$,
and has the least total weight (among all such paths).
Then all these vertices are mutually different.
For each $i=1,2,\cdots,p-1$ there is an edge connecting $v_{i}$ and $v_{i+1}$, and these are all the edges in the subgraph induced by these vertices.
Note that each $v_i$ is either a ball or a triangle in $\cP_{1. k}$. See Figure \ref{fig:duc_8} for an illustration.

\begin{figure}
    \centering
    \includegraphics{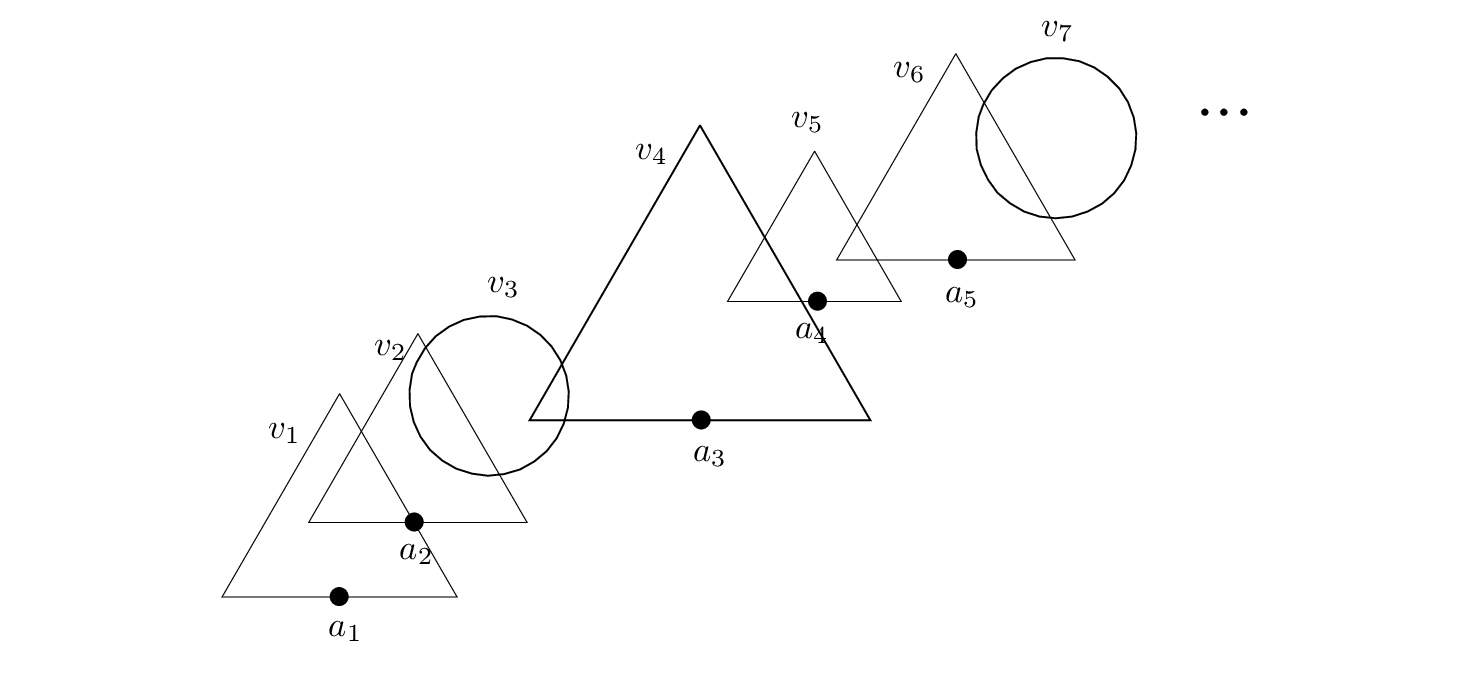}
    \caption{The path $\gamma_{path}$}
    \label{fig:duc_8}
\end{figure}

Suppose all the triangles in $\gamma_{path}$ are $\left\{\caT_{a_{i},r(a_{i})+1}:1 \leq i \leq m\right\}$.
Let $r_{i}:=r(a_{i})$ and $s_i:=I(a_i)$.
We claim that these $a_{i}$, $r_{i}$ and $s_i$ for $1 \leq i \leq m$ satisfy all the conditions.

Condition 2 follows from the definition of $r_i=r(a_i)$.
As $\gamma_{path}$ is a least weighted path, we have that $v_{i'} \cap v_{i''}=\emptyset$ whenever $|i'-i''|>1$, thus Condition 3 follows as well.

We next verify Condition 1. 
For this, we need to show that in the path, triangles constitute a substantial
fraction.
This is incorporated in Claims \ref{cla:length,i>0} and \ref{cla:length:i=0} below.
Denote $\ell_{i}:=\diam(v_{i})$, for each $1 \leq i \leq p$.
As $r(a_{\infty}) < \frac{n}{100}$, we have $a_{\infty} \not\in Q_{\frac{n}{2}-\frac{n}{20}}$; also note that $a_0 \in Q_{\frac{n}{4}}$, so we have 
\begin{equation}\label{eq:totallength}
    \ell_{total}:=\sum_{i=1}^{p} \ell_{i}\geq \dist(Q_{\frac{n}{4}},\Z^3 \setminus Q_{\frac{n}{2}-\frac{n}{20}}) \geq \frac{n}{5}.
\end{equation}
For each $1 \leq i \leq d$ and $1 \leq t \leq N$, denote $\mathcal{V}_{i,t}:=\left\{v \in \gamma_{path}:\exists j\in\Z_+, v=\tilde{B}_{i}^{(j,t)}\right\}$. 
\begin{cla}\label{cla:length,i>0}
If $\mathcal{V}_{i,t} \neq \emptyset$, then $\sum_{i' : v_{i'} \in \mathcal{V}_{i,t}} \ell_{i'} \leq \ell_{total} l_{i}^{-\frac{\varepsilon}{4}}$, provided that $\varepsilon$ is small enough and $C_{\varepsilon,N}$ is large enough.
\end{cla}
\begin{proof}
    Since $\mathcal{V}_{i,t} \neq \emptyset$ and $E$ is $(1,2\varepsilon)$-normal in $Q_{n}$, we have $C_{\varepsilon,N} \leq l_{i} \leq n^{1-\varepsilon}$.
    
    \noindent{\textbf{Case 1:}} $|\mathcal{V}_{i,t}| = 1$.
    Suppose $\{v_{i'}\} = \mathcal{V}_{i,t}$. Then by \eqref{eq:totallength}, when $C_{\varepsilon,N}$ is large enough we have
    \begin{equation}
          \ell_{i'} \leq 2l_{i}^{1+\frac{2}{3}\varepsilon} \leq
          \frac{nl_{i}^{-\frac{\varepsilon}{4}}}{5} \leq \ell_{total} l_{i}^{-\frac{\varepsilon}{4}}.
    \end{equation}
    
    \noindent{\textbf{Case 2:}} $|\mathcal{V}_{i,t}| > 1$.
    Write $\mathcal{V}_{i,t}=\left\{v_{i_1},v_{i_2},\cdots,v_{i_q}\right\}$, where $1 \leq i_1 < i_2 < \cdots < i_q \leq p$, and $q \geq 2$. 
    For each $w \in \left\{1,2,\cdots,q-1\right\}$, consider the part of $\gamma_{path}$ between $v_{i_{w}}$ and $v_{i_{w+1}}$.
    By letting $C_{\varepsilon,N}$ large enough we have
    \begin{equation}\label{eq:twovertexF}
        \sum_{i'=i_w}^{i_{w+1}} \ell_{i'} \geq \dist(v_{i_{w}},v_{i_{w+1}}) \geq l_{i}^{1+\varepsilon}-2l_{i}^{1+\frac{2}{3}\varepsilon}\geq 2(\ell_{i_{w}}+\ell_{i_{w+1}}) l_{i}^{\frac{\varepsilon}{4} }.
    \end{equation}
    Summing \eqref{eq:twovertexF} through all $w \in \left\{1,2,\cdots,q-1\right\}$, we get
    \begin{equation}
        \ell_{total} \geq \frac{1}{2}\sum_{w \in \left\{1,2,\cdots,q-1\right\}} \sum_{i'=i_w}^{i_{w+1}} \ell_{i'}
        \geq \left(\sum_{v_{i'} \in \mathcal{V}_{i,t}} \ell_{i'}\right) l_{i}^{\frac{\varepsilon}{4}}.
    \end{equation}
    Then the claim follows as well.
\end{proof}
Let $\mathcal{V}_{0}:=\left\{v_{i'} \in \gamma_{path}:\exists j\in \Z_{+}, v_{i'}=\tilde{B}_{j}\right\}$.
\begin{cla}\label{cla:length:i=0}
If $\mathcal{V}_{0} \neq \emptyset$, then $\sum_{v_{i'} \in \mathcal{V}_{0}} \ell_{i'} \leq \varepsilon^{\frac{1}{4}} \ell_{total}$,
provided that $\varepsilon$ is small enough and $C_{\varepsilon,N}$ is large enough.
\end{cla}
This is by the same arguments as the proof of Claim \ref{cla:length,i>0}.

From Claim \ref{cla:length,i>0} and Claim \ref{cla:length:i=0}, by making $\varepsilon$ small and $C_{\varepsilon,N}$ large enough, from $l_1>C_{\varepsilon,N}$ and $l_{i+1}\geq l_{i}^{1+2\varepsilon}$, we have
\begin{equation}
    \sum_{i':\text{$v_{i'}$ is a 2D ball}} \ell_{i'}=
    \sum_{v_{i'} \in \mathcal{V}_{0}} \ell_{i'} +
    \sum_{1 \leq i \leq d,1 \leq t\leq N}\sum_{v_{i'} \in \mathcal{V}_{i,t}} \ell_{i'} \leq  \varepsilon^{\frac{1}{4}} \ell_{total} + N \ell_{total} \sum_{i=1}^{\infty} l_{i}^{-\frac{\varepsilon}{4}} \leq \frac{\ell_{total}}{100} .
\end{equation}

Now we have that
\begin{equation}\label{eq:sumoftran}
\sum_{i=1}^m (r_i+1) \geq (2\sqrt{2})^{-1} \sum_{i':v_{i'}\text{\;is\;a\;triangle}} \ell_{i'} \geq (2\sqrt{2})^{-1}\frac{99}{100} \ell_{total} > \frac{n}{100},
\end{equation}
where the last inequality is due to \eqref{eq:totallength}. Then Condition 1 follows.

It remains to check Condition 4.
We prove by contradiction.
    Suppose for some $1 \leq i' \leq m$, $E$ is not $(\varepsilon^{-\frac{1}{2}},\varepsilon)$-normal in $\fT_{a_{i'},r_{i'}}$.
    There are only two cases:
    
    \noindent{\textbf{Case 1:}} There exists $1\leq i\leq d$ and $E_{i}^{(j,t)}$, such that 
    \begin{equation}
        E_{i}^{(j,t)} \cap \fT_{a_{i'},r_{i'}} \neq \emptyset
    \end{equation}
    and 
    \begin{equation}\label{eq:sparsity-case1}
        l_{i} > \diam(\fT_{a_{i'},r_{i'}})^{1-\frac{\varepsilon}{2}}.
    \end{equation}
    Recall that $B_{i}^{(j,t)}$ is the ball with radius $l_{i}^{1+\frac{2}{3}\varepsilon}$ and the same center as $E_{i}^{(j,t)}$. By \eqref{eq:sparsity-case1} and letting $C_{\varepsilon,N}$ large enough, we have
    \begin{equation}
        \radius(B_{i}^{(j,t)})-l_i=l_{i}^{1+\frac{2}{3}\varepsilon}-l_{i} >\diam(\fT_{a_{i'},r_{i'}}) + 3.
    \end{equation}
    This implies that $\caT_{\pi_{k}(a_{i'}),r_{i'}+1} \subset B_{i}^{(j,t)}$ and $\caT_{\pi_{k}(a_{i'}),r_{i'}+1} \subset \tilde{B}_{i}^{(j,t)}$.
    If we substitute $\caT_{\pi_{k}(a_{i'}),r_{i'}+1}$ by $\tilde{B}_{i}^{(j,t)}$ in the path $\gamma_{path}$, then the new path has lower weight than $\gamma_{path}$. This contradicts with the fact that $\gamma_{path}$ is a least weight path.
    
    \noindent{\textbf{Case 2:}} $E_{0} \cap \fT_{a_{i'},r_{i'}} \neq \emptyset$ and $\varepsilon^{-\frac{1}{2}} > \diam(\fT_{a_{i'},r_{i'}})$.
    
    Then $\caT_{\pi_{k}(a_{i'}),r_{i'}+1} \subset B_{j}$ and $\caT_{\pi_{k}(a_{i'}),r_{i'}+1} \subset \tilde{B}_{j}$ for some $j \in \Z_{+}$, since $\radius(B_{j})-1=\varepsilon^{-\frac{2}{3}}-1>\varepsilon^{-\frac{1}{2}}+3 > \diam(\fT_{a_{i'},r_{i'}}) + 3$. By the same reason as Case 1, we reach a contradiction. Thus Condition 4 holds and the conclusion follows.
\end{proof}

Now we work on each tetrahedron $\fT_{a_i, r_i}$.
We will construct a pyramid in each of them, and show that on the boundary of the pyramid, the number of points $b$ such that $b\not\in E$, $|u(b)|\geq \exp(-C_{2}n^{3})$, is at least in the order of $r_i^2+1$.

\definecolor{ffxfqq}{rgb}{1.,0.4980392156862745,0.}
\definecolor{xdxdff}{rgb}{0.49019607843137253,0.49019607843137253,1.}
\definecolor{ffqqqq}{rgb}{1.,0.,0.}
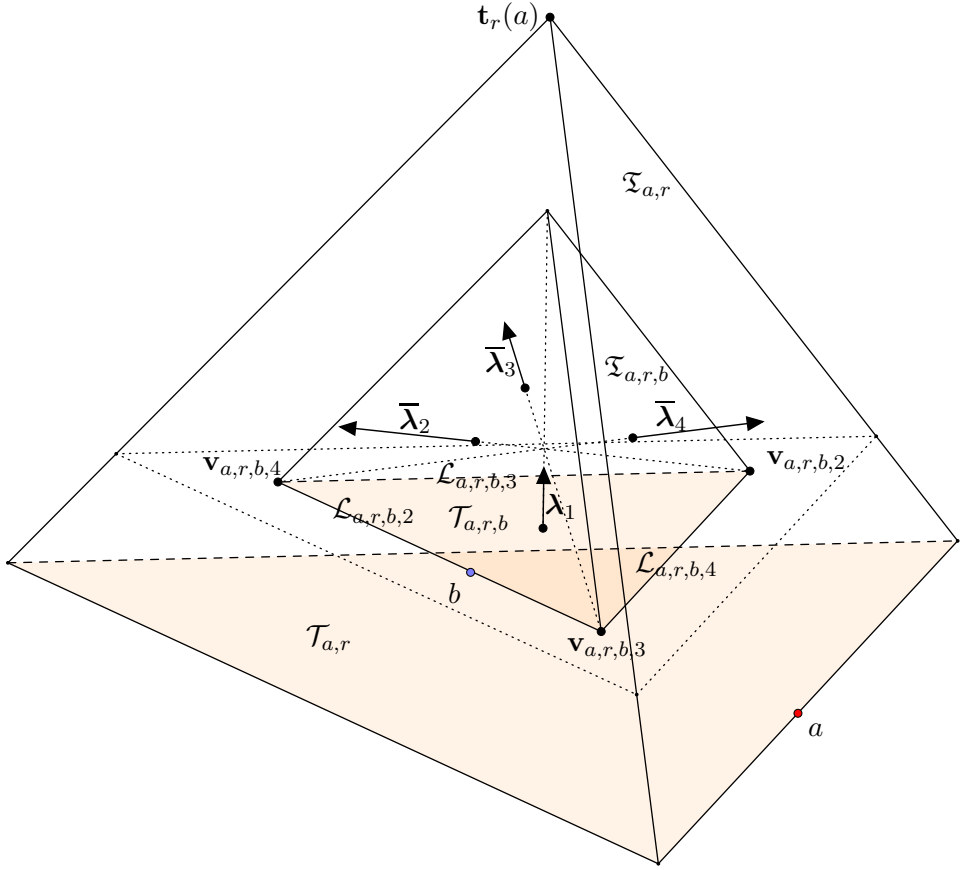
\begin{figure}[!ht]
\centering
\begin{tikzpicture}[line cap=round,line join=round,>=triangle 45,x=1.2cm,y=1.0cm]
\clip(-1.2907274908887831,-3.9369515295122215) rectangle (9.735456906044003,7.6586600831992556);
\fill[line width=0.pt,color=ffxfqq,fill=ffxfqq,fill opacity=0.10000000149011612] (1.7958949427825415,1.2839659485570531) -- (5.377163216471651,-0.705855602279004) -- (7.025266878165092,1.4289533343128125) -- cycle;
\fill[line width=0.pt,color=ffxfqq,fill=ffxfqq,fill opacity=0.10000000149011612] (-1.1960004123504433,0.2086020006538712) -- (6.010003772890715,-3.7951925568639444) -- (9.326214951798109,0.5003365563248241) -- cycle;
\draw [line width=0.5pt] (-1.1960004123504433,0.2086020006538712)-- (6.010003772890715,-3.7951925568639444);
\draw [line width=0.5pt,dash pattern=on 3pt off 3pt] (-1.1960004123504433,0.2086020006538712)-- (9.326214951798109,0.5003365563248241);
\draw [line width=0.5pt] (9.326214951798109,0.5003365563248241)-- (6.010003772890715,-3.7951925568639444);
\draw [line width=0.5pt,dotted] (0.004190049586180855,1.6591257919140512)-- (5.770285669832124,-1.544627865314141);
\draw [line width=0.5pt,dotted] (5.770285669832124,-1.544627865314141)-- (8.423849318000979,1.8925657539909513);
\draw [line width=0.5pt,dotted] (0.004190049586180855,1.6591257919140512)-- (8.423849318000979,1.8925657539909513);
\draw [line width=0.5pt] (4.780948388681161,4.891635879938136)-- (7.025266878165092,1.4289533343128125);
\draw [line width=0.5pt] (4.780948388681161,4.891635879938136)-- (5.377163216471651,-0.705855602279004);
\draw [line width=0.5pt] (4.780948388681161,4.891635879938136)-- (1.7958949427825415,1.2839659485570531);
\draw [line width=0.5pt] (1.7958949427825415,1.2839659485570531)-- (5.377163216471651,-0.705855602279004);
\draw [line width=0.5pt] (5.377163216471651,-0.705855602279004)-- (7.025266878165092,1.4289533343128125);
\draw [line width=0.5pt,dash pattern=on 3pt off 3pt] (7.025266878165092,1.4289533343128125)-- (1.7958949427825415,1.2839659485570531);
\draw [line width=0.5pt] (4.8103375608045145,7.467729951356388)-- (-1.1960004123504433,0.2086020006538712);
\draw [line width=0.5pt] (4.8103375608045145,7.467729951356388)-- (9.326214951798109,0.5003365563248241);
\draw [line width=0.5pt] (4.8103375608045145,7.467729951356388)-- (6.010003772890715,-3.7951925568639444);
\draw (7.557001680270136,-1.7932952781050366) node[anchor=north west] {$a$};
\draw (3.930002396408951,0.0817280060325498) node[anchor=north east] {$b$};
\draw (3.433644207027424,1.6626409123505885) node[anchor=north west] {$\mathcal{L}_{a,r,b,3}$};
\draw (2.2874945974521417,1.2013855816678514) node[anchor=north west] {$\mathcal{L}_{a,r,b,2}$};
\draw (5.642078820599309,0.48853643424907556) node[anchor=north west] {$\mathcal{L}_{a,r,b,4}$};
\draw (4.887297370391196,-0.657613175326211) node[anchor=north west] {$\mathbf{v}_{a,r,b,3}$};
\draw (0.8478188683514828,1.7465055179292681) node[anchor=north west] {$\mathbf{v}_{a,r,b,4}$};
\draw (7.0957319839630815,1.8303701235079475) node[anchor=north west] {$\mathbf{v}_{a,r,b,2}$};
\draw (3.568072956028273,1.1013855816678514) node[anchor=north west] {$\mathcal{T}_{a,r,b}$};
\draw (1.9939684779267648,-0.489883964168852) node[anchor=north west] {$\mathcal{T}_{a,r}$};
\draw (5.3066203982845925,3.0936848008184245) node[anchor=north west] {$\mathfrak{T}_{a,r,b}$};
\draw (5.502304477968177,5.548367637496072) node[anchor=north west] {$\mathfrak{T}_{a,r}$};
\draw (4.8103375608045145,7.467729951356388)
node[anchor=east] {$\bt_r(a)$};
\draw [->,line width=0.55pt,] (5.727792827772634,1.8715778706573143) -- (7.188448679027976,2.089869092196835);
\draw [->,line width=0.55pt,] (4.534036736542932,2.5348517209360004) -- (4.302961127455124,3.4230321261470174);
\draw [->,line width=0.55pt,] (3.984668849311784,1.8232487420720602) -- (2.4627590277118667,2.020605330475444);
\draw [line width=0.5pt,dotted] (4.780948388681161,4.891635879938136)-- (4.732775012473095,0.6690212268636206);
\draw [line width=0.5pt,dotted] (3.984668849311784,1.8232487420720602)-- (7.025266878165092,1.4289533343128125);
\draw [line width=0.5pt,dotted] (4.534036736542932,2.5348517209360004)-- (5.377163216471651,-0.705855602279004);
\draw [line width=0.5pt,dotted] (5.727792827772634,1.8715778706573143)-- (1.7958949427825415,1.2839659485570531);
\draw [->,line width=0.55pt,] (4.732775012473095,0.6690212268636206) -- (4.742217852186804,1.4967289044306487);
\draw (3.9787641432888385,3.228113549819273) node[anchor=north west] {$\olambda_3$};
\draw (4.642177434129781,1.2572953187203044) node[anchor=north west] {$\blambda_1$};
\draw (5.86571776880912,2.5012869681373835) node[anchor=north west] {$\olambda_4$};
\draw (3.028298613397141,2.473332099611157) node[anchor=north west] {$\olambda_2$};
\begin{scriptsize}
\draw [fill=black] (-1.1960004123504433,0.2086020006538712) circle (0.5pt);
\draw [fill=black] (6.010003772890715,-3.7951925568639444) circle (0.5pt);
\draw [fill=black] (9.326214951798109,0.5003365563248241) circle (0.5pt);
\draw [fill=black] (4.8103375608045145,7.467729951356388) circle (0.5pt);
\draw [fill=black] (0.004190049586180855,1.6591257919140512) circle (0.5pt);
\draw [fill=black] (5.770285669832124,-1.544627865314141) circle (0.5pt);
\draw [fill=black] (8.423849318000979,1.8925657539909513) circle (0.5pt);
\draw [fill=black] (1.7958949427825415,1.2839659485570531) circle (1.5pt);
\draw [fill=black] (7.025266878165092,1.4289533343128125) circle (1.5pt);
\draw [fill=black] (5.377163216471651,-0.705855602279004) circle (1.5pt);
\draw [fill=black] (4.780948388681161,4.891635879938136) circle (0.5pt);
\draw [fill=black] (4.8103375608045145,7.467729951356388) circle (1.5pt);
\draw [fill=ffqqqq] (7.557001680270136,-1.7932952781050366) circle (1.5pt);
\draw [fill=xdxdff] (3.930002396408951,0.0817280060325498) circle (1.5pt);
\draw [fill=black] (5.727792827772634,1.8715778706573143) circle (1.5pt);
\draw [fill=black] (4.534036736542932,2.5348517209360004) circle (1.5pt);
\draw [fill=black] (3.984668849311784,1.8232487420720602) circle (1.5pt);
\draw [fill=black] (4.732775012473095,0.6690212268636206) circle (1.5pt);
\end{scriptsize}
\end{tikzpicture}
\caption{An illustration of the constructions in Definition \ref{def:tetra} and \ref{defn:rt}. The colored triangles are $\caT_{a,r}$ and  $\caT_{a,r,b}$.} \label{fig:td}
\end{figure}

We start by defining a family of regular tetrahedrons. Recall that in Definition \ref{def:tetra}, we have defined the tetrahedron $\fT_{a,r}$ with one face being $\caT_{a,r}$.
\begin{defn}  \label{defn:rt}
Let $a\in \Z^3$, $r \in \Z_+$.
For each $b \in \fT_{a,r} \cap \Z^3$, 
we define a regular tetrahedron $\fT_{a,r,b}$ characterized by the following conditions. Its four faces are orthogonal to $\blambda_1, \olambda_2, \olambda_3, \olambda_4$ respectively.
For $\tau \in \left\{2,3,4\right\}$, we consider the distances between the faces of $\fT_{a,r}$ and $\fT_{a,r,b}$ that are orthogonal to $\olambda_{\tau}$,
and they are the same for each $\tau$.
The point $b$ is at the boundary of the face orthogonal to $\blambda_1$.
Formally, we denote
\begin{equation}
F_{a,r,b}:= \max \left\{F : b\cdot \olambda_\tau \leq \bt_r(a) \cdot \olambda_\tau - F,\; \forall \tau \in \{2,3,4\} \right\}.
\end{equation}
Then $F_{a,r,b}\geq 0$ since $b \in \fT_{a,r}$,
and $\frac{F_{a,r,b}}{\sqrt{3}}$ would be the distance between the faces of $\fT_{a,r}$ and $\fT_{a,r,b}$ that are orthogonal to $\olambda_{\tau}$, for each $\tau \in \left\{2,3,4\right\}$.
Define
\begin{equation}
\fT_{a,r,b}:= \left\{c \in \R^3: c\cdot \blambda_1 \geq b \cdot \blambda_1, \;
b\cdot \olambda_\tau \leq \bt_r(a) \cdot \olambda_\tau - F_{a,r,b},\; \forall \tau \in \{2,3,4\} \right\},
\end{equation}
and let $\mathring\fT_{a,r,b}$ be the interior of $\fT_{a,r,b}$.
We denote $\caT_{a,r,b}:= \fT_{a, r, b} \cap \cP_{1, b\cdot \blambda_1}$ to be the face of $\fT_{a, r, b}$ orthogonal to $\blambda_1$, and we denote its three edges as
\begin{equation}
\cL_{a,r,b,\tau} := 
\left\{c\in \caT_{a,r,b}: c\cdot \olambda_\tau = \bt_r(a)\cdot \olambda_\tau - F_{a, r,b}\right\},
\;\forall \tau \in \{2,3,4\}. 
\end{equation}
Then $b$ is on one of these three edges.
We denote the three vertices by 
\begin{equation}
\cV_{a,r,b,\tau}:= \bigcap_{\tau' \in \left\{2, 3, 4\right\}\setminus \left\{\tau\right\}}\cL_{a,r,b,\tau'},\; \tau \in \left\{2,3,4\right\},
\end{equation}
or equivalently, $\cV_{a,r,b,\tau}$ is the unique point characterized by
$\cV_{a,r,b,\tau} \cdot \blambda_1=b\cdot \blambda_1$, and $\cV_{a,r,b,\tau} \cdot \olambda_{\tau'}=\bt_r(a)\cdot \olambda_{\tau'}-F_{a,r,b}$ for $\tau'\in\{2,3,4\}\setminus\{\tau\}$.
As $b\cdot \blambda_1$ and each $\bt_r(a)\cdot \olambda_{\tau'}-F_{a,r,b}$ are integers and have the same parity, we have $\cV_{a,r,b,\tau}\in\Z^3$.
We also denote the interior of these three edges by 
\begin{equation}
\mathring\cL_{a,r,b,\tau} := \cL_{a,r,b,\tau} \setminus \left\{\cV_{a,r,b,2},\cV_{a,r,b,3},\cV_{a,r,b,4}\right\}, \;\tau \in \left\{2,3,4\right\}.
\end{equation}
\end{defn}

We now define the pyramid using these tetrahedrons.
\begin{defn}
Take any $a\in \Z^3$, $r \in \Z_+$.
For any $b \in \fT_{a,r}\bigcap\Z^3$ let
\begin{equation}
\mathring\fH_{a,r,b} := \left\{c \in \R^3: 
c \cdot \blambda_1 > b \cdot \blambda_1
\right\} \setminus \fT_{a,r,b},
\end{equation}
which is an open half space minus a regular tetrahedron.
Let $\fH_{a,r,b}$ be the closure of $\mathring\fH_{a,r,b}$.

Let $\Gamma \subset \Z^3$, such that $a\in\Gamma$ and $\oT_{a, r} \cap \Gamma = \emptyset$.
We consider the collection of sets $\left\{\fH_{a, r, b}\right\}_{b \in \fT_{a,r} \cap \Gamma}$.
They form a partially ordered set (POSET) by inclusion of sets.
We take all the maximal elements in $\left\{\fH_{a, r, b}\right\}_{b \in \fT_{a,r} \cap \Gamma}$, and denote them as $\fH_{a, r, b_1}, \cdots, \fH_{a, r, b_m}$.
In particular $\fH_{a, r, a} = \fH_{a, r}$ is maximal since $\oT_{a, r} \cap \Gamma = \emptyset$, so we can assume that $b_1 = a$.
(For each $2 \leq i \leq m$, the choice of each $b_i \in \fT_{a,r} \cap \Gamma$ may not be unique, but always gives the same $\fH_{a, r, b_i}$.)
We note that since each $\fH_{a, r, b_i}$ is maximal, all the numbers $b_i \cdot \blambda_1$ for $1 \leq i \leq m$ must be mutually different, so we can assume that $b_1 \cdot \blambda_1 < \cdots < b_m \cdot \blambda_1$.

The \emph{pyramid} is defined as
\begin{equation}
\fP_{a, r, \Gamma} :=
\fT_{a,r,b_m}\cup\bigcup_{i=1}^{m-1} \left(\fT_{a,r,b_i}\cap \left\{c \in \R^3: 
c \cdot \blambda_1 \leq b_{i+1} \cdot \blambda_1
\right\}\right),
\end{equation}
and we let $\mathring\fP_{a, r, \Gamma}$ be the interior of $\fP_{a, r, \Gamma}$.
Note that in this definition, $\fP_{a, 0, \Gamma}:=\left\{a\right\}$.
Finally, let $\partial\fP_{a, r, \Gamma}:=\fP_{a, r, \Gamma} \setminus (\mathring \fP_{a, r, \Gamma}\cup \mathring \caT_{a,r})$ be the boundary of the pyramid (without the interior of its basement). 
See Figure \ref{fig:py} for an example of pyramid.
\end{defn}

In words, we construct the pyramid $\fP_{a,r,\Gamma}$ by stacking together some ``truncated'' regular tetrahedrons $\fT_{a,r,b}$, for $b \in \Gamma$,
so that $\fP_{a,r,\Gamma}$ intersects $\Gamma$ only at its boundary.
Indeed, for each $b\in\fT_{a,r}\cap\Gamma$ we have $b \in \fH_{a,r,b}$, and $\mathring \fP_{a, r, \Gamma} \cap \fH_{a,r,b} = \emptyset$.

Our key step towards proving Proposition \ref{prop:lvh} is the following estimate about points on the boundary of a pyramid.

\definecolor{ffqqqq}{rgb}{1.,0.,0.}
\definecolor{uuuuuu}{rgb}{0.26666666666666666,0.26666666666666666,0.26666666666666666}
\definecolor{xdxdff}{rgb}{0.49019607843137253,0.49019607843137253,1.}
\begin{figure}[!ht]
\centering
\begin{tikzpicture}[line cap=round,line join=round,>=triangle 45,x=1.3cm,y=1.0cm]
\clip(-0.5,-3.45) rectangle (9.5,7.93);
\draw [line width=0.3pt] (-0.3209199692347259,-0.01641594405543717)-- (6.65250911683071,-3.41280109478163);
\draw [line width=0.3pt,dash pattern=on 3pt off 3pt] (-0.3209199692347259,-0.01641594405543717)-- (9.326214951798109,0.5003365563248241);
\draw [line width=0.3pt] (9.326214951798109,0.5003365563248241)-- (6.65250911683071,-3.41280109478163);
\draw [line width=0.3pt] (0.30528548633098307,0.9366983387294354)-- (6.403765792489545,-2.0335459554618005);
\draw [line width=0.3pt] (6.403765792489545,-2.0335459554618005)-- (8.742004581149812,1.3886144488429621);
\draw [line width=0.3pt,dash pattern=on 3pt off 3pt] (0.30528548633098307,0.9366983387294354)-- (8.742004581149812,1.3886144488429621);
\draw [line width=0.3pt] (0.30528548633098307,0.9366983387294354)-- (-0.3209199692347259,-0.01641594405543717);
\draw [line width=0.3pt] (6.403765792489545,-2.0335459554618005)-- (6.65250911683071,-3.41280109478163);
\draw [line width=0.3pt] (8.742004581149812,1.3886144488429621)-- (9.326214951798109,0.5003365563248241);
\draw [line width=0.3pt] (6.092246975298758,-1.5039644534722274)-- (7.849348585644739,1.0676649226324464);
\draw [line width=0.3pt] (6.092246975298758,-1.5039644534722274)-- (7.849348585644739,1.0676649226324464);
\draw [line width=0.3pt,dotted] (0.30528548633098307,0.9366983387294354)-- (4.67,7.58);
\draw [line width=0.3pt,dotted] (4.67,7.58)-- (6.403765792489545,-2.0335459554618005);
\draw [line width=0.3pt,dotted] (8.742004581149812,1.3886144488429621)-- (4.67,7.58);
\draw [line width=0.3pt,dotted] (4.789384889219156,5.720266519599415)-- (7.849348585644739,1.0676649226324464);
\draw [line width=0.3pt,dash pattern=on 3pt off 3pt] (1.5094602990268404,0.7280663629503787)-- (7.849348585644739,1.0676649226324464);
\draw [line width=0.3pt] (6.092246975298758,-1.5039644534722274)-- (1.5094602990268404,0.7280663629503787);
\draw [line width=0.3pt,dotted] (4.789384889219156,5.720266519599415)-- (6.092246975298758,-1.5039644534722274);
\draw [line width=0.3pt] (1.9357859682784129,1.376954169452178)-- (5.922900535916804,-0.5649565526100568);
\draw [line width=0.3pt] (5.922900535916804,-0.5649565526100568)-- (7.451613505190447,1.6724115974551057);
\draw [line width=0.3pt] (1.9357859682784129,1.376954169452178)-- (1.5094602990268404,0.7280663629503787);
\draw [line width=0.3pt] (7.451613505190447,1.6724115974551057)-- (7.849348585644739,1.0676649226324464);
\draw [line width=0.3pt] (5.922900535916804,-0.5649565526100568)-- (6.092246975298758,-1.5039644534722274);
\draw [line width=0.3pt,dash pattern=on 3pt off 3pt] (1.9357859682784129,1.376954169452178)-- (7.451613505190447,1.6724115974551057);
\draw [line width=0.3pt,dotted] (1.9357859682784129,1.376954169452178)-- (4.789384889219156,5.720266519599415);
\draw [line width=0.3pt,dotted] (3.198124031599888,1.1582449907887755)-- (4.914536227926674,3.770705429970543);
\draw [line width=0.3pt,dotted] (4.914536227926674,3.770705429970543)-- (5.5963349440818115,-0.009795557259423717);
\draw [line width=0.3pt,dotted] (4.914536227926674,3.770705429970543)-- (6.515841033703964,1.3359597807678745);
\draw [line width=0.3pt] (3.198124031599888,1.1582449907887755)-- (5.5963349440818115,-0.009795557259423717);
\draw [line width=0.3pt] (5.5963349440818115,-0.009795557259423717)-- (6.515841033703964,1.3359597807678745);
\draw [line width=0.3pt,dash pattern=on 3pt off 3pt] (3.198124031599888,1.1582449907887755)-- (6.515841033703964,1.3359597807678745);
\draw [line width=0.3pt] (3.5343802166767158,1.6700427878583475)-- (3.198124031599888,1.1582449907887755);
\draw [line width=0.3pt] (5.462766191904897,0.7308288931876199)-- (5.5963349440818115,-0.009795557259423717);
\draw [line width=0.3pt] (6.20213512652842,1.8129421078483692)-- (6.515841033703964,1.3359597807678745);
\draw [line width=0.3pt] (3.5343802166767158,1.6700427878583475)-- (5.462766191904897,0.7308288931876199);
\draw [line width=0.3pt] (5.462766191904897,0.7308288931876199)-- (6.20213512652842,1.8129421078483692);
\draw [line width=0.3pt,dash pattern=on 3pt off 3pt] (3.5343802166767158,1.6700427878583475)-- (6.20213512652842,1.8129421078483692);
\draw [line width=0.3pt] (4.3197227276794425,1.5339765306950959)-- (5.259598871398082,1.0762130202829052);
\draw [line width=0.3pt] (5.259598871398082,1.0762130202829052)-- (5.619959936032505,1.603624237916336);
\draw [line width=0.3pt,dash pattern=on 3pt off 3pt] (4.3197227276794425,1.5339765306950959)-- (5.619959936032505,1.603624237916336);
\draw [line width=0.3pt] (4.992397039900687,2.557818608310598)-- (4.3197227276794425,1.5339765306950959);
\draw [line width=0.3pt] (4.992397039900687,2.557818608310598)-- (5.259598871398082,1.0762130202829052);
\draw [line width=0.3pt] (4.992397039900687,2.557818608310598)-- (5.619959936032505,1.603624237916336);
\draw [line width=0.3pt,] (1.9357859682784129,1.376954169452178)-- (3.393105538161571,1.4550160539136827);
\draw [line width=0.3pt,] (3.5343802166767158,1.6700427878583475)-- (4.441027210536704,1.718607683537161);
\draw [line width=0.3pt,] (5.5067904702262425,1.775695703323975)-- (6.20213512652842,1.8129421078483692);
\draw [line width=0.3pt,] (6.333935530956374,1.6125427459740107)-- (7.451613505190447,1.6724115974551057);
\draw [line width=0.3pt,] (0.30528548633098307,0.9366983387294354)-- (1.69545787445072,1.011163458518966);
\draw [line width=0.3pt,] (7.675824526548343,1.3315041032760435)-- (8.742004581149812,1.3886144488429621);
\draw [line width=0.3pt,dotted] (-0.08423795352736203,0.34382526375264616)-- (6.558493533166348,-2.8914947349479383);
\draw [line width=0.3pt,dotted] (6.558493533166348,-2.8914947349479383)-- (9.105405492389503,0.8360720569354899);
\draw [line width=0.3pt,dotted] (0.09352349462935808,0.6143865845792646)-- (6.487882564895148,-2.4999644651875137);
\draw [line width=0.3pt,dotted] (6.487882564895148,-2.4999644651875137)-- (8.939565222887543,1.088228219319889);
\draw [line width=0.3pt,dotted] (8.939565222887543,1.088228219319889)-- (0.09352349462935808,0.6143865845792646);
\draw [line width=0.3pt,dotted] (-0.08423795352736203,0.34382526375264616)-- (9.105405492389503,0.8360720569354899);
\draw [line width=0.3pt,dotted] (0.9117591437279704,0.8316224011939087)-- (6.246871665143767,-1.7668261843661333);
\draw [line width=0.3pt,dotted] (6.246871665143767,-1.7668261843661333)-- (8.292425051098597,1.2269706152828215);
\draw [line width=0.3pt,dotted] (8.292425051098597,1.2269706152828215)-- (0.9117591437279704,0.8316224011939087);
\draw [line width=0.3pt,dotted] (2.5245950967771402,1.2749387390177755)-- (5.770576204493181,-0.30600542743460774);
\draw [line width=0.3pt,dotted] (5.770576204493181,-0.30600542743460774)-- (7.015128708115337,1.5154759027140576);
\draw [line width=0.3pt,dotted] (7.015128708115337,1.5154759027140576)-- (2.5245950967771402,1.2749387390177755);
\draw [line width=0.3pt,dotted] (4.730127308459957,6.643359112189763)-- (0.9117591437279704,0.8316224011939087);
\draw [line width=0.3pt,dotted] (4.730127308459957,6.643359112189763)-- (6.246871665143767,-1.7668261843661333);
\draw [line width=0.3pt,dotted] (4.730127308459957,6.643359112189763)-- (8.292425051098597,1.2269706152828215);
\draw [line width=0.3pt,dotted] (4.84776089230103,4.810906816764349)-- (2.5245950967771402,1.2749387390177755);
\draw [line width=0.3pt,dotted] (4.84776089230103,4.810906816764349)-- (5.770576204493181,-0.30600542743460774);
\draw [line width=0.3pt,dotted] (7.015128708115337,1.5154759027140576)-- (4.84776089230103,4.810906816764349);
\draw (7.9196363826290176,-1.5606247371911564) node[anchor=north] {$a=b_1$};
\draw (2.9615477485690636,6.408067459875722) node[anchor=north west] {$\mathfrak T_{a,r}$};
\draw (3.878482414470606,-0.4156323793916066) node[anchor=north west] {$b_2$};
\draw (6.206426786456868,1.2369824254309496) node[anchor=north west] {$b_3$};
\draw (4.752768956376728,1.3862508594149225) node[anchor=north west] {$b_4$};

\draw (4.67,7.58) node[anchor=east] {$\bt_r(a)$};
\begin{scriptsize}
\draw [fill=uuuuuu] (4.67,7.58) circle (1.5pt);

\draw [fill=black] (-0.3209199692347259,-0.01641594405543717) circle (0.5pt);
\draw [fill=black] (6.65250911683071,-3.41280109478163) circle (0.5pt);
\draw [fill=black] (9.326214951798109,0.5003365563248241) circle (0.5pt);
\draw [fill=black] (0.30528548633098307,0.9366983387294354) circle (0.5pt);
\draw [fill=black] (6.403765792489545,-2.0335459554618005) circle (0.5pt);
\draw [fill=black] (8.742004581149812,1.3886144488429621) circle (0.5pt);
\draw [fill=black] (1.5094602990268404,0.7280663629503787) circle (0.5pt);
\draw [fill=black] (7.849348585644739,1.0676649226324464) circle (0.5pt);
\draw [fill=black] (6.092246975298758,-1.5039644534722274) circle (0.5pt);
\draw [fill=black] (1.9357859682784129,1.376954169452178) circle (0.5pt);
\draw [fill=black] (7.451613505190447,1.6724115974551057) circle (0.5pt);
\draw [fill=black] (5.922900535916804,-0.5649565526100568) circle (0.5pt);
\draw [fill=black] (3.198124031599888,1.1582449907887755) circle (0.5pt);
\draw [fill=black] (6.515841033703964,1.3359597807678745) circle (0.5pt);
\draw [fill=black] (5.5963349440818115,-0.009795557259423717) circle (0.5pt);
\draw [fill=black] (3.5343802166767158,1.6700427878583475) circle (0.5pt);
\draw [fill=black] (6.20213512652842,1.8129421078483692) circle (0.5pt);
\draw [fill=black] (5.462766191904897,0.7308288931876199) circle (0.5pt);
\draw [fill=xdxdff] (4.3197227276794425,1.5339765306950959) circle (0.5pt);
\draw [fill=uuuuuu] (5.619959936032505,1.603624237916336) circle (0.5pt);
\draw [fill=uuuuuu] (5.259598871398082,1.0762130202829052) circle (0.5pt);
\draw [fill=uuuuuu] (4.992397039900687,2.557818608310598) circle (0.5pt);
\draw [fill=ffqqqq] (7.9196363826290176,-1.5606247371911564) circle (1.5pt);
\draw [fill=ffqqqq] (3.9299307231322262,-0.4508156052667598) circle (1.5pt);
\draw [fill=ffqqqq] (6.0922469752987585,-1.5039644534722265) circle (1.5pt);
\draw [fill=ffqqqq] (6.206426786456868,0.889824254309496) circle (1.5pt);
\draw [fill=ffqqqq] (4.735641147369931,1.3314048651463593) circle (1.5pt);
\draw [fill=ffqqqq] (7.1110057427514874,-0.012946541861660554) circle (1.5pt);
\draw [fill=ffqqqq] (5.308709525705433,-.0815867351027142) circle (1.5pt);
\draw [fill=ffqqqq] (8.033059086513155,0.8473721163899153) circle (1.5pt);
\draw [fill=ffqqqq] (7.685278912284848,-0.7474991531891553) circle (1.5pt);
\draw [fill=ffqqqq] (1.3230749707444311,-0.3416017563590088) circle (1.5pt);
\end{scriptsize}
\end{tikzpicture}
\caption{Pyramid $\fP_{a,r,\Gamma}$, where $\Gamma$ is the collection of red points.} \label{fig:py}
\end{figure}
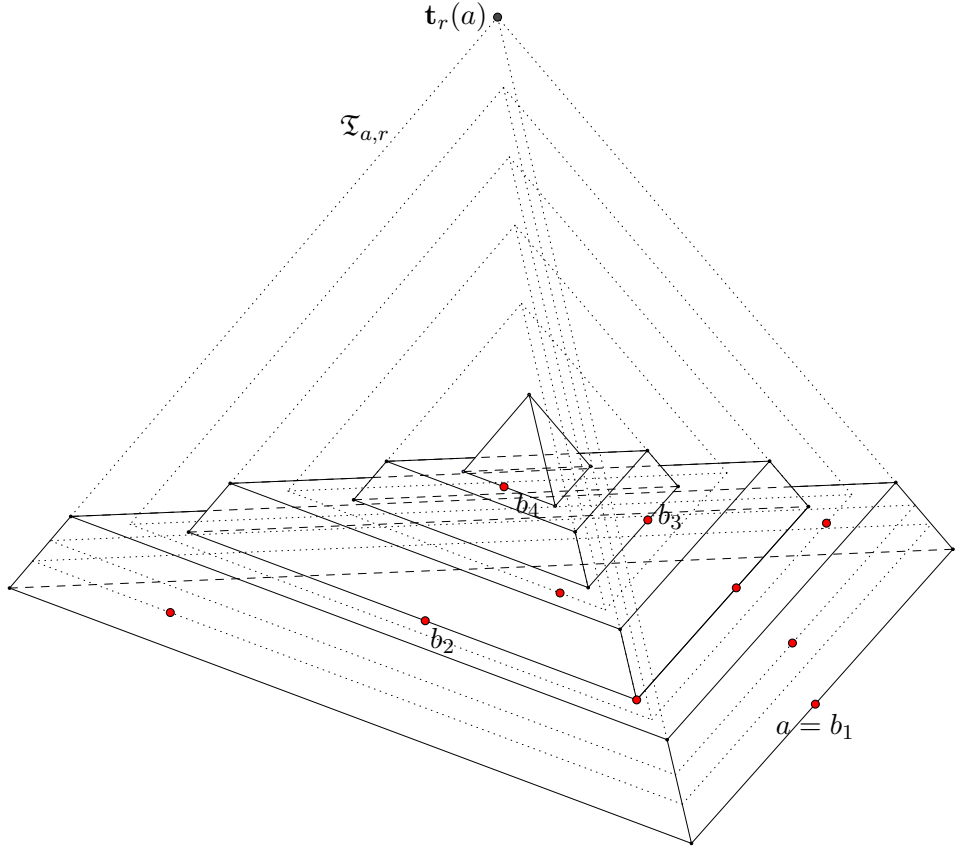

\begin{prop}\label{prop:tetraF}
There exists a constant $C_{9}$, such that for any $K\in \R_+$, $N \in \Z_+$, and any small enough $\varepsilon \in \R_+$, there are small $C_{10} \in \R_+$ depending only on $K$ and large $C_{\varepsilon, N} \in \R_+$ depending only on $\varepsilon, N$, such that the following statement holds.

Take any $g\in\R_+$, $n, r\in \Z_+$ with $0 \leq r<\frac{n}{32}$, and functions $u, V$ 
satisfying $\Delta u = Vu$ in $Q_n$ and $\|V\|_{\infty}\leq K$.
Suppose we have that
\begin{enumerate}
    \item $\Gamma:=\left\{b \in Q_{n}: |u(b)|\geq \exp(3 C_{10} n) g\right\}$, and $a \in \Gamma \cap Q_{\frac{n}{2}}$;
    \item $|u(b)| < g$ for each $b \in \oT_{a,r} \cap \Z^3$, and either $|u(b)| < g$ for each $b \in \oT_{a-\frac{\blambda_1}{3}, r} \cap \Z^3$ or $|u(b)| < g$ for each $b \in \oT_{a+\frac{\blambda_1}{3} ,r} \cap \Z^3$;
    \item $\Vec{l}$ is a vector of positive reals,
    $E$ is an $(N,\Vec{l},\varepsilon^{-1},\varepsilon)$-graded set;
    in addition, the first scale length of $E$ is $l_1>C_{\varepsilon,N}$, and $E$ is $(\varepsilon^{-\frac{1}{2}},\varepsilon)$-normal in $\fT_{a, r}$;
    \item for each $b \in Q_{n}$ with $b \cdot \blambda_1 \geq a \cdot \blambda_1$, $g \leq |u(b)| \leq \exp(3C_{10} n) g$ implies $b \in E$.
\end{enumerate}
Then
\begin{equation}   
    \left| \left\{b \in \partial\fP_{a,r,\Gamma}\cap \Z^3: |u(b)|\geq \exp(C_{10}n)g \right\}
    \setminus E \right| \geq C_{9} (r^2 +1).
\end{equation}
\end{prop}
The proof of Proposition \ref{prop:tetraF} is left for the next subsection. 
We now finish the proof of Proposition \ref{prop:lvh} assuming it.

\begin{proof} [Proof of Proposition \ref{prop:lvh}]
The idea is to first apply Lemma \ref{lem:bas} to find some triangles $\caT_{a_i,r_i}$ in $\cP_{1,k}\cup\cP_{1,k+1}$, and build pyramids using these triangles as basements,
then apply Proposition \ref{prop:tetraF} to lower bound the number of desired points on the boundary of each pyramid and finally sum them up.

For the parameters, we take $C_{7}=\max\left\{6C_{10},\log(K+11)\right\}$ where $C_{10}$ is the constant in Proposition \ref{prop:tetraF}. 
We leave $C_8$ to be determined.
We require that $\varepsilon$ is small as required by Lemma \ref{lem:bas} and Proposition \ref{prop:tetraF}; and for each such $\varepsilon$ we let $C_{\varepsilon, N}$ be large enough as required by Lemma \ref{lem:bas} and Proposition \ref{prop:tetraF}.

Without loss of generality, we assume $\tau=1$.
We can also assume $n > 100$, by letting $C_{\varepsilon, N} > 100$.
Denote
\begin{equation}
\Upsilon:=\left\{a \in Q_{n}: |u(a)| \geq \exp(- C_{7} n^{3})|u(a_0)|, a \cdot \blambda_1 \geq k \right\}\setminus E.
\end{equation}
If $|\Upsilon| \geq n^2$, the conclusion follows by letting $h = 3n$ and $C_{8}<\frac{1}{3}$.
Now we assume that $|\Upsilon| < n^2$.

The interval $[\exp(-C_{7} n^3)|u(a_0)|, |u(a_0)|)$ is the union of $2n^2$ disjoint intervals, which are
\begin{equation}
\left[\exp\left(-\frac{C_{7}(i+1)n}{2}\right)|u(a_0)|,\exp\left(-\frac{C_{7}in}{2}\right)|u(a_0)|\right), \; i = 0, \cdots, 2n^2-1.
\end{equation}
By the Pigeonhole principle, at least $n^2$ of these intervals do not intersect the set $\left\{|u(a)| : a \in \Upsilon \right\}$; i.e., we can find $\exp(-C_{7} n^3)|u(a_0)| \leq g_1, \cdots, g_{n^2} \leq |u(a_0)|$, such that $g_i \leq g_{i+1}\exp\left(-\frac{C_{7}n}{2}\right)$, for each $1 \leq i \leq n^2-1$, and
    \begin{equation}
        \left\{a \in Q_n: |u(a)| \in \bigcup_{i=1}^{n^2} \left[g_i, g_i \exp\left(\frac{C_{7} n}{2}\right)\right), a \cdot \blambda_1 \geq k  \right\} \subset E .
    \end{equation}
We remark that actually we just need $g_1,\cdots, g_{100n}$ to apply Lemma \ref{lem:bas}, rather than $n^2$ numbers; but we cannot get a better quantitative lower bound for $|u|$ by optimizing this, since applying the Pigeonhole principle to $2n^2$ parts or $n^2+100n$ parts does not make any essential difference.

As we assume that $a_0 \in \cP_{1, k} \cap Q_{\frac{n}{4}}$ and $0 \leq k \leq \frac{n}{10}$, we can apply Lemma \ref{lem:bas} with $D=\frac{C_{7}}{2}$.
Then we can find some $a_1, \cdots, a_m$, $r_1, \cdots, r_m$ and $g_{s_1}, \cdots, g_{s_m}$, satisfying the conditions there.
In particular, we have $|u(a_{i})| \geq g_{s_i} \exp\left(\frac{C_{7}n}{2}\right) > \exp(-C_{7} n^3)|u(a_0)|$, for each $1\leq i \leq m$.

If $m > n$, we can just take $h=2$, and \eqref{eq:lvh}  holds by taking $C_{8}$ small.
Now assume that $m\leq n$.
We argue by contradiction, assuming that \eqref{eq:lvh} does not hold.

As $C_{7} \geq 6 C_{10}$, 
we can apply Proposition \ref{prop:tetraF} to $a=a_i$, $r=r_{i}$ and $g=g_{s_i}$ for each $i=1,2,\cdots,m$, and get that
\begin{equation}
\left|\Upsilon \cap \fT_{a_i, r_i}\right|\geq
    \left| \left\{b \in \fT_{a_i,r_i}\cap \Z^3: |u(b)|\geq \exp(C_{10}n)g_{s_i} \right\}
    \setminus E \right| \geq C_{9} (r_{i}^2 +1).
\end{equation}
As we have assumed that $\eqref{eq:lvh}$ does not hold, for each $h \in \Z_+$,
\begin{equation}\label{eq:opt}
     C_{9}\sum_{i=1}^m \mathds{1}_{h > 4r_i}(r_i^2+1) 
     \leq \sum_{i=1}^m \mathds{1}_{h > 4r_i} \left|\Upsilon \cap \fT_{a_i, r_i}\right|
     \leq 2\left| \left(\bigcup_{i=0}^h \cP_{1, k + i} \right) \cap \Upsilon \right| \leq 2C_{8}hn(\log_2 n)^{-1}
\end{equation}
where the second inequality is due to the fact that any point is contained in at most two tetrahedrons $\fT_{a_i, r_i}$, by Conclusion $3$ in Lemma \ref{lem:bas}.

Take $l := \left\lfloor \log_2 n\right\rfloor - 5$.
For each $0 \leq j \leq l$, let $M_j=|\left\{i:1 \leq i \leq m, 2^j \leq r_i +1 <2^{j+1}\right\}|$.
Then we have that
\begin{equation}  \label{eq:lvhpf1}
\sum_{j=0}^{l}2^j M_j \geq \frac{1}{2}\sum_{i=1}^m (r_i +1) \geq \frac{n}{200} ,
\end{equation}
by Lemma \ref{lem:bas}.
For each $0 \leq s \leq l$, by taking $h = 2^{s+3}$ in equation \eqref{eq:opt} we get
\begin{equation}  \label{eq:lvhpf2}
C_{9}
\sum_{j=0}^{s} 2^{2j} M_j \leq C_{8}2^{s+4}n (l+5)^{-1} .
\end{equation}
Multiplying both sides of \eqref{eq:lvhpf2} by $2^{-s}$ and summing over all $s\in\Z_{\geq 0}$, we get
\begin{equation}  \label{eq:lvhpf3}
\sum_{j=0}^{l}2^j M_j\leq
\sum_{s=0}^l\sum_{j=0}^{s} 2^{2j-s} M_j \leq
\sum_{s=0}^l 2^4 C_{8}(C_{9})^{-1} n (l+5)^{-1} < 2^4 C_{8}(C_{9})^{-1} n.
\end{equation}
This contradicts with \eqref{eq:lvhpf1} whenever $C_8 < (200\cdot 2^4)^{-1}C_9$.
\end{proof}

\subsection{Multi-layer structure of the pyramid and estimates on the boundary}   \label{ssec:multila}

The purpose of this subsection is to prove Proposition  \ref{prop:tetraF}.
We first show that, under slightly different conditions, there are many points in $\Gamma$ on the boundary of a pyramid without removing the graded set.
\begin{prop}   \label{prop:pmt}
There exists a constant $C_{9}'$, so that for any $K \in \R_+$, there is $C_{10} >K+11$, relying only on $K$, and the following is true.

Take any $g\in\R_+$, $n,r\in \Z$ with $0 \leq r<\frac{n}{32}$, and functions $u, V$ 
satisfying $\Delta u = Vu$ in $Q_n$ and $\|V\|_{\infty}\leq K$.
Suppose we have
\begin{enumerate}
\item $\Gamma:=\left\{b \in Q_{n}: |u(b)|\geq \exp(3 C_{10} n) g\right\}$, and $a \in \Gamma \cap Q_{\frac{n}{2}}$;
    \item $|u(b)| < g$ for each $b \in \oT_{a,r} \cap \Z^3$, and either $|u(b)| < g$ for each $b \in \oT_{a-\frac{\blambda_1}{3} ,r} \cap \Z^3$ or for each $b \in \oT_{a+\frac{\blambda_1}{3} ,r} \cap \Z^3$;
    \item $h:= \max\{a \cdot \blambda_1\}\cup \left\{ b\cdot \blambda_1: b \in \mathring \fP_{a,r,\Gamma} \cap \Z^3, |\cL_{a,r,b,2}\cap \Z^3| \geq \frac{r}{4}  \right\}$, and $|u(b)| \leq \exp(C_{10} n)g$ for each $b \in \mathring \fP_{a,r,\Gamma} \cap \Z^3$ with $b \cdot \blambda_1 \leq h$.
\end{enumerate}
Then
\begin{equation}   \label{eq:pmt}
    \left| \left\{b \in \partial\fP_{a,r,\Gamma}: |u(b)|\geq \exp(C_{10}n)g \right\} \right| \geq C_{9}' (r^2+1) .
\end{equation}
\end{prop}
To prove Proposition \ref{prop:pmt}, we analyze the structure of the pyramid boundary $\partial\fP_{a,r,\Gamma}$.
Specifically, we study faces of it and estimate the number of lattice points $b$ with $|u(b)|\geq \exp(C_{10}n)g$ on each face.
For some of the faces, we can show that the number of such points is proportional to the area of the face.
This is by observing that the lattice $\Z^3$ restricted to the face is a triangular lattice, and then using results from Section \ref{sec:pol}.
Finally we sum up the points on all the faces and get the conclusion.
\begin{proof} [Proof of Proposition \ref{prop:pmt}]
We can assume that $r>100$, since otherwise the statement holds by taking $C_{9}'<10^{-5}$.

We take $a=b_1,\cdots,b_m$ from the definition of $\fP_{a,r,\Gamma}$.
As $\fH_{a, r, b_1}, \cdots, \fH_{a, r, b_m}$ are all the maximal elements in $\left\{\fH_{a, r, b}\right\}_{b \in \fT_{a,r} \cap \Gamma}$, we have that
\begin{equation}
    \bigcup_{b \in \fT_{a,r} \cap \Gamma } \fH_{a, r, b}
    =
    \bigcup_{i=1}^m \fH_{a, r, b_i} .
\end{equation}
We can also characterize $\mathring\fP_{a,r,\Gamma}$ as the half space $\{b\in\R^3: b\cdot \blambda_1>a\cdot \blambda_1\}$ minus $\bigcup_{i=1}^m \fH_{a,r,b_i}$.

For each $s \in \Z$, we take $m_s \in \left\{1, \cdots, m\right\}$ to be the maximum number such that $b_{m_s} \cdot \blambda_1 \leq s$.

We first study the faces of $\partial \fP_{a,r,\Gamma}$ that are orthogonal to $\blambda_1$.
For $2\leq i \leq m$, we denote 
\begin{equation}
\mathring\cR_{i}:=\mathring\fT_{a,r,b_{i-1}}\cap \cP_{1, b_i \cdot \blambda_1}=\left\{
b\in\cP_{1,b_i\cdot\blambda_1}: b\cdot \olambda_\tau < \bt_r(a)\cdot\olambda_\tau - F_{a,r,b_{i-1}},\;\forall \tau\in\{2,3,4\}
\right\}.
\end{equation}
Let $\cR_i$ be the closure of $\mathring\cR_i$, then $\cR_i \supset \caT_{a,r,b_i}$ and it has the same center as $\caT_{a,r,b_i}$.
We denote the side length of $\cR_i$ to be $\theta_i$.
Note that the three vertices of $\cR_i$ are in $\frac{1}{2}\Z^3$, so $\frac{\theta_i}{\sqrt{2}}\in \frac{1}{2}\Z_{\geq 0}$.
Further, for each $1 \leq i \leq m+1$, we denote the side length of $\caT_{a,r,b_i}$ to be $\vartheta_i$.
Note that the vertices of $\caT_{a,r,b_i}$ are $\cV_{a,r,b_i,\tau}$, for $\tau \in \{2,3,4\}$, and each of them is in $\Z^3$.
Thus we have $\frac{\vartheta_i}{\sqrt{2}} = |\cL_{a,r,b_i,2}\cap\Z^3|-1 \in \Z_{\geq 0}$.
We also obviously have that
$2\sqrt{2}r = \vartheta_1 > \theta_2 > \vartheta_2 > \cdots > \theta_m > \vartheta_m\geq 0$.
For simplicity of notations, we also denote $b_{m+1}:= \argmax_{b \in \fP_{a,r,\Gamma}} b \cdot \blambda_1$, and $\theta_{m+1}=\vartheta_{m+1}=0$.

The following results will be useful in analyzing the face $\cR_i$, for $1 \leq i \leq m_{h+1}$.
\begin{cla}   \label{lem:gam1}
For any $2\leq i \leq m_{h+1}$ and $b \in \mathring\cR_i \cap \Z^3$, if $b+\be_1-\be_3, b+\be_2-\be_3\in\mathring\cR_i$, then we have
\begin{equation} \label{eq:gam1}
|u(c)| < \exp(C_{10}n)g,\;\forall c \in 
\left\{b-\be_3, b-\be_1-\be_3, b-\be_2-\be_3, b-2\be_3\right\} .
\end{equation}
\end{cla}
\begin{cla}   \label{lem:edg}
If $C_{10} > K+11$, then for each $2 \leq i \leq m_h$, there exists $\tau_i \in \left\{2,3,4\right\}$, such that $b_i \in \cL_{a,r,b_i,\tau_i}$, and $|u(b)|\geq \exp(2C_{10}n)g$, $\forall b \in \mathring\cL_{a,r,b_i,\tau_i} \cap \Z^3$.
\end{cla}
We continue our proof assuming these claims.
Fix $2 \leq i \leq m_{h+1}$.
For any $b \in \mathring\cR_i \cap \Z^3$ with $b+\be_1-\be_3, b+\be_2-\be_3\in\mathring\cR_i$,
since $\Delta u(b-\be_3)=(V u)(b-\be_3)$, and $|V(b-\be_3)| \leq K$,
by Claim \ref{lem:gam1} we have
\begin{multline}
\left|u(b)+u(b+\be_1-\be_3)+u(b+\be_2-\be_3) \right| \\
\leq
(K+6)\max_{c \in \left\{b-\be_3, b-\be_1-\be_3, b-\be_2-\be_3, b-2\be_3\right\}} |u(c)|
\leq
(K+6)\exp(C_{10} n)g.
\end{multline}
We take $C_{10} > 2\ln(C_4(K+6))$ where $C_{4}$ is the constant in Theorem \ref{thm:bou}.
Then if $i \leq m_h$,
using Claim \ref{lem:edg} and $b_i\in\Gamma$, we have
\begin{equation} \label{eq:pmtp2}
\left|u(b)+u(b+\be_1-\be_3)+u(b+\be_2-\be_3) \right|
<
C_4^{-2n}\min_{c \in \left(\mathring\cL_{a,r,b_i,\tau_i}\cap \Z^3\right) \cup \left\{ b_i\right\}} |u(c)|,
\end{equation}
where $\tau_i \in \left\{2,3,4\right\}$ is  given by Claim \ref{lem:edg}.
If $m_{h}< m_{h+1}$ and $i = m_{h+1}$, as $b_i\in\Gamma$ we have
\begin{equation} \label{eq:pmtp22}
\left|u(b)+u(b+\be_1-\be_3)+u(b+\be_2-\be_3) \right|
<
C_4^{-2n} |u(b_i)|.
\end{equation}
Without loss of generality, we assume that $\tau_i = 2$ in the former case,
and $b_i \in \cL_{a,r,b_{m_{h+1}},2}$ in the later case.
We consider the following trapezoid in $\cR_i$:
\begin{equation}
\mathring\cW_i:= \left\{ b \in \mathring\cR_i: b \cdot \olambda_2 \geq b_i\cdot \olambda_2 \right\} ,
\end{equation}
and let $\cW_i$ be the closure of $\mathring \cW_i$.
See Figure \ref{fig:loc} for an illustration of $\cW_i$.
Then $\mathring{\cW}_i\cap\Z^3$ can be treated as $P_{\mathbf{0}; \frac{\vartheta_i}{\sqrt{2}} + 2\left\lceil\frac{\theta_i - \vartheta_i}{3\sqrt{2}} \right\rceil-2, \left\lceil\frac{\theta_i - \vartheta_i}{3\sqrt{2}} \right\rceil-1}$ (see Definition \ref{defn:pt}). 
We apply Corollary \ref{cor:poly} to $\mathring{\cW}_i\cap\Z^3$, with $L = \mathring \cL_{a,r,b_i,2} \cap \Z^3$ if $\vartheta_i \geq 2\sqrt{2}$ (thus $\mathring \cL_{a,r,b_i,2} \cap \Z^3$ is not empty) and $i \leq m_{h}$, and with $L = \left\{b_i\right\}$ otherwise.
If $i\leq m_h$,
we have $\frac{\epsilon_2(\theta_{i} - \vartheta_{i})^2}{(3\sqrt{2})^2}\geq
\frac{\epsilon_2(\theta_i+2\vartheta_i)(\theta_i-\vartheta_i)}{5\cdot 3\sqrt{2}\cdot 3\sqrt{2}}$ when $\vartheta_i =\sqrt{2}$ or $0$, since $\theta_{i} - \vartheta_{i}\geq \frac{\sqrt{2}}{2}$.
Thus we always have
\begin{equation}
\left| \left\{ b\in \mathring{\cW}_i\cap\Z^3: |u(b)| \geq C_4^{-\frac{2(\theta_i - \vartheta_i)}{3\sqrt{2}}} \min_{c \in \left(\mathring\cL_{a,r,b_i,2}\cap \Z^3\right) \cup \left\{ b_i\right\} } |u(c)| \right\} \right| 
\geq 
\frac{\epsilon_2(\theta_i+2\vartheta_i)(\theta_i-\vartheta_i)}{5\cdot 3\sqrt{2}\cdot 3\sqrt{2}}.
\end{equation}
Since $\frac{\theta_i - \vartheta_i}{3\sqrt{2}} < n$, and $C_4^{-2n} \min_{c \in \left(\mathring\cL_{a,r,b_i,2}\cap \Z^3\right)\cup \left\{ b_i\right\} } |u(c)| \geq \exp(C_{10} n)g$ by Claim \ref{lem:edg}, 
we have
\begin{equation}  \label{eq:pmtp1}
\left|\left\{b \in \mathring{\cW}_i\cap\Z^3: |u(b)|\geq \exp(C_{10}n)g \right\}\right|
\geq
\frac{\epsilon_2(\theta_i+2\vartheta_i)(\theta_i-\vartheta_i)}{5\cdot 3\sqrt{2}\cdot 3\sqrt{2}}
\geq \frac{\epsilon_2\theta_i(\theta_i-\vartheta_i)}{5\cdot 3\sqrt{2}\cdot 3\sqrt{2}}.
\end{equation}
If $m_{h}<m_{h+1}$ and $i=m_{h+1}$, we have
\begin{equation}
\left| \left\{ b\in \mathring{\cW}_i\cap\Z^3: |u(b)| \geq C_4^{-\frac{2(\theta_i - \vartheta_i)}{3\sqrt{2}}} |u(b_{i})| \right\} \right| \geq \frac{\epsilon_2(\theta_{i} - \vartheta_{i})^2}{(3\sqrt{2})^2}.
\end{equation}
Since $\frac{\theta_i - \vartheta_i}{3\sqrt{2}} < n$, and $C_4^{-2n}|u(b_{i})| \geq \exp(C_{10} n)g$, we have
\begin{equation}  \label{eq:pmtp11}
\left|\left\{b \in \mathring\cW_{i}\cap \Z^3: |u(b)|\geq \exp(C_{10}n)g \right\}\right|
\geq
\frac{\epsilon_2(\theta_{i} - \vartheta_{i})^2}{(3\sqrt{2})^2}.
\end{equation}
For the cases where $\tau_i = 3, 4$, $i\leq m_h$, and the cases where $m_{h}<m_{h+1}=i$ and $b_{i} \in \cL_{a,r,b_{m_{h+1}},3}$ or $\cL_{a,r,b_{m_{h+1}},4}$, we can argue similarly and get \eqref{eq:pmtp1} and \eqref{eq:pmtp11} as well.

We then study other faces of $\fP_{a,r,\Gamma}$.
Again we fix $2 \leq i \leq m_h$,
and assume that $\tau_i = 2$, for $\tau_i$ given by Claim \ref{lem:edg}.
We define
\begin{equation}
\haS_i := \left\{b \in \cP_{2, b_i\cdot \blambda_2}:
b \cdot  \olambda_\tau  < \bt_r(a)\cdot \olambda_\tau -F_{a,r,b_i}, \forall \tau\in\{3,4\},\;
b_i \cdot  \blambda_1  \leq b \cdot  \blambda_1  < b_{i+1} \cdot  \blambda_1 \right\}.
\end{equation}
Let $\mathring\cS_i:=\left\{b \in \haS_i : b\cdot \blambda_1 < h+1\right\}$, and $\cS_i$ be the closure of $\mathring\cS_i$.
Then $\cS_i\subset \cP_{2,\blambda_2\cdot b_i}$ and is a trapezoid.
It is a face of $\partial\fP_{a, r, \Gamma}$, for $2\leq i<m_h$, and for $i=m_h$ when $m_{h+1}>m_h$;
and it is part of a face of $\partial\fP_{a, r, \Gamma}$ for $i=m_h$ when $m_{h+1}=m_h$. See Figure \ref{fig:loc} for an illustration.

\definecolor{qqwuqq}{rgb}{0.,0.39215686274509803,0.}
\definecolor{ffwwqq}{rgb}{1.,0.4,0.}
\definecolor{xdxdff}{rgb}{0.49019607843137253,0.49019607843137253,1.}
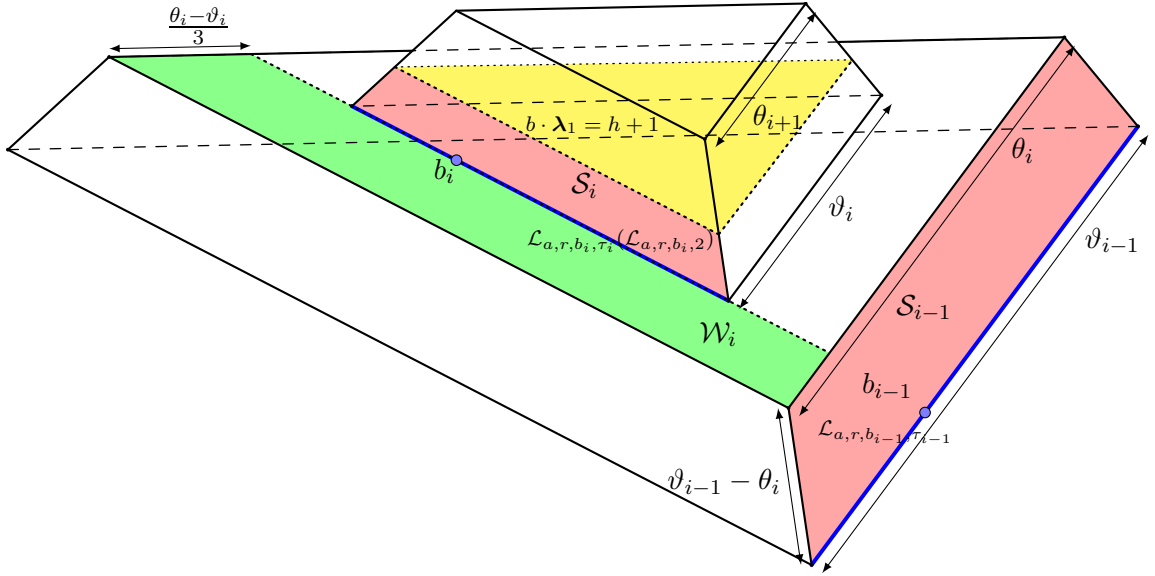
\begin{figure}[!ht]
\centering
\begin{tikzpicture}[line cap=round,line join=round,>=triangle 45,x=1.4cm,y=1.1cm]
\clip(-2.344640960073626,-5.222725711060293) rectangle (10.231429621044287,2.0001618009178756);
\fill[line width=0.pt,dotted,color=yellow,fill=yellow,fill opacity=0.5999999964237213] (2.269264772726905,1.219250248647084)-- (5.345853220770641,-0.8055611613999891)-- (6.615853220770641,1.3055611613999891)-- cycle;

\fill[line width=0.pt,dotted,color=red,fill=red,fill opacity=0.34999998807907104] (6.007850044520397,-2.9059175123729366) -- (6.2269771474771325,-4.804141828805885) -- (9.326214951798109,0.5003365563248241) -- (8.627694919387823,1.578059319528951) -- cycle;
\fill[line width=0.pt,dotted,color=green,fill=green,fill opacity=0.4599999964237213] (-0.44377693176686517,1.34012583582345) -- (0.9038860666342065,1.3754733689836565) -- (6.397055683620749,-2.2397753734572263) -- (6.007850044520397,-2.9059175123729366) -- cycle;
\fill[line width=0.pt,color=red,fill=red,fill opacity=0.34999998807907104] (1.8623434259349265,0.744678763228153) -- (2.269264772726905,1.219250248647084) -- (5.345853220770641,-0.8055611613999891) -- (5.43859832432003,-1.608980767701723) -- cycle;
\draw [line width=0.8pt] (-1.4052023056282101,0.21886475824701107)-- (6.2269771474771325,-4.804141828805885);
\draw [line width=0.5pt,dash pattern=on 4pt off 4pt] (-1.4052023056282101,0.21886475824701107)-- (9.326214951798109,0.5003365563248241);

\draw [line width=1.5pt,color=blue] (9.326214951798109,0.5003365563248241)-- (6.2269771474771325,-4.804141828805885);

\draw[{latex}-{latex}] (9.426214951798109,0.4003365563248241)-- (6.3269771474771325,-4.904141828805885);
\draw (9.426214951798109-0.7,0.4003365563248241-1) node[anchor=north west] {$\vartheta_{i-1}$};

\draw [line width=0.8pt] (-0.44377693176686517,1.34012583582345)-- (6.007850044520397,-2.9059175123729366);
\draw [line width=0.8pt] (6.007850044520397,-2.9059175123729366)-- (8.627694919387823,1.578059319528951);

\draw[{latex}-{latex}] (6.107850044520397,-3.0059175123729366)-- (8.727694919387823,1.478059319528951);
\draw (8.027694919387823,0.478059319528951) node[anchor=north west] {$\theta_{i}$};

\draw [line width=0.5pt,dash pattern=on 4pt off 4pt] (-0.44377693176686517,1.34012583582345)-- (8.627694919387823,1.578059319528951);

\draw [line width=1.5pt,color=blue] (1.8623434259349265,0.7446787632281529)-- (5.43859832432003,-1.6089807677017232);

\draw [line width=0.8pt] (5.43859832432003,-1.6089807677017232)-- (6.890826281886398,0.876569647453035);

\draw[{latex}-{latex}] (5.53859832432003,-1.7089807677017232)-- (6.990826281886398,0.776569647453035);
\draw (6.290826281886398,0.776569647453035-1) node[anchor=north west] {$\vartheta_{i}$};

\draw [line width=0.5pt,dash pattern=on 4pt off 4pt] (6.890826281886398,0.876569647453035)-- (1.8623434259349265,0.7446787632281529);
\draw (7.278804095295231,-2.9290015030676897) node[anchor=south east] {$b_{i-1}$};
\draw (2.543676510203298,0.2326981070407311) node[anchor=north west] {$b_i$};
\draw [line width=0.8pt] (-1.4052023056282101,0.21886475824701107)-- (-0.44377693176686517,1.34012583582345);
\draw [line width=0.8pt] (6.2269771474771325,-4.804141828805885)-- (6.007850044520397,-2.9059175123729366);

\draw[{latex}-{latex}] (6.1269771474771325,-4.804141828805885)-- (5.907850044520397,-2.9059175123729366);
\draw (6.0269771474771325,-3.804141828805885) node[anchor=east] {$\vartheta_{i-1}-\theta_i$};

\draw [line width=0.8pt] (9.326214951798109,0.5003365563248241)-- (8.627694919387823,1.578059319528951);
\draw [line width=0.8pt] (1.8623434259349265,0.7446787632281529)-- (2.8525324521436684,1.8994854125181257);
\draw [line width=0.8pt] (2.8525324521436684,1.8994854125181257)-- (6.171408124758177,1.9865354153983188);
\draw [line width=0.8pt] (6.171408124758177,1.9865354153983188)-- (6.890826281886398,0.876569647453035);
\draw [line width=0.8pt] (6.171408124758177,1.9865354153983188)-- (5.212915438688216,0.3460340868215254);

\draw[{latex}-{latex}] 
(6.271408124758177,1.8865354153983188)-- (5.312915438688216,0.2460340868215254);
\draw (5.571408124758177,0.8865354153983188) node[anchor=north west] {$\theta_{i+1}$};

\draw [line width=0.8pt] (5.212915438688216,0.3460340868215254)-- (2.8525324521436684,1.8994854125181257);
\draw [line width=0.8pt] (5.212915438688216,0.3460340868215254)-- (5.43859832432003,-1.6089807677017232);

\draw [line width=0.8pt] (-0.44377693176686517,1.34012583582345)-- (2.437713859387965,1.4157037743822871);
\draw [line width=0.8pt] (6.4727930317243345,1.521538908994902)-- (8.627694919387823,1.578059319528951);
\draw (3.8499186026424517,0.05457418534448205) node[anchor=north west] {$\mathcal{S}_i$};
\draw (5.067098734233482,-1.7415086917593625) node[anchor=north west] {$\mathcal{W}_i$};
\draw (6.937399912044088,-1.4149481686495726) node[anchor=north west] {$\mathcal{S}_{i-1}$};
\draw [line width=0.8pt,dotted] (0.9038860666342065,1.3754733689836565)-- (6.397055683620749,-2.2397753734572263);

\draw[{latex}-{latex}] (0.9038860666342065,1.4754733689836565)-- (-0.44377693176686517,1.44012583582345);
\draw (0.4038860666342065,1.3754733689836565) node[anchor=south] {$\frac{\theta_i-\vartheta_i}{3}$};

\draw [line width=0.8pt,dotted] (2.269264772726905,1.219250248647084)-- (5.345853220770641,-0.8055611613999891);
\draw [line width=0.5pt,dotted] (2.269264772726905,1.219250248647084)-- (6.615853220770641,1.3055611613999891);
\draw [line width=0.8pt,dotted] (5.345853220770641,-0.8055611613999891)-- (6.615853220770641,1.3055611613999891);

\begin{scriptsize}
\draw (3.436319910416923,0.7239488259098984) node[anchor=north west] {$b\cdot\blambda_1 = h+1$};
\draw (3.4499186026424517,-0.6457418534448205) node[anchor=north west] {$\cL_{a,r,b_i,\tau_i}(\cL_{a,r,b_i,2})$};
\draw (7.6499186026424517,-3.4457418534448205) node[anchor=south east] {$\cL_{a,r,b_{i-1},\tau_{i-1}}$};

\draw [fill=black] (-1.4052023056282101,0.21886475824701107) circle (0.5pt);
\draw [fill=black] (6.2269771474771325,-4.804141828805885) circle (0.5pt);
\draw [fill=black] (9.326214951798109,0.5003365563248241) circle (0.5pt);
\draw [fill=black] (4.8103375608045145,7.467729951356388) circle (0.5pt);
\draw [fill=black] (-0.44377693176686517,1.34012583582345) circle (0.5pt);
\draw [fill=black] (6.007850044520397,-2.9059175123729366) circle (0.5pt);
\draw [fill=black] (8.627694919387823,1.578059319528951) circle (0.5pt);
\draw [fill=black] (1.8623434259349265,0.7446787632281529) circle (0.5pt);
\draw [fill=black] (6.890826281886398,0.876569647453035) circle (0.5pt);
\draw [fill=black] (5.43859832432003,-1.6089807677017232) circle (0.5pt);
\draw [fill=black] (4.7747952399962115,4.141321860245865) circle (0.5pt);
\draw [fill=black] (6.171408124758177,1.9865354153983188) circle (0.5pt);
\draw [fill=black] (5.212915438688216,0.3460340868215254) circle (0.5pt);
\draw [fill=xdxdff] (2.854966286799833,0.09139863541943694) circle (2.0pt);
\draw [fill=xdxdff] (7.304222595652947,-2.9603900562370793) circle (2.0pt);
\end{scriptsize}
\end{tikzpicture}
\caption{Faces $\cS_i$, $\cW_i$, and $\cS_{i-1}$, in the pyramid boundary $\partial \fP_{a,r,\Gamma}$.
The yellow triangle is the intersection of $\fP_{a,r,\Gamma}$ with the plane $b\cdot\blambda_1=h+1$, and the blue lines are $\cL_{a,r,b_{i},\tau_{i}}$ and $\cL_{a,r,b_{i-1},\tau_{i-1}}$.} \label{fig:loc}
\end{figure}
\begin{cla} \label{lem:gam2}
For $b \in \mathring\cS_i\cap \Z^3$, if $b + \be_1 + \be_2, b + \be_1 + \be_3 \in \mathring\cS_i$, then
\begin{equation} \label{eq:gam2}
|u(c)| < \exp(C_{10}n)g,\;\forall c \in 
\left\{b + \be_1, b + \be_1 - \be_2, b + \be_1 - \be_3, b + 2\be_1\right\} .
\end{equation}
\end{cla}
We leave the proof of this claim for later as well.
By Claim \ref{lem:gam2}, and arguing as for \eqref{eq:pmtp2} above, we conclude that $\forall b \in \mathring\cS_i\cap \Z^3$ with $b + \be_1 +\be_2, b + \be_1+\be_3 \in \mathring\cS_i$,
\begin{equation}
|u(b) + u(b+\be_1+\be_2)+u(b+\be_1+\be_3)| < C_4^{-2n}\min_{c\in \left(\mathring\cL_{a,r,b_i,2}\cap \Z^3\right)\cup \{b_i\}}|u(c)|.
\end{equation}
Let's first assume that $\mathring\cS_i \cap \Z^3 \neq \emptyset$.
Then we have
$\mathring \cL_{a,r,b_i,2} \cap \Z^3 \neq \emptyset$, and $\vartheta_{i} \geq 2\sqrt{2}$.
If $i < m_{h+1}$, then $b_{i+1} \cdot \blambda_1 \leq h+1$, 
so we treat $\mathring\cS_i \cap \Z^3$ as $P^r_{\mathbf{0}; \frac{\vartheta_i}{\sqrt{2}}-2, \left\lceil\frac{\vartheta_i-\theta_{i+1}}{\sqrt{2}}\right\rceil-1}$ (from Definition \ref{defn:3rt}), and $\cL_{a,r,b_i,2} \cap \Z^3$ is its upper edge.
If $i=m_h=m_{h+1} \geq 2$, then $b_{i+1}\cdot \blambda_1 > h+1$, and we treat $\mathring\cS_{i} \cap \Z^3$ as $P^r_{\mathbf{0}; \frac{\vartheta_{i}}{\sqrt{2}}-2, \left\lceil\frac{\vartheta_{i}-\theta_{i+1}}{\sqrt{2}} - \frac{b_{i+1}\cdot \blambda_1 - (h+1)}{2} \right\rceil-1}$.
We apply Corollary \ref{cor:poly2} to the trapezoid, with
$L = \mathring \cL_{a,r,b_i,2} \cap \Z^3$ if it is not empty; similar to the study of $\cW_i$, we conclude that
\begin{equation}  \label{eq:pmtp3}
\left| \left\{b \in \cS_i \cap \Z^3: |u(b)|>\exp(C_{10}n)g \right\} \right| > \frac{\epsilon_3\vartheta_i (\vartheta_i - \theta_{i+1})}{\sqrt{2} \cdot \sqrt{2}},
\end{equation}
if $2 \leq i < m_{h+1}$, and
\begin{equation}  \label{eq:pmtp31}
\left| \left\{b \in \cS_{i} \cap \Z^3: |u(b)|>\exp(C_{10}n)g \right\} \right| 
> \frac{\epsilon_3 \vartheta_{i}}{\sqrt{2} } \left(\frac{\vartheta_{i}-\theta_{i+1}}{\sqrt{2}} -  \frac{b_{i+1}\cdot \blambda_1 - (h+1)}{2}  \right),
\end{equation}
if $i = m_h = m_{h+1} \geq 2$.
In the case where $\mathring\cS_i \cap \Z^3 = \emptyset$, we have $\vartheta_{i} \leq\sqrt{2}$, and these inequalities still hold, since $b_i\in \cS_i \cap \Z^3$ and $|u(b_i)|>\exp(C_{10}n)g$.

When $\tau_i = 3,4$, we can define $\cS_i$ analogously, and obtain \eqref{eq:pmtp3} and \eqref{eq:pmtp31} as well.

In addition, we consider
\begin{equation}
\haS_1 := \left\{b \in \cP_{4, a\cdot \blambda_4 }:
b \cdot  \olambda_\tau  < \bt_r(a)\cdot \olambda_2,\;
\forall \tau \in \{2,3\},\;
a \cdot  \blambda_1  \leq b \cdot  \blambda_1  < b_{2} \cdot  \blambda_1  \right\},
\end{equation}
and let $\mathring\cS_1:=\left\{b \in \haS_1 : b\cdot \blambda_1 < h+1\right\}$, and $\cS_1$ be the closure of $\mathring\cS_1$.
We treat $\cS_1$ differently (from $\cS_i$ for $2\leq i \leq m_h$) because Claim \ref{lem:edg} cannot be extended to $i=1$.
Also note that by taking $\cS_1\subset \cP_{4,a\cdot\blambda_4}$, $\cS_1$ is defined as (possibly part of) the face in $\partial \fP_{a,r,\Gamma}$ that contains $a=b_1$.

Using arguments similar to those for $\cS_i$, $2 \leq i \leq m_h$,
we treat $\mathring\cS_1 \cap \Z^3$ as $P^r_{\mathbf{0};\frac{\vartheta_1}{\sqrt{2}}-2, \left\lceil\frac{\vartheta_1 - \theta_2}{\sqrt{2}}\right\rceil-1}$ if $m_{h+1} > 1$,
and as $P^r_{\mathbf{0};\frac{\vartheta_1}{\sqrt{2}}-2, \left\lceil\frac{\vartheta_1 - \theta_2}{\sqrt{2}} -  \frac{b_{2}\cdot \blambda_1 - (h+1)}{2} \right\rceil-1}$ if $m_{h+1}=1$. 
Then we apply Corollary \ref{cor:poly2} to it with $L = \left\{ a \right\}$.
We conclude that
\begin{equation}  \label{eq:pmtp4}
\left| \left\{b \in \cS_{1} \cap \Z^3: |u(b)|>\exp(C_{10}n)g \right\} \right|
> 
\begin{cases}
\frac{\epsilon_3(\vartheta_1 - \theta_2)^2}{(\sqrt{2})^2}, & \; m_{h+1} > 1,\\
\epsilon_3 \left(\frac{\vartheta_1-\theta_{2}}{\sqrt{2}} - \frac{b_{2}\cdot \blambda_1 - (h+1)}{2}  \right)^2  & \; m_{h+1} = 1.
\end{cases}
\end{equation}
We now put together the bounds we've obtained so far, for all $\cS_i$ and $\cW_i$ that are contained in $\{b\in\R^3: b\cdot \blambda_1 \leq h+1\}$.

\noindent{\textbf{Case 1:}} $m_h = m_{h+1}$.
In this case we consider $\cS_i$ for $1 \leq i \leq m_h$ and $\cW_i$ for $2 \leq i \leq m_h$.

We first show that $h\neq a \cdot \blambda_1$.
For this we argue by contradiction.
Assume the contrary, i.e. $h= a \cdot \blambda_1$.
From the definition of $h$ we have that $|\cL_{a,r,c,2}\cap \Z^3|< \frac{r}{4}$, 
for any $c \in \mathring\fP_{a,r,\Gamma} \cap \Z^3$ with $c\cdot \blambda_1 = h+1= a \cdot \blambda_1+1$.
As we assumed that $r>100$, we must have $b_2 \cdot \blambda_1 = a \cdot \blambda_1+1 = h+1$, and this implies $m_{h+1}=2$.
However, by $h= a \cdot \blambda_1$ we have $m_h=1$.
This contradicts with the assumption that $m_h = m_{h+1}$.

We next show that
\begin{equation}  \label{eq:pmtpc1}
\frac{\theta_{m_h+1}}{\sqrt{2}}+ \frac{b_{m_h+1}\cdot \blambda_1 - (h+1)}{2} < \frac{r}{2}.
\end{equation}
By the definition of $h$ and $h\neq a\cdot \blambda_1$, we can find $c \in \mathring\fP_{a,r,\Gamma} \cap \Z^3$ with $c\cdot \blambda_1 = h$ and $|\cL_{a,r,c,2}\cap \Z^3|\geq \frac{r}{4}$.
Since we assumed that $r > 100$, using $m_{h}=m_{h+1}$ we have $\mathring\fP_{a,r,\Gamma} \cap \cP_{1, h+1} \cap \Z^3 \neq \emptyset$.
This implies that $b_{m_h+1} \cdot \blambda_1 = b_{m_{h+1}+1} \cdot \blambda_1 > h+1$ (since otherwise, by the definiton of $m_{h+1}$, we must have $m_{h+1} = m$ and $b_{m+1}\cdot \blambda_1 \leq h+1$, implying $\mathring\fP_{a,r,\Gamma} \cap \cP_{1, h+1} = \emptyset$).
Also note that $b_{m_h} \cdot \blambda_1 \leq h$, 
so we can find $b \in \Z^3$, and $b$ in the closure of $\haS_{m_h}$, such that $b \cdot \blambda_1 = h+1$ or $h+2$.
As $|\cL_{a,r,b,2}\cap \Z^3| = \frac{\theta_{m_h+1}}{\sqrt{2}} + 1 + \frac{(b_{m_h+1}-b)\cdot \blambda_1 }{2}$, 
we have $|\cL_{a,r,b,2}\cap \Z^3| \geq \frac{\theta_{m_h+1}}{\sqrt{2}}+ \frac{b_{m_h+1}\cdot \blambda_1 - (h+1)}{2}$.

On the other hand, using $|\cL_{a,r,c,2}\cap \Z^3|\geq \frac{r}{4}$ and $r > 100$ again, we have $\mathring\fP_{a,r,\Gamma} \cap \cP_{1, b\cdot \blambda_1} \cap \Z^3 \neq \emptyset$.
By the maximum property of $h$, for any $b' \in \mathring\fP_{a,r,\Gamma} \cap \cP_{1, b\cdot \blambda_1} \cap \Z^3$, we have $|\cL_{a,r,b',2}\cap \Z^3|< \frac{r}{4}$.
Then $|\cL_{a,r,b,2}\cap \Z^3| < \frac{r}{4}+3 < \frac{r}{2}$, and \eqref{eq:pmtpc1} follows.

If $m_h = m_{h+1} = 1$, by \eqref{eq:pmtp4} we have that
\begin{multline} \label{eq:pmtp901}
\left| \left\{b \in \fP_{a,r,\Gamma} \cap \Z^3: |u(b)|>\exp(C_{10}n)g \right\}\right| 
>
\epsilon_3 \left(\frac{\vartheta_1-\theta_{2}}{\sqrt{2}} - \frac{b_{2}\cdot \blambda_1 - (h+1)}{2}  \right)^2
\\
>
\epsilon_3\left(2r - \frac{r}{2}\right)^2 > \epsilon_3 r^2,
\end{multline}
where we use \eqref{eq:pmtpc1} and the fact that $\vartheta_1 = 2\sqrt{2}r$.

If $m_h = m_{h+1}> 1$,
we note that for all $2 \leq i \leq m_{h}$, these $\cW_i$ are mutually disjoint; and for all $1 \leq i \leq m_h$, these $\cS_i$ are mutually disjoint.
By
\eqref{eq:pmtp1},\eqref{eq:pmtp3},\eqref{eq:pmtp31},\eqref{eq:pmtp4} and taking a small enough $\epsilon_4>0$, we have that
\begin{equation}  \label{eq:pmtp9}
\begin{split}
&\left| \left\{b \in \fP_{a,r,\Gamma} \cap \Z^3: |u(b)|>\exp(C_{10}n)g \right\}\right| 
\\
>&
\epsilon_4\left(\frac{(\vartheta_1 - \theta_2)^2}{2}
+ \sum_{i=2}^{m_h} \theta_i(\theta_i - \vartheta_{i}) + \vartheta_i(\vartheta_i - \theta_{i+1}) 
- \vartheta_{m_h} \frac{b_{m_h+1}\cdot \blambda_1 - (h+1)}{\sqrt{2}} 
\right)  \\
=&
\epsilon_4\Bigg(
\frac{\vartheta_1^2}{4}+
\frac{(\vartheta_1 - 2\theta_2)^2}{4}
+ \frac{\sum_{i=2}^{m_h} (\theta_i - \vartheta_i)^2 + \sum_{i=2}^{m_h-1} (\vartheta_i - \theta_{i+1})^2}{2}
\\
&
+ \frac{\left(\vartheta_{m_h} - \theta_{m_h+1} -  \frac{b_{m_h+1}\cdot \blambda_1 - (h+1)}{\sqrt{2}} \right)^2}{2}
-
\frac{\left(\theta_{m_h+1} + \frac{b_{m_h+1}\cdot \blambda_1 - (h+1)}{\sqrt{2}} \right)^2}{2}
\Bigg) 
\\
\geq &
\epsilon_4\left(
\frac{\vartheta_1^2}{4}
-
\frac{\left(\theta_{m_h+1} + \frac{b_{m_h+1}\cdot \blambda_1 - (h+1)}{\sqrt{2}} \right)^2}{2}
\right) .
\end{split}
\end{equation}
Using \eqref{eq:pmtpc1}, we get
\begin{equation} \label{eq:pmtp902}
\left| \left\{b \in \fP_{a,r,\Gamma} \cap \Z^3: |u(b)|>\exp(C_{10}n)g \right\}\right|
>\epsilon_4\left(
2r^2
-
\frac{r^2}{4}
\right)
> \epsilon_4 r^2.
\end{equation}

\noindent{\textbf{Case 2:}} $m_h < m_{h+1}$.
In this case we consider $\cS_i$ for $1 \leq i \leq m_h$ and $\cW_i$ for $2 \leq i \leq m_h+1=m_{h+1}$.
Similar to the first case, by
\eqref{eq:pmtp1},\eqref{eq:pmtp11},\eqref{eq:pmtp3},\eqref{eq:pmtp4} and taking a small enough $\epsilon_5>0$, we have
\begin{equation}  \label{eq:pmtp91}
\begin{split}
&\left| \left\{b \in \fP_{a,r,\Gamma} \cap \Z^3: |u(b)|>\exp(C_{10}n)g \right\}\right| 
\\
\geq &
\epsilon_5\left(\frac{(\vartheta_1 - \theta_2)^2}{2}
+ \sum_{i=2}^{m_h} \theta_i(\theta_i - \vartheta_{i}) + \vartheta_i(\vartheta_i - \theta_{i+1}) 
+(\theta_{m_{h+1}} - \vartheta_{m_{h+1}})^2
\right) \\
= &
\epsilon_5\left(
\frac{\vartheta_1^2}{4}
+
\frac{(\vartheta_1 - 2\theta_2)^2}{4}
+ \sum_{i=2}^{m_h} 
\frac{(\theta_i - \vartheta_i)^2 + (\vartheta_i - \theta_{i+1})^2}{2}
+ \frac{(\theta_{m_{h+1}} - 2\vartheta_{m_{h+1}})^2}{2}-\vartheta_{m_{h+1}}^2
\right) \\
\geq &
\epsilon_5\left(\frac{\vartheta_1^2}{4}
-\vartheta_{m_{h+1}}^2
\right) .
\end{split}
\end{equation}
We now show that $\vartheta_{m_{h+1}} < r$.
Since $m_{h+1} > m_h$, we have $b_{m_{h+1}} \cdot \blambda_1 = h+1$.
If $\vartheta_{m_{h+1}} \geq r$, then $|\cL_{a,r,b_{m_{h+1}},2}\cap \Z^3| \geq \frac{r}{\sqrt{2}} + 1$, and we can find $b \in \mathring\fP_{a,r,\Gamma} \cap \cP_{1, h+1}$, such that $|\cL_{a,r,b,2}\cap \Z^3| \geq \frac{r}{\sqrt{2}} - 2 > \frac{r}{4}$.
This contradicts with the definition of $h$.

With $\vartheta_{m_{h+1}} < r$, and note that $2\sqrt{2}r = \vartheta_1 \geq \theta_2$, we have
\begin{equation}  \label{eq:pmtp92}
\left| \left\{b \in \fP_{a,r,\Gamma} \cap \Z^3: |u(b)|>\exp(C_{10}n)g \right\}\right|
>
\epsilon_5\left(2r^2
-r^2
\right)
= \epsilon_5 r^2.
\end{equation}
By taking $C_{9}'$ small enough, we get \eqref{eq:pmt} by each of  \eqref{eq:pmtp901}, \eqref{eq:pmtp902}, and \eqref{eq:pmtp92}.
\end{proof}
\begin{figure}[!ht]
\centering
\begin{tikzpicture}[line cap=round,line join=round,>=triangle 45,x=1.25cm,y=0.85cm]
\clip(0,2.9) rectangle (9.1,9.1);

\draw [line width=0.5pt] (0,6)-- (6,3);
\draw [line width=0.4pt,dash pattern=on 4pt off 4pt] (0,6)-- (9,7.5);
\draw [line width=0.5pt] (9,7.5)-- (6,3);
\draw [line width=0.5pt] (1.5,8)-- (5.5,6);
\draw [line width=0.5pt] (1.5,8)-- (7.5,9);
\draw [line width=0.5pt] (7.5,9)-- (5.5,6);

\draw [line width=0.5pt] (0,6)-- (1.5,8);
\draw [line width=0.5pt] (6,3)-- (5.5,6);
\draw [line width=0.5pt] (9,7.5)-- (7.5,9);

\fill[line width=0.pt,dotted,color=green,fill=green,fill opacity=0.25] (1.5,8) -- (5.5,6) -- (7.5,9) -- cycle;

\draw (4.7,9) node[anchor=north west] {$\cR_i\subset \cP_{1,b_i\cdot \blambda_1}$};
\draw [fill=black] (2,5) circle (1.5pt);
\draw (2.2,5.2) node[anchor=north east] {$b_{i-1}$};

\draw [fill=black] (5,7) circle (.7pt);
\draw [fill=black] (4.3,7.35) circle (.7pt);
\draw [fill=black] (5.35,7.525) circle (.7pt);

\draw [color=red,fill=red] (4.9125,6.9125) circle (.7pt);
\draw [color=blue,fill=blue] (4.825,6.825) circle (.7pt);
\draw [color=blue,fill=blue] (5.525,6.475) circle (.7pt);
\draw [color=blue,fill=blue] (4.475,6.3) circle (.7pt);

\draw [line width=0.5pt,dotted,color=red] (5,7)-- (4.825,6.825);
\draw [line width=0.5pt,dotted,color=red] (4.3,7.35)-- (5.525,6.475);
\draw [line width=0.5pt,dotted,color=red] (5.35,7.525)-- (4.475,6.3);

\draw [line width=0.001pt,dotted,fill=black,fill opacity=0.2] (5,7)-- (4.3,7.35) -- (5.35,7.525) -- cycle;

\draw [line width=0.001pt,dotted,fill=blue,fill opacity=0.2] (4.825,6.825) -- (5.525,6.475) -- (4.475,6.3) -- cycle;

\begin{scriptsize}
\draw (5,7) node[anchor=west] {$b$};
\draw (4.3,7.35) node[anchor=south east] {$b+\be_2-\be_3$};
\draw (5.35,7.525) node[anchor=south west] {$b+\be_1-\be_3$};

\draw (4.9125,6.9125) node[anchor=south east,color=red] {$b-\be_3$};
\draw (4.825,6.825) node[anchor=east,color=blue] {$b-2\be_3$};
\draw (5.525,6.475) node[anchor=west,color=blue] {$b-\be_2-\be_3$};
\draw (4.475,6.3) node[anchor=north,color=blue] {$b-\be_1-\be_3$};
\end{scriptsize}

\end{tikzpicture}
\caption{An illustration of points in Claim \ref{lem:gam1}. 
The red point ($b-\be_3$) is in $\cP_{1,b_i\cdot \blambda_1-1}$ and the blue points are in $\cP_{1,b_i\cdot \blambda_1-2}$.
The point $c$ is among the red and blue points and is in $\mathring\fP_{a,r,\Gamma}\cup\oT_{a,r}\cup \oT_{a-\frac{\blambda_1}{3},r}$.} 
\end{figure}

It remains to prove Claim \ref{lem:gam1}, \ref{lem:edg}, and \ref{lem:gam2}.
\begin{proof}[Proof of Claim \ref{lem:gam1}]
Take any $c \in \left\{b-\be_3, b-\be_1-\be_3, b-\be_1-\be_3, b-2\be_3\right\}$.
Since $c \cdot \olambda_{2} \leq b \cdot \olambda_2,  c \cdot \olambda_{3} \leq b \cdot \olambda_3$, and $c \cdot \olambda_{4} \leq (b+\be_1-\be_3) \cdot \olambda_4$, and $b, b+\be_1-\be_3 \in \mathring\cR_i \subset \mathring\fT_{a,r}$, we have that 
\begin{equation}  \label{eq:gam1p1}
c\cdot \olambda_\tau < \bt_r(a)\cdot \olambda_\tau - F_{a,r,b_{i-1}} \leq \bt_r(a)\cdot \olambda_\tau,\;
\forall \tau \in \{2,3,4\}.
\end{equation}

We first consider the case where $c \not\in \mathring\fT_{a,r}$.
Then we have that $a \cdot \blambda_1 \geq c\cdot \blambda_1 \geq b \cdot \blambda_1 - 2 = b_i \cdot \blambda_1 - 2 \geq a \cdot \blambda_1 - 1$,
where the last inequality is due to $b_i \in \mathring \fT_{a,r}$.
If $c \cdot \blambda_1 = a \cdot \blambda_1$, we have $c \in \mathring \caT_{a,r}$ by \eqref{eq:gam1p1}; and by the second condition of Proposition \ref{prop:pmt} we have that $|u(c)| < g$.
If $c \cdot \blambda_1 = a \cdot \blambda_1-1$, we have 
$c \in \mathring \caT_{a-\frac{\blambda_1}{3}, r}$  by \eqref{eq:gam1p1}.
As $b_i \cdot \blambda_1 > a\cdot \blambda_1$, and $b_i \cdot \blambda_1 = b \cdot \blambda_1 \leq c \cdot \blambda_1 +2$, 
we have that $b_i \cdot \blambda_1 = a\cdot \blambda_1 + 1$, thus $b_i \in \mathring \caT_{a+\frac{\blambda_1}{3}, r}$.
Since $|u(b_i)|\geq \exp(3C_{10} n)g$, by the second condition of Proposition \ref{prop:pmt} we have $|u(c)| < g$.

Now we assume that $c \in \mathring\fT_{a,r}$.
For any $j$, if $i\leq j \leq m$,
as $c\cdot \blambda_1 < b_i \cdot \blambda_1$, we have that  $c\cdot \blambda_1 < b_j \cdot \blambda_1$, and thus $c\not\in  \fH_{a,r,b_j}$.
If $1 \leq j \leq i-1$, we have $b_j\cdot \blambda_1 \leq b_{i-1}\cdot \blambda_1$, so $F_{a,r,b_j} \leq F_{a,r,b_{i-1}}$ (since otherwise $\fH_{a,r,b_{i-1}}$ is not maximal).
By \eqref{eq:gam1p1} we have that
\begin{equation}
c\cdot \olambda_\tau < \bt_r(a)\cdot \olambda_\tau - F_{a,r,b_{j}},\;
\forall \tau \in \{2,3,4\},
\end{equation}
thus $c\not\in  \fH_{a,r,b_j}$.
Then by the definition of $\mathring\fP_{a,r,\Gamma}$, we have that $c \in \mathring\fP_{a,r,\Gamma}$.
As $c \cdot \blambda_1 \leq b_{i} \cdot \blambda_1 - 1\leq b_{m_{h+1}} \cdot \blambda_1 - 1 \leq h$, we have $|u(c)| \leq \exp(C_{10}n)g$ by the third condition of Proposition \ref{prop:pmt}.
\end{proof}
Claim \ref{lem:gam2} can be proved in a similar way.

\begin{figure}[!ht]
\centering
\begin{tikzpicture}[line cap=round,line join=round,>=triangle 45,x=1.25cm,y=0.85cm]
\clip(0,2.9) rectangle (9.1,9.1);

\draw [line width=0.5pt] (0,6)-- (6,3);
\draw [line width=0.4pt,dash pattern=on 4pt off 4pt] (0,6)-- (9,7.5);
\draw [line width=0.5pt] (9,7.5)-- (6,3);
\draw [line width=0.5pt] (1.5,8)-- (5.5,6);
\draw [line width=0.5pt] (1.5,8)-- (7.5,9);
\draw [line width=0.5pt] (7.5,9)-- (5.5,6);

\draw [line width=0.5pt] (0,6)-- (1.5,8);
\draw [line width=0.5pt] (6,3)-- (5.5,6);
\draw [line width=0.5pt] (9,7.5)-- (7.5,9);

\fill[line width=0.pt,dotted,color=yellow,fill=yellow,fill opacity=0.25] (1.5,8) -- (5.5,6) -- (6,3) -- (0,6) -- cycle;

\draw (1,7.5) node[anchor=north west] {$\cS_i\subset \cP_{2,b_i\cdot \blambda_2}$};
\draw [fill=black] (2,5) circle (1.5pt);
\draw (2.1,5.1) node[anchor=north east] {$b_{i}$};

\draw [fill=black] (5.3,5.4) circle (.7pt);
\draw [fill=black] (4.6,5.75) circle (.7pt);
\draw [color=blue,fill=blue] (5.65,5.925) circle (.7pt);

\draw [color=red,fill=red] (5.2125,5.3125) circle (.7pt);
\draw [color=blue,fill=blue] (5.125,5.225) circle (.7pt);
\draw [color=blue,fill=blue] (5.825,4.875) circle (.7pt);
\draw [color=black,fill=black] (4.775,4.7) circle (.7pt);

\draw [line width=0.5pt,dotted,color=red] (5.3,5.4)-- (5.125,5.225);
\draw [line width=0.5pt,dotted,color=red] (4.6,5.75)-- (5.825,4.875);
\draw [line width=0.5pt,dotted,color=red] (5.65,5.925)-- (4.775,4.7);

\draw [line width=0.001pt,dotted,fill=black,fill opacity=0.2] (5.3,5.4)-- (4.6,5.75) -- (4.775,4.7) -- cycle;

\draw [line width=0.001pt,dotted,fill=blue,fill opacity=0.2] (5.125,5.225) -- (5.825,4.875) -- (5.65,5.925) -- cycle;

\begin{scriptsize}
\draw (5.3,5.4) node[anchor=west] {$b+\be_1+\be_3$};
\draw (4.6,5.75) node[anchor=south east] {$b+\be_1+\be_2$};
\draw (5.65,5.925) node[anchor=south,color=blue] {$b+2\be_1$};

\draw (5.2125,5.3125) node[anchor=south east,color=red] {$b+\be_1$};
\draw (5.125,5.225) node[anchor=east,color=blue] {$b+\be_1-\be_3$};
\draw (5.825,4.875) node[anchor=west,color=blue] {$b+\be_1-\be_2$};
\draw (4.775,4.7) node[anchor=north,color=black] {$b$};
\end{scriptsize}

\end{tikzpicture}
\caption{An illustration of points in Claim \ref{lem:gam2}. 
The red point ($b+\be_1$) is in $\cP_{2,b_i\cdot \blambda_2-1}$ and the blue points are in $\cP_{2,b_i\cdot \blambda_2-2}$.
The point $c$ is among the red and blue points and is in $\mathring\fP_{a,r,\Gamma}\cup \oT_{a,r}$.} 
\end{figure}
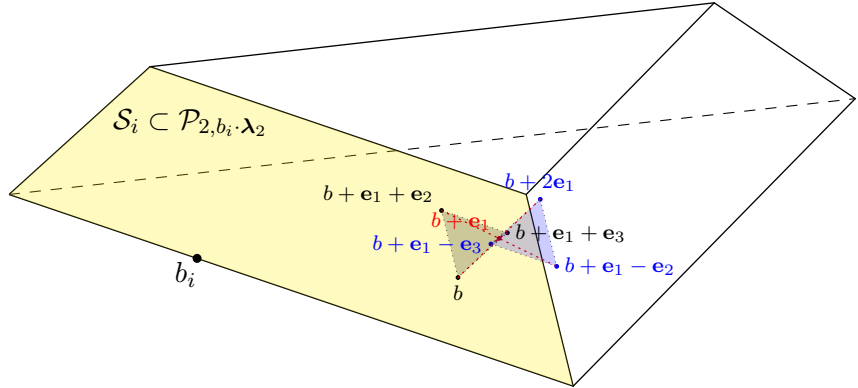

\begin{proof}[Proof of Claim \ref{lem:gam2}]
We take $c \in \left\{b + \be_1, b + \be_1 - \be_2, b + \be_1 - \be_3, b + 2\be_1 \right\}$,
then $c \cdot \olambda_2 < b \cdot \olambda_2 = b_i \cdot \olambda_2$, and $c \cdot \olambda_\tau \leq b \cdot \olambda_\tau + 2$ for $\tau \in \{3,4\}$.
Since $b, b+\be_1+\be_2, b+\be_1+\be_3 \in \mathring\cS_i$, we have that $b \cdot \olambda_3 + 2 = (b+\be_1+\be_3) \cdot \olambda_3 < \bt_r(a)\cdot \olambda_3 - F_{a,r,b_{i}}$,
and $b \cdot \olambda_4 + 2 = (b+\be_1+\be_2) \cdot \olambda_4 < \bt_r(a)\cdot \olambda_4 + F_{a,r,b_{i}}$; then
\begin{equation}  \label{eq:gam2p1}
c\cdot \olambda_\tau < \bt_r(a)\cdot \olambda_\tau - F_{a,r,b_{i}} \leq \bt_r(a)\cdot \olambda_\tau,\;
\forall \tau \in \{2,3,4\},
\end{equation}
We claim that $c\not\in  \fH_{a,r,b_j}$ for any $1 \leq j \leq m$:
for $j > i$, note that $b + \be_1 + \be_2 \in \mathring\cS_i$, so $c\cdot \blambda_1 \leq b \cdot \blambda_1 + 2 = (b+\be_1+\be_2) \cdot \blambda_1 < b_{i+1} \cdot \blambda_1 $; for $j \leq i$, this is implied by \eqref{eq:gam2p1}.
Thus $c \in \mathring\fP_{a,r,\Gamma}\cup\oT_{a,r}$, since we also have $c \cdot \blambda_1 \geq b \cdot \blambda_1 \geq b_i \cdot \blambda_1 \geq a \cdot \blambda_1$.
If $c\in\oT_{a,r}$, by the second condition of Proposition \ref{prop:pmt}, we have $|u(c)|\leq g<\exp(C_{10} n)g$.
If $c \in \mathring\fP_{a,r,\Gamma}$, using the fact that $b + \be_1 + \be_2 \in \mathring\cS_i$ again, we have $c\cdot \blambda_1 \leq b\cdot \blambda_1+2 = (b + \be_1 + \be_2)\cdot \blambda_1 \leq h$, and this implies that $|u(c)|\leq \exp(C_{10} n)g$ by the third condition of Proposition \ref{prop:pmt}.
\end{proof}
Lastly, we prove Claim \ref{lem:edg}, using Claim \ref{lem:gam1} above and the local cone property (from Section \ref{sec:cone}).
\begin{proof}[Proof of Claim \ref{lem:edg}]
Throughout this proof, we assume that $\left(\bigcup_{\tau \in \left\{2,3,4\right\}}\mathring\cL_{a,r,b_i,\tau}\right) \cap \Z^3 \neq \emptyset$.
We first show that we can find point $b \in \left(\bigcup_{\tau \in \left\{2,3,4\right\}}\mathring\cL_{a,r,b_i,\tau}\right) \cap \Z^3$, such that 
\begin{equation}
    |u(b)| \geq (K+11)^{-1}\exp(3C_{10}n)g.
\end{equation}
This is obviously true if $b_i \in \bigcup_{\tau \in \left\{2,3,4\right\}}\mathring\cL_{a,r,b_i,\tau}$; otherwise, by symmetry we assume that $b_i = \cV_{a,r,b_i,4}$.
By Lemma \ref{lem:walk},
\begin{equation}
\max_{c \in \left\{b_i - \be_3, b_i -\be_3+\be_1, b_i -\be_3+\be_2, b_i -\be_3-\be_1, b_i -\be_3-\be_2, b_i -2\be_3\right\}} |u(c)| \geq (K+11)^{-1}\exp(3C_{10}n)g.
\end{equation}
As $b_i, b_i - \be_3 + \be_1, b_i - \be_3 + \be_2 \in \mathring\cR_i$, by Claim \ref{lem:gam1}, we have
\begin{equation}
\max_{c \in \left\{b_i -\be_3+\be_1, b_i -\be_3+\be_2\right\}} |u(c)| \geq (K+11)^{-1}\exp(3C_{10}n)g.
\end{equation}
Note that $b_i -\be_3+\be_1, b_i -\be_3+\be_2 \in \bigcup_{\tau \in \left\{2,3,4\right\}}\mathring\cL_{a,r,b_i,\tau}$,
so we can choose $b \in \left\{b_i -\be_3+\be_1, b_i -\be_3+\be_2\right\}$ and the condition is satisfied. 

Now by symmetry we assume that there is $b \in \mathring \cL_{a,r,b_i,4}\cap \Z^3$ so that 
\begin{equation}
    |u(b)| \geq (K+11)^{-1}\exp(3C_{10}n)g.
\end{equation}
We prove that, for any $b' \in \mathring\cL_{a,r,b_i,4}\cap \Z^3$, we have $|u(b')| \geq \exp(2C_{10}n)g$.
We argue by contradiction, and assume that there is $b' \in \mathring\cL_{a,r,b_i,4}\cap \Z^3$ so that $|u(b')|<\exp(2C_{10}n)g$.
Without loss of generality, we also assume that $b' \cdot \be_1 < b \cdot \be_1$.
Consider the sequence of points in $\mathring\cL_{a,r,b_i,4}\cap \Z^3$ between $b$ and $b'$.
We iterate this sequence from $b$ to $b'$, by adding $-\be_1+\be_2$ at each step. We let $c$ be the first one such that $|u(c-\be_1+\be_2)| < (K+11)^{-1}|u(c)|$.
The existence of such $c$ is ensured by that $|u(b')|<(K+11)\exp(-C_{10}n)|u(b)|$, $|\mathring\cL_{a,r,b_i,4}\cap \Z^3| < 2r < \frac{n}{16}$, and $C_{10}>K+11$.
For such $c$ we also have $c, c-\be_1+\be_2\in \mathring\cL_{a,r,b_i,4}\cap \Z^3$, and
$|u(c)| \geq (K+11)^{-1-2r}\exp(3C_{10}n)g > \exp\left(\frac{5C_{10}n}{2}\right)g$.
Since $c, c-\be_1+\be_2, c-\be_1+\be_3 \in \mathring\cR_i$, by Claim \ref{lem:gam1} we have
\begin{equation}
|u(c')| < \exp(C_{10}n)g < (K+11)^{-1}|u(c)|,\;\forall c' \in 
\left\{c-\be_1, c-\be_1-\be_3, c-\be_1-\be_2, c-2\be_1\right\} .
\end{equation}
For $c-\be_1+\be_3$, as $c-\be_1+\be_3 \in \mathring\fP_{a,r,\Gamma}$ and $(c-\be_1+\be_3) \cdot \blambda_1 = c\cdot \blambda_1 \leq h$, we have $|u(c-\be_1+\be_3)| \leq \exp(C_{10}n)g< (K+11)^{-1}|u(c)|$ by the third condition of Proposition \ref{prop:pmt}.
Then we get a contradiction with Lemma \ref{lem:walk}.
\end{proof}

The next step is to control the points in a graded set $E$.
\begin{prop}\label{prop:pmtF}
For $C_{9}'$ from Proposition \ref{prop:pmt},
any small enough $\varepsilon>0$, 
and any $N \in \Z_+$, there exists $C_{\varepsilon,N}>0$ such that the following is true.

Let $n \in \Z_+$, $r \in \Z$, $0 \leq r < \frac{n}{32}$. 
 Let $\Gamma \subset Q_{n}$, $a \in \Gamma \cap Q_{\frac{n}{2}}$ such that $\oT_{a, r} \cap \Gamma = \emptyset$.
Suppose that $\Vec{l}$ is a vector of positive reals, and $E$ is an $(N,\Vec{l},\varepsilon^{-1},\varepsilon)$-graded set with the first scale length $l_1>C_{\varepsilon,N}$.
If $E$ is $(\varepsilon^{-\frac{1}{2}},\varepsilon)$-normal in $\fT_{a,r}$, then
\begin{equation}\label{eq:pyramidF}
    \left|E\cap \partial\fP_{a, r, \Gamma}\cap \Z^3\right|\leq \frac{C'_{9}}{2}(r^2+1).
\end{equation}
\end{prop}
\begin{proof}
    If $r<\frac{1}{10\sqrt{\varepsilon}}$, since $E$ is $(\varepsilon^{-\frac{1}{2}},\varepsilon)$-normal in $\fT_{a,r}$, we have $E \cap \fT_{a,r}=\emptyset$ when $C_{\varepsilon,N}$ is large, and our conclusion holds.
    
    From now on, we assume that $r\geq \frac{1}{10\sqrt{\varepsilon}}$.
    Denote $\pi:=\pi_{a \cdot \blambda_1}$ for the simplicity of notations.
    Evidently, for any two $b_1,b_2 \in \partial\fP_{a, r, \Gamma}$, 
    \begin{equation}\label{eq:distpym}
        \frac{1}{10} |b_1-b_2| \leq |\pi(b_1)-\pi(b_2)| \leq |b_1-b_2|.
    \end{equation}
    Suppose $\vec{l}=(l_1,\cdots,l_d)$, where $l_{i+1} \geq l_{i}^{1+2\varepsilon}$ for each $1\leq i\leq d-1$.
    Write $E=\bigcup_{i=0}^{d}E_{i}$, where $E_{0}$ is a $\varepsilon^{-1}$-unitscattered set and $E_{i}$ is an $(N,l_{i},\varepsilon)$-scattered set.
    It suffices to prove that there exists a universal constant $C$ such that for each $1 \leq i \leq d$, 
    \begin{equation} \label{eq:pym,i>0}
        \left|E_{i} \cap \partial\fP_{a, r, \Gamma} \cap \Z^3\right| \leq C N l_{i}^{-\varepsilon} r^{2},
    \end{equation}
    and 
    \begin{equation}\label{eq:pym,i=0}
        \left|E_{0} \cap \partial\fP_{a, r, \Gamma} \cap \Z^3\right| \leq C \varepsilon^{2} r^{2}. 
    \end{equation}
    Then with \eqref{eq:pym,i>0} and \eqref{eq:pym,i=0} we can take $\varepsilon$ small enough, such that $C\varepsilon^{2}<\frac{C'_{9}}{4}$;
    and take $C_{\varepsilon,N}$ large enough, such that
    \begin{equation}
        \sum_{i=1}^{\infty} C N l_{i}^{-\varepsilon} \leq \sum_{i=1}^{\infty} C N l_{1}^{-\varepsilon(1+2\varepsilon)^{i-1}} \leq \frac{C'_{9}}{4}.
    \end{equation}
    Thus we get \eqref{eq:pyramidF}.
    
We first prove \eqref{eq:pym,i>0}.
As in Definition \ref{def:F}, for each $1 \leq i \leq d$, we write $E_{i}=\bigcup_{j \in \Z_{+}, 1 \leq t \leq N} E^{(j,t)}_{i}$ where each $E^{(j,t)}_{i}$ is a open ball with radius $l_i$, and
$\dist(E^{(j,t)}_i,E^{(j',t)}_i) \geq l_i^{1+\varepsilon}$ for each $j\neq j'$.

\begin{cla}\label{cla:numbpym}For any $1 \leq i \leq d$, we have
$\left|\left\{(j,t):E_{i}^{(j,t)} \cap \partial\fP_{a, r, \Gamma} \neq \emptyset\right\}\right| < C N l_{i}^{-2-\varepsilon} r^{2}$, where $C$ is a universal constant.
\end{cla}
\begin{proof}
The
proof is via a simple packing argument.
Assume that $E_i\cap \fT_{a,r}\neq \emptyset$ (since otherwise the claim obviously holds).
Denote $\tilde{\caT}_{a,r}$ to be the closed equilateral triangle in $\cP_{1,a\cdot \blambda_1}$, such that it has the same center and orientation as $\caT_{a,r}$, and its side length is $100 r$.
For any $j, t$, let $B^{(j,t)}_{i}$ be the open ball with radius $l_{i}^{1+\frac{\varepsilon}{2}}$ and with the same center as $E^{(j,t)}_{i}$.
    Since $E$ is $(\varepsilon^{-\frac{1}{2}},\varepsilon)$-normal in $\fT_{a,r}$, we have $\diam(B^{(j,t)}_{i}) \leq 10 r^{1-\frac{\varepsilon^2}{4}}$.
    Suppose $E_{i}^{(j,t)} \cap \partial\fP_{a, r, \Gamma} \neq \emptyset$, we then have $\pi(B^{(j,t)}_{i}) \subset \tilde{\caT}_{a,r}$.
    In addition, if for some $j'\neq j$ we have $E_{i}^{(j',t)} \cap \partial\fP_{a, r, \Gamma} \neq \emptyset$ as well, by $\dist(E^{(j,t)},E^{(j',t)}) \geq l_i^{1+\varepsilon}$ and \eqref{eq:distpym}, we have that (when $C_{\varepsilon, N}$ is large enough) $\pi(B^{(j,t)}_{i}) \cap \pi(B^{(j',t)}_{i})=\emptyset$. Thus for any $t$, 
    \begin{equation}
        \left|\left\{j:E_{i}^{(j,t)} \cap \partial\fP_{a, r, \Gamma} \neq \emptyset\right\}\right| l_{i}^{2+\varepsilon} < \Area(\tilde{\caT}_{a,r}),
    \end{equation}
    since $\Area(\pi(B_{i}^{(j,t)}))>l_{i}^{2+\varepsilon}$ for any $j, t$.
    Our claim follows by observing that $\Area(\tilde{\caT}_{a,r}) \leq C r^{2}$.
\end{proof}
\begin{cla}\label{cla:numpym2}
    There exists some universal constant $C$ such that for any $j \in \Z_{+}$, $t\in \left\{1,2,\cdots,N\right\}$ and $i\in \left\{1,2,\cdots,d\right\}$, $\left|E^{(j,t)}_{i} \cap \partial\fP_{a, r, \Gamma}\cap \Z^3\right|\leq C l_{i}^{2}$.
\end{cla}
\begin{proof}
    By \eqref{eq:distpym}, $\pi$ is a injection from $\partial\fP_{a, r, \Gamma}$, so we only need to show 
    \begin{equation}
        \left|\pi(E_{i}^{(j,t)})\cap \pi(\Z^3)\right| \leq C l_{i}^2.
    \end{equation}
We note that $\pi(\Z^3)$ is a triangular lattice on $\cP_{1, a\cdot \blambda_1}$, with constant lattice length $\frac{\sqrt{6}}{3}$ and $\pi(E_{i}^{(j,t)})$ is a 2D ball with radius at least $C_{\varepsilon,N}$.
Assuming $C_{\varepsilon,N} > 10$, we have
\begin{equation}
    \left|\pi(E_{i}^{(j,t)})\cap \pi(\Z^3)\right| \leq 10 \Area(\pi(E_{i}^{(j,t)}))
\end{equation}
and our claim follows.
\end{proof}
    Now by Claim \ref{cla:numpym2},
\begin{equation}
\left|E_{i} \cap \partial\fP_{a, r, \Gamma}\cap \Z^3\right| \leq \sum_{j,t}\left|E^{(j,t)}_{i} \cap \partial\fP_{a, r, \Gamma}\cap \Z^3\right|
\leq \sum_{j,t}\left|\left\{(j,t):E_{i}^{(j,t)} \cap \partial\fP_{a, r, \Gamma} \neq \emptyset\right\}\right| C l_{i}^{2}.    
\end{equation} 
    Then by Claim \ref{cla:numbpym}, we get \eqref{eq:pym,i>0}. 
    
As for \eqref{eq:pym,i=0}, since by \eqref{eq:distpym} $\pi$ is a injection on $\partial\fP_{a, r, \Gamma}$, we only need to show 
\begin{equation}\label{eq:pi_pym:i=0}
    \left|\pi\left(E_{0} \cap \partial\fP_{a, r, \Gamma}\cap \Z^3\right)\right|\leq C \varepsilon^{2} r^{2}
\end{equation} 
for some universal constant $C$. By \eqref{eq:distpym} and the fact that $E_{0}$ is $\varepsilon^{-1}$-unitscattered, we have
\begin{equation}
    |\pi(b)-\pi(b')| \geq \frac{\varepsilon^{-1}}{10}
\end{equation}
for any $b \neq b' \in E_{0}\cap \partial\fP_{a, r, \Gamma}\cap \Z^3$
(since $b$ and $b'$ are centers of different unit balls in $E_0$).
Thus \eqref{eq:pi_pym:i=0} follows from $\Area(\pi(\fP_{a, r, \Gamma})) < 100 r^{2}$.
\end{proof}

\begin{proof}[Proof of Proposition \ref{prop:tetraF}]
We assume that $r > 1000$, since otherwise the statement holds by taking $C_{9}$ small enough.

To apply Proposition \ref{prop:pmt}, we need to check its third condition.
We argue by contradiction, and assume that there exists $b \in \mathring \fP_{a,r,\Gamma} \cap \Z^3$ with $b\cdot \blambda_1 \leq h$, and $|u(b)|>\exp(C_{10}n)g$.
Consider the triangle $\cP_{1, b \cdot \blambda_1} \cap \mathring \fP_{a,r,\Gamma}$,
and write it as $\{c \in \cP_{1, b \cdot \blambda_1} : c\cdot \olambda_\tau < \bt_r(a) \cdot \olambda_\tau - F', \forall \tau = 2,3,4\}$ for some $F'\geq 0$.
From the definition of $h$, the its side length is at least $\sqrt{2}\left(\frac{r}{4}-1\right)$.
Consider the three sets
\begin{equation}  \label{eq:green}
\left\{ c \in \cP_{1, b \cdot \blambda_1} : c\cdot \olambda_\tau > \bt_r(a) \cdot \olambda_\tau - F' - \frac{r}{10} \right\}
\end{equation}
where $\tau \in \{2,3,4\}$ (see Figure \ref{fig:walk-inside-pyramid}).
The intersection of all three of them is empty, so by symmetry, we can assume that $b$ is not in the first one, i.e.
\begin{equation}
b \cdot \olambda_2 \leq \bt_r(a) \cdot \olambda_2 - F' - \frac{r}{10}.
\end{equation}
\begin{figure}
    \centering
    \definecolor{qqqqcc}{rgb}{0.,0.,0.8}
\definecolor{qqzzqq}{rgb}{0.,0.6,0.}
\begin{tikzpicture}[line cap=round,line join=round,>=triangle 45,x=2.0cm,y=2.0cm]
\clip(1.339681657794315,-0.13876646084657006) rectangle (9.786494711898724,3.200160763145978);
\fill[line width=0.pt,color=qqzzqq,fill=qqzzqq,fill opacity=0.30000001192092896] (4.162182892026156,3.0413236594414337) -- (3.2365217283477787,3.0438356300348617) -- (6.100728766782767,1.6554827274940687) -- (6.409701149774464,1.9518952575186244) -- cycle;
\fill[line width=0.pt,dash pattern=on 1pt off 1pt on 1pt off 4pt,color=qqzzqq,fill=qqzzqq,fill opacity=0.30000001192092896] (6.100728766782767,1.6554827274940687) -- (7.5357489937787046,3.0321687989210733) -- (6.610087830100328,3.034680769514502) -- (5.484039986096087,1.954407228112053) -- cycle;
\fill[line width=0.pt,color=qqzzqq,fill=qqzzqq,fill opacity=0.30000001192092896] (3.2365217283477787,3.0438356300348617) -- (3.8532105090344633,2.7449111294168755) -- (7.226776610787008,2.735756268896516) -- (7.5357489937787046,3.0321687989210733) -- cycle;
\draw [line width=0.8pt] (1.424859893370625,2.4012319417301593)-- (6.6678515185230856,-0.14017744072664295);
\draw [line width=0.4pt,dash pattern=on 1pt off 1pt] (1.424859893370625,2.4012319417301593)-- (9.294686426440611,2.3798755603649764);
\draw [line width=0.8pt] (9.294686426440611,2.3798755603649764)-- (6.6678515185230856,-0.14017744072664295);
\draw [line width=0.8pt] (1.7308224457995385,2.85911671167349)-- (6.499978690076196,0.5473872083377506);
\draw [line width=0.8pt] (6.499978690076196,0.5473872083377506)-- (8.889413386801364,2.839690413326131);
\draw [line width=0.4pt,dash pattern=on 1pt off 1pt] (1.7308224457995385,2.85911671167349)-- (8.889413386801364,2.839690413326131);
\draw [line width=0.8pt] (2.953425672950927,2.620171455552697)-- (6.2560554077862225,1.0193040891763174);
\draw [line width=0.8pt] (6.2560554077862225,1.0193040891763174)-- (7.910733441939953,2.606718788608358);
\draw [line width=0.4pt,dash pattern=on 1pt off 1pt] (7.910733441939953,2.606718788608358)-- (2.953425672950927,2.620171455552697);
\draw [line width=0.8pt] (1.424859893370625,2.4012319417301593)-- (1.7308224457995385,2.85911671167349);
\draw [line width=0.8pt] (6.6678515185230856,-0.14017744072664295)-- (6.499978690076196,0.5473872083377506);
\draw [line width=0.8pt] (9.294686426440611,2.3798755603649764)-- (8.889413386801364,2.839690413326131);
\draw [line width=0.8pt] (2.953425672950927,2.620171455552697)-- (3.2365217283477787,3.0438356300348617);
\draw [line width=0.8pt] (3.2365217283477787,3.0438356300348617)-- (7.5357489937787046,3.0321687989210733);
\draw [line width=0.8pt] (7.5357489937787046,3.0321687989210733)-- (7.910733441939953,2.606718788608358);
\draw [line width=0.8pt] (7.5357489937787046,3.0321687989210733)-- (6.100728766782767,1.6554827274940687);
\draw [line width=0.8pt] (6.100728766782767,1.6554827274940687)-- (3.2365217283477787,3.0438356300348617);
\draw [line width=0.8pt] (6.100728766782767,1.6554827274940687)-- (6.2560554077862225,1.0193040891763174);
\draw [line width=0.8pt] (1.7308224457995385,2.85911671167349)-- (3.1105889948942,2.8553724333855808);
\draw [line width=0.8pt] (7.702557452106849,2.8429111892683148)-- (8.889413386801364,2.839690413326131);
\draw [line width=0.6pt,dash pattern=on 1pt off 1pt on 1pt off 4pt] (5.184242227946971,2.2164614287336644)-- (5.403541832306778,2.110161416844389);
\draw [line width=0.6pt,dash pattern=on 1pt off 1pt on 1pt off 4pt] (5.403541832306778,2.110161416844389)-- (5.59396851638688,2.10964465515082);
\draw [line width=0.6pt,dash pattern=on 1pt off 1pt on 1pt off 4pt] (5.59396851638688,2.10964465515082)-- (5.682663552738143,1.746372217873078);
\draw [line width=0.6pt,dash pattern=on 1pt off 1pt on 1pt off 4pt] (5.682663552738143,1.746372217873078)-- (5.83811652113854,1.56999832467991);
\draw [line width=0.6pt,dash pattern=on 1pt off 1pt on 1pt off 4pt] (6.019592807427639,1.4820322225682896)-- (5.83811652113854,1.56999832467991);
\draw [line width=0.6pt,dash pattern=on 1pt off 1pt on 1pt off 4pt] (6.151897521806393,0.9401455732933459)-- (6.019592807427639,1.4820322225682896);
\draw [line width=0.6pt,dash pattern=on 1pt off 1pt on 1pt off 4pt] (6.151897521806393,0.9401455732933459)-- (6.344498232362743,0.8467871840623653);
\draw (4.857449081968653,2.437920100801263) node[anchor=north west] {$b=c_0$};
\begin{scriptsize}
\draw [fill=black] (1.424859893370625,2.4012319417301593) circle (0.5pt);
\draw [fill=black] (6.6678515185230856,-0.14017744072664295) circle (0.5pt);
\draw [fill=black] (9.294686426440611,2.3798755603649764) circle (0.5pt);
\draw [fill=black] (4.8103375608045145,7.467729951356388) circle (0.5pt);
\draw [fill=black] (1.7308224457995385,2.85911671167349) circle (0.5pt);
\draw [fill=black] (6.499978690076196,0.5473872083377506) circle (0.5pt);
\draw [fill=black] (8.889413386801364,2.839690413326131) circle (0.5pt);
\draw [fill=black] (2.953425672950927,2.620171455552697) circle (0.5pt);
\draw [fill=black] (7.910733441939953,2.606718788608358) circle (0.5pt);
\draw [fill=black] (6.2560554077862225,1.0193040891763174) circle (0.5pt);
\draw [fill=black] (5.0859828102633236,5.8116255287963945) circle (0.5pt);
\draw [fill=black] (3.2365217283477787,3.0438356300348617) circle (0.5pt);
\draw [fill=black] (7.5357489937787046,3.0321687989210733) circle (0.5pt);
\draw [fill=black] (6.100728766782767,1.6554827274940687) circle (0.5pt);
\draw [fill=qqqqcc] (5.184242227946971,2.2164614287336644) circle (1.5pt);
\draw [fill=qqqqcc] (5.403541832306778,2.110161416844389) circle (1.5pt);
\draw [fill=qqqqcc] (5.59396851638688,2.10964465515082) circle (1.5pt);
\draw [fill=qqqqcc] (5.682663552738143,1.746372217873078) circle (1.5pt);
\draw [fill=qqqqcc] (5.83811652113854,1.56999832467991) circle (1.5pt);
\draw [fill=qqqqcc] (6.019592807427639,1.4820322225682896) circle (1.5pt);
\draw [fill=qqqqcc] (6.151897521806393,0.9401455732933459) circle (1.5pt);
\draw [fill=qqqqcc] (6.344498232362743,0.8467871840623653) circle (1.5pt);
\draw [fill=qqqqcc] (5.637855116270412,1.9298962412668146) circle (1.5pt);
\draw [fill=qqqqcc] (6.068873390976172,1.2801914157699705) circle (1.5pt);
\draw [fill=qqqqcc] (6.109965329561339,1.111889227773991) circle (1.5pt);
\end{scriptsize}
\end{tikzpicture}
    \caption{The three green areas are given by \eqref{eq:green} and do not have common intersection, so $b=c_0 \in \cP_{1,b\cdot \blambda_1}\cap \mathring{\mathfrak{P}}_{a,r,\Gamma}$ is outside one of them, and we can construct a path in $\mathring{\mathfrak{P}}_{a,r,\Gamma}$ from it by using the cone property.}
    \label{fig:walk-inside-pyramid}
\end{figure}
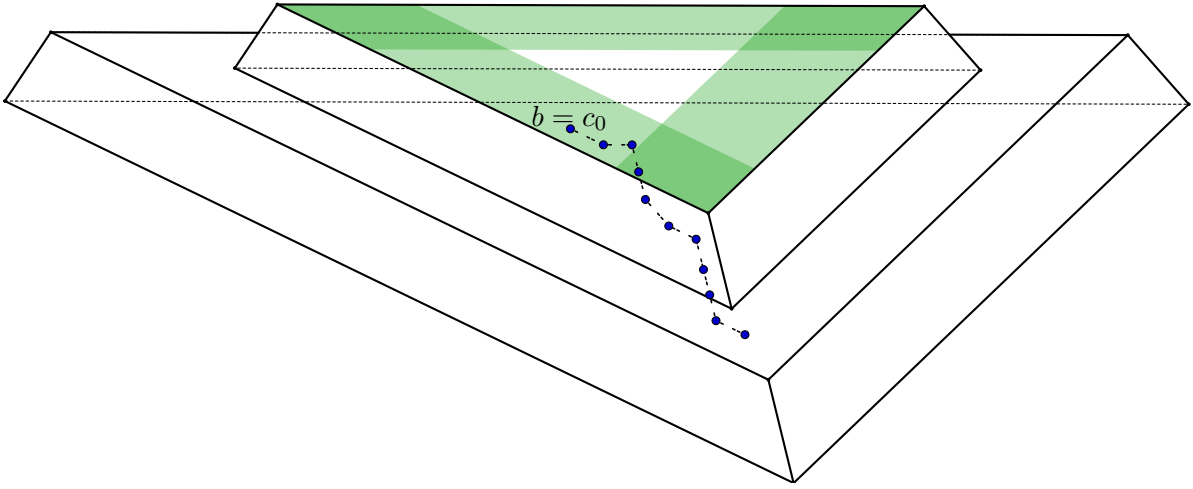

Now we apply Lemma \ref{lem:chain}, starting from $b$ and in the $-\be_1$ direction.
Since $r < \frac{n}{32}$ and $a \in Q_{\frac{n}{2}}$, we can find a sequence of points $b=c_0, c_1, \cdots, c_r$, such that for any $1 \leq i \leq r$, we have $|u(c_i)| \geq (K+11)^{-1}|u(c_{i-1})|$, and $c_i - c_{i-1} \in \left\{-\be_1, -\be_1 + \be_2, -\be_1+\be_3, -\be_1-\be_2, -\be_1-\be_3, -2\be_1 \right\}$.
Then we have that $c_i \cdot \olambda_2 \leq c_{i-1} \cdot \olambda_2 + 2$, $c_i \cdot \olambda_3 \leq c_{i-1} \cdot \olambda_3$ and $c_i \cdot \olambda_4 \leq c_{i-1} \cdot \olambda_4$.
This means that for $1 \leq i \leq \frac{r}{30}$,
\begin{equation}  \label{eq:tetra:p001}
\begin{split}
&
c_i \cdot \olambda_2 \leq b\cdot \olambda_2 + \frac{r}{15} < \bt_r(a) \cdot \olambda_2 - F', \\
&
c_i \cdot \olambda_\tau \leq b \cdot \olambda_\tau \leq \bt_r(a) \cdot \olambda_\tau - F',\;\forall \tau \in \{3,4\}.
\end{split}
\end{equation}
Also, for $i \leq \frac{r}{30}$, we have
\begin{equation} \label{eq:tetra:p1}
    |u(c_i)| \geq (K+11)^{-\frac{r}{30}}|u(c_{0})| > \exp\left(\frac{C_{10}n}{2}\right)g,
\end{equation}
when $C_{10}>K+11$.
Since $c_{i-1} \blambda_1 - 2 \leq c_i \cdot \blambda_1 \leq c_{i-1} \blambda_1$, by the second condition of Proposition \ref{prop:tetraF}, we have that $a \cdot \blambda_1 < c_i \cdot \blambda_1\leq b \cdot \blambda_1$ for each $1 \leq i \leq \frac{r}{30}$.
With \eqref{eq:tetra:p001} this implies that $c_i \in \mathring \fP_{a,r,\Gamma}$ for each $1 \leq i \leq \frac{r}{30}$.
See Figure \ref{fig:walk-inside-pyramid} for an illustration.

By the definition of the pyramid $\fP_{a,r,\Gamma}$, for $0 \leq i \leq \frac{r}{30}$ we have that $c_i \not\in \Gamma$, thus $c_i \in E$ by \eqref{eq:tetra:p1} and the fourth condition of Proposition \ref{prop:tetraF}. 

For $l \in \R_+$ with $1 \leq l < (2\sqrt{2}r)^{1-\frac{\varepsilon}{2}}$, and any $(1,l,\varepsilon)$-scattered set $Z$, the number of balls in $Z$ that intersect $\left\{c_i\right\}_{i=1}^{\left\lfloor \frac{r}{30} \right\rfloor}$ is at most $2\left\lfloor \frac{r}{30} \right\rfloor l^{-1-\varepsilon}+1$.
This is because, otherwise, there must exist $1\leq i_1 < i_2 \leq \left\lfloor \frac{r}{30} \right\rfloor$, such that $|i_1-i_2|< \frac{l^{1+\varepsilon}}{2}$, and $c_{i_1}$ and $c_{i_2}$ are contained in different balls.
By construction the distance between $c_{i_1}$ and $c_{i_2}$ is at most $2|i_1-i_2|$, and this contradicts with the fact that $Z$ is $(1,l,\varepsilon)$-scattered.
For each ball in $Z$, it contains at most $2l$ points in $\left\{c_i\right\}_{i=1}^{\left\lfloor \frac{r}{30} \right\rfloor}$. This is because $c_i\cdot\be_1\leq c_{i-1}\cdot\be_1-1$ for $1<i\leq \left\lfloor \frac{r}{30} \right\rfloor$, and the diameter of each ball is $2l$.
Thus we have
\begin{equation} \label{eq:tetra:p2}
\left| Z \cap \left\{c_i\right\}_{i=1}^{\left\lfloor \frac{r}{30} \right\rfloor} \right| \leq 2l\cdot\left( 2\left\lfloor \frac{r}{30} \right\rfloor l^{-1-\varepsilon}+1\right) < r l^{-\varepsilon} + 2l.
\end{equation}
Similarly, for any $\varepsilon^{-1}$-unitscattered set $Z$, we have
\begin{equation} \label{eq:tetra:p22}
 \left| Z \cap \left\{c_i\right\}_{i=1}^{\left\lfloor \frac{r}{30} \right\rfloor} \right| < 
 r\varepsilon+2.
\end{equation}
For the set $E$ which is $(\varepsilon^{-\frac{1}{2}},\varepsilon)$-normal in $\fP_{a,r,\Gamma}$, using \eqref{eq:tetra:p2} and \eqref{eq:tetra:p22} we have
\begin{equation}  \label{eq:tetra:p30}
\left| E \cap \left\{c_i\right\}_{i=1}^{\left\lfloor \frac{r}{30} \right\rfloor} \right|
<
r\varepsilon+2
+
\sum_{1 \leq i \leq d: l_i < (2\sqrt{2}r)^{1-\frac{\varepsilon}{2}}}
Nr l^{-\varepsilon}_i + 2Nl_i
\end{equation}
We have that 
\begin{equation} 
Nr\sum_{i=1}^{d} l_i^{-\varepsilon}
\leq
Nr\sum_{i=1}^{\infty} l_1^{-\varepsilon(1+2\varepsilon)^{i-1}} 
<
Nr\sum_{i=1}^{\infty} C_{\varepsilon, N}^{-\varepsilon(1+2\varepsilon)^{i-1}} 
<
Nr\sum_{i=1}^{\infty} C_{\varepsilon, N}^{-\varepsilon} C_{\varepsilon, N}^{-2(i-1)\varepsilon^2}
=
\frac{NrC_{\varepsilon, N}^{-\varepsilon}}{1-C_{\varepsilon, N}^{-2\varepsilon^2}},
\end{equation}
and when $C_{\varepsilon, N}$ is large enough this is less than $\frac{r}{100}$.

Also, when $(2\sqrt{2}r)^{1-\frac{\varepsilon}{2}} > C_{\varepsilon, N}>100$, and $\varepsilon < \frac{1}{200}$, we have
\begin{equation} \label{eq:tetra:p3}
\sum_{1 \leq i \leq d: l_i < (2\sqrt{2}r)^{1-\frac{\varepsilon}{2}}}
2Nl_i
< 2\left(\frac{\log\left(\frac{\log(2\sqrt{2}r)}{\log(C_{\varepsilon, N})}\right)}{\log(1+2\varepsilon)} + 1 \right)N(2\sqrt{2}r)^{1-\frac{\varepsilon}{2}}
<
\frac{4\log(\log(2\sqrt{2}r))}{\varepsilon}N(2\sqrt{2}r)^{1-\frac{\varepsilon}{2}},
\end{equation}
where the first inequality is due to that there are at most $\left\lceil \frac{\log\left(\frac{\log(2\sqrt{2}r)}{\log(C_{\varepsilon, N})}\right)}{\log(1+2\varepsilon)} \right\rceil$ terms in the summation, and each is at most $2N(2\sqrt{2}r)^{1-\frac{\varepsilon}{2}}$.
We further have that \eqref{eq:tetra:p3} is less than $\frac{r}{100}$ when $C_{\varepsilon, N}$ is large enough.
When $(2\sqrt{2}r)^{1-\frac{\varepsilon}{2}} \leq C_{\varepsilon, N}$,
the left hand side of \eqref{eq:tetra:p3} is zero.
Thus the left hand side of \eqref{eq:tetra:p30} is less than $\frac{3r}{100}+2<\frac{r}{30}$ when $\varepsilon < \frac{1}{100}$
and
$C_{\varepsilon, N}$ is large enough.
This contradicts with the fact that $c_i \in E$ for each $0 \leq i \leq \frac{r}{30}$.

Finally, the conclusion follows from Proposition \ref{prop:pmt} and \ref{prop:pmtF}, by taking $C_{9}= \frac{1}{2}C'_{9}$ and the same $C_{10}$ as in Proposition \ref{prop:pmt}.
\end{proof}

\subsection{Proof of Theorem \ref{thm:wqucF}}\label{ssec:deta}
In this subsection we assemble results in previous subsections together and finish the proof of Theorem \ref{thm:wqucF}.

\begin{proof}[Proof of Theorem \ref{thm:wqucF}]
By taking $C_{\varepsilon, N}$ large we can assume that $n>100$.

We prove the result for $C_{3}=\frac{1}{60}C_{8}$ and $C_{2}=\max\left\{2C_{7}, 2\log(K+11) \right\}$, where $C_{8},C_{7}$ are the constants in Proposition \ref{prop:lvh}.
We let $\varepsilon$ be small enough,
and $C_{\varepsilon,N}$ be the same as required by Proposition \ref{prop:lvh}. 
By Proposition \ref{prop:estp}, there exists $\tau \in \left\{1,2,3,4\right\}$, and
     \begin{equation}
     a_i \in \left(\cP_{\tau, i}\cup \cP_{\tau, i+1}\right) \cap \cC \cap Q_{\frac{n}{10} + 1}
    \end{equation}
     for $i=0,1,\cdots,\left\lfloor \frac{n}{10} \right\rfloor - 1$, such that $|u(a_{i})|\geq (K+11)^{-n}|u(\mathbf{0})|$.

For each $i=0,1,\cdots,\left\lfloor \frac{n}{10} \right\rfloor-1$, we apply Proposition \ref{prop:lvh} to $a_i$, and find $h_i \in \Z_+$, such that
\begin{multline}
\left| \left\{a \in Q_n \cap \bigcup_{j=0}^{h_i} \cP_{\tau, a_i\cdot \blambda_1+j}: |u(a)| \geq \exp(-C_{7}n^3)|u(a_i)| \geq \exp(-C_{2}n^3)|u(\mathbf{0})| \right\} \setminus E \right| \\ > C_{8}h_i n (\log_2(n))^{-1}.
\end{multline}
Now for some $m\in \Z_{\geq 0}$, we define a sequence of nonnegative integers $i_1 < \cdots < i_m$ inductively.
Let $i_1 := 0$.
Given $i_k$, if $a_{i_k} \cdot \blambda_{\tau} + h_k + 1 \leq \left\lfloor \frac{n}{10} \right\rfloor-1$, we let $i_{k+1} := a_{i_k} \cdot \blambda_{\tau} + h_{i_k} + 1$; otherwise, let $m=k$ and the process terminates.

Obviously, the sets 
\begin{equation}
\left\{a \in Q_n \cap \bigcup_{j=0}^{h_{i_k}} \cP_{\tau, a_{i_k}\cdot \blambda_1+j} : |u(a)| \geq \exp(-C_{2}n^3)|u(\mathbf{0})| \right\} \setminus E  
\end{equation}
for $k=1, \cdots, m$ are mutually disjoint.
Besides, we have that $a_{i_1}\cdot \blambda_{\tau} \leq 1$ and $a_{i_m}\cdot \blambda_{\tau} + h_{i_m} \geq \left\lfloor \frac{n}{10} \right\rfloor-1$; and for each $1 \leq k < m$,
$a_{i_{k+1}} \cdot \blambda_{\tau} - a_{i_k} \cdot \blambda_{\tau} \leq h_{i_k} + 2$.
This implies that $\sum_{j=1}^m (h_{i_k}+2) \geq \left\lfloor \frac{n}{10} \right\rfloor-2$, thus $\sum_{j=1}^m h_{i_k} >  \frac{n}{60}$, and
\begin{multline}
\left| \left\{ a \in Q_{n} : |u(a)| \geq \exp(-C_{2} n^{3})  |u(\mathbf{0})| \right\}\setminus E \right| \geq C_{8}\left(\sum_{k=1}^m h_{i_k}\right) n (\log_2(n))^{-1} \\ > C_3n^2(\log_2(n))^{-1}
\end{multline}
which is \eqref{eq:wqucF}.
\end{proof}

\section{Recursive construction: proof of discrete unique continuation}  \label{sec:quc}
We deduce Theorem \ref{thm:qucF} from Theorem \ref{thm:wqucF} in this section.
The key step is the following result.
\begin{theorem}\label{thm:low}
  There exist universal constants $\beta$ and $\alpha >
  \frac{5}{4}$ such that  for any positive integers $m \leq n$ and any positive real $K$, the following is true.
  For any $u,V:\Z^3 \rightarrow \R$ such that $\Delta u=V u$ in
  $Q_{n}$ and $\| V \|_{\infty} \leq K$, we can find a subset $\Theta \subset Q_n$ with $|\Theta| \geq \beta \left(\frac{n}{m}\right)^{\alpha} $, such that
  \begin{enumerate}
        \item  $|u(b)| \geq ( K+11 )^{-12n} | u ( \mathbf{0} ) |$ for each $b \in \Theta$.
        \item $Q_m(b) \cap Q_m(b')=\emptyset$ for $b,b' \in \Theta$, $b \neq b'$.
        \item $Q_m(b)\subset Q_n $ for each $b \in \Theta$.
  \end{enumerate}
\end{theorem}
The proof of Theorem \ref{thm:low} is based on the cone property, i.e. Lemma \ref{lem:chain}, and induction on $\frac{n}{m}$.
We first set up some notations.
\begin{defn}
A set $B \subset \Z^3$ is called a \emph{cuboid} if there are integers $t_{\tau}\leq k_{\tau}$, for $\tau=1,2,3$, such that 
\begin{equation}
B=\left\{b \in \Z^3: t_{\tau} \leq b\cdot \be_{\tau} \leq k_{\tau},\tau=1,2,3 \right\} .
\end{equation}
We denote $p^+(B):=k_1$, $p^-(B):=t_1$, and $q^+(B):=k_2$, $q^-(B):=t_2$.
A cuboid is called \emph{even} if $t_{\tau}, k_{\tau}$ are even for each $\tau = 1,2,3$.
\end{defn}
\begin{proof} [Proof of Theorem \ref{thm:low}]
Without loss of generality we assume that $u ( \mathbf{0} ) =1$.

Take $\alpha=1.251>\frac{5}{4}$, and leave $\beta$ to be determined.
We denote $f_{m} ( x ) = \beta (\frac{x}{m}) ^{\alpha}$ for $x>0$. Then we have the following two inequalities:
  \begin{equation}
      4\cdot 4^{-\alpha} +4\cdot 8^{-\alpha} >1, \;
      6\cdot 4^{-\alpha} >1.
  \end{equation}
  This implies that there exists universal $N_0>10^8$ such that, for any positive integers $m,n$ with $n>N_0 m$  and any real $\beta >0$, we have 
  \begin{equation}   \label{eq:low:pf1}
      4f_{m}\left(\frac{n}{4}-3\right)+4f_{m}\left(\frac{n}{8}-2\right)> f_{m}(n+7)
  \end{equation}
   and 
   \begin{equation}   \label{eq:low:pf2}
      4f_{m}\left(\frac{n}{4}-3\right)+2f_{m}\left(\frac{n}{4}-2\right)> f_{m}(n+7).
  \end{equation}
  We let $\beta=(N_{0}+7)^{-\alpha}$, and fix $m \in \Z_+$.
  We need to show that, when $n \geq m$, there is $\Theta \subset Q_n$, such that $|\Theta| \geq f_{m}(n)$, and $\Theta$ satisfies the three conditions in the statement.
  For this, we do induction on $n$.
  First, it holds trivially when $m \leq n \leq N_0 m+7$ by the choice of $\beta$.
  For simplicity of notations below, we only work on $n$ that divides $8$.
  At each step, we take some $n>N_0 m\geq 10^8 m$ with $\frac{n}{8}\in\Z$ and suppose our conclusion holds for all smaller $n$. Then we show that we can find a subset $\Theta \subset Q_{n}$ with $|\Theta| \geq f_m(n+7) $, such that the conditions in the statement are satisfied.
  Thus the conclusion holds for $n,n+1,\cdots, n+7$.

  By Lemma \ref{lem:chain}, and using the notations in Definition \ref{defn:cone},
  we pick $a_1 \in \cC^{3}_{\mathbf{0}} (  \frac{n}{2}  ) \cup \cC_{\mathbf{0}}^{3} (  \frac{n}{2}  +1 )$ and $a_2 \in \cC^{3}_{\mathbf{0}} ( -  \frac{n}{2}
   ) \cup \cC_{\mathbf{0}}^{3} ( -  \frac{n}{2}  -1 ) $ such that $|u(a_1)|, | u ( a_2 ) | \geq (
  K+11 )^{-n}$.
  For simplicity of notations, we denote $Q^1$ as the even cuboid such that $Q_{ \frac{n}{2}  -2} ( a_1 )\subset Q^1\subset Q_{ \frac{n}{2}  -1} ( a_1 )$;
  and $Q^2$ as the even cuboid such that $Q_{\frac{n}{2}-2}(a_2) \subset Q^2 \subset Q_{\frac{n}{2}-1}(a_2)$.

  Then we use Lemma \ref{lem:chain} again to pick 
  \begin{equation}
      \begin{split}
          &a_{11} \in \cC_{a_1}^{3}\left( \frac{n}{4} -1\right)\cup \cC_{a_1}^{3}\left( \frac{n}{4} \right),\\
          &a_{12}\in \cC_{a_1}^{3}\left(- \frac{n}{4} +1\right) \cup \cC_{a_1}^{3}\left(- \frac{n}{4} \right),\\
          &a_{21} \in \cC_{a_2}^{3}\left( \frac{n}{4}  -1\right)\cup
          \cC_{a_2}^{3}\left( \frac{n}{4} \right),\\ 
          &a_{22}\in \cC_{a_2}^{3}\left(- \frac{n}{4} +1\right) \cup \cC_{a_2}^{3}\left(- \frac{n}{4} \right),
      \end{split}
  \end{equation}
   such that $|u(a_{11})|, |u(a_{12})|,|u(a_{21})|,|u(a_{22})|\geq ( K+11 )^{-2n}$. 
   For $i,j\in \left\{1,2\right\}$, let $Q^{ij}$ be an even cuboid such that $Q_{\frac{n}{4}-3}(a_{ij})\subset Q^{ij}\subset Q_{\frac{n}{4}-2}(a_{ij})$.
   Comparing the coordinates of $a_{ij}$'s, we see $Q^{ij}$'s are pairwise disjoint.
  
  By inductive hypothesis,
  we can find $4f (\frac{n}{4}-3)$ points in $Q^{11}\cup Q^{12}\cup Q^{21}\cup Q^{22}$, such that for each $b$ among them,
  \begin{equation}
      | u(b) | \geq ( K+11
  )^{-2n} ( K+11 )^{-12  (\frac{n}{4}-3) } \geq ( K+11 )^{-12n}
  \end{equation}
   and all $Q_m(b)$ are mutually disjoint, and contained in $Q^{11} \cup Q^{12}\cup Q^{21} \cup Q^{22}$.

  Let $B$ be the minimal cuboid containing $Q^1\cup Q^2$, $B_1$ be the minimal cuboid containing $Q^{11}\cup Q^{12}$, and $B_2$ be the minimal cuboid containing $Q^{21}\cup Q^{22}$.

  Let $g^{(r)}:=p^+(Q_n)-p^+(B)$, $g^{(l)}:=p^-(B)-p^-(Q_n)$, $g^{(r)}_{1}:=p^+(Q^1)-p^+(B_1)$, $g^{(l)}_{1}:=p^-(B_1)-p^-(Q^1)$, $g^{(r)}_{2}:=p^+(Q^2)-p^+(B_2)$ and $g^{(l)}_{2}:=p^-(B_2)-p^-(Q^2)$. 
  
  Similarly, in the $\be_2$-direction, let $h^{(u)}:=q^+(Q_n)-q^+(B)$, $h^{(d)}:=q^-(B)-q^-(Q_n)$, $h^{(u)}_{1}:=q^+(Q^1)-q^+(B_1)$, $h^{(d)}_{1}:=q^-(B_1)-q^-(Q^1)$, $h^{(u)}_{2}:=q^+(Q^2)-q^+(B_2)$ and $h^{(d)}_{2}:=q^-(B_2)-q^-(Q^2)$. See Figure \ref{fig:duc_7} for an illustration of these definitions.
  \begin{figure}
      \centering
      \includegraphics{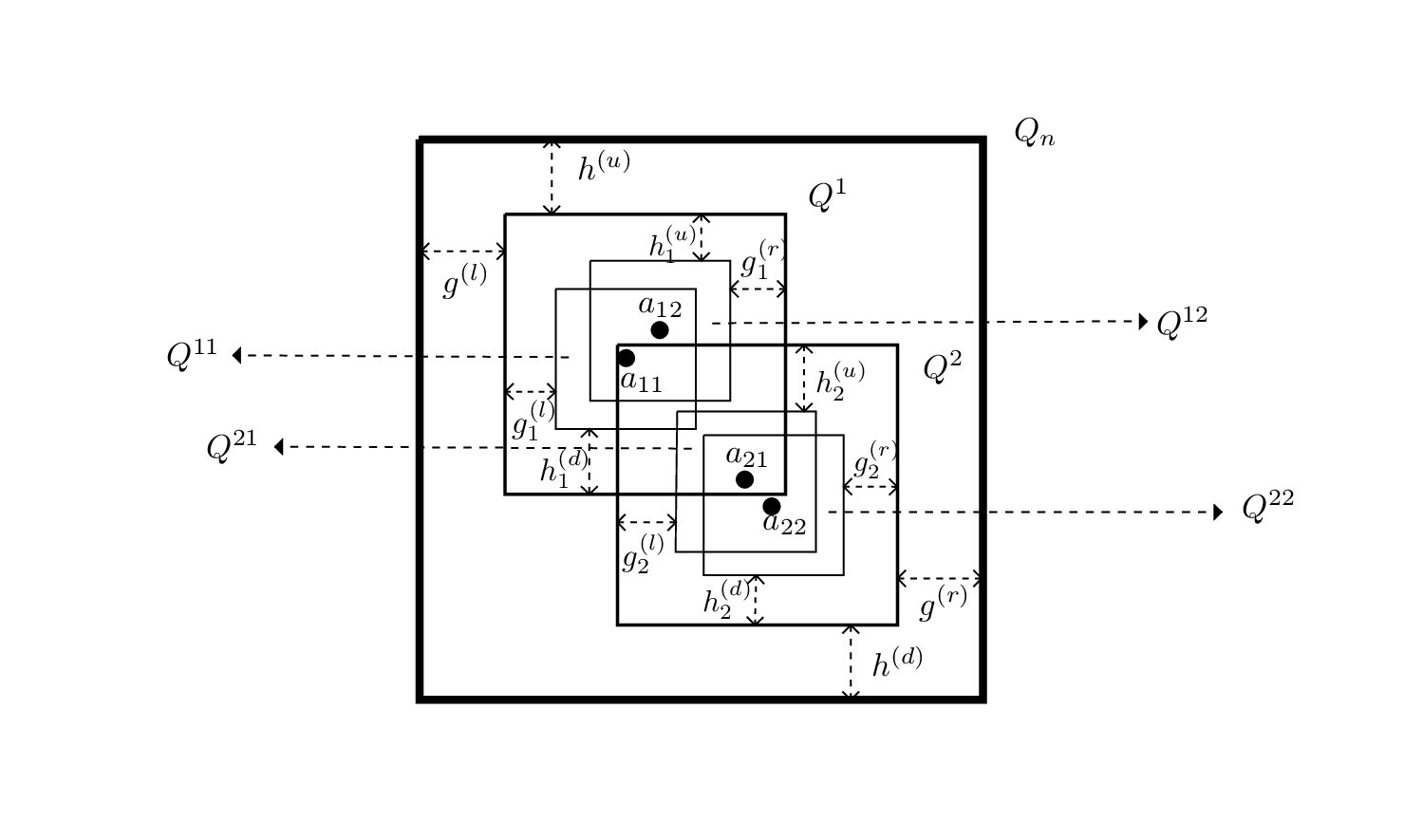}
      \caption{The projection onto the $\be_1\be_2$ plane.}
      \label{fig:duc_7}
  \end{figure}
  
  From the above definitions, 
  \begin{equation}\label{sumofg}
      g^{(r)}+g^{(l)}+h^{(u)}+h^{(d)}=4n-(p^+(B)-p^-(B))-(q^+(B)-q^-(B)).
  \end{equation}
  Observe that 
  \begin{equation}\label{eq:distfromQn}
      (p^+(B)-p^-(B))+(q^+(B)-q^-(B))\leq |(a_1-a_2)\cdot \be_1|+|(a_1-a_2)\cdot \be_2|+4\left( \frac{n}{2} -1 \right).
  \end{equation}
  As $a_1 \in \cC^{3}_{\mathbf{0}} (  \frac{n}{2}  ) \cup \cC_{\mathbf{0}}^{3} (  \frac{n}{2}  +1 ) $, we have $|a_1\cdot \be_1|+|a_1\cdot \be_2|\leq |a_1\cdot \be_3| \leq \frac{n}{2}+1$; and similarly, we have $|a_2\cdot \be_1|+|a_2\cdot \be_2|\leq \frac{n}{2}+1$.
  Using these and \eqref{eq:distfromQn}, and triangle inequality, we have
  \begin{equation}
      (p^+(B)-p^-(B))+(q^+(B)-q^-(B)) \leq 3n-2.
  \end{equation}
  Thus with \eqref{sumofg} we have
  \begin{equation}
      g^{(r)}+g^{(l)}+h^{(u)}+h^{(d)} \geq n+2.
  \end{equation}
   The same argument applying to smaller cubes $Q^{1}$ and $Q^{2}$, we have 
   \begin{equation}
       g^{(r)}_{1}+g^{(l)}_{1}+h^{(u)}_{1}+h^{(d)}_{1} \geq \frac{n}{2}+2
   \end{equation}
    and 
    \begin{equation}
        g^{(r)}_{2}+g^{(l)}_{2}+h^{(u)}_{2}+h^{(d)}_{2} \geq \frac{n}{2}+2.
    \end{equation}
     Summing them together we get 
     \begin{equation}
         g^{(r)}+g^{(l)}+g^{(r)}_{1}+g^{(l)}_{1}+g^{(r)}_{2}+g^{(l)}_{2}+h^{(u)}+h^{(d)}+h_{1}^{(u)}+h^{(d)}_{1}+h^{(u)}_{2}+h^{(d)}_{2} \geq 2n+6.
     \end{equation}
      As these $g$'s and $h$'s are exchangeable, we assume without loss of generality that 
  \begin{equation}\label{eq:pickcase}
      g^{(r)}+g^{(l)}+g^{(r)}_{1}+g^{(l)}_{1}+g^{(r)}_{2}+g^{(l)}_{2} \geq n + 3.
  \end{equation}
  By symmetry, we assume without loss of generality that $a_1 \cdot \be_1 \leq a_2 \cdot \be_1$; consequently $p^-(Q^1)\leq p^-(Q^2)$.
  We discuss two possible cases.
  
  \noindent{\textbf{Case 1:}} $p^+(B_2) \leq p^+(Q^1)$ or $p^-(B_1) \geq p^-(Q^2)$. By symmetry again, it suffices to consider the scenario for $p^+(B_2) \leq p^+(Q^1)$. See Figure \ref{fig:case_1} for an illustration.

\begin{figure}
    \centering
    \includegraphics{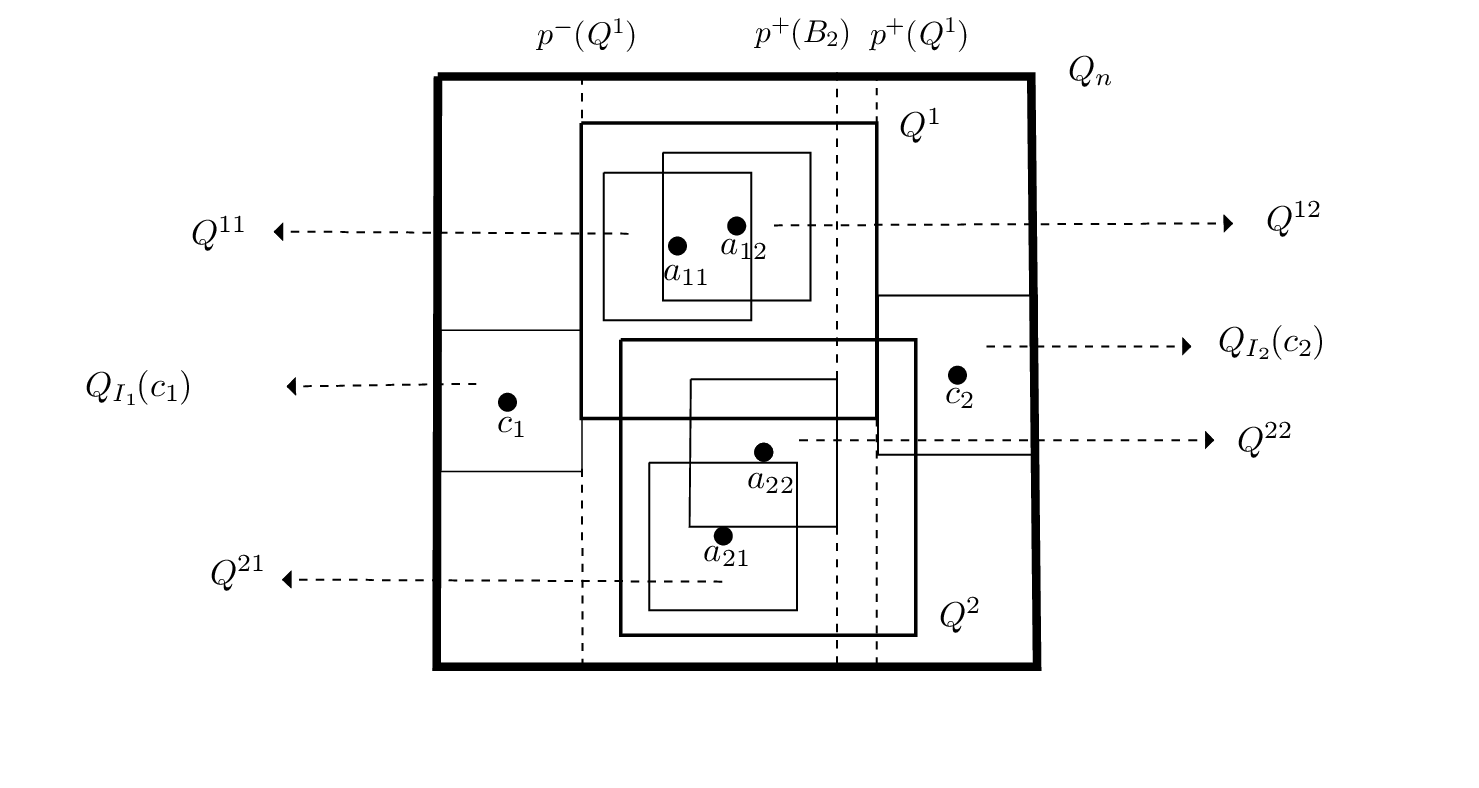}
    \caption{The projection onto the $\be_1\be_2$ plane in Case 1.}
    \label{fig:case_1}
\end{figure}

Consider cuboids
  \begin{equation}
  \begin{split}
      & U_l:=\left\{b \in \Z^3: |b\cdot \be_2|, |b\cdot \be_3| \leq n-1, -n+1 \leq b\cdot \be_1 \leq p^-(Q^1)-1\right\}, \\
      & U_r:=\left\{b \in \Z^3: |b\cdot \be_2|, |b\cdot \be_3| \leq n-1, p^+(Q^1)+1 \leq b\cdot \be_1 \leq n-1\right\}.
  \end{split}
  \end{equation}
Then $U_l,U_r,B_1,B_2$ are mutually disjoint, since $p^+(B_2) \leq p^+(Q^1)$ and $p^-(Q^1) \leq p^-(Q^2)$.
Now we use Lemma \ref{lem:chain} to pick points
  \begin{equation}
  \begin{split}
      & c_1 \in  \cC^{1}_{\mathbf{0}} \left(  \frac{1}{2}(p^-(Q^1) -n)  \right) \cup \cC_{\mathbf{0}}^{1} \left( \frac{1}{2}(p^-(Q^1) -n)  +1 \right)  ,\\
      & c_2 \in  \cC^{1}_{\mathbf{0}} \left(  \frac{1}{2}(p^+(Q^1) +n)  \right) \cup \cC_{\mathbf{0}}^{1} \left( \frac{1}{2}(p^+(Q^1) +n)  +1\right)  ,
  \end{split}
  \end{equation}
   such that $|u(c_1)|, |u(c_2)| \geq (K+11)^{-n}$.
   Denote $I_1:=\frac{p^-(Q^1)+n}{2}-2, I_2:=\frac{n-p^+(Q^1)}{2}-2$.
   Then $I_1, I_2 \leq \frac{n}{2}$.
   We also have
   \begin{equation}\label{conv1}
       (p^-(Q^1)+n)+(n-p^+(Q^1)) = 2n+p^-(Q^1)-p^+(Q^1)\geq n+2,
   \end{equation}
    so
    \begin{equation}\label{eq:sumofI}
        I_1+I_2 \geq \frac{n}{2}-3.
    \end{equation}

   We use inductive hypothesis on $Q_{I_1}(c_1) \subset U_l$, if $I_1 > m$; and on $Q_{I_2}(c_2) \subset U_r$, if $I_2 > m$.
   Note that $U_l, U_r, B_1, B_2$ are mutually disjoint.
   Thus we get $f_{m}(I_1)\mathds{1}_{I_1 > m} + f_m(I_2)\mathds{1}_{I_2 > m}$ points in $\Z^3$, such that for each point $b$ among them,
\begin{itemize}
    \item $|u(b)| \geq (K+11)^{-n}(K+11)^{-12\cdot \frac{n}{2}} \geq (K+11)^{-12n}$,
    \item $Q_m(b)\cap Q_m(b')=\emptyset$ for another $b'\neq b$ among them,
    \item $Q_m(b)\subset Q_n\setminus (Q^{11}\cup Q^{12}\cup Q^{21}\cup Q^{22})$.
\end{itemize}
We now show that
\begin{equation}  \label{eq:51pfr}
      f_{m}(I_1)\mathds{1}_{I_1 > m} + f_m(I_2)\mathds{1}_{I_2 > m} \geq 2f_{m}\left(\frac{n}{4}-2\right).
\end{equation}
If $I_1, I_2>m$, \eqref{eq:51pfr} follows by convexity and monotonicity of the function $f_{m}$, and \eqref{eq:sumofI}.
If $I_1\leq m$, by \eqref{eq:sumofI} and the assumption that $n>N_0 m \geq 10^8 m$, we have $I_2\geq \frac{n}{2}-3-m > 10^7 m$.
Then by monotonicity of $f_m$ we have $f_m(I_2)\mathds{1}_{I_2 > m} = f_m(I_2) \geq f_m\left(\frac{n}{2}-3-m\right) \geq 2f_{m}\left(\frac{n}{4}-2\right)$, which implies \eqref{eq:51pfr}.
The case when $I_2\leq m$ is symmetric.

Now together with the $4f_{m}\left(\frac{n}{4}-3\right)$ points we found in $Q^{11}\cup Q^{12}\cup Q^{21}\cup Q^{22}$, we have a set of at least $4f_{m}\left(\frac{n}{4}-3\right)+2f_{m}\left(\frac{n}{4}-2\right)$ points in $Q_n$, satisfying all the three conditions.

\noindent{\textbf{Case 2:}} $p^+(B_2) > p^+(Q^1)$ and $p^-(B_1) < p^-(Q^2)$. See Figure \ref{fig:case_2} for an illustration.

\begin{figure}
    \centering
    \includegraphics{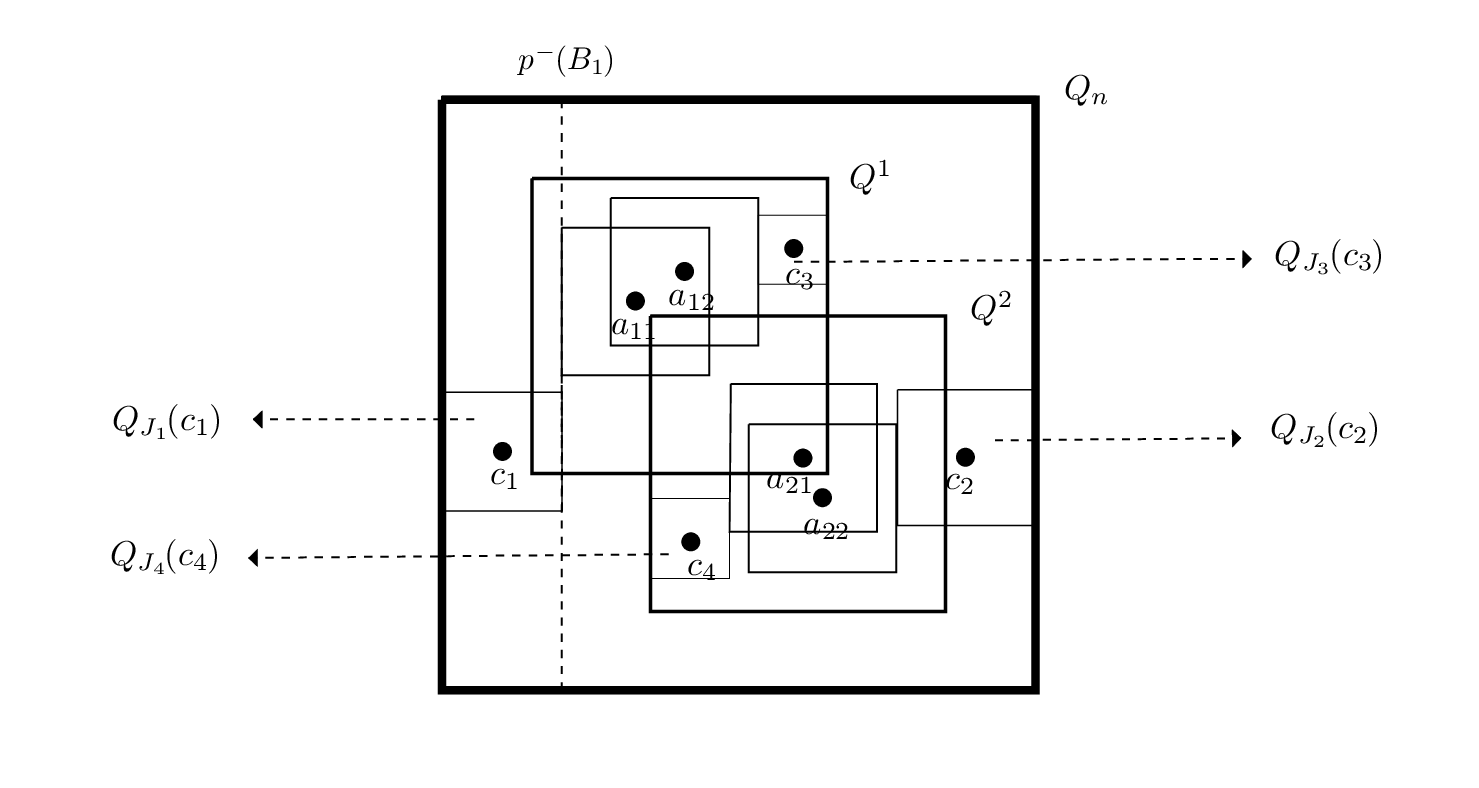}
    \caption{The projection onto the $\be_1\be_2$ plane in Case 2.}
    \label{fig:case_2}
\end{figure}

Denote
  \begin{equation}
  \begin{split}
& U_1:=\left\{b \in \Z^3: |b\cdot \be_2|, |b\cdot \be_3| \leq n-1, -n+1 \leq b\cdot \be_1 \leq p^-(B_1)-1\right\}, \\
& U_2:=\left\{b \in \Z^3: |b\cdot \be_2|, |b\cdot \be_3| \leq n-1, p^+(B_2)+1 \leq b\cdot \be_1 \leq n-1\right\}, \\
& U_3:=\left\{b \in \Z^3: |b\cdot \be_2| \leq n-1, 1 \leq b\cdot \be_3 \leq n-1, p^+(B_1)+1 \leq b\cdot \be_1 \leq p^+(Q^1)-1\right\}, \\
& U_4:=\left\{b \in \Z^3: |b\cdot \be_2| \leq n-1, -n+1 \leq b\cdot \be_3 \leq -1, p^-(Q^2)+1 \leq b\cdot \be_1 \leq p^-(B_2)-1\right\}.
  \end{split}
  \end{equation}
We note that $U_1$, $U_2$, $U_3$, $U_4$, $B_1$ and $B_2$ are mutually disjoint.
  
We use Lemma \ref{lem:chain} to pick the following points:
\begin{equation}
    \begin{split}
  & c_1 \in \cC^{1}_{\mathbf{0}} \left(  \frac{1}{2}\left(p^-\left(B_1\right) -n\right)  \right) \cup \cC_{\mathbf{0}}^{1} \left( \frac{1}{2}\left(p^-\left(B_1\right) -n\right)  +1 \right) ,\\
  & c_2 \in \cC^{1}_{\mathbf{0}} \left(  \frac{1}{2}\left(p^+\left(B_2\right) +n\right)  \right) \cup \cC_{\mathbf{0}}^{1} \left( \frac{1}{2}\left(p^+\left(B_2\right) +n\right)  +1 \right),\\
  & c_3 \in \cC^{1}_{a_1} \left(  \frac{1}{2}\left(p^+\left(B_1\right) + p^+\left(Q^1\right)\right)  -a_1 \cdot \be_1 \right) \cup \cC_{a_1}^{1} \left(  \frac{1}{2}\left(p^+\left(B_1\right) + p^+\left(Q^1\right)\right)  -a_1 \cdot \be_1 +1 \right),\\
  & c_4 \in \cC^{1}_{a_2} \left(   \frac{1}{2}\left(p^-\left(B_2\right) + p^-\left(Q^2\right)\right)  -a_2 \cdot \be_1 \right) \cup \cC_{a_2}^{1} \left(  \frac{1}{2}\left(p^-\left(B_2\right) + p^-\left(Q^2\right)\right)  -a_2 \cdot \be_1 +1 \right),
    \end{split}
\end{equation}
such that $|u(c_i)| \geq (K+11)^{-3n}$ for each $i=1,2,3,4$.

Denote $J_1:=\frac{p^-(B_1)+n}{2} -2$, $J_2:=\frac{n-p^+(B_2)}{2} -2$, $J_3:=\frac{p^+(Q^1)-p^+(B_1)}{2} -2$, and $J_4:=\frac{p^-(B_2)-p^-(Q^2)}{2} -2$.

For each $i=1,2,3,4$, if $J_i > m$, we use inductive hypothesis on $Q_{J_i}(c_i) \subset U_i$ (note that $Q_{J_i}(c_i)$ is disjoint from $Q_{11}$, so $J_i \leq \frac{3n}{4}$).
As the sets $B_1$, $B_2$, $U_1$, $U_2$, $U_3$ and $U_4$ are mutually disjoint, we can find $\sum_{i=1}^4 f_{m}(J_i)\mathds{1}_{J_i>m}$ points in $\bigcup_{i=1}^4 U_i$, such that for each point $b$ among them, 
\begin{itemize}
    \item $|u(b)|\geq (K+11)^{-3n}(K+11)^{-12 \cdot\frac{3n}{4}}=(K+11)^{-12n}$,
    \item $Q_m(b)\cap Q_m(b')=\emptyset$ for another $b'\neq b$ among them,
    \item $Q_m(b)\subset Q_n\setminus (Q^{11}\cup Q^{12}\cup Q^{21}\cup Q^{22})$.
\end{itemize}

By \eqref{eq:pickcase}, we have
  \begin{multline}
      (p^-(B_1)+n)+(n-p^+(B_2))+(p^+(Q^1)-p^+(B_1))+(p^-(B_2)-p^-(Q^2))\\=g^{(r)}+g^{(l)}+g^{(r)}_{1}+g^{(l)}_{1}+g^{(r)}_{2}+g^{(l)}_{2} \geq n+3,
  \end{multline}
thus $J_1+J_3+J_3+J_4 \geq \frac{n}{2}-7$.
Similar to \eqref{eq:51pfr} above, by monotonicity and convexity of $f_{m}$, and $n>N_0 m \geq 10^8 m$, we have
   \begin{equation}
\sum_{i=1}^4 f_{m}(J_i)\mathds{1}_{J_i>m} \geq 4f_{m}\left(\frac{n}{8}-2\right).
   \end{equation}
This implies that, together with the $4f_{m}\left(\frac{n}{4}-3\right)$ points we found in $Q^{11}\cup Q^{12}\cup Q^{21}\cup Q^{22}$, we have a set of at least $4f_{m}\left(\frac{n}{4}-3\right)+4f_{m}\left(\frac{n}{8}-2\right)$ points in $Q_n$, satisfying all the three conditions.

In conclusion, by \eqref{eq:low:pf1} and \eqref{eq:low:pf2}, in each case, we can always find a $\Theta \subset Q_n$ satisfying the three conditions, with $|\Theta| \geq f_{m}(n+7)$. Thus Theorem \ref{thm:low} follows from the principle of induction.
\end{proof}

Now we prove Theorem \ref{thm:qucF}.
\begin{proof}[Proof of Theorem \ref{thm:qucF}]
    Let $p:=\frac{1}{3}\alpha + \frac{13}{12}$, then $p>\frac{3}{2}$ since $\alpha>\frac{5}{4}$.
    Without loss of generality, we assume that $u(\mathbf{0})=1$. 

    Suppose $\vec{l}=(l_1,l_2,\cdots,l_d)$. Since $E$ is $(N,\Vec{l},\varepsilon^{-1},\varepsilon)$-graded, we can write $E=\bigcup_{i=0}^{d}E_{i}$ where $E_i$ is an $(N,l_i,\varepsilon)$-scattered set for $i>0$ and $E_0$ is a $\varepsilon^{-1}$-unitscattered set. We also write $E_{i}=\bigcup_{j \in \Z_+,1 \leq t \leq N}E_{i}^{(j,t)}$, where each $E_{i}^{(j,t)}$ is an open ball with radius $l_{i}$ and 
    \begin{equation}\label{eq:distqucF}
        \dist(E_{i}^{(j,t)},E_{i}^{(j',t)}) \geq l_{i}^{1+\varepsilon}
    \end{equation}
    whenever $j \neq j'$. 

    We assume without loss of generality that $l_{d} \leq 4n^{1-\frac{\varepsilon}{2}}$. Otherwise, since $E$ is $(1,\varepsilon)$-normal in $Q_{n}$, we can replace $E$ by $E_{0} \cup \left(\bigcup_{l_{i} \leq 4n^{1-\frac{\varepsilon}{2}} }  E_{i}\right)$.

    Let $n_{k}:=\left\lfloor l_{d-k} \right\rfloor$ for $k=0,1,\cdots,d$.
    \begin{cla}\label{cla:1/3scale}
    We can assume there is $M \in \Z_{+}$ such that $n^{\frac{1}{3}(1-4\varepsilon)}+1 \leq n_{M} \leq n^{\frac{1}{3}}$.
    \end{cla}
     \begin{proof}
          Suppose there is no such $M\in \Z_{+}$, we then add a level of empty set with scale length equal $n^{\frac{1}{3}(1-2\varepsilon)}$. More specifically, let $k$ be the largest nonnegative integer satisfying $l_{k}\leq n^{\frac{1}{3}(1-4\varepsilon)}$, then $l_{k+1} > n^{\frac{1}{3}}$. We let $l'_{i}=l_{i}$ and $E'_{i}=E_{i}$ for each $0 \leq i \leq k$.
          Let $l'_{k+1}=n^{\frac{1}{3}(1-2\varepsilon)}$ and $E'_{k+1}$
          be any $(N,l'_{k},\varepsilon)$-scattered
          set that is disjoint from $Q_n$.
          Let $l'_{i}=l_{i-1}$ and $E'_{i}=E_{i-1}$ for $i \geq k+2$.
          Then for each $1\leq i \leq d+1$, we have $(l_{i-1}')^{1+2\varepsilon} \leq l_i'$, and $E'_{i}$ is $(N,l'_{i},\varepsilon)$-scattered.
          Also, as $n > C_{\varepsilon, N}^4$ we still have $l_1' > C_{\varepsilon, N}$.
          Evidently, by replacing $E$ with $\bigcup_{i=0}^{d+1} E'_{i}$, our claim holds with $M=k+1$.
     \end{proof}
     Now we inductively construct subsets $\Theta_{k} \subset Q_n$ for $k=0,1,\cdots,M$, such that the following conditions hold. 
    \begin{enumerate}
        \item $|\Theta_{k}| \geq \left(\frac{\beta}{2}\right)^{2k+2} \left(\frac{n}{n_{k}}\right)^{\alpha}$.
        \item For any $a \in \Theta_{k}$, we have $|u(a)| \geq (K+11)^{-24(k+1) n}$.
        \item For any $a,a' \in \Theta_{k}$ with $a \neq a'$, we have $Q_{n_{k}}(a) \cap Q_{n_{k}}(a')=\emptyset$.
        \item For any $a \in \Theta_{k}$, we have $Q_{n_{k}}(a) \subset Q_n$.
        \item When $k>0$, for any $a \in \Theta_{k}$, there exists $a' \in \Theta_{k-1}$ such that $Q_{n_{k}}(a) \subset Q_{n_{k-1}}(a')$.
        \item For any $a \in \Theta_{k}$ and $d-k\leq i\leq d$, we have $E_{i} \cap Q_{n_{k}}(a) = \emptyset$.
    \end{enumerate}
    Let $n'_{0}:=\min \left\{\left\lfloor \frac{1}{4}n_0^{1+\varepsilon} \right\rfloor, n\right\}$. By using Theorem \ref{thm:low} for $m=n'_{0}$, we get a subset $\Theta'_{0} \subset Q_n$ such that $|\Theta'_{0}| \geq \beta \left(\frac{n}{n'_{0}}\right)^{\alpha}$ and $\Theta'_{0}$ satisfies Condition $1$ to $3$ in Theorem \ref{thm:low}. For each fixed $t \in \left\{1,2,\cdots,N\right\}$ and $j\neq j' \in \Z_+$, by definition we have $\dist (E_{d}^{(j,t)},E_{d}^{(j',t)}) \geq 4n'_{0}$. This implies
    \begin{equation}\label{eq:nonem-n_0}
        \left|\left\{(j,t):E_{d}^{(j,t)} \cap Q_{n'_{0}}(a) \neq \emptyset\right\}\right| \leq N,
    \end{equation}
    for each $a \in \Theta'_{0}$.
    For each $a \in \Theta'_{0}$, by using Theorem \ref{thm:low} for $Q_{n'_{0}}(a)$ and $m=n_{0}$, we get a subset $\Theta^{(a)}_{0} \subset Q_{n'_{0}}(a)$ such that $|\Theta^{(a)}_{0}| \geq \beta (\frac{n'_{0}}{n_{0}}) ^{\alpha}$ and $\Theta^{(a)}_{0}$ satisfies Condition $1$ to $3$ in Theorem \ref{thm:low}. 
    For each $j, t$ we have $\left|\left\{b \in \Theta^{(a)}_{0}:Q_{n_{0}}(b) \cap E_{d}^{(j, t)} \neq \emptyset \right\}\right| \leq 100$.
    This is because for each $b \in \Theta^{(a)}_{0}$ with $Q_{n_{0}}(b) \cap E_{d}^{(j, t)} \neq \emptyset$, the cube $Q_{n_{0}}(b)$ is contained in the closed ball of radius $2\sqrt{3}n_0+l_d<(2\sqrt{3}+1)n_0+1$ with the same center as $E_{d}^{(j, t)}$.
    As we have $Q_{n_{0}}(b) \cap Q_{n_{0}}(b')=\emptyset$ for $b \neq b' \in \Theta^{(a)}_{0}$, the number of such $b \in \Theta^{(a)}_{0}$ is at most $\frac{(2(2\sqrt{3}+1)n_0+2)^3}{(2n_0+1)^3}<100$.
    Thus by \eqref{eq:nonem-n_0}, we have
    \begin{equation}
        \left|\left\{b \in \Theta^{(a)}_{0}:Q_{n_{0}}(b) \cap E_{d} \neq \emptyset \right\}\right| \leq 100 N.
    \end{equation}
    Let $\tilde{\Theta}^{(a)}_{0}:=\Theta^{(a)}_{0} \setminus \left\{b \in \Theta^{(a)}_{0}:Q_{n_{0}}(b) \cap E_{d} \neq \emptyset \right\}$ for each $a \in \Theta'_{0}$,
    and $\Theta_{0}=\bigcup_{a \in \Theta'_{0}} \tilde{\Theta}^{(a)}_{0}$.
    Now we check the conditions.
    Condition 6 is from the definition, and 
    Condition 5 automatically holds since $k=0$.
    Condition 2 to 4 hold by the conditions in Theorem \ref{thm:low}.
    For Condition 1, recall that $l_{d} \geq l_{1} \geq C_{\varepsilon,N}$, and $l_d\leq 4n^{1-\frac{\varepsilon}{2}}$. By letting $C_{\varepsilon,N}$ large enough we have $n'_{0}> l_{d} ^{1+\frac{\varepsilon}{2}}$, and then $\frac{1}{2} \beta ( \frac{n'_{0}}{n_{0}}) ^{\alpha} > \frac{1}{2}\beta l_d^{\frac{1}{2} \alpha \varepsilon} \geq
    \frac{1}{2}\beta C_{\varepsilon,N}^{\frac{1}{2} \alpha \varepsilon} > 100N$.
    Thus for each $a \in \Theta'_{0}$ we have
    $|\tilde{\Theta}^{(a)}_{0}|\geq |\Theta^{(a)}_{0}|-100N \geq \frac{1}{2} \beta ( \frac{n'_{0}}{n_{0}}) ^{\alpha}$.
    This implies that
    \begin{equation}
        |\Theta_{0}| = \sum_{a \in \Theta'_{0}} |\tilde{\Theta}^{(a)}_{0}| \geq \left(\frac{1}{2} \beta \left(\frac{n'_{0}}{n_{0}}\right) ^{\alpha}\right)\left(\beta \left(\frac{n}{n'_{0}}\right) ^{\alpha}\right) > \left(\frac{\beta}{2}\right)^{2} \left(\frac{n}{n_{0}}\right)^{\alpha}.
    \end{equation}
    Suppose we have constructed $\Theta_{k}$, for some $0\leq k < M$, we proceed to construct $\Theta_{k+1}$. 
    Note that as $l_{d-k-1}^{1+2\varepsilon} \leq l_{d-k}$, we have $n_{k} \geq n_{k+1}^{1+2\varepsilon} -1$. Let $n'_{k+1}=\left\lfloor \frac{1}{4} n_{k+1}^{1+\varepsilon}\right\rfloor$. Take an arbitrary $a_{0} \in \Theta_{k}$, use Theorem \ref{thm:low} for $Q_{n_{k}}(a_0)$ with $m=n'_{k+1}$, we get a subset $\Theta'^{(a_0)}_{k+1} \subset Q_{n_{k}}(a_0)$ such that $|\Theta'^{(a_0)}_{k+1}| \geq \beta \left(\frac{n_k}{n'_{k+1}}\right)^{\alpha}$ and $\Theta'^{(a_0)}_{k+1}$ satisfies Condition $1$ to $3$ in Theorem \ref{thm:low}.
    For each fixed $t \in \left\{1,2,\cdots,N\right\}$ and $j \neq j' \in \Z_+$, by definition we have $\dist (E_{d-k-1}^{(j,t)},E_{d-k-1}^{(j',t)}) \geq 4n'_{k+1}$.
    This implies, for each $a \in \Theta'^{(a_0)}_{k+1}$,
    \begin{equation}\label{eq:nonem-n_k+1}
        \left|\left\{(j,t):E_{d-k-1}^{(j,t)} \cap Q_{n'_{k+1}}(a) \neq \emptyset\right\}\right| \leq N.
    \end{equation}
    
    For each $a \in \Theta'^{(a_0)}_{k+1}$, by using Theorem \ref{thm:low} for $Q_{n'_{k+1}}(a)$ and $m=n_{k+1}$, we get a subset $\Theta^{(a)}_{k+1} \subset Q_{n'_{k+1}}(a)$ such that $|\Theta^{(a)}_{k+1}| \geq \beta \left(\frac{n'_{k+1}}{n_{k+1}}\right)^{\alpha}$ and $\Theta^{(a)}_{k+1}$ satisfies Condition $1$ to $3$ in Theorem \ref{thm:low}. By \eqref{eq:nonem-n_k+1}, 
    \begin{equation}
        \left|\left\{b \in \Theta^{(a)}_{k+1}:Q_{n_{k+1}}(b) \cap E_{d-k-1} \neq \emptyset \right\}\right| \leq 100 N.
    \end{equation}
    
     Let $\tilde{\Theta}^{(a)}_{k+1}:=\Theta^{(a)}_{k+1} \setminus \left\{b \in \Theta^{(a)}_{k+1}:Q_{n_{k+1}}(b) \cap E_{d-k-1} \neq \emptyset \right\}$. Then $|\tilde{\Theta}^{(a)}_{k+1}|\geq |\Theta^{(a)}_{k+1}|-100N \geq \frac{1}{2} \beta \left(\frac{n'_{k+1}}{n_{k+1}}\right)^{\alpha}$,
     when $C_{\varepsilon, N}$ is large enough; and for each $b \in \tilde{\Theta}^{(a)}_{k+1}$, $Q_{n_{k+1}}(b) \cap E_{i} \neq \emptyset$ implies $i \leq d-k-2$. Then 
     \begin{equation}
         \left|\bigcup_{a \in \Theta'^{(a_0)}_{k+1}} \tilde{\Theta}^{(a)}_{k+1}\right|  = \sum_{a \in \Theta'^{(a_{0})}_{k+1}} |\tilde{\Theta}^{(a)}_{k+1}| \geq \left(\frac{\beta}{2}\right)^{2} \left(\frac{n_{k}}{n_{k+1}}\right)^{\alpha}.
     \end{equation}
     
     Now let $\Theta_{k+1}:=\bigcup_{a_{0} \in \Theta_{k}} \bigcup_{a \in \Theta'^{(a_0)}_{k+1}} \tilde{\Theta}^{(a)}_{k+1}$. Then Condition $2$ to $6$ hold for $k+1$ obviously. As for Condition $1$,
     \begin{equation}
         |\Theta_{k+1}| = \sum_{a_0 \in \Theta_{k}} \left|\bigcup_{a \in \Theta'^{(a_0)}_{k+1}} \tilde{\Theta}^{(a)}_{k+1}\right| \geq |\Theta_{k}| \left(\frac{\beta}{2}\right)^{2} \left(\frac{n_{k}}{n_{k+1}}\right)^{\alpha} \geq \left(\frac{\beta}{2}\right)^{2k+4} \left(\frac{n}{n_{k+1}}\right)^{\alpha},
     \end{equation}
     where the second inequality is true since Condition $1$ holds for $k$.
     
     Inductively, we have constructed $\Theta_{M}$ such that
     \begin{enumerate}
        \item $|\Theta_{M}| \geq \left(\frac{\beta}{2}\right)^{2M+2} \left(\frac{n}{n_{M}}\right)^{\alpha}$.
        \item For any $a \in \Theta_{M}$, we have $|u(a)| \geq (K+11)^{-24 (M+1) n}$.
        \item For any $a,a' \in \Theta_{M}$ with $a \neq a'$, we have $Q_{n_{M}}(a) \cap Q_{n_{M}}(a')=\emptyset$.
        \item For any $a \in \Theta_{M}$, we have $Q_{n_{M}}(a) \subset Q_n$.
        \item 
        For any $a \in \Theta_{M}$ and $d-M\leq i\leq d$, we have $E_{i} \cap Q_{n_{M}}(a) = \emptyset$.
    \end{enumerate}
    As $l_{d-k-1}^{1+2\varepsilon} \leq l_{d-k}$ for each $0 \leq k < M$, we have $n_{M} \leq l_d^{\left(\frac{1}{1+2\varepsilon}\right)^{M}}\leq n^{\left(\frac{1}{1+2\varepsilon}\right)^{M}}$.
    Note that $n_{M} > n^{\frac{1}{3}(1-4\varepsilon)}$, thus $\left(\frac{1}{1+2\varepsilon}\right)^{M} \geq \frac{1}{3}(1-4\varepsilon)$.
    From this we have
    \begin{equation}\label{eq:upperforM}
         M < 2 \varepsilon^{-1} .
    \end{equation}
    Since $l_{d-M-1}^{1+2\varepsilon} \leq l_{d-M}$ and $l_{d-M}\geq l_1 \geq C_{\varepsilon, N}$ we have
    $l_{d-M-1}< n_M^{1-\varepsilon}$ when $C_{\varepsilon, N}$ is large enough.
    Then for each $a \in \Theta_{M}$, by Condition $5$ we have that $E$ is $(1,2\varepsilon)$-normal in $Q_{n_{M}}(a)$.
    For any $a \in \Theta_{M}$, we apply Theorem \ref{thm:wqucF} to $Q_{n_{M}}(a)$, then
    \begin{equation}\label{eq:finallowerbound}
        \left|\left\{b\in Q_{n_M}(a): |u(b)| \geq (K+11)^{-24 (M+1) n} \exp(-C_{2}n^{3}_{M})\right\}\setminus E\right| \geq C_{3} \frac{n_{M}^{2}}{\log(n_{M})}.
    \end{equation}
    Let $C_{\varepsilon,K}=C_{2}+96 \log(K+11) \varepsilon^{-1}$.
    From \eqref{eq:finallowerbound}, \eqref{eq:upperforM} and $n^{\frac{1}{3}(1-4\varepsilon)}<n_{M}<n^{\frac{1}{3}}$, we have
    \begin{equation}
        \left|\left\{b\in Q_{n_M}(a): |u(b)| \geq \exp(-C_{\varepsilon,K}n )\right\}\setminus E\right| \geq C_{3}\frac{n_{M}^{2}}{\log(n_{M})}.
    \end{equation}
    Since $Q_{n_{M}}(a) \cap Q_{n_{M}}(a')=\emptyset$ when $a \neq a' \in \Theta_{M}$, in total we have
    \begin{multline}
        \left|\left\{b\in Q_{n}: |u(b)| \geq \exp(-C_{\varepsilon,K}n )\right\}\setminus E\right| \geq C_{3}\frac{n_{M}^{2}}{\log(n_{M})}
        |\Theta_{M}|
        \\
        \geq C_{3} \left(\frac{\beta}{2}\right)^{2M+2} n^{\frac{2}{3}(1-4\varepsilon)+\frac{2}{3}\alpha}(\log(n_{M}))^{-1} \geq n^{p},
    \end{multline}
    where the last inequality holds by taking $\varepsilon$ small enough, and then $C_{\varepsilon,N}$ large enough (recall that we require $n>C_{\varepsilon,N}^{4}$).
\end{proof}

\addcontentsline{toc}{section}{References}
\bibliographystyle{halpha}
\bibliography{bibliography}

\begin{appendices}
\section{Auxiliary lemmas for the framework}  \label{sec:aux}
In our general framework several results from \cite{ding2020localization} are used, and some of them are also used in Appendix \ref{app:proof-of-main} below as well.
For the convenience of readers we record them here.

There are a couple of results from linear algebra.
The first of them is an estimate on the number of almost orthonormal vectors, which appears in \cite{tao2019cheap} as well as \cite{ding2020localization}.
\begin{lemma}[\protect{\cite{tao2019cheap}\cite[Lemma 5.2]{ding2020localization}}]\label{lem:app-almost-orth}
Assume $v_{1},\cdots,v_{m} \in \R^{n}$ such that $|v_{i}\cdot v_{j}-\mathds{1}_{i=j}|\leq (5n)^{-\frac{1}{2}}$, then $m\leq \frac{5-\sqrt{5}}{2}n$.
\end{lemma}
The second one is about the variation of eigenvalues.
\begin{lemma}[\protect{\cite[Lemma 5.1]{ding2020localization}}]   \label{lem:vareigen}
Suppose the real symmetric $n\times n$ matrix $A$ has eigenvalues $\lambda_{1}\geq \cdots \geq \lambda_{n}\in \R$ with orthonormal eigenbasis $v_{1},\cdots,v_{n}\in \R^{n}$. If 
\begin{enumerate}
    \item $1\leq i\leq j\leq n$, $1\leq k\leq n$
    \item $0<r_{1}<r_{2}<r_{3}<r_{4}<r_{5}<1$
    \item $r_{1}\leq c\min\{r_{3}r_{5},r_{2}r_{3}/r_{4}\}$ where $c>0$ is a universal constant
    \item $0<\lambda_{j}\leq \lambda_{i}<r_{1}<r_{2}<\lambda_{i-1}$
    \item $v^{2}_{j,k}\geq r_{3}$
    \item $\sum_{r_{2}<\lambda_{\ell}<r_{5}} v_{\ell,k}^{2}\leq r_{4}$
\end{enumerate}
then the $i$-th largest eigenvalue $\lambda'_{i}$ (counting with multiplicity) of $A+e_{k} e_{k}^{\dag}$ is at least $r_{1}$, where $e_{k}$ is the $k$-th standard basis element and $e_{k}^{\dag}$ is its transpose. 
\end{lemma}
We then state the generalized Sperner's theorem, used in the proof of our 3D Wegner estimate (Lemma \ref{lem:WegnerF}).
\begin{theorem}[\protect{\cite[Theorem 4.2]{ding2020localization}}]  \label{thm:sperner}
Suppose $\rho \in (0,1]$, and $\mathcal{A}$ is a set of subsets of $\{1,\cdots,n\}$ satisfying the following. For every $A\in \mathcal{A}$, there is a set $B(A)\subset \{1,\cdots,n\}\setminus A$ such that $|B(A)|\geq \rho (n-|A|)$, and $A'\cap B(A)=\emptyset$ for any $A\subset A'\in \mathcal{A}$.
Then 
\begin{equation}
    |\mathcal{A}|\leq 2^{n}n^{-\frac{1}{2}}\rho^{-1}.
\end{equation}
\end{theorem}
For the next several results, in \cite{ding2020localization} they are stated and proved in the 2D lattice setting, but the proofs work, essentially verbatim, in the 3D setting.

The following covering lemma is used in the multi-scale analysis.
Recall that by ``dyadic'' we mean an integer power of $2$.
\begin{lemma}[\protect{\cite[Lemma 8.1]{ding2020localization}}]  \label{lem:scov}
There is a constant $C>1$ such that following holds.
Suppose $K\geq 1$ is an integer, $\alpha\geq C^{K}$ is a dyadic scale, $L_{0}\geq \alpha L_{1}\geq L_{1}\geq \alpha L_{2}\geq L_{2}$ are dyadic scales, $Q\subset \Z^{3}$ is an $L_{0}$-cube, and $Q''_{1},\cdots Q''_{K}\subset Q$ are $L_{2}$-cubes. Then there is a dyadic scale $L_{3}\in [L_{1},\alpha L_{1}]$ and disjoint $L_{3}$-cubes $Q'_{1},\cdots,Q'_{K}\subset Q$, such that
for each $Q''_{k}$ there is $Q'_{j}$ with $Q''_{k}\subset Q'_{j}$ and $\dist(Q''_{k},Q\setminus Q'_{j})\geq \frac{1}{8}L_{3}$.
\end{lemma}
We need the following continuity of resolvent estimate.
It is stated in a slightly different way from \cite[Lemma 6.4]{ding2020localization}, so we add a proof here.
\begin{lemma}
[\protect{\cite[Lemma 6.4]{ding2020localization}}]
\label{lem:app-pertu}
If for $\lambda\in \R$, $\alpha>\beta>0$, and a cube $Q\subset \Z^{3}$, we have
\begin{equation}
    |(H_{Q}-\lambda)^{-1}(a,b)|\leq \exp(\alpha-\beta|a-b|) \text{ for $a,b\in Q$},
\end{equation}
then for $\lambda'$ with $|\lambda'-\lambda|\leq \frac{1}{2} |Q|^{-1} \exp(-\alpha)$, we have
\begin{equation}   \label{eq:app-pertu}
    |(H_{Q}-\lambda')^{-1}(a,b)|\leq 2\exp(\alpha-\beta|a-b|) \text{ for $a,b\in Q$}.
\end{equation}
\end{lemma}
\begin{proof}
     We first prove \eqref{eq:app-pertu} assuming $\lambda'$ is not an eigenvalue of $H_Q$. 
     By resolvent identity we have,
     \begin{equation}
         (H_{Q}-\lambda')^{-1}=(H_{Q}-\lambda)^{-1} + (H_{Q}-\lambda')^{-1}(\lambda'-\lambda)(H_{Q}-\lambda)^{-1}.
     \end{equation}
     Let $\gamma=\max_{a,b\in Q} \exp(\beta |a-b|-\alpha) |(H_{Q}-\lambda')^{-1}(a,b)|$. 
     Then for any $a,b\in Q$,
     \begin{align}
     \begin{split}                     &|(H_{Q}-\lambda')^{-1}(a,b)|\\ \leq &|(H_{Q}-\lambda)^{-1}(a,b)|+|\lambda'-\lambda|\sum_{c\in Q}|(H_{Q}-\lambda')^{-1}(a,c)||(H_{Q}-\lambda)^{-1}(c,b)|\\ \leq
         &\exp(\alpha-\beta|a-b|)+|\lambda'-\lambda|\sum_{c \in Q} \exp(\alpha-\beta|a-c|) \exp(\alpha-\beta|c-b|) \gamma\\ \leq
         &\exp(\alpha-\beta|a-b|)+|\lambda'-\lambda||Q|\exp(2\alpha-\beta|a-b|) \gamma\\ \leq
         &\exp(\alpha-\beta|a-b|)+\frac{1}{2}\exp(\alpha-\beta|a-b|) \gamma.
     \end{split}
     \end{align}
     This implies $\gamma \leq 1+\frac{1}{2}\gamma$ and thus $\gamma\leq 2$ and \eqref{eq:app-pertu} follows.
    
    Now we can deduce that $|\det (H_{Q}-\lambda')^{-1}|$ is uniformly bounded for $\lambda'$ that is not an eigenvalue of $H_Q$ and satisfies $|\lambda'-\lambda|\leq \frac{1}{2} |Q|^{-1} \exp(-\alpha)$.
    By continuity of the determinant (as a function of $\lambda'$), we conclude that $H_{Q}$ has no eigenvalue in $\left[\lambda-\frac{1}{2} |Q|^{-1} \exp(-\alpha),\lambda+\frac{1}{2} |Q|^{-1} \exp(-\alpha)\right]$. Thus our conclusion follows.
\end{proof}
We also need the following result to deduce exponential decay of the resolvent in a cube from the decay of the resolvent in subcubes. 
\begin{lemma}[\protect{\cite[Lemma 6.2]{ding2020localization}}]  \label{lem:propdr}
Suppose
\begin{enumerate}
    \item $\varepsilon>\delta>0$ are small,
    \item $K\geq 1$ is an integer and $\lambda\in [0,13]$,
    \item $L_{0}\geq \cdots\geq L_{6}$ are large enough (depending on $\varepsilon,\delta,K$) with $L_{k}^{1-\varepsilon}\geq L_{k+1}$ ,
    \item $1\geq m\geq 2L_{5}^{-\delta}$ represents the exponential decay rate,
    \item $Q\subset \Z^{3}$ is an $L_{0}$-cube,
    \item $Q'_{1},\cdots,Q'_{K}\subset Q$ are disjoint $L_{2}$-cubes with $\|(H_{Q'_{k}}-\lambda)^{-1}\|\leq \exp(L_{4})$,
    \item for all $a\in Q$, one of the following holds
    \begin{itemize}
        \item there is $Q'_{k}$ with $a\in Q'_{k}$ and $\dist(a,Q\setminus Q'_{k})\geq \frac{1}{8}L_{2}$
        \item there is an $L_{5}$-cube $Q''\subset Q$ such that $a\in Q''$, $\dist(a,Q\setminus Q'')\geq \frac{1}{8}L_{5}$, and $|(H_{Q''}-\lambda)^{-1}(b,b')|\leq \exp(L_{6}-m|b-b'|)$ for $b,b'\in Q''$.
    \end{itemize}
Then $|(H_{Q}-\lambda)^{-1}(a,a')|\leq \exp(L_{1}-\tilde{m}|a-a'|)$ for $a,a'\in Q$ where $\tilde{m}=m-L_{5}^{-\delta}$.
\end{enumerate}
\end{lemma}

\section{The principal eigenvalue }\label{sec:app}
This appendix sets up the base case in the induction proof of Theorem \ref{thm:multiscale}.
We follow \cite[Section 7]{ding2020localization}, and generalize their result to higher dimensions.
We take $d \in \Z$, $d>2$, and denote $Q_{n}:=\left\{a \in \Z^d:\|a\|_{\infty} \leq n\right\}$ instead. 
\begin{theorem}\label{thm:principal}
  Let $\overline{V}:Q_{n}\rightarrow \left[0,1\right]$ be any potential function, and $R>0$ large enough, such that for any $a \in Q_n$, there exists $b\in Q_n$ with $\overline{V}(b)=1$ and $|a-b|<R$.
  Let $\overline{H}:\ell^2(Q_n) \rightarrow \ell^2(Q_n)$, $\overline{H}=-\Delta+\overline{V}$, with Dirichlet boundary condition.
  Then its principal eigenvalue is no less than $CR^{-d}$, where $C$ is a constant depending only on $d$.
\end{theorem}
\begin{proof}
Let $\lambda_0$ denote the principal eigenvalue, then by e.g. \cite[Exercise 6.14]{evans2010partial} we have
    \begin{equation}\label{eq:appendix-min-max}
        \lambda_{0}=\sup_{u:Q_{n} \rightarrow \R_{+}} \min_{Q_{n}} \frac{\overline{H}u}{u}.
    \end{equation}
    Hence we lower bound $\lambda_0$ by constructing a function $u$.
    Let $\tilde{G}:\Z^d \rightarrow \R$ be the lattice Green's function; i.e. for any $a \in \Z^d$, $\tilde{G}(a)$ is the expected number of times that a (discrete time) simple random walk starting at $\mathbf{0}$ gets to $a$.
    Let $G:=\tilde{G}/2d$.
    Then $G$ is the only function such that $-\Delta G= \delta_{\mathbf{0}}$ (where $\delta_{\mathbf{0}}(\mathbf{0})=1$ and $\delta_{\mathbf{0}}(a)=0$ for $a \neq \mathbf{0}$), and
    $0 \leq G(a) \leq G(\mathbf{0})$ for any $a\in \Z^d$.
    In addition, for any $a \in \Z^d$ with $a \neq \mathbf{0}$, by e.g. \cite[Theorem 4.3.1]{lawler2010random} we have
    \begin{equation}\label{eq:grennfun}
        G(a)=\frac{C_{d}}{|a|^{d-2}}+O\left(\frac{1}{|a|^{d}}\right),
    \end{equation}
    where $C_d$ is a constant depending only on $d$. Hence 
    \begin{equation}
        \frac{4 C_{d}}{5 |a|^{d-2}} \leq G(a) \leq \frac{3 C_{d}}{2 |a|^{d-2}}
    \end{equation}
    when $|a|$ is large enough.
    
    We define $u:\Z^d \rightarrow \R_+$ as
    \begin{equation}
        u(a):=1 + G(\mathbf{0}) - G(a) -\varepsilon_{d} R^{-d} |a|^2, \; \forall a \in \Z^d,
    \end{equation}
    where $\varepsilon_{d} > 0$ is a small enough constant depending on $d$. Then 
    \begin{equation}  \label{eq:ap03}
        -\Delta u =-\delta_{\mathbf{0}}+2d\varepsilon_{d} R^{-d},
    \end{equation}
    and for any $a \in \Z^d$ with $|a|<3R$, we have $0 < u(a) \leq 1+G(\mathbf{0}) $.
    
    Assume that $R$ is large enough.
    For any $a$ with $2R<|a|<3R$, we have $u(a) \geq 1+G(\mathbf{0})-\frac{3C_{d}}{2 (2R)^{d-2}}-9 \varepsilon_{d} R^{-d+2}$; and for any $a$ with $|a|<R$, $u(a) \leq 1+G(\mathbf{0})-\frac{4 C_{d}}{5 R^{d-2}} \leq 1+G(\mathbf{0})-\frac{3 C_{d}}{2 (2R)^{d-2}}-9\varepsilon_{d} R^{-d+2} $, as long as $\varepsilon_{d}<\frac{C_{d}}{180}$ (also note that here we have $d>2$). Thus 
    \begin{equation}  \label{eq:ap05}
        \min_{2R<|a|<3R} u(a) \geq \max_{|a|<R} u(a)
    \end{equation}
    
    Now we define $u_0 : Q_n \rightarrow \R_+$, as $u_{0}(a):=\min_{|a-b|<3R,\overline{V}(b)=1} u(a-b),\; \forall a \in Q_n$. Pick an arbitrary $a' \in Q_n$, by \eqref{eq:ap05} there is $b'$ with $|a'-b'| \leq 2R$ such that $u_{0}(a')=u(a'-b')$ and $\overline{V}(b')=1$. For any $a'' \in Q_{n}$ with $|a''-a'|=1$, since $|a''-b'| \leq 2R+1 <3R$, we have
    \begin{equation}
        u_{0}(a'')=\min_{|a''-b|<3R,\overline{V}(b)=1} u(a''-b) \leq u(a''-b').
    \end{equation}
    Thus by \eqref{eq:ap03}, and Dirichlet boundary condition,
    \begin{align}  \label{eq:ap1}
    \begin{split}
                \overline{H} u_{0} (a')= &2d u_{0}(a')-\sum_{a''\in Q_{n}, |a'-a''|=1} u_{0}(a'') +\overline{V}(a')u_{0}(a')\\
                \geq &2d u(a'-b')-\sum_{a''\in Q_{n}, |a'-a''|=1} u(a''-b') +\overline{V}(a')u(a'-b')\\
        \geq &-\Delta u(a'-b') +\overline{V}(a')u(a'-b')\\
        = &-\delta_{\mathbf{0}}(a'-b')+  2d\varepsilon_{d}R^{-d} +\overline{V}(a')u(a'-b') \\
        \geq &2d\varepsilon_{d}R^{-d}.
    \end{split}
    \end{align}
    Since $a'$ is arbitrary and $0<u_{0}(a')\leq 1+G(\mathbf{0})$, by \eqref{eq:appendix-min-max} and letting $C=\frac{2d \varepsilon_{d}}{1+G(\mathbf{0})}$, we have
    $\lambda_{0} \geq C R^{-d}$.
\end{proof}

\begin{rem}
The exponent in $R^{-d}$ is optimal. Consider a potential $\overline{V}$ such that $\overline{V}(a)=1$ only if $a \in \lceil R \rceil \Z^{d} \cap Q_{n}$ and $\overline{V}(a)=0$ otherwise. In this case we have that $\lambda_{0} \leq 8d R^{-d}+4d n^{-1}$. To see this, consider the test function $\phi(a)=1-\overline{V}(a)$ for $a \in Q_{n}$ and use the variational principle $\lambda_{0} \leq \frac{\langle \phi , \overline{H} \phi \rangle}{\|\phi\|_{2}^{2}}$.  
\end{rem}

\begin{cor}\label{cor:lifshiz}
Let $\overline{H}$, $C$ be defined as in Theorem \ref{thm:principal}.
Let $0\leq \lambda<\frac{CR^{-d}}{2} $. Then $\|(\overline{H}-\lambda)^{-1}\| \leq \frac{2R^{d}}{C}$ and
\begin{equation}
    |(\overline{H}-\lambda)^{-1}(a,b)| \leq \frac{2R^{d}}{C} \exp\left(-\frac{C R^{-d}}{8d+2} |a-b|\right)
\end{equation}
for any $a,b \in Q_{n}$.
\end{cor}
\begin{proof}
    As the principal eigenvalue of $\overline{H}$ is no less than $C R^{-d}$, we have $\|(\overline{H}-\lambda)^{-1}\| \leq \frac{2R^{d}}{C} $.
    Let $T:= I - \frac{1}{4 d +1} (\overline{H}-\lambda) $. Since any eigenvalue of $\overline{H}$ is in $\left[C R^{-d},4d+1\right]$, the eigenvalues of $T$ are in $\left[0, 1 - \frac{C}{8d+2} R^{-d}\right]$, so $\|T\| \leq 1 - \frac{C}{8d+2} R^{-d}$.
    
    Note that for each $i>0$ and $a,b \in Q_{n}$, $T^{i}(a,b)=0$ if $|a-b|>i$. Then we have
    \begin{multline}
    |(\overline{H}-\lambda)^{-1}(a,b)|=(4d+1)^{-1}|(I-T)^{-1}(a,b)|\leq (4d+1)^{-1} \sum_{i \geq 0}|T^{i}(a,b)|\\
    = (4d+1)^{-1} \sum_{i \geq |a-b|}|T^{i}(a,b)|\leq (4d+1)^{-1} \sum_{i \geq |a-b|} \|T\|^{i}\leq  \frac{2R^{d}}{C} \exp\left(-\frac{C R^{-d}}{8d+2} |a-b|\right),    
    \end{multline}
    so the corollary follows.
\end{proof}
Finally, we have the following result, which implies the base case in the induction proof of Theorem \ref{thm:multiscale}.
\begin{prop}\label{prop:basecase}
Let $d=3$, and $V$ be the Bernoulli potential, i.e. $\mathbb{P}(V(a)=0)=\mathbb{P}(V(a)=1)=\frac{1}{2}$ for each $a \in \Z^3$ independently. 
For any $0<\delta<\frac{1}{10}$ and $\varepsilon>0$, there exists $C_{\delta,\varepsilon}$ such that for any $n>C_{\delta,\varepsilon}$ and $0\leq \lambda< \frac{Cn^{-\frac{3 \delta}{10}}}{2}$, with probability at least $1-n^{-1}$ the following is true.

Take any $V':\Z^3\to [0,1]$ such that $V'_{Q_{n} \cap \lceil \varepsilon^{-1}\rceil\Z^{3}} = V_{Q_{n} \cap \lceil \varepsilon^{-1}\rceil\Z^{3}}$.
Let $H'_{Q_n}$ be the restriction of $-\Delta+V'$ on $Q_n$ with Dirichlet boundary condition.
Then we have
\begin{equation}\label{eq:wegn2}
    \|(H'_{Q_{n}}-\lambda)^{-1}\| \leq \exp(n^{2\delta}),
\end{equation}
and
\begin{equation}  \label{eq:ap2}
    \text{$|(H'_{Q_{n}}-\lambda)^{-1}(a,b)| \leq n^{2\delta} \exp(-n^{-\delta}|a-b|)$  for any $a,b \in Q_{n}$}.
\end{equation}
\end{prop}
\begin{proof}
     Let $R:=n^{\frac{\delta}{10}}$,
     and let $A$ denote the following event:
    \begin{equation}
        \forall a \in Q_{n}, \exists b \in Q_{n}\cap \lceil \varepsilon^{-1}\rceil\Z^{3}, \; \text{s.t.} \; |a-b| \leq R, V(b)=1.
    \end{equation}
    Then $A$ only depends on $V_{Q_{n} \cap \lceil \varepsilon^{-1}\rceil\Z^{3}}$.
    
    Using Corollary \ref{cor:lifshiz} with $d=3$, we have that \eqref{eq:wegn2} and \eqref{eq:ap2} hold under the event $A$,  when $n$ is large enough.
    
    Finally, since there are $(2n+1)^{3}$ points in $Q_{n}$, and inside each ball of radius $R$, there are at least $\frac{1}{8}n^{\frac{3\delta}{10}} \varepsilon^{3}$ points in $\lceil \varepsilon^{-1}\rceil\Z^{3} \cap Q_{n}$,
    we have $\mathbb{P}(A^{c}) \leq (2n+1)^{3} 2^{- \frac{1}{8} n^{\frac{3\delta}{10}} \varepsilon^{3}} \leq n^{-1}$, when $n$ is large enough.
\end{proof}

\section{Deducing Anderson localization from the resolvent estimate}\label{app:proof-of-main}
The arguments in this appendix originally come from \cite[Section 7]{bourgain2005localization} (see also \cite[Section 6, 7]{germinet2012comprehensive} and \cite[Section 6]{bourgain2005anderson}).
These previous works are about the continuous space model.
For completeness and for the reader's convenience, we adapt the arguments for the lattice model, thus deducing Theorem \ref{thm:main} from Theorem \ref{thm:resonentexpo}.

As in Section \ref{sec:fra}, in this appendix, by ``dyadic'' we mean an integer power of $2$, and by ``dyadic cube'', we mean a cube $Q_{2^{n}}(a)$ for some $a \in 2^{n-1}\Z^3$ and $n \in \Z_{+}$.

For any $k \in \Z_{+}$, we define
\begin{equation}
\Omega_{k}:=\{u:\Z^{3}\rightarrow \R: |u(a)|\leq k (1+|a|)^{k},\;\; \forall a\in \Z^{3},\;\; \text{and}\;\; u(\mathbf{0})=1\}.    
\end{equation}
Since the law of $H$ is invariant under translation, 
to prove Theorem \ref{thm:main}, it suffices to show that for any $k\in \Z_{+}$, almost surely \begin{equation}\label{eq:apb0}
\inf_{t>0} \sup_{a \in \Z^3} \exp(t|a|) |u(a)|<\infty,  
\end{equation}
for any $u\in \Omega_{k}$ and $\lambda \in [0, \lambda_*]$ with $H u=\lambda u$.

Denote $\mathcal{I}=(0,\lambda_{*})$.
We first see that it suffices to prove \eqref{eq:apb0} for any $u\in \Omega_{k}$ and $\lambda \in \mathcal{I}$ with $H u=\lambda u$, by applying the following lemma to $\lambda=0$ and $\lambda=\lambda_{*}$.
\begin{lemma}\label{cla:app-not-eigenvalue}
Suppose $\lambda\in [0,\lambda_{*}]$ and $k\in\Z_{+}$. Then almost surely, there is no $u\in\Omega_k$ with $H u =\lambda u$.
\end{lemma}
\begin{proof}
     Let $L_{i}=2^{i}$ for $i\in \Z_{+}$. 
     By Theorem \ref{thm:resonentexpo} and the Borel-Cantelli lemma, almost surely, there exists $i'>0$, such that for any $i>i'$,
     \begin{equation}
         \left|(H_{Q_{L_{i}}}-\lambda)^{-1}(a,b)\right| \leq \exp\left(L_{i}^{1-\lambda_{*}}-\lambda_{*} |a-b|\right), \;\forall a,b \in Q_{L_{i}}.
     \end{equation}
     Assume there exists $u\in\Omega_k$ with $H u=\lambda u$.
     For each large enough $i$ we have
     \begin{equation}
         |u(\mathbf{0})|=\left|\sum_{\substack{a\in Q_{L_{i}},a'\in \Z^{3}\setminus Q_{L_{i}}\\|a-a'|=1}} (H_{Q_{L_{i}}}-\lambda)^{-1}(\boo,a) u(a')\right|\leq
         6\cdot (2L_i +1)^2\exp\left(-\frac{\lambda_{*}L_{i}}{2}\right)k(1+\sqrt{3}L_{i})^{k}
     \end{equation}
     which converges to zero as $i\rightarrow \infty$. Thus $u(\boo)=0$, which contradicts with the fact that $u\in \Omega_{k}$.
\end{proof}

Let us fix $k\in \Z_{+}$ and denote by $\sigma_k(H)$ the set of all $\lambda \in \mathcal{I}$, such that $Hu=\lambda u$ for some $u \in \Omega_k$.
For each $L\in \Z_{+}$, denote by $\sigma(H_{Q_{L}})$ the set of eigenvalues of $H_{Q_{L}}$. The first key step is to prove that for any large enough $L$, with high probability, the distance between any $\lambda\in \sigma_k(H)$ and $\sigma(H_{Q_{L}})$ is small, exponentially in $L$.

\begin{prop}\label{prop:first-reduction}
There exist $\kappa',c_1 >0$ such that for any dyadic $L$ large enough, 
we can find a $V_{Q_{L}}$-measurable event  $\mathcal{E}_{wloc}^{(L)}$, such that
\begin{equation}
\prob\left[\mathcal{E}_{wloc}^{(L)}\right] \geq 1-L^{-\kappa'},
\end{equation}
and under the event
$\mathcal{E}_{wloc}^{(L)}$, we have $\dist (\lambda, \sigma(H_{Q_{L}})\cap \mathcal{I}) \leq \exp(-c_1 L)$ for any $$\lambda \in \sigma_k(H)\cap \left[\exp(-c_1\sqrt{L}),\lambda_{*}-\exp(-c_1\sqrt{L})\right].$$
\end{prop}
The next key step is to strengthen Proposition \ref{prop:first-reduction} so that each $\lambda\in\sigma_k(H)$ is not only exponentially close to $\sigma(H_{Q_{L}})$, but also exponentially close to a finite subset $S\subset \sigma(H_{Q_{L}})$ with $|S|<L^{\delta'}$ for arbitrarily small $\delta'$.

\begin{prop}\label{prop:second-reduction}
For any $\delta'>0$, there exist $\kappa'',c_2 >0$ such that for each dyadic $L$ large enough (depending on $\delta'$), 
we can find a $V_{Q_{L}}$-measurable event  $\mathcal{E}_{sloc}^{(L)}$ with
\begin{equation}
\prob\left[\mathcal{E}_{sloc}^{(L)}\right] \geq 1-L^{-\kappa''},
\end{equation}
and under the event $\mathcal{E}_{sloc}^{(L)}$, there exists a finite set $S\subset \sigma(H_{Q_L})\cap \mathcal{I}$ with $|S|<L^{\delta'}$ such that $\dist (\lambda, S) \leq \exp(-c_2 L)$ for any $\lambda \in \sigma_k(H) \cap \left[\exp(-L^{c_2}),\lambda_{*}-\exp(-L^{c_2}) \right]$.
\end{prop}
Proposition \ref{prop:first-reduction} and \ref{prop:second-reduction} are discrete versions of \cite[Lemma 6.1]{bourgain2005anderson} and \cite[Lemma 6.4]{bourgain2005anderson} respectively. See also \cite[Proposition 6.3, 6.9]{germinet2012comprehensive}.
Now we leave the proofs of these two propositions to the next two subsections,
and prove localization assuming them.

\begin{proof}[Proof of Theorem \ref{thm:main}]
We apply Proposition \ref{prop:second-reduction} with $\delta'<\kappa_{0}$ where $\kappa_{0}$ is the constant in Theorem \ref{thm:resonentexpo}. Take large enough dyadic $L$, and consider the annulus $A_{L}=Q_{5L}\setminus Q_{2L}$.
We cover $A_L$ by $2L$-cubes $\{Q^{(j)}:1\leq j\leq 1000\}$ that are disjoint with $Q_{L}$, such that for each $a\in A_{L}$ there is $1\leq j\leq 1000$ with $a\in Q^{(j)}$ and $\dist(a,\Z^{3}\setminus Q^{(j)})\geq \frac{1}{8}L$. Apply Theorem \ref{thm:resonentexpo} to each of $Q^{(j)}$'s and to each energy $\lambda\in S\subset \sigma(H_{Q_L})\cap \mathcal{I}$, we have 
\begin{equation}
\prob\left[\mathcal{E}^{(L)}_{ann}\big|\; \mathcal{E}_{sloc}^{(L)}\right]\geq 1-1000L^{\delta'-\kappa_{0}}    
\end{equation}
where $\mathcal{E}^{(L)}_{ann}$ denotes the event:
\begin{equation}
\left|(H_{Q^{(j)}}-\lambda)^{-1}(a,b)\right| \leq \exp\left(L^{1-\lambda_{*}}-\lambda_{*} |a-b|\right), \;\forall 1\leq j\leq 1000,\;\forall a,b \in Q^{(j)},\;\text{and }\forall \lambda\in S.
\end{equation}
Then by Proposition \ref{prop:second-reduction} we have
\begin{equation} \label{eq:apb05}
\prob\left[\mathcal{E}^{(L)}_{ann}\cap \mathcal{E}_{sloc}^{(L)}\right]\geq (1-L^{-\kappa''})(1-1000L^{\delta'-\kappa_{0}})\geq 1-L^{-\kappa'''},
\end{equation}
for some constant $\kappa'''>0$ and large enough $L$.

Under the event $\mathcal{E}^{(L)}_{ann}\cap \mathcal{E}_{sloc}^{(L)}$, we take any $u\in \Omega_{k}$ with $H u=\lambda u$ and $\lambda\in [\exp(-L^{c_2}),\lambda_{*}-\exp(-L^{c_2})]$, and $\lambda'\in S$ with $|\lambda-\lambda'|<\exp(-c_2 L)$. Thus using Lemma \ref{lem:app-pertu}, we have
\begin{equation}  \label{eq:apb1}
\begin{split}
\|u\|_{\ell^{\infty}(A_{L})}&\leq 2\exp\left(L^{1-\lambda_{*}}-\frac{1}{8}\lambda_{*} L\right) \|u\|_{\ell^{1}(Q_{6L})}
\\
&\leq 
2\exp\left(L^{1-\lambda_{*}}-\frac{1}{8}\lambda_{*} L\right)
k(6\sqrt{3}L+1)^{k} (12L+1)^{3} \leq \exp(-c' L)
\end{split}
\end{equation}
for some constant $c'<\frac{\lambda_{*}}{8}$ and large enough $L$. 

Now we consider the event
\begin{equation}
\mathcal{E}_{loc}=\bigcup_{i'\geq 0} \bigcap_{i\geq i'} (\mathcal{E}^{(2^{i})}_{ann}\cap \mathcal{E}_{sloc}^{(2^{i})}).  
\end{equation}
We have $\prob[\mathcal{E}_{loc}]=1$ by \eqref{eq:apb05}.
Note that for any $\lambda\in \mathcal{I}$, we have $\lambda\in [\exp(-L^{c_2 }),\lambda_{*}-\exp(-L^{c_2})]$ for large enough $L$.
We also have that $\bigcup_{i\geq i'}A_{2^{i}}=\Z^{3}\setminus Q_{2^{i'+1}}$ for any $i'\in \Z_{+}$. By \eqref{eq:apb1} we have that \eqref{eq:apb0} holds under the event $\mathcal{E}_{loc}$.
Then localization is proved.
\end{proof}

\subsection{The first spectral reduction}

For simplicity of notations, for any $\lambda\in \R$, dyadic scale $L$, and $a\in\Z^3$, we say $Q_{L}(a)$ is \emph{$\lambda$-good} if
\begin{equation}
    \left|(H_{Q_L(a)}-\lambda)^{-1}(b,b')\right| \leq \exp\left(L^{1-\lambda_{*}}-\lambda_{*} |b-b'|\right), \;\forall b,b' \in Q_{L}(a).
\end{equation}
Otherwise, we call it \emph{$\lambda$-bad}. 
By Theorem \ref{thm:resonentexpo}, for any large enough dyadic scale $L$ and $\lambda\in [0, \lambda_{*}]$, we have
\begin{equation}\label{eq:app-prob-control}
    \prob[\text{$Q_{L}(a)$ is $\lambda$-bad}]\leq L^{-\kappa_{0}}.
\end{equation}

\begin{proof}[Proof of Proposition \ref{prop:first-reduction}]
Throughout the proof, we use $C$ to denote large universal constants.
For a dyadic scale $L$, we construct a graph $G_{L}$ whose vertices are all the dyadic $2L$-cubes. The edges are given as follows: for any $a\not=a'\in \frac{L}{2}\Z^{3}$, there is an edge connecting $Q_{L}(a)$ and $Q_{L}(a')$ if and only if $Q_{L}(a)\cap Q_{L}(a')\not=\emptyset$.  

Fix large dyadic scale $L$. Take the dyadic scale $L_{0}\in \left\{\sqrt{L},\sqrt{2L}\right\}$.
For any $\lambda\in \mathcal{I}$, denote by $\mathcal{E}^{\lambda}_{per}$ the event that there is a path of $\lambda$-bad $2L_0$-cubes $\overline{Q}_{1},\cdots,\overline{Q}_{m}$ in $G_{L_{0}}$ such that 
\begin{equation}\label{eq:app-cond-path}
    \text{$\overline{Q}_{1}\cap Q_{\frac{L}{2}}\not=\emptyset$ and $\overline{Q}_{m}\cap Q_{L}=\emptyset$}.
\end{equation}
Under the event $\mathcal{E}^{\lambda}_{per}$, suppose that
$\Gamma_{0}=(\overline{Q}_{1},\cdots,\overline{Q}_{m})$ is such a path with the shortest length.
Since $\dist(Q_{\frac{L}{2}},\Z^{3}\setminus Q_{L})\geq \frac{L}{2}$, we have $m\geq \frac{L}{4\sqrt{3}L_{0}}$.
By definition of dyadic cubes and that $\Gamma_{0}$ has the shortest length, there are at least $\frac{m}{1000}$ disjoint $\lambda$-bad cubes in $\Gamma_{0}$. Hence,
\begin{equation}\label{eq:app-prob-per}
    \prob[\mathcal{E}^{\lambda}_{per}] \leq \sum_{m\geq \frac{L}{4\sqrt{3}L_{0}}} C L^{3} 1000^{m}(L_{0}^{-\kappa_{0}})^{\frac{m}{1000}}\leq 2C L^{3} (1000 L_{0}^{-\frac{\kappa_{0}}{1000}})^{\frac{L}{4\sqrt{3}L_{0}}}\leq L_{0}^{-c'L_{0}}
\end{equation}
for some $c'>0$.
Here the first inequality is by \eqref{eq:app-prob-control}, and counting the total number of $G_{L_0}$ paths with length $m$ and one end intersecting $Q_{\frac{L}{2}}$. 
\begin{cla}\label{cla:app-eigen-0}
Under the event $(\mathcal{E}^{\lambda}_{per})^{c}$, any $\lambda'\in \sigma_{k}(H)$ with $|\lambda'-\lambda|\leq \exp(- L^{1-\frac{\lambda_{*}}{2}}_{0})$ satisfies $\dist(\lambda',\sigma(H_{Q_{\frac{3}{2}L}}))\leq \exp(-\epsilon' L_{0})$ for a universal constant $\epsilon'>0$.
\end{cla}
\begin{proof}
Denote the set of all the $\lambda$-bad $L_{0}$-cubes contained in $Q_{\frac{3}{2}L}$ by $\mathcal{S}$. We consider $\Z^{3}$ as a graph with edges between nearest neighbors.
Consider the set $S_{0}:=(\bigcup \mathcal{S})\cup Q_{\frac{L}{2}}\subset Q_{\frac{3}{2}L}$. Let $S_{1}$ be the maximal connected component of $S_{0}$ which contains $Q_{\frac{L}{2}}$. Then $(\mathcal{E}^{\lambda}_{per})^{c}$ 
implies $S_{1}\subset Q_{L+2L_{0}}$. Denote
\begin{equation}
    \partial^{-}S_{1}=\{a\in S_{1}:\text{$|a-a'|=1$ for some $a'\in \Z^{3}\setminus S_{1}$}\},
\end{equation}
and
\begin{equation}
    \partial^{+}S_{1}=\{a\in \Z^{3}\setminus S_{1}:\text{$|a-a'|=1$ for some $a'\in S_{1}$}\}.
\end{equation}
Assume $\lambda'$ satisfies the hypothesis in the claim, then there is $u\in \Omega_{k}$ such that $H u= \lambda' u$. For any $a'\in \partial^{-}S_{1}\cup \partial^{+}S_{1}$, there is a dyadic $L_{0}$-cube $Q'$ such that $a'\in Q'$ and $\dist(a',\Z^{3}\setminus Q')\geq \frac{1}{8}L_{0}$. By maximality of $S_{1}$, we have $Q'$ is $\lambda$-good. Thus by Lemma \ref{lem:app-pertu},
\begin{align}
\begin{split}\label{eq:app-smallness-bound}
        |u(a')|
        \leq &2\exp(L_{0}^{1-\lambda_{*}}-\frac{1}{8}\lambda_{*} L_{0})\|u\|_{\ell^{1}(Q_{L+4L_{0}})}\\
    \leq &2\exp(L_{0}^{1-\lambda_{*}}-\frac{1}{8}\lambda_{*} L_{0})(2L+8L_{0}+1)^{3}k(\sqrt{3}L+4\sqrt{3}L_{0}+1)^{k}\\
        \leq &\exp(-\frac{1}{10}\lambda_{*} L_{0})
\end{split}
\end{align}
for large enough $L_{0}$.
Let $u_{*}:Q_{\frac{3}{2}L}\rightarrow \R$ be defined by $u_{*}=u$ on $S_{1}$ and $u_{*}=0$ on $Q_{\frac{3}{2}L}\setminus S_{1}$. Then
\begin{equation}
(H_{Q_{\frac{3}{2}L}}-\lambda') u_{*}(a)=
        \begin{cases}
    0 & \text{if\;} a\in Q_{\frac{3}{2}L}\setminus (\partial^{-}S_{1}\cup \partial^{+}S_{1}),\\
    \sum_{|a'-a|=1,a'\in \partial^{+}S_{1}} u(a') & \text{if\;} a\in \partial^{-}S_{1},\\
    -\sum_{|a'-a|=1,a'\in \partial^{-}S_{1}} u(a')  &\text{if\;} a\in \partial^{+}S_{1}.
    \end{cases}
\end{equation}
By \eqref{eq:app-smallness-bound}, we have
\begin{equation}\label{eq:app-weyl}
    \|(H_{Q_{\frac{3}{2}L}}-\lambda') u_{*}\|_{\ell^{2}(Q_{\frac{3}{2}L})}\leq 6(3L+1)^{\frac{3}{2}}\exp(-\frac{1}{10}\lambda_{*} L_{0})\leq \exp(-\epsilon'L_{0})\|u_{*}\|_{\ell^{2}(Q_{\frac{3}{2}L})}
\end{equation}
for large enough $L$. Here, we used $\|u_{*}\|_{\ell^{2}(Q_{\frac{3}{2}L})}\geq 1$ since $\mathbf{0}\in S_{1}$ and $u(\mathbf{0})=1$. By expanding $u_{*}$ into a linear combination of eigenvectors of $H_{Q_{\frac{3}{2}L}}$, \eqref{eq:app-weyl} guarantees that there is an eigenvalue $\lambda_{0}$ of $H_{Q_{\frac{3}{2}L}}$ such that $|\lambda'-\lambda_{0}|\leq \exp(-\epsilon'L_{0})$. Our claim follows.
\end{proof}
Denote $\lambda^{(h)}=h\exp(-L_{0})$ for $h\in \Z_{+}$ and let
\begin{equation}
    \mathcal{E}^{0}_{trap}=\bigcap_{\lambda^{(h)}\in \mathcal{I}} (\mathcal{E}^{\lambda^{(h)}}_{per})^{c}.
\end{equation}
Then by \eqref{eq:app-prob-per},
\begin{equation}\label{eq:app-prob-trap-0}
    \prob[\mathcal{E}^{0}_{trap}]\geq 1-\lambda_{*}\exp(L_{0})L^{-c' L_{0}}_{0}\geq 1-L^{-10}
\end{equation}
for large $L$.
\begin{cla}\label{cla:app-trap-0}
Under the event $\mathcal{E}^{0}_{trap}$, any $\lambda\in [0, \lambda_{*}]\cap \sigma_{k}(H)$ satisfies 
\begin{equation}
    \dist(\lambda,\sigma(H_{Q_{\frac{3}{2}L}}))\leq \exp(-\epsilon' L_{0}).
\end{equation}
\end{cla}
\begin{proof}
For any $\lambda\in [0, \lambda_{*}]$, there exists an $h\in \Z_{+}$ such that $\lambda^{(h)}\in \mathcal{I}$ and $|\lambda-\lambda^{(h)}|\leq \exp(-L_{0}^{1-\frac{\lambda_{*}}{2}})$. Our claim follows from Claim \ref{cla:app-eigen-0}.
\end{proof}

Let $q$ be the smallest positive integer such that $2^{\frac{1}{q}}-1<\frac{\lambda_{*}}{2}$ and let $\tau=2^{\frac{1}{q}}-1$.
Define $\tL_{1}=L_{0}^{1+\tau}$ and $\tL_{i+1}=\tL_{i}^{1+\tau}$ for $i=1,2,\cdots,q-1$.
Then $L\leq \tL_{q}=L_{0}^{2}\leq 2L$.
Let $L_i$ be the (unique) dyadic scale such that $L_i\in [\tL_i, 2\tL_i)$ for each $i=1,\cdots,q$. Let $M_{i}=\frac{3}{2}L+C'\sum_{1\leq j\leq i}L_{j}$ for each $i=1,\cdots,q$ and $M_{0}=\frac{3}{2}L$. Here $C'$ is a large constant to be determined. Then
\begin{equation}\label{eq:app-bound-M}
    M_{i}\leq \frac{3}{2}L + 4C' i L\leq \left(\frac{3}{2}+4C' q\right)L
\end{equation}
for each $0\leq i\leq q$.
In addition, we denote $M_{q+1}=2^wL$ where $w$ is the smallest integer with $2^w>3+8C'q$, and let $L_{q+1}=L_q$.

For any $\lambda\in \mathcal{I}$ and any $j\in \{1,\cdots,q+1\}$, denote by $\mathcal{E}^{\lambda,j}_{per}$ the following event: there exists a path of $\lambda$-bad $2L_j$-cubes in $G_{L_{j}}$, say $\overline Q_{1},\cdots,\overline Q_{m}$, such that 
\begin{equation}\label{eq:app-cond-path-M}
\begin{split}
&\overline Q_{i}\subset Q_{M_{j}}\setminus Q_{M_{j-1}}, \; \forall i\in\{1,\cdots,m\},\\
&\overline{Q}_{1}\cap Q_{M_{j-1}+10L_{j}}\not=\emptyset,\\
&\overline{Q}_{m}\cap Q_{M_{j}-10L_{j}}\not=\emptyset.
\end{split}
\end{equation}
Under the event $\mathcal{E}^{\lambda,j}_{per}$, suppose that $\Gamma_{0}=(\overline{Q}_{1},\cdots,\overline{Q}_{m})$ in $G_{L_{j}}$ is such a path with the shortest length. Since $\dist(Q_{M_{j-1}+10L_{j}},\Z^{3}\setminus Q_{M_{j}-10L_{j}})\geq (C'-20) L_{j}$, we have $m\geq \frac{C'}{4}$ when $C'$ is large enough.
By definition of dyadic cubes and that $\Gamma_{0}$ has the shortest length, there are at least $\frac{m}{1000}$ disjoint $\lambda$-bad cubes in $\Gamma_{0}$. Hence,
\begin{equation}\label{eq:app-prob-per-j}
    \prob[\mathcal{E}^{\lambda,j}_{per}] \leq \sum_{m\geq \frac{C'}{4}}
    C (C' L)^{3} 1000^{m}
    (L_{j}^{-\kappa_{0}})^{\frac{m}{1000}}
    \leq 2C (C' L)^{3} (1000 L_{j}^{-\frac{\kappa_{0}}{1000}})^{\frac{C'}{4}}\leq L^{-10}.
\end{equation}
Here the first inequality is by \eqref{eq:app-prob-control} and counting the number of paths in $G_{L_j}$ with length $m$ and one end intersecting $Q_{M_{j-1}+10L_j}$,
and the last inequality is by taking $C'$ large enough.

By adapting the proof of Claim \ref{cla:app-eigen-0} we can get the following result.
\begin{cla}\label{cla:app-eigen-j}
Under the event $(\mathcal{E}^{\lambda,j}_{per})^{c}$, any $\lambda'\in \sigma_{k}(H)$ with $|\lambda'-\lambda|\leq \exp(- L_{j}^{1-\frac{\lambda_{*}}{2}})$ satisfies $\dist(\lambda',\sigma(H_{Q_{M_{j}}} ))\leq \exp(-\epsilon'' L_{j})$ for a universal constant $\epsilon''>0$.
\end{cla}
Note that, given $\lambda\in \mathcal{I}$, the event $\mathcal{E}^{\lambda,j}_{per}$ is $V_{Q_{M_{j}}\setminus Q_{M_{j-1}}}$-measurable. Hence, the event $\mathcal{E}^{j}_{trap}:=\left(\bigcup_{\lambda\in \sigma(H_{Q_{M_{j-1}}})\cap \mathcal{I}} \mathcal{E}^{\lambda,j}_{per}\right)^{c}$ satisfies
\begin{equation}\label{eq:app-prob-trap-j}
    \prob[\mathcal{E}^{j}_{trap}| V_{Q_{M_{j-1}}}]\geq 1-(M_{j-1}+1)^{3}L^{-10} \geq 1-L^{-6}
\end{equation}
by \eqref{eq:app-bound-M} and \eqref{eq:app-prob-per-j} for large enough $L$. For each $0\leq j\leq q+1$, $\mathcal{E}^{j}_{trap}$ is $V_{Q_{M_{j}}}$-measurable, thus the event $\mathcal{E}_{trap}:=\bigcap_{0\leq j\leq q+1}\mathcal{E}^{j}_{trap}$ is $V_{Q_{M_{q+1}}}$-measurable. By \eqref{eq:app-prob-trap-0} and \eqref{eq:app-prob-trap-j}, we have
\begin{equation}\label{eq:app-prob-trap}
    \prob[\mathcal{E}_{trap}]\geq 1- (q+2) L^{-6}\geq 1-L^{-5}.
\end{equation}
\begin{cla}\label{cla:app-eigen-bootstrap}
Under the event $\mathcal{E}_{trap}$, any $\lambda\in [\exp(-\epsilon''' L_{0}/2),\lambda_{*}-\exp(-\epsilon''' L_{0}/2)]\cap \sigma_{k}(H)$ satisfies 
\begin{equation}
    \dist(\lambda,\sigma(H_{Q_{M_{q+1}}}))\leq \exp(-\epsilon''' L)
\end{equation}
for some $\epsilon'''>0$. 
\end{cla}
\begin{proof}
Let $\epsilon'''=\min\{\epsilon',\epsilon''\}$.
Let $\lambda\in [\exp(-\epsilon''' L_{0}/2),\lambda_{*}-\exp(-\epsilon''' L_{0}/2)]\cap \sigma_{k}(H)$. 
We inductively prove that, $\dist(\lambda,\sigma(H_{Q_{L_{j}}}))\leq \exp(-\epsilon''' L_{j})$ for any $0\leq j\leq q+1$.
Thus, in particular, we have 
\begin{equation}
    \dist(\lambda,\sigma(H_{Q_{M_{q+1}}}))\leq \exp(-\epsilon''' L_{q+1})\leq \exp(-\epsilon''' L),
\end{equation}
and the claim follows.

For the case $j=0$, by Claim \ref{cla:app-trap-0}, $\dist(\lambda,\sigma(H_{Q_{M_{0}}}))\leq \exp(-\epsilon' L_{0})$. Assume the conclusion holds for some $j<q+1$, then $|\lambda-\lambda_{0}|\leq \exp(-\epsilon''' L_{j})$ for some $\lambda_{0}\in \sigma(H_{Q_{M_{j}}})$.
As $\lambda \in [\exp(-\epsilon''' L_{0}/2),\lambda_{*}-\exp(-\epsilon''' L_{0}/2)]$, we must have $\lambda_0 \in \mathcal{I}$.
Since $\tau<\frac{\lambda_{*}}{2}$, for $L$ large enough we have
$\epsilon'''L_{j} > L_{j+1}^{1-\frac{\lambda_{*}}{2}}$
and $|\lambda-\lambda_{0}|\leq \exp(- L_{j+1}^{1-\frac{\lambda_{*}}{2}})$. Thus Claim \ref{cla:app-eigen-j} implies $\dist(\lambda,\sigma(H_{Q_{M_{j+1}}}))\leq \exp(-\epsilon'' L_{j+1})$.
\end{proof}
Finally, since $M_{q+1}=2^wL$ and $w$ is a constant, the proposition follows from Claim \ref{cla:app-eigen-bootstrap} and \eqref{eq:app-prob-trap}.
\end{proof}

\subsection{The second spectral reduction}

For any positive integers $L''>L'$, we denote the annulus $A_{L'',L'}=Q_{L''}\setminus Q_{L'}$. 
Take any $\delta>0$. For $\lambda\in \mathcal{I}$ and $L''>2L'$, let $\mathcal{E}_{L'',L'}^{(\lambda)}$ denote the following event: there exists a subset $G^{(\lambda)}_{L'',L'}\subset A_{L'',L'}$ with $|G^{(\lambda)}_{L'',L'}|\leq (L')^{\frac{\delta}{2}}$ such that, for any $a\in A_{L'',2L'}\setminus G^{(\lambda)}_{L'',L'}$, there is a $\lambda$-good cube $Q_{L'''}(b)\subset A_{L'',L'}$ such that  $\dist(a,Q_{L''}\setminus Q_{L'''}(b))\geq \frac{1}{8}L'''$, and $(L')^{\frac{\delta}{10}}\leq L'''\leq L'$. Note that, $\mathcal{E}^{(\lambda)}_{L'',L'}$ is $V_{A_{L'',L'}}$-measurable.

\begin{lemma}\label{lemma:app-supp-annu}
Let $\varepsilon, \delta>0$ be small enough.
Suppose $L',L''$ are dyadic, satisfying $(L')^{1+\frac{1}{2}\varepsilon}<L''<(L')^{1+\varepsilon}$,
and $L'$ is large enough (depending on $\varepsilon, \delta$).
Then for any $\lambda\in\mathcal{I}$ we have $\prob[\mathcal{E}_{L'',L'}^{(\lambda)}]\geq 1-(L')^{-10}$.
\end{lemma}
\begin{proof}
Let $\tL^{(0)}=L'$, $\tL^{(i+1)}=(\tL^{(i)})^{1-\varepsilon}$,
and $L^{(i)}$ be the (unique) dyadic scale with $L^{(i)}\in[\tL^{(i)}, 2\tL^{(i)})$,
for $i\in \Z_{\geq 0}$. Let $M'\in \Z_{+}$ such that $\frac{1}{10}\delta<(1-\varepsilon)^{M'}<\frac{1}{6}\delta$. For any dyadic $2L^{(M')}$-cube $Q\subset A_{L'',L'}$, we call it hereditary bad if there are $\lambda$-bad dyadic cubes $Q^{(0)},\cdots,Q^{(M')}=Q$ such that, $Q^{(i+1)}\subset Q^{(i)}\subset A_{L'',L'}$ for each $0\leq i\leq M'-1$ and $Q^{(i)}$ is a dyadic $2L^{(i)}$-cube. By \eqref{eq:app-prob-control}, and the same arguments in the proof of Claim \ref{cla:readyprob}, the following is true. For small enough $\varepsilon$, there exists $N\in \Z_{+}$ depending on $\varepsilon,\delta$, such that with probability at least $1-(L')^{-10}$,
\begin{equation}\label{eq:app-hereditary}
    |\{Q\subset A_{L'',L'}:\text{$Q$ is a hereditary bad $2L^{(M')}$-cube}\}|<N.
\end{equation}
Let $G^{(\lambda)}_{L'',L'}=\bigcup\{Q\subset A_{L'',L'}:\text{$Q$ is a hereditary bad $2L^{(M')}$-cube}\}$. Then \eqref{eq:app-hereditary} implies $|G^{(\lambda)}_{L'',L'}|\leq N (2L^{(M')}+1)^{3}\leq (L')^{\frac{\delta}{2}}$ for large enough $L'$. For each $a\in A_{L'',2L'}\setminus G^{(\lambda)}_{L'',L'}$, there is $0\leq i'\leq M'$ and a $\lambda$-good cube $Q_{L^{(i')}}(b)\subset A_{L'',L'}$ such that $\dist(a,Q_{L''}\setminus Q_{L^{(i')}}(b))\geq \frac{1}{8}L^{(i')}$. Since $(L')^{\frac{\delta}{10}}\leq L^{(i')}\leq L'$, our claim follows.
\end{proof}
For any large enough dyadic scales $L',L''$ with $(L')^{1+\frac{1}{2}\varepsilon}<L''<(L')^{1+\varepsilon}$, we denote $\mathcal{E}_{L'',L'}^{supp}=\bigcap_{\lambda\in \sigma(H_{Q_{L'}})\cap \mathcal{I}}\mathcal{E}_{L'',L'}^{(\lambda)}$. Then by Lemma \ref{lemma:app-supp-annu}, as each $\mathcal{E}_{L'',L'}^{(\lambda)}$ is $V_{A_{L'',L'}}$-measurable, we have
\begin{equation}\label{eq:app-prob-supp-energy}
    \prob[\mathcal{E}_{L'',L'}^{supp}]\geq 1-(L')^{-6}.
\end{equation}

\begin{proof}[Proof of Proposition \ref{prop:second-reduction}]
In this proof we let $\varepsilon>0$ be a small universal constant, and $\delta>0$ be a number depending on $\delta'$. Both of them are to be determined.

Now we fix dyadic scale $L$ large enough (depending on $\epsilon,\delta$ and thus depending on $\delta'$). Let $\tL_{0}=L$, $\tL_{i+1}=\tL^{1-\frac{3}{4}\varepsilon}_{i}$,
and $L_{i}$ be the (unique) dyadic scale with $L_{i}\in[\tL_{i}, 2\tL_{i})$,
for $i\in \Z_{\geq 0}$. Pick $M\in \Z_{+}$ such that $\frac{1}{10}\delta<(1-\frac{3}{4}\varepsilon)^{M}< \frac{1}{6}\delta$. Write $\overline{L_{i}}=\frac{1}{16} L_{i}$ for $0\leq i\leq M$ and let
\begin{equation}
    \mathcal{E}^{supp}=\bigcap_{0\leq i\leq M-1} \mathcal{E}^{supp}_{L_{i},\overline{L_{i+1}}}.
\end{equation}
Then by \eqref{eq:app-prob-supp-energy},
\begin{equation}\label{eq:app-prob-supp-fin}
    \prob[\mathcal{E}^{supp}]\geq 1-M \left(\frac{L_{M}}{16}\right)^{-6}\geq 1-L^{-\frac{\delta}{2}}
\end{equation}
as $L$ is large enough. For $0\leq i\leq M$, denote by $\Theta_{i}$ the set of eigenvalues $\lambda \in \sigma(H_{Q_{L_{i}}})$ such that,
\begin{equation}
\lambda \in  [(M-i+1)\exp(-L^{\frac{\delta}{20}}),\lambda_{*}-(M-i+1)\exp(-L^{\frac{\delta}{20}})],    
\end{equation}
and
\begin{equation}
    \dist(\lambda,\sigma(H_{Q_{\overline{L_{j}}}})),\dist(\lambda,\sigma(H_{Q_{L_{j}}})) \leq 2^{i}\exp(-c' L_{j}) \quad \forall j\in \{i,i+1,\cdots,M\}. 
\end{equation}
Here the constant $c'=\frac{c_1}{20}$ where $c_1$ is the constant from Proposition \ref{prop:first-reduction}.

\begin{cla}\label{cla:app-supp-i}
Under the event $\mathcal{E}^{supp}$, for any $1\leq i\leq M$ and $\lambda\in \Theta_{i}$, there exists $G^{(i-1)}\subset Q_{L_{i-1}}$ with $10\leq|G^{(i-1)}|\leq L^{\frac{2}{3}\delta}$ such that the following holds. 
For any $\lambda'\in \sigma(H_{Q_{L_{i-1}}})$ and $u\in \ell^{2}(Q_{L_{i-1}})$ with $|\lambda-\lambda'|\leq 2^{i-1}\exp(-c' L_{i})$ and $H_{Q_{L_{i-1}}} u=\lambda' u$, we have $\|u\|_{\ell^{2}(G^{(i-1)})}\geq (1-|G^{(i-1)}|^{-2})\|u\|_{\ell^{2}(Q_{L_{i-1}})}$.
\end{cla}
\begin{proof}
Since $\lambda\in \Theta_{i}$, there are $\lambda^{(j)}\in \sigma(H_{Q_{\overline{L_{j}}}})$ such that $|\lambda-\lambda^{(j)}|\leq 2^{i}\exp(-c' L_{j})$ for each $i\leq j\leq M$. 
Let
\begin{equation}
    G^{(i-1)}_{*}=\bigcup_{i-1\leq j\leq M-1} G^{(\lambda^{(j+1)})}_{L_{j},\overline{L_{j+1}}}.
\end{equation}
Then $|G^{(i-1)}_{*}|\leq M L^{\frac{\delta}{2}}$. Suppose $\lambda'$ and $u$ satisfy the hypothesis. Then 
\begin{equation}
    |\lambda'-\lambda^{(j)}|\leq |\lambda'-\lambda|+|\lambda-\lambda^{(j)}|\leq 2^{i-1}\exp(-c' L_{i})+2^{i}\exp(-c' L_{j})\leq 2^{M+1}\exp(-c' L_{j})
\end{equation}
for each $i\leq j\leq M$. Denote $L'_{j}=\frac{1}{2}L_{j}$ for each $i\leq j\leq M-1$ and $L'_{i-1}=L_{i-1}$. Pick an arbitrary $a\in Q_{L_{i-1}}\setminus Q_{L_{M}}$, there exists $j'\in \{i-1,\cdots,M-1\}$ such that $a\in A_{L'_{j'},2\overline{L_{j'+1}}}$. If $a\not \in G^{(i-1)}_{*}$, by definition of $G^{\lambda^{(j'+1)}}_{L_{j'},\overline{L_{j'+1}}}$, there exists a $\lambda^{(j'+1)}$-good cube $Q_{L'''}(b)$ such that $\overline{L_{j'+1}}\geq L'''\geq \overline{L_{j'+1}}^{\frac{\delta}{10}}\geq L^{\frac{\delta^{2}}{100}}$, and 
$\dist(a,Q_{L_{j'}}\setminus Q_{L'''}(b))\geq \frac{1}{8}L'''$.
Then since $a\in Q_{L'_{j'}}$, we have
$\dist(a,Q_{L_{i-1}}\setminus Q_{L'''}(b))\geq \frac{1}{8}L'''$. We also have that
\begin{equation}
    |\lambda'-\lambda^{(j'+1)}|\leq 2^{M+1}\exp(-c' L_{j'+1})\leq 2^{M+1}\exp(-16 c' L''').
\end{equation}
Then by Claim \ref{lem:app-pertu} we have,
\begin{equation}
    |u(a)|\leq 2\exp\left((L''')^{1-\lambda_{*}}-\frac{1}{8}\lambda_{*}L'''\right) \|u\|_{\ell^{1}(Q_{L_{i-1}})}\leq L^{-10} \|u\|_{\ell^{2}(Q_{L_{i-1}})}.
\end{equation}
Hence, by letting $G^{(i-1)}=G^{(i-1)}_{*}\cup Q_{L_{M}}$, we have $10\leq |G^{(i-1)}|\leq |G^{(i-1)}_{*}|+ |Q_{L_{M}}|\leq ML^{\frac{\delta}{2}}+100L^{\frac{\delta}{2}}\leq L^{\frac{2}{3}\delta}$,
and
\begin{equation}
    \|u\|_{\ell^{2}(G^{(i-1)})}\geq \big(1- (2L_{i-1}+1)^{3} L^{-20}\big)^{\frac{1}{2}} \|u\|_{\ell^{2}(Q_{L_{i-1}})} \geq (1-|G^{(i-1)}|^{-2}) \|u\|_{\ell^{2}(Q_{L_{i-1}})}.
\end{equation}
Thus our claim follows.
\end{proof}
\begin{cla}\label{cla:app-num-eig}
Under the event $\mathcal{E}^{supp}$, for any $1\leq i\leq M$ and $\lambda\in \Theta_{i}$, we have
\begin{equation}
    |\{\lambda'\in \sigma(H_{Q_{L_{i-1}}}):|\lambda-\lambda'|\leq 2^{i-1}\exp(-c' L_{i})\}|\leq  2L^{\frac{2}{3}\delta}.
\end{equation}
\end{cla}
\begin{proof}
Let $\lambda_{1},\cdots,\lambda_{p} \in \sigma(H_{Q_{L_{i-1}}})$ be all the eigenvalues (counting with multiplicity) in the interval
\begin{equation}
    [\lambda-2^{i-1}\exp(-c' L_{i}),\lambda+2^{i-1}\exp(-c' L_{i})]. 
\end{equation}
Let $u_{1},\cdots,u_{p}$ be the corresponding (mutually orthogonal) eigenvectors with $H_{Q_{L_{i-1}}}u_{s}=\lambda_{s} u_{s}$ and $\|u_{s}\|_{\ell^{2}(Q_{L_{i-1}})}=1$ for $1\leq s\leq p$. By Claim \ref{cla:app-supp-i}, $\|u_{s}\|_{\ell^{2}(G^{(i-1)})}\geq 1-|G^{(i-1)}|^{-2}$ for $1\leq s\leq p$. Thus we have
\begin{equation}
    |\langle u_{s_{1}},u_{s_{2}} \rangle_{\ell^{2}(G^{(i-1)})}-\mathds{1}_{s_{1}=s_{2}}|\leq 2|G^{(i-1)}|^{-2}
\end{equation}
for $1\leq s_{1},s_{2}\leq p$. By Lemma \ref{lem:app-almost-orth}, we have $p\leq 2|G^{(i-1)}|\leq 2L^{\frac{2}{3}\delta}$.
\end{proof}
\begin{cla}
We have $|\Theta_{0}|\leq L^{M \delta}$ under the event $\mathcal{E}^{supp}$.
\end{cla}
\begin{proof}
Suppose $\mathcal{E}^{supp}$ holds. For each $1\leq i\leq M$ and $\lambda\in \Theta_{i-1}$, there are $\lambda^{(j)}\in \sigma(H_{Q_{L_{j}}})$ and $\overline{\lambda^{(j)}}\in \sigma(H_{Q_{\overline{L_{j}}}})$ with $|\lambda-\lambda^{(j)}|,  |\lambda-\overline{\lambda^{(j)}}|\leq 2^{i-1}\exp(-c' L_{j})$, for $i\leq j\leq M$. In particular, $|\lambda-\lambda^{(i)}|\leq 2^{i-1}\exp(-c' L_{i})$. Thus $|\lambda^{(i)}-\lambda^{(j)}|\leq 2^{i-1}(\exp(-c' L_{j})+\exp(-c' L_{i}))\leq 2^{i}\exp(-c' L_{j})$ and similarly $|\lambda^{(i)}-\overline{\lambda^{(j)}}|\leq 2^{i}\exp(-c' L_{j})$ for $i\leq j\leq M$. Moreover, $\lambda\in \Theta_{i-1}$ implies that $\lambda \in [(M-i+2)\exp(-L^{\frac{\delta}{20}}),\lambda_{*}-(M-i+2)\exp(-L^{\frac{\delta}{20}})]$ and thus
\begin{equation}
    \lambda^{(i)}\in [(M-i+1)\exp(-L^{\frac{\delta}{20}}),\lambda_{*}-(M-i+1)\exp(-L^{\frac{\delta}{20}})].
\end{equation}
These imply $\lambda^{(i)}\in \Theta_{i}$. Hence, we have
\begin{equation}
    \Theta_{i-1}\subset \{\lambda\in \sigma(H_{Q_{L_{i-1}}}):\dist(\lambda,\Theta_{i})\leq 2^{i-1}\exp(-c' L_{i})\}.
\end{equation}
Together with Claim \ref{cla:app-num-eig}, we have $|\Theta_{i-1}|\leq 2L^{\frac{2}{3}\delta} |\Theta_{i}|$ for $1\leq i \leq M$. Since $|\Theta_{M}|\leq |\sigma(H_{Q_{L_{M}}})|\leq 10L_{M}^{3}\leq 100  L^{\frac{\delta}{2}}$, we have $|\Theta_{0}|\leq 100 L^{\frac{\delta}{2}}\cdot 2^{M} L^{\frac{2}{3} M \delta} \leq L^{M\delta}$.  
\end{proof}
Now we denote $\mathcal{E}_{sloc}^{(L)}=\mathcal{E}^{supp}\cap\bigcap_{0\leq i\leq M} \mathcal{E}^{(L_{i})}_{wloc}\cap \mathcal{E}_{wloc}^{(\overline{L_{i}})}$.
By Proposition \ref{prop:first-reduction} and \eqref{eq:app-prob-supp-fin}, 
\begin{equation}
    \prob[ \mathcal{E}_{sloc}^{(L)}]\geq 1- L^{-\frac{\delta}{2}} -2(M+1) L^{-\frac{1}{20}\kappa'\delta}\geq 1-L^{-\kappa''}
\end{equation}
for some small $\kappa''>0$ depending on $\delta,M$.
Take $c_2=\min\{\frac{\delta}{30}, c'\}$.
Under the event $\mathcal{E}_{sloc}^{(L)}$, for any 
\begin{equation}\label{eq:app-lamb-interval}
    \lambda\in \sigma_{k}(H)\cap [\exp(-L^{c_2}),\lambda_{*}-\exp(-L^{c_2})],
\end{equation}
we claim that
\begin{equation}\label{eq:app-close-theta-0}
    \dist(\lambda,\Theta_{0})\leq \exp(-c_2 L). 
\end{equation}
To see this, by definition of $\mathcal{E}^{(L_{i})}_{wloc}$ and $ \mathcal{E}_{wloc}^{(\overline{L_{i}})}$, \eqref{eq:app-lamb-interval} implies $\dist(\lambda,\sigma(H_{Q_{L_{i}}}))\leq \exp(-c_1 L_{i})$ and $\dist(\lambda,\sigma(H_{Q_{\overline{L_{i}}}})) \leq \exp(-\frac{c_1}{16} L_{i})$ for each $0\leq i\leq M$. In particular, there is $\lambda_{0}\in \sigma(H_{Q_{L}})$ such that $|\lambda-\lambda_{0}|\leq \exp(-c_1 L)$.
Since $c'=\frac{c_1}{20}$, we have $\lambda_{0}\in \left[(M+1)\exp(-L^{\frac{\delta}{20}}),\lambda_{*}-(M+1)\exp(-L^{\frac{\delta}{20}})\right]$ by \eqref{eq:app-lamb-interval},
and also
\begin{equation}
\begin{split}
    &\dist(\lambda_{0},\sigma(H_{Q_{L_{i}}}))\leq |\lambda-\lambda_{0}|+\dist(\lambda,\sigma(H_{Q_{L_{i}}}))\leq  \exp(-c_1 L)+\exp(-c_1 L_{i})\leq \exp(-c' L_{i}),\\
    &\dist(\lambda_{0},\sigma(H_{Q_{\overline{L_{i}}}}))\leq |\lambda-\lambda_{0}|+\dist(\lambda,\sigma(H_{Q_{\overline{L_{i}}}}))\leq  \exp(-c_1 L)+\exp(-\frac{c_1}{16} L_{i})\leq \exp(-c' L_{i}),
\end{split}
\end{equation}
for $0\leq i\leq M$. Hence $\lambda_{0}\in \Theta_{0}$ and \eqref{eq:app-close-theta-0} follows.

Finally, observe that $|\Theta_{0}|\leq L^{M \delta}\leq L^{\frac{\log(\frac{1}{10}\delta)}{\log(1-\frac{3}{4}\varepsilon)} \delta}\leq L^{\delta'}$ by taking $\delta$ small enough (depending on $\delta'$), the proposition follows by letting $S=\Theta_{0}$.
\end{proof}
\end{appendices}

\end{document}